\documentclass[10pt]{article}
\usepackage{hyperref}\hypersetup{hidelinks}
\usepackage{mathtools}
\usepackage{amsfonts}
\usepackage{amssymb}
\usepackage{amsthm}

\usepackage[shortlabels]{enumitem}
\usepackage{microtype}
\usepackage{fullpage}
\usepackage{mathrsfs}
\usepackage{mleftright}
\mleftright

\usepackage{xparse}

\DeclareMathOperator{\Real}{Re}
\DeclareMathOperator{\Imaginary}{Im}
\DeclareMathOperator{\supp}{supp}

\NewDocumentCommand{\transpose}{}{\top}

\NewDocumentCommand{\R}{}{\mathbb{R}}
\NewDocumentCommand{\C}{}{\mathbb{C}}

\NewDocumentCommand{\Z}{}{\mathbb{Z}}
\NewDocumentCommand{\M}{}{\mathbb{M}}
\NewDocumentCommand{\EmptySet}{}{\varnothing}
\NewDocumentCommand{\Zgeq}{}{\Z_{\geq}}
\NewDocumentCommand{\Zg}{}{\Z_{>}}
\NewDocumentCommand{\Rn}{}{\R^n}
\NewDocumentCommand{\Rngeq}{}{\Rn_{\geq}}
\NewDocumentCommand{\Rng}{}{\Rn_{>}}
\NewDocumentCommand{\Rnmo}{}{\R^{n-1}}
\NewDocumentCommand{\psit}{}{\tilde{\psi}}
\NewDocumentCommand{\phih}{}{\hat{\phi}}
\NewDocumentCommand{\vt}{}{\tilde{v}}
\NewDocumentCommand{\at}{}{\tilde{a}}
\NewDocumentCommand{\ah}{}{\hat{a}}
\NewDocumentCommand{\uh}{}{\hat{u}}

\NewDocumentCommand{\Wt}{}{\widetilde{W}}
\NewDocumentCommand{\Wh}{}{\widehat{W}}
\NewDocumentCommand{\Omegat}{}{\widetilde{\Omega}}
\NewDocumentCommand{\Tt}{}{\widetilde{T}}

\NewDocumentCommand{\Kt}{}{\widetilde{K}}
\NewDocumentCommand{\At}{}{\widetilde{A}}
\NewDocumentCommand{\sB}{}{\mathcal{B}}

\NewDocumentCommand{\JSet}{}{J}
\NewDocumentCommand{\JComplement}{}{J^c}

\NewDocumentCommand{\wtEst}{}{\tilde{w}}

\NewDocumentCommand{\Compact}{}{\mathcal{K}}

\NewDocumentCommand{\Omegageq}{}{\Omega_{\geq}}

\NewDocumentCommand{\Ropn}{}{\R^{1+n}}
\NewDocumentCommand{\Ropngeq}{}{\R^{1+n}_{\geq}}

\NewDocumentCommand{\Cubenmo}{m}{Q^{n-1}(#1)}
\NewDocumentCommand{\CubenmoOne}{}{\Cubenmo{1}}
\NewDocumentCommand{\CubenmoOneHalf}{}{\Cubenmo{1/2}}

\NewDocumentCommand{\DSet}{}{\mathcal{D}}
\NewDocumentCommand{\DSetL}{m}{\mathcal{D}_{#1}}

\NewDocumentCommand{\Cuben}{m}{Q^n(#1)}
\NewDocumentCommand{\CubenOne}{}{\Cuben{1}}
\NewDocumentCommand{\Cubengeq}{m}{Q_{\geq}^n(#1)}
\NewDocumentCommand{\Cubeng}{m}{Q_{>}^n(#1)}
\NewDocumentCommand{\CubengeqOne}{}{\Cubengeq{1}}

\NewDocumentCommand{\CubenOneHalf}{}{\Cuben{1/2}}

\NewDocumentCommand{\CubengeqOneHalf}{}{\Cubengeq{1/2}}
\NewDocumentCommand{\CubengeqThreeFourths}{}{\Cubengeq{3/4}}
\NewDocumentCommand{\CubengOneHalf}{}{\Cubeng{1/2}}

\NewDocumentCommand{\CubengeqOneHalfClosure}{}{\overline{\Cubengeq{1/2}}}

\NewDocumentCommand{\CubengeqThreeFourthsClosure}{}{\overline{\Cubengeq{3/4}}}
\NewDocumentCommand{\CubengeqSevenEighths}{}{\Cubengeq{7/8}}
\NewDocumentCommand{\CubengeqSevenEighthsClosure}{}{\overline{\Cubengeq{7/8}}}
\NewDocumentCommand{\CubenSevenEighthsClosure}{}{\overline{\Cuben{7/8}}}

\NewDocumentCommand{\CubengeqDeltaOne}{}{\Cubengeq{\delta_1}}
\NewDocumentCommand{\CubengeqDeltaOneClosure}{}{\overline{\Cubengeq{\delta_1}}}
\NewDocumentCommand{\CubengeqDeltaTwoClosure}{}{\overline{\Cubengeq{\delta_2}}}
\NewDocumentCommand{\CubengeqDeltaTwo}{}{\Cubengeq{\delta_2}}
\NewDocumentCommand{\CubenDeltaTwoClosure}{}{\overline{\Cuben{\delta_2}}}

\NewDocumentCommand{\Bn}{m o}{B^n(\IfValueT{#2}{#2,}#1)}
\NewDocumentCommand{\Bnmo}{m o}{B^{n-1}(\IfValueT{#2}{#2,}#1)}

\NewDocumentCommand{\sconst}{}{(\mathrm{s.c.})}
\NewDocumentCommand{\lconst}{}{(\mathrm{l.c.})}

\NewDocumentCommand{\MatrixSpace}{o o o}{\mathbb{M}^{\IfNoValueTF{#1}{N}{#1} \times \IfNoValueTF{#2}{N}{#2} }(\IfNoValueTF{#3}{\C}{#3})}

\NewDocumentCommand{\MatrixNorm}{s m o o o}{\IfBooleanT{#1}{\left}\| #2 \IfBooleanT{#1}{\right}\|_{\MatrixSpace[#3][#4][#5]}}

\NewDocumentCommand{\GL}{m o}{\mathrm{GL}_{#1}\IfValueT{#2}{(#2)}}

\NewDocumentCommand{\CNip}{s m m o}{\IfBooleanT{#1}{\left}\langle #2, #3 \IfBooleanT{#1}{\right}\rangle_{\C^{\IfNoValueTF{#4}{N}{#4}}}}

\NewDocumentCommand{\ManifoldN}{}{\mathfrak{N}}
\NewDocumentCommand{\BoundaryN}{}{\partial \ManifoldN}
\DeclareMathOperator{\Int}{Int}
\NewDocumentCommand{\InteriorN}{}{\Int(\ManifoldN)}

\NewDocumentCommand{\TangentSpace}{m m}{T_{#1}#2}

\NewDocumentCommand{\CinftySpace}{s o o}{C^\infty\IfValueT{#2}{\IfBooleanT{#1}{\left}  (#2 \IfValueT{#3}{;#3}  \IfBooleanT{#1}{\right})}}
\NewDocumentCommand{\CinftycptSpace}{s o o}{C^\infty_{\mathrm{cpt}}\IfValueT{#2}{\IfBooleanT{#1}{\left}(#2 \IfValueT{#3}{;#3}  \IfBooleanT{#1}{\right})}}
\NewDocumentCommand{\CinftycpttSpace}{s o o}{C^\infty_{\mathrm{cpt},t}\IfValueT{#2}{\IfBooleanT{#1}{\left}(#2 \IfValueT{#3}{;#3}  \IfBooleanT{#1}{\right})}}

\NewDocumentCommand{\LtSpace}{s o o o}{L^2\IfNoValueTF{#2}{}{\IfBooleanT{#1}{\left}(#2\IfValueT{#3}{,#3}\IfValueT{#4}{;#4} \IfBooleanT{#1}{\right})}}

\NewDocumentCommand{\LtNorm}{s m o o o}{\IfBooleanTF{#1}{\left\|#2\right\|}{\|#2\|}_{\LtSpace[#3][#4][#5]}}
\NewDocumentCommand{\LtSpaceNoMeasure}{s o o}{L^2\IfNoValueTF{#2}{}{\IfBooleanT{#1}{\left}(#2\IfValueT{#3}{;#3} \IfBooleanT{#1}{\right})}}
\NewDocumentCommand{\LtSpaceNoMeasurecpts}{s o o}{L_{\mathrm{cpt},s}^2\IfNoValueTF{#2}{}{\IfBooleanT{#1}{\left}(#2\IfValueT{#3}{;#3} \IfBooleanT{#1}{\right})}}

\NewDocumentCommand{\LtlocSpace}{s o}{L^2_{\mathrm{loc}}\IfValueT{#2}{  \IfBooleanT{#1}{\left}(#2\IfBooleanT{#1}{\right})  }}

\NewDocumentCommand{\LtxnSpace}{o o}{L^2_{x_n}\IfValueT{#1}{(#1\IfValueT{#2}{;#2})}}

\NewDocumentCommand{\VlsSpace}{m m o}{V^{#1,#2}\IfValueT{#3}{(#3)}}
\NewDocumentCommand{\VlsNorm}{s m m m o}{\IfBooleanT{#1}{\left}\| #2 \IfBooleanT{#1}{\right}\|_{\VlsSpace{#3}{#4}[#5]}}

\NewDocumentCommand{\QNorm}{s m}{\IfBooleanT{#1}{\left}\| #2 \IfBooleanT{#1}{\right}\|_{\DomainQ}}

\NewDocumentCommand{\Formt}{s o o}{\FormtSymbol\IfValueT{#2}{ \IfBooleanTF{#1}{\left( #2 \IfValueT{#3}{,#3} \right)}{(#2\IfValueT{#3}{,#3})}     }}
\NewDocumentCommand{\FormQepsilon}{s o o}{\FormQepsilonSymbol\IfValueT{#2}{ \IfBooleanTF{#1}{\left( #2 \IfValueT{#3}{,#3} \right)}{(#2\IfValueT{#3}{,#3})}     }}
\NewDocumentCommand{\FormtSymbol}{}{\mathfrak{t}}
\NewDocumentCommand{\FormQepsilonSymbol}{}{\mathcal{Q}_{\epsilon}}
\NewDocumentCommand{\FormQSymbol}{}{\mathcal{Q}}
\NewDocumentCommand{\FormQ}{s o o}{\FormQSymbol\IfValueT{#2}{ \IfBooleanTF{#1}{\left( #2 \IfValueT{#3}{,#3} \right)}{(#2\IfValueT{#3}{,#3})}     }}
\NewDocumentCommand{\FormQh}{s o o}{\widehat{\FormQSymbol}\IfValueT{#2}{ \IfBooleanTF{#1}{\left( #2 \IfValueT{#3}{,#3} \right)}{(#2\IfValueT{#3}{,#3})}     }}
\NewDocumentCommand{\FormQF}{s o o}{\FormQSymbol_{\mathrm{F}}\IfValueT{#2}{ \IfBooleanTF{#1}{\left( #2 \IfValueT{#3}{,#3} \right)}{(#2\IfValueT{#3}{,#3})}     }}
\NewDocumentCommand{\FormRpsi}{o o}{\mathcal{R}_{\psi}\IfValueT{#1}{(#1,#2)}}

\NewDocumentCommand{\FormQFt}{s o o}{\widetilde{\FormQSymbol}_{\mathrm{F}}\IfValueT{#2}{ \IfBooleanTF{#1}{\left( #2 \IfValueT{#3}{,#3} \right)}{(#2\IfValueT{#3}{,#3})}     }}
\NewDocumentCommand{\FormQZero}{s o o}{\FormQSymbol_0\IfValueT{#2}{ \IfBooleanTF{#1}{\left( #2 \IfValueT{#3}{,#3} \right)}{(#2\IfValueT{#3}{,#3})}     }}

\NewDocumentCommand{\FormQH}{s o o}{\FormQSymbol^{H}\IfValueT{#2}{ \IfBooleanTF{#1}{\left( #2 \IfValueT{#3}{,#3} \right)}{(#2\IfValueT{#3}{,#3})}     }}
\NewDocumentCommand{\FormQHh}{s o o}{\widehat{\FormQSymbol}^{H}\IfValueT{#2}{ \IfBooleanTF{#1}{\left( #2 \IfValueT{#3}{,#3} \right)}{(#2\IfValueT{#3}{,#3})}     }}

\NewDocumentCommand{\QDomainQ}{}{(\FormQ,\DomainQ)}
\NewDocumentCommand{\tDomaint}{}{(\Formt,\HsSpace{\kappa}[\Gamma])}
\NewDocumentCommand{\QepsilonDomainQepsilon}{}{(\FormQepsilon, \DSet)}

\NewDocumentCommand{\DomainSymbol}{}{\mathrm{Dom}}
\NewDocumentCommand{\Domain}{m}{\DomainSymbol(#1)}
\NewDocumentCommand{\DomainQ}{}{\Domain{\FormQSymbol}}

\NewDocumentCommand{\Ltip}{s m m o o o}{\IfBooleanTF{#1}{\left \langle #2,#3 \right\rangle}{\langle#2,#3 \rangle}_{\LtSpace[#4][#5][#6]}}

\NewDocumentCommand{\CoreB}{}{\mathscr{B}}

\NewDocumentCommand{\opP}{}{\mathscr{P}}
\NewDocumentCommand{\opL}{}{\mathscr{L}}
\NewDocumentCommand{\opLh}{}{\widehat{\opL}}
\NewDocumentCommand{\DomainL}{}{\Domain{\opL}}
\NewDocumentCommand{\DomainLepsilon}{}{\Domain{\opL_{\epsilon}}}
\NewDocumentCommand{\LDomainL}{}{(\opL, \DomainL)}
\NewDocumentCommand{\LepsilonDomainLepsilon}{}{(\opL_\epsilon, \DomainLepsilon)}
\NewDocumentCommand{\DomainLNorm}{s m}{\IfBooleanT{#1}{\left}\|#2\IfBooleanT{#1}{\right}\|_{\DomainL}}

\NewDocumentCommand{\Vol}{o}{\mathrm{Vol}\IfValueT{#1}{(#1)}}
\NewDocumentCommand{\Mult}{m}{\mathrm{Mult}(#1)}

\NewDocumentCommand{\HsSpace}{m o}{H^{#1}\IfValueT{#2}{(#2)}}
\NewDocumentCommand{\HsNorm}{s m m o}{\IfBooleanT{#1}{\left}\| #2  \IfBooleanT{#1}{\right}\|_{\HsSpace{#3}[#4]} }

\NewDocumentCommand{\HstxpSpace}{m o}{H_{t,x'}^{#1}\IfValueT{#2}{(#2)}}
\NewDocumentCommand{\HstxpNorm}{s m m o}{\IfBooleanT{#1}{\left}\| #2  \IfBooleanT{#1}{\right}\|_{\HstxpSpace{#3}[#4]} }

\NewDocumentCommand{\HsWSpace}{s m o}{H_W^{#2}\IfValueT{#3}{\IfBooleanT{#1}{\left}(#3\IfBooleanT{#1}{\right})}}
\NewDocumentCommand{\HsWNorm}{s m m o}{\IfBooleanT{#1}{\left}\| #2  \IfBooleanT{#1}{\right}\|_{\HsWSpace{#3}[#4]} }

\NewDocumentCommand{\HsWhSpace}{s m o}{H_{\Wh}^{#2}\IfValueT{#3}{\IfBooleanT{#1}{\left}(#3\IfBooleanT{#1}{\right})}}
\NewDocumentCommand{\HsWtSpace}{m o}{H_{\Wt}^{#1}\IfValueT{#2}{(#2)}}

\NewDocumentCommand{\HsWlocSpace}{s m o}{H_{W,\mathrm{loc}}^{#2}\IfValueT{#3}{\IfBooleanT{#1}{\left}(#3\IfBooleanT{#1}{\right})}}
\NewDocumentCommand{\HsWtlocSpace}{s m o}{H_{{\Wt},\mathrm{loc}}^{#2}\IfValueT{#3}{\IfBooleanT{#1}{\left}(#3\IfBooleanT{#1}{\right})}}

\NewDocumentCommand{\HsYSpace}{m o}{H_Y^{#1}\IfValueT{#2}{(#2)}}
\NewDocumentCommand{\HsYNorm}{s m m o}{\IfBooleanT{#1}{\left}\| #2  \IfBooleanT{#1}{\right}\|_{\HsYSpace{#3}[#4]} }

\NewDocumentCommand{\HsYgeqSpace}{m o o}{H_{Y,\geq}^{#1}\IfValueT{#2}{(#2\IfValueT{#3}{;#3})}}
\NewDocumentCommand{\HsYgeqNorm}{s m m o}{\IfBooleanT{#1}{\left}\| #2  \IfBooleanT{#1}{\right}\|_{\HsYgeqSpace{#3}[#4]} }

\NewDocumentCommand{\HsWgeqSpace}{m o o}{H_{W,\geq}^{#1}\IfValueT{#2}{(#2\IfValueT{#3}{;#3})}}
\NewDocumentCommand{\HsWgeqNorm}{s m m o}{\IfBooleanT{#1}{\left}\| #2  \IfBooleanT{#1}{\right}\|_{\HsWgeqSpace{#3}[#4]} }

\NewDocumentCommand{\HsWgeqcptSpace}{m o}{H_{W,\geq,\mathrm{cpt}}^{#1}\IfValueT{#2}{(#2)}}

\NewDocumentCommand{\HsYgeqcptSpace}{m o}{H_{Y,\geq,\mathrm{cpt}}^{#1}\IfValueT{#2}{(#2)}}

\NewDocumentCommand{\GsSpace}{m o}{G^{#1}\IfValueT{#2}{(#2)}}
\NewDocumentCommand{\GsNorm}{s m m o}{\IfBooleanT{#1}{\left}\| #2  \IfBooleanT{#1}{\right}\|_{\GsSpace{#3}[#4]} }

\NewDocumentCommand{\HscptSpace}{m o}{H^{#1}_{\mathrm{cpt}}\IfValueT{#2}{(#2)}}

\NewDocumentCommand{\FourierSymbol}{}{\mathscr{F}}
\NewDocumentCommand{\Fouriertxp}{}{\FourierSymbol_{(t,x')\rightarrow (\tau,\xi')}}
\NewDocumentCommand{\Fouriertx}{}{\FourierSymbol_{(t,x)\rightarrow (\tau,\xi)}}

\NewDocumentCommand{\JBracket}{s m}{\IfBooleanT{#1}{\left}\langle #2 \IfBooleanT{#1}{\right}\rangle }

\NewDocumentCommand{\Symbolstx}{m}{\Psi^{#1}_{t,x}}
\NewDocumentCommand{\Symbolstxp}{m}{\Psi^{#1}_{t,x'}}
\NewDocumentCommand{\Lambdatx}{o}{\Lambda_{t,x}\IfValueT{#1}{^{#1}}}
\NewDocumentCommand{\Lambdatxp}{o}{\Lambda_{t,x'}\IfValueT{#1}{^{#1}}}
\NewDocumentCommand{\Lambdax}{o}{\Lambda_{x}\IfValueT{#1}{^{#1}}}

\NewDocumentCommand{\SymbolstxNorm}{s m m m}{\IfBooleanT{#1}{\left}\| #2\IfBooleanT{#1}{\right}\|_{\Symbolstx{#3},#4}}
\NewDocumentCommand{\SymbolstxpNorm}{s m m m}{\IfBooleanT{#1}{\left}\| #2\IfBooleanT{#1}{\right}\|_{\Symbolstxp{#3},#4}}

\NewDocumentCommand{\SchwartzSymbol}{}{\mathscr{S}}
\NewDocumentCommand{\TemperedDistributions}{o}{\SchwartzSymbol'\IfValueT{#1}{(#1)}}
\NewDocumentCommand{\SchwartzSpace}{o}{\SchwartzSymbol\IfValueT{#1}{(#1)}}

\NewDocumentCommand{\TestFunctionsSymbol}{}{\mathscr{D}}
\NewDocumentCommand{\TestFunctionsZero}{s o}{\TestFunctionsSymbol_0\IfValueT{#2}{\IfBooleanTF{#1}{\left(#2\right)}{(#2)}}} 
\NewDocumentCommand{\DistributionsZero}{s o}{\TestFunctionsSymbol_0'\IfValueT{#2}{\IfBooleanTF{#1}{\left(#2\right)}{(#2)}}} 
\NewDocumentCommand{\TestFunctionsZeroCM}{s o}{\TestFunctionsSymbol_0\IfValueT{#2}{\IfBooleanTF{#1}{\left(#2;\C^M\right)}{(#2;\C^M)}}} 
\NewDocumentCommand{\DistributionsZeroCM}{s o}{\TestFunctionsSymbol_0'\IfValueT{#2}{\IfBooleanTF{#1}{\left(#2;\C^M\right)}{(#2;\C^M)}}}
\NewDocumentCommand{\Distributions}{s o o}{\TestFunctionsSymbol'\IfValueT{#2}{\IfBooleanT{#1}{\left}(#2\IfValueT{#3}{;#3}\IfBooleanT{#1}{\right})}}

\NewDocumentCommand{\Stx}{m o}{S_{t,x\IfValueT{#2}{,#2}}^{(#1,-\infty)}}
\NewDocumentCommand{\Stxp}{m o}{S_{t,x'\IfValueT{#2}{,#2}}^{(#1,-\infty)}}

\NewDocumentCommand{\Stxt}{m o}{\widetilde{S}_{t,x\IfValueT{#2}{,#2}}^{(#1,-\infty)}}
\NewDocumentCommand{\Stxpt}{m o}{\widetilde{S}_{t,x'\IfValueT{#2}{,#2}}^{(#1,-\infty)}}

\NewDocumentCommand{\CjSpace}{m o}{C^{#1}\IfValueT{#2}{(#2)}}
\NewDocumentCommand{\CjNorm}{s m m o o}{\IfBooleanT{#1}{\left}\|  #2 \IfBooleanT{#1}{\right}\|_{C^{#3} \IfValueT{#4}{(#4 \IfValueT{#5}{;#5} )} }}

\NewDocumentCommand{\Extension}{}{\mathscr{E}}

\NewDocumentCommand{\FilteredSheaf}{m o o}{%
  #1_{%
    \IfValueTF{#3}{#3}{\bullet}%
  }%
  \IfValueT{#2}{(#2)}%
}

\NewDocumentCommand{\FilteredSheafF}{o o}{\FilteredSheaf{\mathcal{F}}[#1][#2]}
\NewDocumentCommand{\FilteredSheafG}{o o}{\FilteredSheaf{\mathcal{G}}[#1][#2]}
\NewDocumentCommand{\FilteredSheafFh}{o o}{\FilteredSheaf{\widehat{\mathcal{F}}}[#1][#2]}
\NewDocumentCommand{\RestrictFilteredSheaf}{m m o}{%
  #1\big|_{#2}^{\#}\IfValueT{#3}{(#3)}
}

\NewDocumentCommand{\TLSymbol}{}{\mathscr{F}}
\NewDocumentCommand{\TLSpace}{m m m o o}{\TLSymbol^{#1}_{#2,#3}\IfValueT{#4}{(#4\IfValueT{#5}{,#5})}}

\NewDocumentCommand{\floor}{s m}{\IfBooleanT{#1}{\left}\lfloor #2\IfBooleanT{#1}{\right}\rfloor}

\NewDocumentCommand{\LtOpNorm}{s m}{\IfBooleanT{#1}{\left}\|#2\IfBooleanT{#1}{\right}\|_{L^2\rightarrow L^2}}

\NewDocumentCommand{\FSpace}{m}{\mathcal{F}_{#1}}

\NewDocumentCommand{\DomainLN}{m}{\DomainSymbol(\opL^{#1})}

\NewDocumentCommand{\EllipOmega}{}{\Gamma}

\NewDocumentCommand{\DomainQNorm}{m}{\|#1\|_{\DomainQ}}

\begin{document}

\newtheorem{theorem}{Theorem}[section]
\newtheorem{corollary}[theorem]{Corollary}
\newtheorem{proposition}[theorem]{Proposition}
\newtheorem{lemma}[theorem]{Lemma}
\newtheorem{conjecture}[theorem]{Conjecture}
\newtheorem{problem}[theorem]{Problem}

\theoremstyle{remark}
\newtheorem{remark}[theorem]{Remark}

\theoremstyle{definition}
\newtheorem{definition}[theorem]{Definition}

\theoremstyle{definition}
\newtheorem{assumption}[theorem]{Assumption}

\theoremstyle{remark}
\newtheorem{example}[theorem]{Example}

\theoremstyle{definition}
\newtheorem{goal}[theorem]{Goal}

\theoremstyle{remark}
\newtheorem{question}[theorem]{Question}

\numberwithin{equation}{section}

\title{A Priori Estimates for Maximally Subelliptic Quadratic Forms}

\author{Brian Street\footnote{The author was partially supported by National Science Foundation Grant 2153069.}}
\date{}

\maketitle

\begin{abstract}
    We prove a priori subelliptic estimates, near a non-characteristic boundary point, for the heat operators associated to a wide class of maximally subelliptic quadratic forms. This is the third paper in a series devoted to studying general maximally subelliptic boundary value problems.
\end{abstract}

\section{Introduction}\label{Section::Intro}
We establish a priori subelliptic estimates (near a non-characteristic boundary point--see Definition \ref{Defn::Intro::NonChar::NonCharDefn}) 
for the heat operator corresponding to certain
maximally subelliptic quadratic forms. These are quadratic forms defined in terms of H\"ormander vector fields
which satisfy a  maximally subelliptic estimate (see Assumption \ref{Assumption::Intro2::MaxSub}), and can be viewed as maximally subelliptic boundary value problems;
we do not restrict attention to Dirichlet boundary conditions.
This is the third paper in a series
studying general maximally subelliptic boundary value problems.\footnote{The first paper in the series \cite{StreetCarnotCaratheodoryBallsOnManifoldsWithBoundary} studied the adapted Carnot--Carath\'eodory geometry on manifolds with boundary,
while the second paper \cite{StreetFunctionSpacesAndTraceTheoremsForMaximallySubellipticBoundaryValueProblems} studied the adapted function spaces. In forthcoming papers, these ideas will
be combined to provide a general study of maximally subelliptic boundary value problems.}
For applications in the future papers in the series,
it is important that we keep track of what constants depend on in this paper (see Section \ref{Section::Intro::Quant}).

Maximal subellipticity (also known as maximal hypoellipticity) is a far-reaching generalization of ellipticity, which has its roots
in work of H\"ormander \cite{HormanderHypoellipticSecondOrderDifferentialEquations} and Kohn \cite{KohnLecturesOnDegenerateEllipticProblems},
and was first made explicit in the work of Folland and Stein \cite{FollandSteinEstimatesForTheBarPartialBComplex},
Folland \cite{FollandSubellipticEstimatesAndFunctionSpacesOnNilpotentLieGroups},  Rothschild and Stein \cite{RothschildSteinHypoellipticDifferentialOperatorsAndNilpotentGroups},
and Helffer and Nourrigat \cite{HelfferNourrigatHypoellipticiteMaximalePourDesOperateursPolynomesDeChampsDeVecteurs}.
Since these original papers, a vast program followed (covering thousands of papers), generalizing the classical
theory of elliptic PDEs to this more general and more degenerate setting.  
Most of this work has addressed the interior theory:
see
 \cite{BramantiAnInvitationToHypoellipticOperators,BramantiBrandoliniHormanderOperators} for a friendly introduction
 and
\cite{StreetMaximalSubellipticity}
for a more detailed history and an up-to-date account of the interior theory of maximally subelliptic PDEs (covering both the
linear and fully nonlinear theory). 
While there have been some examples of maximally subelliptic boundary value problems studied in the literature
(an early, important work is that of Jerison \cite{JerisonDirichletProblemForTheKohnLaplacianI,JerisonDirichletProblemForTheKohnLaplacianII}), there
is no general theory 
similar to what is known for elliptic boundary value problems. 
This is the third paper in a series devoted to introducing such a theory. 
While the results in this paper are not sharp, they are an important step in obtaining sharp results
(see Section \ref{Section::Intro::Quant}).

The main theorem of this paper is Theorem \ref{Thm::Result::MainThm::New}. 
The easiest way to use the main theorem in practice is via Corollary \ref{Cor::Core::MainCor}.
In this introduction, we describe some further corollaries 
which are easier to understand and which already
contain the main ideas and some of the main applications.
For any smooth manifold with boundary, \(\ManifoldN\), let \(\BoundaryN\) denote the boundary,
\(\InteriorN\) the interior, \(\CinftycptSpace[\ManifoldN]\) the smooth functions with compact support (possibly nonzero
on the boundary), and \(\TestFunctionsZero[\ManifoldN]\) those functions in \(\CinftycptSpace[\ManifoldN]\)
which vanish to infinite order at \(\BoundaryN\). We give \(\CinftycptSpace[\ManifoldN]\)
the usual locally convex topology, and \(\TestFunctionsZero[\ManifoldN]\) the topology as a closed subspace
of \(\CinftycptSpace[\ManifoldN]\). Let \(\DistributionsZero[\ManifoldN]\) be the dual space--this is the
space of distributions we use throughout the paper.

Let \(\ManifoldN\) be a compact,\footnote{The main
 results in this paper
 (Theorem \ref{Thm::Result::MainThm::New} and Corollary \ref{Cor::Core::MainCor})
 are local, and in the main results we do not assume compactness.} 
smooth manifold with boundary and dimension of \(\ManifoldN\) equal to \(n\geq 2\), and 
let \(\Vol\) be a smooth, strictly positive density on \(\ManifoldN\).
Let \(W_1,\ldots, W_r\) be smooth vector fields on \(\ManifoldN\)
satisfying H\"ormander's condition: the Lie algebra generated by \(W_1,\ldots, W_r\)
spans the tangent space at every point (see Definition \ref{Defn::Result::HormandersCondition}).
For a list \(\alpha=(\alpha_1,\alpha_2,\ldots,\alpha_L)\in \{1,\ldots, r\}^L\),
set \(W^{\alpha}=W_{\alpha_1}W_{\alpha_2}\cdots W_{\alpha_L}\) and \(|\alpha|=L\).

Fix \(\kappa\in \Zg=\left\{ 1,2,\ldots \right\}\) and for lists \(|\alpha|,|\beta|\leq \kappa\)
let \(a_{\alpha,\beta}\in \CinftySpace[\ManifoldN]\).
Define a sesqui-linear form
\begin{equation}\label{Eqn::Intro2::DefineFormQF}
    \FormQF[f][g]:=\sum_{|\alpha|,|\beta|\leq \kappa}
    \Ltip{W^\alpha f}{a_{\alpha,\beta}W^\beta g}[\ManifoldN][\Vol].
\end{equation}
Note that
\begin{equation*}
    \FormQF[f][g]=\Ltip*{f}{\opL_0 g}[\ManifoldN][\Vol],\quad f\in \TestFunctionsZero[\ManifoldN], g\in \DistributionsZero[\ManifoldN],
\end{equation*}
where
\begin{equation}\label{Eqn::Intro2::opL0}
    \opL_0=\sum_{|\alpha|,|\beta|\leq \kappa} \left( W^{\alpha} \right)^{*}a_{\alpha,\beta} W^\beta,
\end{equation}
and \(*\) denotes the formal \(\LtSpace[\ManifoldN][\Vol]\) adjoint.
Many different forms \(\FormQF\) give rise to the same formula \eqref{Eqn::Intro2::opL0}.
As we discuss, and as is well-known, the choice of \(\FormQF\) and domain for \(\FormQF\) corresponds
to boundary conditions for \(\opL_0\)--see Example \ref{Example::Intro::Examples} and Section \ref{Section::BoundaryCond}.

We turn to (an example of) the type of boundary conditions we consider in this paper.
Let \(X\) be in the \(\CinftySpace[\ManifoldN][\R]\) module generated by \(W_1,\ldots, W_r\),
let \(\JSet\subseteq \left\{ 0,1,\ldots, \kappa-1 \right\}\) be any subset, and set
\begin{equation}\label{Eqn::Intro2::CoreB}
    \CoreB
    :=\left\{ f\in \CinftySpace[\ManifoldN] : X^j f\big|_{\BoundaryN}=0, j\in \JSet \right\}.
\end{equation}
For an explanation of how \(\CoreB\) relates to boundary conditions, see Section \ref{Section::BoundaryCond}.

Our main assumption is:
\begin{assumption}[Maximal Subellipticity]\label{Assumption::Intro2::MaxSub}
    We assume \(\exists C_1, C_2\geq 0\), \(\forall f\in\CoreB\),
    \begin{equation}\label{Eqn::Intro2::MaxSub::MaxSubEqn}
        \sum_{|\alpha|\leq \kappa} \LtNorm*{W^\alpha f}[\ManifoldN][\Vol]^2
        \leq
        C_1 \Real \FormQF[f][f]
        +C_2 \LtNorm{f}[\ManifoldN][\Vol]^2.
    \end{equation}
\end{assumption}

\begin{lemma}\label{Lemma::Intro2::QFIsCloseable}
    \((\FormQF, \CoreB)\) is closable. Let \(\QDomainQ\) denote its closure.
    \(\QDomainQ\) is a closed, densely defined, sectorial form.
    See \cite[Chapter 6, Sections 1.1-1.4]{KatoPerturbationTheory} for the relevant definitions.
\end{lemma}

See Section \ref{Section::IntroProofs} for a proof of Lemma \ref{Lemma::Intro2::QFIsCloseable}.
By \cite[Chapter 6, Section 2, Theorem 2.1]{KatoPerturbationTheory},
there is a unique, closed, densely defined, m-sectorial operator associated to \(\QDomainQ\),
call this operator \(\LDomainL\); it satisfies
\begin{equation}\label{Eqn::Intro::DefnOfOpL}
    \Ltip{f}{\opL g}[\ManifoldN][\Vol] = \FormQ[f][g],\quad \forall f\in \DomainQ,\:\forall g\in \DomainL,
\end{equation}
and is maximal with respect to this--see \cite[Chapter 6, Section 2]{KatoPerturbationTheory} for
details. On its domain, \(\opL\) agrees with \(\opL_0\) defined by \eqref{Eqn::Intro2::opL0}
(as can be seen by taking \(f\in \TestFunctionsZero[\ManifoldN]\subseteq \CoreB\subseteq \DomainQ\) in \eqref{Eqn::Intro::DefnOfOpL}).

Let \(\Omega:=\InteriorN\bigcup \left\{ x\in \BoundaryN : X(x)\not \in \TangentSpace{x}{\BoundaryN} \right\}\),
which is an open submanifold with boundary of \(\ManifoldN\).
Our goal is to prove subelliptic estimates for \(\partial_t+\opL\); though we only do so
on \(\R\times \Omega\).

Let \(\R\times \ManifoldN\) have coordinates \((t,\xi)\). We let \(\HsSpace{s}[\R\times \ManifoldN]\)
be the usual \(\LtSpace\)-Sobolev space of order \(s\), except with \(\partial_t\) treated as an operator of
order \(2\kappa\) (see Section \ref{Section::PseudodifferentialOps}
and in particular, Remark \ref{Rmk::PDOs::DefineClassicalHsSpaceOnMfld}).
We write \(\HsNorm{u}{s}[\R\times \ManifoldN]<\infty\) to mean \(u\in \HsSpace{s}[\R\times \ManifoldN]\).
For \(\phi_1,\phi_2\in \CinftycptSpace[\R\times \ManifoldN]\), we write
\(\phi_1\prec \phi_2\) to mean \(\phi_2=1\) on a neighborhood of \(\supp(\phi_1)\).
Fix \(\epsilon_0:=\min\{1/2,1/m\}\) where \(m\) is the order of H\"ormander's condition (see Definition \ref{Defn::Result::HormandersCondition}).\footnote{In the results
which follow, this choice of \(\epsilon_0>0\) is not optimal; see Section \ref{Section::Intro::Quant}.}
All the results which follow are corollaries of our single main result: Theorem \ref{Thm::Result::MainThm::New}.
See Section \ref{Section::IntroProofs} for proofs of the results in this section.

\begin{definition}
    For \(l\in \Zg\) we define vector spaces \(\FSpace{l}\subseteq \LtSpaceNoMeasure[\R\times \ManifoldN]\)
    recursively as follows:
    \begin{itemize}
        \item \(\FSpace{1}:=\LtSpaceNoMeasure[\R][\DomainL]\).
        \item For \(l\geq 1\), \(\FSpace{l+1}:=\left\{ u\in \FSpace{l} : \left( \partial_t+\opL \right)u\in \FSpace{l} \right\}\),
            where \(\partial_t u(t,x)\) is taken in the sense of \(\DistributionsZero[\R\times \ManifoldN]\) and \(\opL u\)
            is defined since \(u\in \FSpace{l}\subseteq \FSpace{1}= \LtSpaceNoMeasure[\R][\DomainL]\).
    \end{itemize}
\end{definition}

\begin{corollary}[Subellipticity]
    \label{Cor::Intro::dtPlusOplSubellipic}
    \(\forall \phi_1,\phi_2\in \CinftycptSpace[\R\times \Omega]\) with \(\phi_1\prec \phi_2\),
    \(\forall l\in \Zg\) with \(l\geq 2\kappa-1\), \(\exists C\geq 0\),
    \(\forall u\in \FSpace{1}\),
    \begin{equation*}
        \HsNorm*{\phi_1 u}{l+1}
        \leq C
        \left( \HsNorm*{\phi_2 \left( \partial_t+\opL \right) u}{l+1-\epsilon_0}
        +\LtNorm*{\phi_2 u} \right),
    \end{equation*}
    where if the right-hand side is finite, so is the left-hand side.
\end{corollary}

\begin{corollary}[Hypoellipticity]
    If \(u\in \FSpace{1}\)
    is such that \(\left( \partial_t +\opL \right)u\) is smooth near some \((t_0,\xi_0)\in \R\times \Omega\),
    then \(u\) is smooth near \((t_0,\xi_0)\).
\end{corollary}

\begin{corollary}
    \(\forall \phi_1,\phi_2\in \CinftycptSpace[\R\times \Omega]\) with \(\phi_1\prec \phi_2\),
    \(\forall N\in \Zg\), \(\forall l'\in \{1,2,\ldots, N\}\), \(\exists C\geq 0\),
    \(\forall u\in \FSpace{N}\),
    \begin{equation*}
        \HsNorm*{\phi_1 u}{2\kappa l'}
        \leq C
        \left( 
            \HsNorm*{\phi_2\left( \partial_t+\opL \right)^N u}{(2\kappa l'-N\epsilon_0)\vee 0}
            +\sum_{k=0}^{N-1} \LtNorm*{\phi_2 \left( \partial_t+\opL \right)^k u}
         \right),
    \end{equation*}
    where if the right-hand side is finite, so is the left-hand side.
    In particular, if \(N\epsilon_0\geq 2\kappa l'\), \(\forall u\in \FSpace{N}\),
    \begin{equation*}
        \HsNorm*{\phi_1 u}{2\kappa l'}\leq C
        \sum_{k=0}^{N} \LtNorm*{\phi_2 \left( \partial_t+\opL \right)^k u}.
    \end{equation*}
\end{corollary}

\begin{corollary}
    \(\forall u\in \FSpace{\infty}:=\bigcap_{l\in \Zg} \FSpace{l}\), 
    \(u\big|_{\R\times\Omega}\in \CinftySpace[\R\times \Omega]\).
    Furthermore
    \(\forall \phi_1,\phi_2\in \CinftycptSpace[\R\times \Omega]\) with \(\phi_1\prec \phi_2\),
    \(\forall l\in \Zg\), \(\exists N\in \Zg\), \(\exists C\geq 0\),
    \(\forall u\in \FSpace{\infty}\),
    \begin{equation*}
        \CjNorm{\phi_1 u}{l}
        \leq C \sum_{k=0}^N \LtNorm*{\phi_2 \left( \partial_t +\opL \right)^k u}.
    \end{equation*} 
\end{corollary}

We also have similar results for \(\opL\) in place of \(\partial_t+\opL\). Here, we use the usual Sobolev spaces
\(\HsSpace{s}[\ManifoldN]\).  As before, if we have \(\HsNorm{f}{s}[\ManifoldN]<\infty\),
this means \(f\in \HsSpace{s}[\ManifoldN]\).
We define \(\DomainLN{N}\) recursively by:
\begin{itemize}
    \item \(\DomainLN{1}=\DomainL\).
    \item \(\DomainLN{N+1}=\left\{ f\in \DomainLN{N} : \opL f\in \DomainLN{N} \right\}\).
\end{itemize}
Set \(\DomainLN{\infty}:=\bigcap_{N}\DomainLN{N}\).\footnote{Because the resolvent set of \(\opL\)
is non-empty, the spaces \(\DomainLN{N}\) are Banach spaces--see 
\cite[Chapter 2, Proposition 2.15]{EngelNagelAShortCourseOnOperatorSemigroups}.
Therefore, \(\DomainLN{\infty}\) is a Fr\'echet space. We do not use these facts.}

\begin{corollary}[Subellipticity]
    \(\forall \phi_1,\phi_2\in \CinftycptSpace[\Omega]\) with \(\phi_1\prec \phi_2\),
    \(\forall l\in \Zg\) with \(l\geq 2\kappa-1\), \(\exists C\geq 0\), \(\forall f\in \DomainL\),
    \begin{equation*}
        \HsNorm*{\phi_1 f}{l+1}
        \leq C
        \left( 
            \HsNorm*{\phi_2 \opL f}{l+1-\epsilon_0}
            +\LtNorm*{\phi_2 f}
         \right),
    \end{equation*}
    where if the right-hand side is finite, so is the left-hand side.
\end{corollary}

\begin{corollary}[Hypoellipticity]
    If \(f\in \DomainL\) is such that \(\opL f\) is smooth near some \(x_0\in \Omega\), then \(f\)
    is smooth near \(x_0\).
\end{corollary}

\begin{corollary}
    \(\forall \phi_1,\phi_2\in \CinftycptSpace[\Omega]\) with \(\phi_1\prec \phi_2\), \(\forall N\in \Zg\),
    \(\forall l'\in \left\{ 1,2,\ldots, N \right\}\), \(\exists C\geq 0\), \(\forall f\in \DomainLN{N}\),
    \begin{equation*}
        \HsNorm*{\phi_1 f}{2\kappa l'} \leq C
        \left( 
            \HsNorm*{\phi_2 \opL^N f}{(2\kappa l'-N\epsilon_0)\vee 0}
            +\sum_{k=0}^{N-1} \LtNorm*{\phi_2 \opL^k f}
         \right),
    \end{equation*}
    where if the right-hand side is finite, so is the left-hand side. In particular, if \(N\epsilon_0\geq 2\kappa l'\),
    \(\forall f\in \DomainLN{N}\),
    \begin{equation*}
        \HsNorm*{\phi_1 f}{2\kappa l'} \leq C \sum_{k=0}^N \LtNorm*{\phi_2 \opL^k f}.
    \end{equation*}
\end{corollary}

\begin{corollary}
    \(\forall f\in \DomainLN{\infty}\), \(f\big|_{\Omega}\in \CinftySpace[\Omega]\).
    Furthermore, \(\forall \phi_1,\phi_2\in \CinftycptSpace[\Omega]\) with \(\phi_1\prec \phi_2\),
    \(\forall l\in \Zg\), \(\exists N\in \Zg\), \(\exists C\geq 0\), \(\forall f\in \DomainLN{\infty}\),
    \begin{equation*}
        \CjNorm*{\phi_1 f}{l}\leq C \sum_{k=0}^N \LtNorm*{\phi_2 \opL^k f}.
    \end{equation*}
\end{corollary}

\begin{corollary}
    \label{Cor::Intro::EigenVectorsSmooth}
    Let \(f\in \DomainL\) satisfy \(\opL f=\lambda f\), where \(\lambda\in \C\). Then,
    \(f\big|_{\Omega}\in \CinftySpace[\Omega]\). Furthermore, \(\forall \phi\in \CinftycptSpace[\Omega]\),
    \(\forall l\in \Zg\), \(\exists N\in \Zg\), \(\exists C\geq 0\), \(\forall \lambda \in \C\), \(\forall f\in \DomainL\) with \(\opL f=\lambda f\),
    \begin{equation*}
        \CjNorm{\phi f}{l}\leq C (1+|\lambda|)^N \LtNorm{f}.
    \end{equation*}
\end{corollary}

\begin{remark}
    The above results hold with \(\LDomainL\) replaced by \((\opL^{*},\Domain{\opL^{*}})\), throughout.
    Indeed, replacing \(\FormQF\) with \(\FormQF^{*}(f,g)=\overline{\FormQF[g][f]}\) has the effect
    of replacing \(\LDomainL\) with \((\opL^{*},\Domain{\opL^{*}})\);
    see \cite[Chapter 6, Section 2, Theorem 2.5]{KatoPerturbationTheory}.
    The  assumptions in this section are unchanged under this replacement.\footnote{The assumptions
    in Section \ref{Section::MainCor} are also unchanged under a similar replacement (thereby
    replacing \(\LDomainL\) with \((\opL^{*},\Domain{\opL^{*}})\)). However,
    this is not true of the assumptions in Section \ref{Section::MainResult}; namely,
    Assumption \ref{Asssumption::Result::BoundaryAssumption::New}\ref{Item::Result::BoundaryAssumption::SOperator::New}
    uses \(\DomainL\), and so one would need a similar assumption for \(\Domain{\opL^{*}}\) to obtain
    the same results for \((\opL^{*},\Domain{\opL^{*}})\).}
\end{remark}

\begin{example}\label{Example::Intro::Examples}
    We describe several examples of the above to keep in mind.
    \begin{enumerate}[(i)]
        \item\label{Item::Intro::Examples::Dirichlet} When \(J=\left\{ 0,\ldots, \kappa-1 \right\}\), this corresponds to Dirichlet
            boundary conditions for \(\opL_0\). Here, the particular choice of \(X\) does not matter,
            only that such an \(X\) exists--see Section \ref{Section::Intro::NonChar}.  Also, in this case the particular choice
            of \(\FormQF\) does not matter, so long as it gives rise to \(\opL_0\);
            however, the choice of \(\FormQF\) does matter in \ref{Item::Intro::Examples::SubLaplace} and \ref{Item::Intro::Examples::NonDirichlet}, below.
        \item\label{Item::Intro::Examples::SubLaplace} Consider an operator of the form
            \begin{equation*}
                \opL_0=-\sum_{1\leq i,j\leq r}a_{i,j}(x) W_iW_j+\sum_{j=1}^r a_j(x)W_j+a(x),
            \end{equation*}
            where 
            \(a_{i,j}\in \CinftySpace[\ManifoldN][\R]\),
            \(a_j,a\in \CinftySpace[\ManifoldN][\C]\), and \(\left( a_{i,j}(x) \right)_{1\leq i,j\leq r}\)
            is a real, symmetric, strictly positive definite matrix for each \(x\).
            Let \(X\) and \(J\subseteq \{0\}\) be as described above.
            In Proposition \ref{Prop::BoundaryConds::SecondOrder::MainSecondOrder}, we show 
            the results in this introduction apply in this situation.
            When \(J=\{0\}\), this corresponds to 
            Dirichlet boundary conditions, as mentioned in \ref{Item::Intro::Examples::Dirichlet},
            and the particular choice of \(\FormQF\) does not play a role.
            However, when \(J=\EmptySet\), in Proposition \ref{Prop::BoundaryConds::SecondOrder::MainSecondOrder}, we show that we may pick different
            choices of \(\FormQF\) to give rise to different ``unstable'' boundary conditions.
        \item\label{Item::Intro::Examples::NonDirichlet} More generally than \ref{Item::Intro::Examples::SubLaplace}, in Section \ref{Section::BoundaryCond::GeneralOps},
            when \(J\subsetneq \left\{ 0,\ldots, \kappa-1 \right\}\),
            we show how to derive the boundary conditions induced by \(\FormQF\) and \(\CoreB\)
            in the general setting of this introduction.

        \item\label{Item::Intro::Examples::HormanderSubLap} The quadratic form
            \begin{equation}\label{Eqn::Intro::Examples::HormanderSubLapForm}
                \FormQF[f][g]=\sum_{j=1}^r \Ltip*{W_j f}{W_j g}[\ManifoldN]
            \end{equation}
            corresponds to the sub-Laplacian \(\opL_0=\sum_{j=1}^r W_j^{*} W_j\). It clearly satisfies Assumption \ref{Assumption::Intro2::MaxSub}
            (with \(\kappa=1\))
            for any subspace \(\CoreB\subseteq \CinftySpace[\ManifoldN][\C^M]\).

        \item\label{Item::Intro::Examples::NonSymmetric} In \ref{Item::Intro::Examples::SubLaplace} and \ref{Item::Intro::Examples::HormanderSubLap}, 
        \(\opL_0\) is symmetric modulo lower order terms. However, this need not be the case. For example, let \(\FormQF\) be given by \eqref{Eqn::Intro::Examples::HormanderSubLapForm},
            and consider the form \((f,g)\mapsto \FormQF[f][g]+\alpha \Ltip{W_1 f}{W_2 g}\), where
            \(\alpha\) is a complex number with
            \(\left\lvert \alpha \right\rvert <2\).
            This form satisfies Assumption \ref{Assumption::Intro2::MaxSub}
            (with \(\kappa=1\))
            for any subspace \(\CoreB\subseteq \CinftySpace[\ManifoldN][\C^M]\) (see Lemma \ref{Lemma::Intro2::PerturbForm} for a proof of a more general version of this). In this case, 
            \(\opL_0 = \sum_{j=1}^r W_j^{*} W_j + \alpha W_1^{*} W_2\).  

        \item \ref{Item::Intro::Examples::SubLaplace}, \ref{Item::Intro::Examples::HormanderSubLap}, and \ref{Item::Intro::Examples::NonSymmetric}
            are all second-order operators (\(\kappa=1\)).  However, our results hold for general \(\kappa\), and second-order is not used in our proof in any way.
            For example, take
            \begin{equation*}
                \FormQF[f][g]=\sum_{|\alpha|\leq \kappa} \Ltip*{W^\alpha f}{W^{\alpha} g}[\ManifoldN].
            \end{equation*}
            \(\FormQF\) clearly satisfies Assumption \ref{Assumption::Intro2::MaxSub}
            for any subspace \(\CoreB\subseteq \CinftySpace[\ManifoldN][\C^M]\).
            This corresponds to the operator \(\opL_0=\sum_{|\alpha|\leq \kappa} \left( W^\alpha \right)^{*} W^\alpha\).
    \end{enumerate}
\end{example}

\begin{remark}
    A main difficulty in this paper is that we do not restrict
    ourselves to Dirichlet boundary conditions (see Example \ref{Example::Intro::Examples}\ref{Item::Intro::Examples::SubLaplace},\ref{Item::Intro::Examples::NonDirichlet}).
    Even in the setting of Dirichlet boundary conditions, our results are new in the
    generality in which they are stated. However, the proof becomes more subtle when
    non-Dirichlet boundary conditions are addressed. Integration by parts becomes
    a major road block, and this adds substantial difficulties.
\end{remark}

\begin{remark}
    Another main difficulty in this paper is that we do not restrict to the special case \(\kappa=1\).
    Many estimates are easier in that case.
\end{remark}

\begin{remark}
    In \cite{KohnNirenbergNonCoerciveBoundaryValueProblems}, Kohn and Nirenberg introduced an ``elliptic regularization''
    procedure which has since become an important tool in the subject. We do not use this elliptic regularization
    and instead proceed directly. However, in Section \ref{Section::EllipticRegularization} we show how one can seamlessly introduce this regularization
    in our framework if desired.
\end{remark}

\begin{remark}
    The results in this introduction hold more generally in the case when the functions take values in \(\C^M\);
    here \(a_{\alpha,\beta}\in \CinftySpace[\ManifoldN][\M^{M\times M}(\C)]\).
    In fact, the rest of this paper works in this greater generality.
\end{remark}

For some previous works on subelliptic boundary value problems, see
\cite{KohnNirenbergNonCoerciveBoundaryValueProblems} (which inspired several methods in this paper),
as well as
\cite{DerridjSurUnTheoremeDeTraces,
JerisonDirichletProblemForTheKohnLaplacianI,
JerisonDirichletProblemForTheKohnLaplacianII,
ParmeggianiXuTheDirichletProblemForSubEllipticSecondOrderEquations,
OrponenVillaSubEllipticBoundaryValueProblemsInFlagDomains}.

\begin{lemma}\label{Lemma::Intro2::PerturbForm}
    Suppose \(\FormQF\) is given by \eqref{Eqn::Intro2::DefineFormQF}, \(\CoreB\subseteq \CinftySpace[\ManifoldN]\) is a subspace,
    and  \(\FormQF\) satisfies Assumption \ref{Assumption::Intro2::MaxSub}.
    Let
    \begin{equation*}
        \FormQZero[f][g]:=\sum_{|\alpha|,|\beta|\leq \kappa} \Ltip*{W^{\alpha} f}{b_{\alpha,\beta}W^{\beta}g}[\ManifoldN][\Vol],\quad b_{\alpha,\beta}\in \CinftySpace[\ManifoldN].
    \end{equation*}
    Define,
    \begin{equation*}
        \FormQFt[f][g]:=\FormQF[f][g]+\FormQZero[f][g].
    \end{equation*}
    Suppose, \(\exists c<1\), \(C\geq 0\) with
    \begin{equation*}
        -\Real \FormQZero[f][f]\leq c \Real \FormQF[f][f] +C \LtNorm{f}^2,\quad \forall f\in \CoreB.
    \end{equation*}
    Then, \(\FormQFt\) also satisfies Assumption \ref{Assumption::Intro2::MaxSub} (with the same \(\CoreB\)).
\end{lemma}
\begin{proof}
    By assumption, we have
    \begin{equation*}
        \Real \FormQFt[f][f]\geq (1-c) \Real \FormQF[f][f] -C\LtNorm{f}^2,\quad \forall f\in \CoreB.
    \end{equation*}
    Combining this with \eqref{Eqn::Intro2::MaxSub::MaxSubEqn} shows, \(\forall f\in \CoreB\),
    \begin{equation*}
        \sum_{|\alpha|\leq \kappa} \LtNorm*{W^{\alpha}f}^2
        \leq C_1 \Real \FormQF[f][f]+ C_2 \LtNorm*{f}^2
        \leq \frac{C_1}{1-c} \Real \FormQFt[f][f]+\left( C_2+\frac{C_1 C}{1-c} \right)\LtNorm*{f}^2.
    \end{equation*}
\end{proof}

\paragraph{AI Usage.} Once an early version of this paper was written, and after the mathematical content and proofs were complete, the author used ChatGPT for proofreading and improving the exposition. 
Otherwise, this article is human generated.

    \subsection{Non-characteristic points}\label{Section::Intro::NonChar}
    \begin{definition}\label{Defn::Intro::NonChar::NonCharDefn}
  We say \(x_0\in \BoundaryN\) is \(W\)-non-characteristic if \(\exists j\) with \(W_j(x_0)\not\in \TangentSpace{x}{\BoundaryN}\).
  In other words, \(\BoundaryN\) is non-characteristic for \(x\mapsto \left( W_1(x),\ldots, W_r(x) \right)\) (in the classical sense) near \(x_0\).
\end{definition}

As first pointed out by Kohn and Nirenberg \cite{KohnNirenbergNonCoerciveBoundaryValueProblems},
Derridj \cite{DerridjSurUnTheoremeDeTraces},
and Jerison \cite{JerisonDirichletProblemForTheKohnLaplacianI,JerisonDirichletProblemForTheKohnLaplacianII}
whether a boundary point is  \(W\)-non-characteristic has important implications
for the subellipticity of operators like \(\opL\). 
In fact, as shown in \cite{JerisonDirichletProblemForTheKohnLaplacianII},
hypoellipticity for the Dirichlet problem for the sublaplacian can fail
at characteristic points in even some of the simplest situations.

In the setting above \(\Omega\cap \BoundaryN\) consists entirely of \(W\)-non-characteristic points.
In fact, given \(x_0\in \BoundaryN\), one can pick \(X\) as above such that
\(x_0\in \Omega\) if and only if \(x_0\) is \(W\)-non-characteristic.
Thus, as described in Example \ref{Example::Intro::Examples}\ref{Item::Intro::Examples::Dirichlet},
we study the Dirichlet problem precisely at the \(W\)-non-characteristic boundary points.
Even when we move beyond the Dirichlet problem, our results apply near \(x_0\in \BoundaryN\)
only if \(x_0\) is \(W\)-non-characteristic.
Furthermore, as studied in \cite{StreetFunctionSpacesAndTraceTheoremsForMaximallySubellipticBoundaryValueProblems},
trace theorems work well near \(W\)-non-characteristic boundary points (see \cite[Chapter \ref*{FS::Chapter::Trace}]{StreetFunctionSpacesAndTraceTheoremsForMaximallySubellipticBoundaryValueProblems}), though
not as well near \(W\)-characteristic boundary points (see \cite[Section \ref*{FS::Section::Trace::CharacteristicFailure}]{StreetFunctionSpacesAndTraceTheoremsForMaximallySubellipticBoundaryValueProblems}).

    \subsection{Quantitative estimates and motivation}\label{Section::Intro::Quant}
    While the results in this paper are perhaps interesting in their own right, 
they are not sharp: even if one takes great care in the arguments, the value of \(\epsilon_0>0\)
 used in the introduction is not optimal.
One of our motivations is to use the results in this paper as a step in later obtaining sharp results; even involving
adapted non-isotropic function spaces as described in \cite{StreetFunctionSpacesAndTraceTheoremsForMaximallySubellipticBoundaryValueProblems}.
To do so, we need to keep careful track of what the constants in this paper depend on, because
the results of this paper will be applied an infinite number of times, and the estimates need
to be uniform over this infinite number of applications.

More precisely, the main estimates of this paper happen in a nice coordinate chart
(see Assumption \ref{Asssumption::Result::BoundaryAssumption::New}).
In Theorems \ref{Thm::EstBdry::MainThm::New} and \ref{Thm::EstBdryInt::MainThm::New}, we take care to describe what each constant depends on, in terms of the various objects
written in this coordinate system. 
Following a proof method from the interior theory (see \cite[Theorem 8.1.1]{StreetMaximalSubellipticity} and \cite[Section \ref*{Heat::Section::Subellip::MaximalSub::Parametricies}]{StreetHypoellipticityAndHigherOrderGaussianBounds}, though similar ideas date
back to \cite{FollandSubellipticEstimatesAndFunctionSpacesOnNilpotentLieGroups})
the plan is the following:
\begin{enumerate}[(i)]
    \item\label{Item::Intro::Quant::APriori} Establish a priori estimates in a fixed coordinate system, while keeping track of the dependence of the constants (this paper).
    \item\label{Item::Intro::Quant::ScaleInv} For each \(\delta\in (0,1]\) in a future paper, 
    we plan to apply the results from \ref{Item::Intro::Quant::APriori} with \(\FormQF\) 
    replaced by \(\delta^{2\kappa} \FormQF\) (i.e., replacing
    \(\opL\) by \(\delta^{2\kappa}\opL\)).  Here the required coordinate charts were given in \cite[Theorem \ref*{CC::Thm::Scaling::MainResult}]{StreetCarnotCaratheodoryBallsOnManifoldsWithBoundary};
    these charts have the effect of ``rescaling'' \(\delta\) to be equal to \(1\).
    This establishes certain hypoelliptic estimates for \(\partial_t+\delta^{2\kappa}\opL\) at 
    every Carnot--Carath\'eodory scale,
    which are appropriately uniform in \(\delta\) and the base point.
    \item\label{Item::Intro::Quant::Heat} We plan to use the estimates from \ref{Item::Intro::Quant::ScaleInv},
        combined with the main result of \cite{StreetHypoellipticityAndHigherOrderGaussianBounds} to conclude appropriate higher
        order Gaussian bounds for the heat operator \(e^{-t\opL}\).
    \item Integrating the heat operator yields a parametrix. In a future paper, we plan to use the bounds from \ref{Item::Intro::Quant::Heat}
        to prove better a priori estimates for \(\opL\) in terms of the function spaces
        from \cite{StreetFunctionSpacesAndTraceTheoremsForMaximallySubellipticBoundaryValueProblems}, and formulate a more general theory of maximally subelliptic boundary
        value problems.
\end{enumerate}
    
\section{H\"ormander vector fields}
Let \(\ManifoldN\) be a smooth manifold with boundary.

\begin{definition}\label{Defn::Result::HormandersCondition}
    Let \(W_1,\ldots, W_r\) be smooth vector fields on some open set \(\Omega\subseteq \ManifoldN\).
    We say \(W_1,\ldots, W_r\) satisfy H\"ormander's condition of order \(m\in \Zg=\left\{ 1,2,3,\ldots \right\}\) on \(\Omega\) if
    \begin{equation*}
        \underbrace{W_1,\ldots, W_r}_{\substack{\text{commutators}\\\text{of order }1}}, \ldots, \underbrace{[W_i,W_j]}_{\substack{\text{commutators}\\\text{of order }2}},\ldots, \underbrace{[W_i,[W_j,W_k]]}_{\substack{\text{commutators}\\\text{of order }3}},\ldots, \text{(commutators of order m)}
    \end{equation*}
    span the tangent space at every point of \(\Omega\).
\end{definition}

Let \(\TestFunctionsZero[\ManifoldN]\) 
be the closed subspace of \(\CinftycptSpace[\ManifoldN]\)
consisting of those \(f\in \CinftycptSpace[\ManifoldN]\) which vanish to infinite
order on \(\BoundaryN\) and let \(\DistributionsZero[\ManifoldN]\)
be the dual. We identify \(L^1_{\mathrm{loc}}(\ManifoldN)\) with a subspace
of \(\DistributionsZero[\ManifoldN]\) in the usual way:
\begin{equation*}
    f(\phi):=\int f(\xi)\phi(\xi) \: d\Vol(\xi),\quad f\in L^1_{\mathrm{loc}}(\ManifoldN),\: \phi\in \TestFunctionsZero[\ManifoldN].
\end{equation*}
\(\DistributionsZero[\ManifoldN]\) is the space of distributions we use throughout this paper.

Fix \(\kappa\in \Zgeq=\{0,1,2,3,\ldots\}\), and define
\begin{equation}\label{Eqn::Intro::DefnHsWSpace}
    \HsWSpace{\kappa}[\ManifoldN]
    :=
    \left\{ 
        f\in \LtSpace[\ManifoldN][\Vol]
        :
        W^{\alpha}f\in \LtSpace[\ManifoldN][\Vol], \forall |\alpha|\leq \kappa
      \right\},
      \quad
      \HsWNorm{f}{\kappa}[\ManifoldN]^2
      :=\sum_{|\alpha|\leq \kappa} 
      \LtNorm*{W^\alpha f}[\ManifoldN][\Vol]^2.
\end{equation}
\(\HsWSpace{\kappa}[\ManifoldN]\) is a Hilbert space.
We write \(\HsWSpace{\kappa}[\ManifoldN;\C^M]\)
for the obvious modification where the functions are allowed to take values in \(\C^M\)
instead of \(\C\).


\section{The main theorem}\label{Section::MainResult}
Let \(\ManifoldN\) be a smooth manifold with boundary of dimension \(n\geq 2\),  let \(\Vol\) be a smooth, 
strictly positive density on \(\ManifoldN\),
and fix \(M\in \Zg\).
Let \(\QDomainQ\) be a closed, densely defined, sectorial, sesqui-linear form on 
\(\LtSpace*[\ManifoldN][\Vol][\C^M]\);
see \cite[Chapter 6, Sections 1.1-1.3]{KatoPerturbationTheory} for the relevant definitions.
We remind the reader that \(\QDomainQ\) is called sectorial if \(\exists \gamma\in \R\) (note, \(\gamma\) can be negative),
and \(C\geq 0\) with, \(\forall f\in \DomainQ\),
\begin{equation}\label{Eqn::Result::SectorialDefn}
    \gamma \LtNorm{f}^2 \leq \Real \FormQ[f][f],
    \quad
    \left|  \Imaginary  \FormQ[f][f] \right| \leq C\left( \Real \FormQ[f][f]-\gamma \LtNorm{f}^2 \right).
\end{equation}
Note that \eqref{Eqn::Result::SectorialDefn} implies
\begin{equation}\label{Eqn::Result::SectorialConsequence}
    \left| \FormQ[f][f] \right|\leq C_1  \Real \FormQ[f][f] + C_2 \LtNorm{f}^2,
\end{equation}
for some \(C_1,C_2\geq 0\).

\(\DomainQ\) can be given the structure of a Hilbert space by defining
\begin{equation}\label{Eqn::Result::DomainQHilbertNorm}
    \DomainQNorm{f}^2
    \coloneqq \Real \FormQ[f][f]+\Gamma \LtNorm{f}^2,
\end{equation}
where \(\Gamma\gg 1\) is sufficiently large (any two such choices of \(\Gamma\) yield equivalent norms); see \cite[Chapter 6, Section 1.3]{KatoPerturbationTheory} for details. We henceforth
equip \(\DomainQ\) with this topology.
Note that, by \eqref{Eqn::Result::SectorialConsequence},
\begin{equation}\label{Eqn::Core::QNormSameAsFormQ}
    \QNorm{u}^2\approx \left| \FormQ[u][u] \right|+\LtNorm{u}^2.
\end{equation}

By \cite[Chapter 6, Section 2, Theorem 2.1]{KatoPerturbationTheory}, there exists
a unique, unbounded, densely defined operator \(\LDomainL\) on \(\LtSpace[\ManifoldN][\Vol][\C^M]\) defined by
\begin{equation*}
    \Ltip{f}{\opL g}[\ManifoldN][\Vol][\C^M]
    =\FormQ[f][g], \quad \forall f\in \DomainQ, g\in \DomainL,
\end{equation*}
and \(\LDomainL\) is the maximal operator with this property (see \cite[Chapter 6, Theorem 2.1]{KatoPerturbationTheory} for details).
\(\LDomainL\) is a closed operator, and we give \(\DomainL\) the usual Banach space structure via the graph norm:
\(\DomainLNorm{f}:=\LtNorm{f}+\LtNorm{\opL f}\).

\begin{goal}
    Given \(x_0\in \ManifoldN\), give conditions on \(\QDomainQ\) so that \(\partial_t+\opL\)
    is subelliptic near \(x_0\).
\end{goal}


Fix \(\Omega\subseteq \ManifoldN\) an open set and let \(W_1,\ldots, W_r\) be smooth vector fields
on \(\Omega\) satisfying H\"ormander's condition of order \(m\in \Zg\).  
We write \(A\Subset B\) to mean \(A\) is a relatively compact subset of \(B\).
Fix \(\kappa\in \Zg\).

Fix \(a_{\alpha,\beta}\in \CinftySpace[\Omega][\MatrixSpace[M][M][\C]]\), for \(|\alpha|,|\beta|\leq \kappa\),
where \(\MatrixSpace[M][M][\C]\) is the space of \(M\times M\) complex matrices.
Define, whenever it makes sense, 
\begin{equation}\label{Eqn::Result::FormQFFormula}
    \FormQF[f][g]:=\sum_{|\alpha|,|\beta|\leq \kappa}
    \Ltip*{W^{\alpha}f}{a_{\alpha,\beta} W^{\beta}g}[\Omega][\Vol][\C^M].
\end{equation}
In particular, \(\FormQF\) is sesqui-linear form \(\HsWSpace{\kappa}[\Omega;\C^M]\times \HsWSpace{\kappa}[\Omega;\C^M]\rightarrow \C\).
Here, \(\mathrm{F}\) stands for ``formula,'' as \(\FormQF\) will give the formula 
for \(\FormQ\); see Assumption \ref{Asssumption::Result::QIsQF}.

As before, let \(\TestFunctionsZeroCM[\Omega]\) consist of those functions in \(\CinftycptSpace[\Omega][\C^M]\) which vanish to infinite
order on \(\Omega\cap\BoundaryN\); this is a closed subspace of \(\CinftycptSpace[\Omega][\C^M]\) and inherits the locally convex topology.
The dual space \(\DistributionsZeroCM[\Omega]\) is the space of distributions we use when we need to define generalized derivatives.
Set
\begin{equation*}
    \HsWlocSpace{\kappa}[\Omega;\C^M]:=\left\{ f:\Omega\rightarrow \C^M \: |\:  \forall U\Subset \Omega\text{ open, }f\big|_{U}\in \HsWSpace{\kappa}[U;\C^M] \right\},
\end{equation*}
given the usual Fr\'echet topology.

\begin{assumption}\label{Asssumption::Result::TestFunctionsInDomain}
    \(\TestFunctionsZeroCM[\Omega]\subseteq \DomainQ\).
\end{assumption}

\begin{assumption}\label{Asssumption::Result::DomainIsHkappaW}
    \(f\mapsto f\big|_{\Omega}\)
    takes \(\DomainQ\rightarrow \HsWlocSpace{\kappa}[\Omega;\C^M]\).

\end{assumption}

\begin{remark}\label{Rmk::Result::DomainIsHkappaW::Continuity}
    We do not assume continuity in Assumption \ref{Asssumption::Result::DomainIsHkappaW}. 
    In fact, continuity follows by the Closed Graph Theorem. Indeed, if \(f_j\rightarrow f\) in \(\DomainQ\)
    and \(f_j\big|_{\Omega} \rightarrow g\) in \(\HsWlocSpace{\kappa}[\Omega;\C^M]\), then these convergences
    imply \(f_j\big|_{\Omega}\rightarrow f\big|_{\Omega}\) and \(f_{j}\big|_{\Omega}\rightarrow g\) in \(\LtlocSpace\). 
    It follows that \(f\big|_{\Omega}=g\) and therefore the graph is closed. The Closed Graph Theorem implies the restriction
    map is continuous \(\DomainQ\rightarrow \HsWlocSpace{\kappa}[\Omega;\C^M]\).
\end{remark}

\begin{remark}\label{Rmk::Result::DomainIsHkappaW::MaxSub}
    Assumption \ref{Asssumption::Result::DomainIsHkappaW} is the ``maximally subelliptic'' assumption.
    Indeed, if \(\Compact_0\Subset \Omega\) is a compact set, then Remark \ref{Rmk::Result::DomainIsHkappaW::Continuity}
    implies \(\exists C_{\Compact_0}\geq 0\) such that
    \begin{equation*}
        \HsWNorm{f}{\kappa}\leq C_{\Compact_0} \DomainQNorm{f},\quad \forall f\in \DomainQ,\:\supp(f)\subseteq \Compact_0.
    \end{equation*}
\end{remark}

\begin{assumption}\label{Asssumption::Result::QIsQF}
    \(\forall f,g\in \DomainQ\), with \(\supp(f)\Subset \Omega\), \(\FormQ[f][g]=\FormQF[f][g]\).
\end{assumption}

\begin{remark}
    Let \(g\in \DomainL\), then by Assumption \ref{Asssumption::Result::QIsQF},
    \begin{equation*}
        \Ltip*{f}{\opL g}=\FormQ[f][g]=\FormQF[f][g],\quad \forall f\in \DomainQ\text{ with }\supp(f)\Subset \Omega.
    \end{equation*}
    In light of \eqref{Eqn::Result::FormQFFormula} and Assumption \ref{Asssumption::Result::TestFunctionsInDomain},
    by taking \(f\in \TestFunctionsZeroCM[\Omega]\),
    we see
    \begin{equation*}
        \opL g \big|_{\Omega}= \sum_{|\alpha|,|\beta|\leq \kappa} \left( W^\alpha \right)^{*} a_{\alpha,\beta} W^{\beta}g\big|_{\Omega},
    \end{equation*}
    where the right-hand side is defined in the sense of \(\DistributionsZeroCM[\Omega]\).
\end{remark}


For our final assumption, we 
follow \cite{KohnNirenbergNonCoerciveBoundaryValueProblems} and
consider operators of the following form.
Let 
\(\Cuben{\delta}:=\left\{ (x_1,\ldots, x_n)\in \Rn : |x_j|< \delta,\: \forall j  \right\}\),
\(\Cubengeq{\delta}=\left\{ (x_1,\ldots, x_n)\in \Cuben{\delta} : x_n\geq 0 \right\}\),
and \(\Cubeng{\delta}=\left\{ (x_1,\ldots, x_n)\in \Cuben{\delta} : x_n> 0 \right\}\).
Let \(K(t,s,x,y')\in \CinftycptSpace*[\R\times \R\times \CubengeqOneHalf\times \CubenmoOneHalf]\)
be such that
\(\exists \delta>0\) with \(\partial_{x_n}K(t,s,(x',x_n),y')=0\) for \(0\leq x_n<\delta\),
and let \(S\) be any operator of the form 
\begin{equation}\label{Eqn::Result::FormOfS}
    S u(t,x)=\iint K(t,s,x,y')u(s,y',x_n)\: ds\: dy',
\end{equation}
(here we let \(S\) act diagonally on functions taking values in \(\C^M\) as \(K\) is scalar valued).

\begin{assumption}\label{Asssumption::Result::BoundaryAssumption::New}
    For every \(x_0\in \Omega\cap \BoundaryN\), \(\exists\) a coordinate chart
    \(\Psi:\CubengeqOne\xrightarrow{\sim}\Psi\left( \CubengeqOne \right)\subseteq \Omega\)
    with \(\Psi(0)=x_0\) and a vector field \(X\) in the \(\CinftySpace[\Psi\left( \CubengeqOne \right)][\R]\)-module
    generated by \(W_1,\ldots, W_r\) such that
    \begin{enumerate}[(i)]
        \item\label{Item::Result::BoundaryAssumption::PullBackX::New} \(\Psi^{*}X=c_0(x)\partial_{x_n}\), where \(c_0(x)\ne 0\), \(\forall x\in \CubengeqOne\).
        \item\label{Item::Result::BoundaryAssumption::SOperator::New} 
        Let \(S\) be any operator as described in \eqref{Eqn::Result::FormOfS}, and set \(T=\Psi_{*}S\Psi^{*}\), where
            \(\Psi\) is acting in the \(x\) variable; i.e.,
            \begin{equation*}
                Tu(t, \Psi(x)) = \iint K(t,s,x,y') u(s, \Psi(y',x_n))\: dy'\: ds.
            \end{equation*}
            Then,
            \(T:\LtSpaceNoMeasure*[\R][\DomainL]\rightarrow \LtSpaceNoMeasure*[\R][\DomainQ]\) (we do not assume continuity).
    \end{enumerate}
\end{assumption}

\begin{remark}
    It is sometimes easier to verify assumptions on a core for \(\QDomainQ\), rather than all elements
    of \(\DomainQ\).
    See Section \ref{Section::MainCor} for such assumptions on a core which imply the above assumptions.
\end{remark}

\begin{remark}
    In most applications, one can verify the stronger assumption
    where \(\DomainL\) is replaced by \(\DomainQ\) in 
    Assumption \ref{Asssumption::Result::BoundaryAssumption::New}\ref{Item::Result::BoundaryAssumption::SOperator::New}.
    In this case, all the assumptions are symmetric in \(\DomainL\) and \(\left( \opL^{*}, \Domain{\opL^{*}} \right)\)
    as can be seen by using the fact that \(\left( \opL^{*}, \Domain{\opL^{*}} \right)\)
    corresponds to the form \(\left( \overline{\FormQ}, \DomainQ \right)\), where
    \(\overline{\FormQ}(f,g)=\overline{\FormQ[g][f]}\), for \(f,g\in \DomainQ\).
    See \cite[Chapter 6, Section 2, Theorem 2.5]{KatoPerturbationTheory}.
\end{remark}

\begin{definition}\label{Defn::Result::FSpace}
    For \(l\in \Zg\) we define vector spaces \(\FSpace{l}\subseteq \LtSpaceNoMeasure[\R\times \ManifoldN][\C^M]\)
    recursively as follows:
    \begin{itemize}
        \item \(\FSpace{1}:=\LtSpaceNoMeasure[\R][\DomainL] \).
        \item For \(l\geq 1\), \(\FSpace{l+1}:=\left\{ u\in \FSpace{l} : \left( \partial_t+\opL \right)u\in \FSpace{l} \right\}\),
            where \(\partial_t u(t,x)\) is taken in the sense of \(\DistributionsZeroCM[\R\times \ManifoldN]\) and \(\opL u\)
            is defined since \(u\in \FSpace{l}\subseteq \FSpace{1}= \LtSpaceNoMeasure[\R][\DomainL]\).
    \end{itemize}
\end{definition}

Set \(\epsilon_0:=\min\{1/m,1/2\}\).
In what follows we use the non-isotropic Sobolev spaces \(\HscptSpace{s}[\R\times \ManifoldN]\) consisting of functions \(u(t,x)\) defined on compact subsets of \(\R\times \ManifoldN\);
where we treat \(\partial_t\) as an operator of order \(2\kappa\)--see Remark \ref{Rmk::PDOs::DefineClassicalHsSpaceOnMfld} for details.
When considering functions supported in a fixed compact set, the equivalence class of the norm
is well-defined. The \(\HsSpace{s}\)-norm is finite precisely when the function lies in the space.

We are not explicit about what the constants in our estimates depend on in this section;
however, in Sections \ref{Section::EstNearBdry} and \ref{Section::EstInt} we work in local coordinates,
we make explicit what the various constants depend on.


\begin{theorem}\label{Thm::Result::MainThm::New}
    Under the above assumptions
    (Assumptions \ref{Asssumption::Result::TestFunctionsInDomain}, \ref{Asssumption::Result::DomainIsHkappaW},
    \ref{Asssumption::Result::QIsQF}, and \ref{Asssumption::Result::BoundaryAssumption::New}), the following holds.
    For all \(N\in \Zg\), \(\forall \phi_1,\phi_2\in \CinftycptSpace[\R\times \Omega]\) with \(\phi_1\prec\phi_2\),
    \(\forall l\in \Zg\) with \(l\geq 2\kappa-1\), \(\forall J\in \Zg\) with \(1\leq J\leq \min\left\{ \floor{\frac{l+1}{2\kappa}},N \right\}\),
    \(\exists C\geq 0\), \(\forall u\in \FSpace{N}\),
    \begin{equation}\label{Eqn::Result::MainThm::MainEqn}
    \begin{split}
         &\HsNorm*{\phi_1 u}{l+1}
         \\&\leq C
         \left(
            \HsNorm*{\phi_2 \left( \partial_t+\opL \right)^J u}{l+1-2\kappa J}
            +\HsNorm*{\phi_2 \left( \partial_t+\opL \right)^N u}{(l+1-N\epsilon_0)\vee 0}
            +\sum_{k=0}^{N-1} \LtNorm*{ \phi_2 \left( \partial_t +\opL \right)^k u}
         \right),
    \end{split}
    \end{equation}
    where if the right-hand side is finite, so is the left-hand side.
    Here, and in the rest of the paper, \(a\vee b=\max\{a,b\}\).
\end{theorem}

We state several simple corollaries of Theorem \ref{Thm::Result::MainThm::New}
which are useful in applications. The corollaries in this section all assume the same assumptions
as Theorem \ref{Thm::Result::MainThm::New}.
The first corollary shows \(\partial_t+\opL\) is subelliptic on \(\R\times \Omega\).


\begin{corollary}\label{Cor::Result::PartialtPluOpLSubEllip}
    \(\forall \phi_1,\phi_2\in \CinftycptSpace[\R\times \Omega]\) with \(\phi_1\prec \phi_2\),
    \(\forall l\in \Zg\) with \(l\geq 2\kappa-1\), \(\exists C\geq 0\),
    \(\forall u\in \FSpace{1}\),
    \begin{equation*}
        \HsNorm*{\phi_1 u}{l+1}
        \leq C
        \left( \HsNorm*{\phi_2 \left( \partial_t+\opL \right) u}{l+1-\epsilon_0}
        +\LtNorm*{\phi_2 u} \right),
    \end{equation*}
    where if the right-hand side is finite, so is the left-hand side.
\end{corollary}
\begin{proof}
    Take \(N=J=1\) in Theorem \ref{Thm::Result::MainThm::New}.
\end{proof}

The next corollary shows \(\partial_t+\opL\) is hypoelliptic on \(\R\times \Omega\).

\begin{corollary}\label{Cor::Result::PartialtPluOpLHypoEllip}
    If \(u\in \FSpace{1}\)
    is such that \(\left( \partial_t +\opL \right)u\) is smooth near some \((t_0,x_0)\in \R\times \Omega\),
    then \(u\) is smooth near \((t_0,x_0)\).
\end{corollary}
\begin{proof}
    This follows from Corollary \ref{Cor::Result::PartialtPluOpLSubEllip} applied to every \(l\in \Zg\) with
    \(l\geq 2\kappa-1\).
\end{proof}

Because \(l\) is required to be \(\geq 2\kappa-1\) in Corollary \ref{Cor::Result::PartialtPluOpLSubEllip},
one cannot iterate that corollary to obtain results when \(\phi_2(\partial_t+\opL)^j u\)
is in \(\HsSpace{s}\) for some large \(j\), where \(s<2\kappa-\epsilon_0\).
However, Theorem \ref{Thm::Result::MainThm::New} does imply such results as the next two corollaries
demonstrate.

\begin{corollary}\label{Cor::Result::BoundSobolevByHighPowers}
    \(\forall \phi_1,\phi_2\in \CinftycptSpace[\R\times \Omega]\) with \(\phi_1\prec \phi_2\),
    \(\forall N\in \Zg\), \(\forall l'\in \{1,2,\ldots, N\}\), \(\exists C\geq 0\),
    \(\forall u\in \FSpace{N}\),
    \begin{equation*}
        \HsNorm*{\phi_1 u}{2\kappa l'}
        \leq C
        \left( 
            \HsNorm*{\phi_2\left( \partial_t+\opL \right)^N u}{(2\kappa l'-N\epsilon_0)\vee 0}
            +\sum_{k=0}^{N-1} \LtNorm*{\phi_2 \left( \partial_t+\opL \right)^k u}
         \right),
    \end{equation*}
    where if the right-hand side is finite, so is the left-hand side.
    In particular, if \(N\epsilon_0\geq 2\kappa l'\), \(\forall u\in \FSpace{N}\),
    \begin{equation}\label{Eqn::Result::BoundSobolevByHighPowers::EstimateWithJustL2}
        \HsNorm*{\phi_1 u}{2\kappa l'}\leq C
        \sum_{k=0}^{N} \LtNorm*{\phi_2 \left( \partial_t+\opL \right)^k u}.
    \end{equation}
\end{corollary}
\begin{proof}
    Take \(l+1=2\kappa l'\) and \(J=l'\) in Theorem \ref{Thm::Result::MainThm::New}.
\end{proof}

\begin{corollary}\label{Cor::Result::partialtPlusOpLInfinityGivesSmooth}
    \(\forall u\in \FSpace{\infty}:=\bigcap_{l\in \Zg} \FSpace{l}\), 
    \(u\big|_{\R\times\Omega}\in \CinftySpace[\R\times \Omega][\C^M]\).
    Furthermore
    \(\forall \phi_1,\phi_2\in \CinftycptSpace[\R\times \Omega]\) with \(\phi_1\prec \phi_2\),
    \(\forall l\in \Zg\), \(\exists N\in \Zg\), \(\exists C\geq 0\),
    \(\forall u\in \FSpace{\infty}\),
    \begin{equation}\label{Eqn::Result::partialtPlusOpLInfinityGivesSmooth::Estimate}
        \CjNorm{\phi_1 u}{l}
        \leq C \sum_{k=0}^N \LtNorm*{\phi_2 \left( \partial_t +\opL \right)^k u}.
    \end{equation} 
\end{corollary}
\begin{proof}
    Given \(l\in \Zg\), the Sobolev embedding theorem shows \(\exists l'\in \Zg\)
    with
    \begin{equation}\label{Eqn::Result::partialtPlusOpLInfinityGivesSmooth::Tmp1}
        \CjNorm{\phi_1 u}{l} \lesssim \HsNorm{\phi_1 u}{2\kappa l'},
    \end{equation}
    where if the right-hand side is finite, then \(\phi_1 u \in \CjSpace{l}\).
    From here, \eqref{Eqn::Result::partialtPlusOpLInfinityGivesSmooth::Estimate}
    follows from \eqref{Eqn::Result::BoundSobolevByHighPowers::EstimateWithJustL2}.
    That \(\phi_1 u\in \CinftySpace[\R\times \Omega][\C^M]\)
    follows from \eqref{Eqn::Result::partialtPlusOpLInfinityGivesSmooth::Tmp1}
    applied to all \(l\in \Zg\).
    Since \(\phi_1\in \CinftycptSpace[\R\times \Omega]\)
    was arbitrary,
    we have \(u\big|_{\R\times\Omega}\in \CinftySpace[\R\times \Omega][\C^M]\),
    completing the proof.
\end{proof}

The above corollaries apply to the operator \(\partial_t+\opL\). We obtain similar results
for \(\opL\), as the next corollaries show.
Our first result shows an analog of Theorem \ref{Thm::Result::MainThm::New}
for \(\opL\).  For these, we use the spaces \(\HscptSpace{s}[\ManifoldN;\C^M]\), which are the usual
(isotropic) Sobolev spaces on \(\ManifoldN\).  As before, the 
equivalence class of the norm \(\HsNorm{f}{s}\) is well-defined for \(f\) supported in a fixed compact set;
and if the norm is finite, the function belongs to \(\HscptSpace{s}[\ManifoldN;\C^M]\).

\begin{corollary}\label{Cor::Result::MainCorWithoutDt}
    For all \(N\in \Zg\), \(\forall \phi_1,\phi_2\in \CinftycptSpace[\Omega]\) with \(\phi_1\prec \phi_2\),
    \(\forall l\in \Zg\) with \(l\geq 2\kappa-1\), \(\forall J\in \Zg\) with \(1\leq J\leq \min\left\{ \floor{\frac{l+1}{2\kappa}}, N \right\}\),
    \(\exists C\geq 0\), \(\forall f\in \DomainLN{N}\),
    \begin{equation}\label{Eqn::Result::MainCorWithoutDt::MainEqn}
        \HsNorm*{\phi_1 f}{l+1} \leq C
        \left( 
            \HsNorm*{\phi_2 \opL^J f}{l+1-2\kappa J} + \HsNorm*{\phi_2 \opL^N f}{(l+1-N\epsilon_0)\vee 0}
            +\sum_{k=0}^{N-1} \LtNorm*{\phi_2 \opL^k f}
         \right),
    \end{equation}
    where if the right-hand side is finite, so is the left-hand side.
\end{corollary}
\begin{proof}
    Let \(\eta_1,\eta_2,\eta_3\in \CinftycptSpace[\R]\) with \(\eta_1\prec \eta_2\prec \eta_3\) and \(\eta_1=1\) on a neighborhood of \(0\).
    For \(f\in \DomainLN{N}\), set \(u(t,x):=\eta_3(t) f(x)\in \FSpace{N}\). 
    Let \(\phih_1(t,x):=\eta_1(t)\phi_1(x)\) and \(\phih_2(t,x):=\eta_2(t)\phi_2(x)\), so that \(\phih_1\prec \phih_2\).
    Theorem \ref{Thm::Result::MainThm::New}
    shows \eqref{Eqn::Result::MainThm::MainEqn} holds for this choice of \(u\), and with \(\phi_1\) and \(\phi_2\)
    replaced by \(\phih_1\) and \(\phih_2\), respectively.
    By the above choices, this implies \eqref{Eqn::Result::MainCorWithoutDt::MainEqn} and completes the proof.
\end{proof}

Similar to the above corollaries about \(\partial_t+\opL\),
we can use Corollary \ref{Cor::Result::MainCorWithoutDt} to deduce results about \(\opL\).
We record these here.  This first shows \(\opL\) is subelliptic on \(\Omega\).

\begin{corollary}\label{Cor::Result::SubellipWithoutDt}
    \(\forall \phi_1,\phi_2\in \CinftycptSpace[\Omega]\) with \(\phi_1\prec \phi_2\),
    \(\forall l\in \Zg\) with \(l\geq 2\kappa-1\), \(\exists C\geq 0\), \(\forall f\in \DomainL\),
    \begin{equation*}
        \HsNorm*{\phi_1 f}{l+1}
        \leq C
        \left( 
            \HsNorm*{\phi_2 \opL f}{l+1-\epsilon_0}
            +\LtNorm*{\phi_2 f}
         \right),
    \end{equation*}
    where if the right-hand side is finite, so is the left-hand side.
\end{corollary}
\begin{proof}
    Take \(N=J=1\) in Corollary \ref{Cor::Result::MainCorWithoutDt}
\end{proof}

The next corollary shows \(\opL\) is hypoelliptic on \(\Omega\).
\begin{corollary}\label{Cor::Result::HypoellipWithoutDt}
    If \(f\in \DomainL\) is such that \(\opL f\) is smooth near some \(x_0\in \Omega\), then \(f\)
    is smooth near \(x_0\).
\end{corollary}
\begin{proof}
    This follows from Corollary \ref{Cor::Result::SubellipWithoutDt} applied to every \(l\in \Zg\) with \(l\geq 2\kappa-1\).
\end{proof}

\begin{corollary}\label{Cor::Result::BoundSobolevByDerivsWithoutDt}
    \(\forall \phi_1,\phi_2\in \CinftycptSpace[\Omega]\) with \(\phi_1\prec \phi_2\), \(\forall N\in \Zg\),
    \(\forall l'\in \left\{ 1,2,\ldots, N \right\}\), \(\exists C\geq 0\), \(\forall f\in \DomainLN{N}\),
    \begin{equation*}
        \HsNorm*{\phi_1 f}{2\kappa l'} \leq C
        \left( 
            \HsNorm*{\phi_2 \opL^N f}{(2\kappa l'-N\epsilon_0)\vee 0}
            +\sum_{k=0}^{N-1} \LtNorm*{\phi_2 \opL^k f}
         \right),
    \end{equation*}
    where if the right-hand side is finite, so is the left-hand side. In particular, if \(N\epsilon_0\geq 2\kappa l'\),
    \(\forall f\in \DomainLN{N}\),
    \begin{equation*}
        \HsNorm*{\phi_1 f}{2\kappa l'} \leq C \sum_{k=0}^N \LtNorm*{\phi_2 \opL^k f}.
    \end{equation*}
\end{corollary}
\begin{proof}
    Take \(l+1=2\kappa l'\) and \(J=l'\) in Corollary \ref{Cor::Result::MainCorWithoutDt}.
\end{proof}

\begin{corollary}\label{Cor::Result::OpLInfinityGivesSmoothWithoutPartialT}
    \(\forall f\in \DomainLN{\infty}\), \(f\big|_{\Omega}\in \CinftySpace[\Omega][\C^M]\).
    Furthermore, \(\forall \phi_1,\phi_2\in \CinftycptSpace[\Omega]\) with \(\phi_1\prec \phi_2\),
    \(\forall l\in \Zg\), \(\exists N\in \Zg\), \(\exists C\geq 0\), \(\forall f\in \DomainLN{\infty}\),
    \begin{equation*}
        \CjNorm*{\phi_1 f}{l}\leq C \sum_{k=0}^N \LtNorm*{\phi_2 \opL^k f}.
    \end{equation*}
\end{corollary}
\begin{proof}
    This follows from Corollary \ref{Cor::Result::BoundSobolevByDerivsWithoutDt} in the same way
    Corollary \ref{Cor::Result::partialtPlusOpLInfinityGivesSmooth} follows from Corollary \ref{Cor::Result::BoundSobolevByHighPowers}.
\end{proof}

\begin{corollary}\label{Cor::Result::EigenvectorsSmooth}
    Let \(f\in \DomainL\) satisfy \(\opL f=\lambda f\), where \(\lambda\in \C\). Then,
    \(f\big|_{\Omega}\in \CinftySpace[\Omega]\). Furthermore, \(\forall \phi\in \CinftycptSpace[\Omega]\),
    \(\forall l\in \Zg\), \(\exists N\in \Zg\), \(\exists C\geq 0\), \(\forall \lambda \in \C\), \(\forall f\in \DomainL\) with \(\opL f=\lambda f\),
    \begin{equation*}
        \CjNorm{\phi f}{l}\leq C (1+|\lambda|)^N \LtNorm{f}.
    \end{equation*}
\end{corollary}
\begin{proof}
    For such \(f\), we have \(f\in \DomainLN{\infty}\), and therefore \(f\big|_{\Omega}\in \CinftySpace[\Omega]\)
    by Corollary \ref{Cor::Result::OpLInfinityGivesSmoothWithoutPartialT}.
    Letting \(\phi_2\in \CinftycptSpace[\Omega]\) with \(\phi\prec \phi_2\), 
    by Corollary \ref{Cor::Result::OpLInfinityGivesSmoothWithoutPartialT}, we have
    \(\exists N\in \Zg\),
    \(\forall f\in \DomainL\)
    with \(\opL f=\lambda f\), 
    \begin{equation*}
        \CjNorm{\phi f}{l}\lesssim \sum_{k=0}^N |\lambda|^k \LtNorm*{\phi_2 f} \lesssim (1+|\lambda|)^N \LtNorm{f},
    \end{equation*}
    as desired.
\end{proof}

\section{Checking on a core}\label{Section::MainCor}
Theorem \ref{Thm::Result::MainThm::New} involves assumptions on all of \(\DomainQ\).
Often \(\QDomainQ\) will be initially defined on a core and/or the relevant properties
will be easier to check on a core,
where the core is usually given by some subspace of smooth functions.  In this section, we present
a corollary of Theorem \ref{Thm::Result::MainThm::New} (Corollary \ref{Cor::Core::MainCor}), where the relevant assumptions
are given on a core.

As in Section \ref{Section::MainResult}, let \(\ManifoldN\)
be a smooth manifold with boundary of dimension \(n\geq 2\), \(\Vol\) a smooth, strictly positive density on \(\ManifoldN\),
and fix \(M\in \Zg\).  Let \(\QDomainQ\) be a closed, densely defined, sectorial,
sesqui-linear form on \(\LtSpace*[\ManifoldN][\Vol][\C^M]\).
Let \(\LDomainL\) be the unique, unbounded, densely defined, m-sectorial operator on \(\LtSpace*[\ManifoldN][\Vol][\C^M]\)
associated to \(\QDomainQ\) (see Section \ref{Section::MainResult} for details).

Fix \(\Omega\Subset\Omegat\subseteq \ManifoldN\) open sets, and let \(W_1,\ldots, W_r\) be smooth vector fields
on \(\Omegat\) satisfying H\"ormander's condition of order \(m\in \Zg\).
Fix \(a_{\alpha,\beta}\in \CinftySpace[\Omegat][\MatrixSpace[M][M][\C]]\), for \(|\alpha|,|\beta|\leq \kappa\),
and define \(\FormQF\) as in \eqref{Eqn::Result::FormQFFormula}.

Let \(\CoreB\subseteq \DomainQ\) be a core for  \(\QDomainQ\) (see \cite[Chapter 6, Section 1.4]{KatoPerturbationTheory}).

\begin{assumption}\label{Assumption::Core::ExistsPsi}
    There exists \(\psi\in \CinftycptSpace[\Omegat][\R]\) with \(\psi=1\) on a neighborhood of \(\overline{\Omega}\)
    such that \(\Mult{\psi}:f\mapsto \psi f\) maps \(\CoreB\rightarrow \CoreB\) and satisfies
    \begin{equation*}
        \Real \FormQ[\psi f][\psi f] \leq C_1 \left( \left| \FormQ[f][f] \right| + \LtNorm{f}^2 \right),\quad \forall f\in \CoreB.
    \end{equation*}
\end{assumption}

Throughout the rest of the section, we let \(\psi\) be as in
Assumption \ref{Assumption::Core::ExistsPsi}.

\begin{assumption}\label{Assumption::Core::TestFunctionsInCore}
    \(\TestFunctionsZeroCM[\Omega]\subseteq \CoreB\).
\end{assumption}

\begin{assumption}\label{Assumption::Core::LocalizationOfCore}
    \(\forall f\in \CoreB\), \(\psi f\in \HsWSpace*{\kappa}[\Omegat;\C^M]\).
\end{assumption}

\begin{assumption}\label{Assumption::Core::QEqualsQF}
    \(\forall f,g\in \CoreB\), \(\FormQ[\psi f][g]=\FormQF[\psi f][g]\) and \(\FormQ[f][\psi g]=\FormQF[f][\psi g]\).
\end{assumption}

\begin{assumption}\label{Assumption::Core::MaxSubAssumption}
    \(\exists C_2\geq 0\),
    \(\forall f\in \CoreB\),
    \begin{equation*}
        \HsWNorm{\psi f}{\kappa}^2
        \leq
        C_2 \left( 
            \left| \FormQ[\psi f][\psi f] \right| + \LtNorm{\psi f}^2.
         \right)
    \end{equation*}
\end{assumption}

For functions \(K(x,y')\in \CinftycptSpace[\CubengeqOneHalf\times \CubenmoOneHalf]\)
such that \(\exists \delta>0\) with \(\partial_{x_n}K((x',x_n),y')=0\) for \(0\leq x_n<\delta\), we consider
operators of the form
\begin{equation}\label{Eqn::MainCor::FormOfS}
    S f(x',x_n)= \int K(x,y') f(y',x_n)\: dy'.
\end{equation}

\begin{assumption}\label{Asssumption::Core::Smoothing}
        For every \(x_0\in \Omega\cap \BoundaryN\), \(\exists\) a coordinate chart
    \(\Psi:\CubengeqOne\xrightarrow{\sim}\Psi\left( \CubengeqOne \right)\subseteq \Omega\)
    with \(\Psi(0)=x_0\) and a vector field \(X\) in the \(\CinftySpace[\Psi\left( \CubengeqOne \right)][\R]\)-module
    generated by \(W_1,\ldots, W_r\) such that
    \begin{enumerate}[(i)]
        \item\label{Item::Core::Smoothing::PullBackX} \(\Psi^{*}X=c_0(x)\partial_{x_n}\), where \(c_0(x)\ne 0\), \(\forall x\in \CubengeqOne\).
        \item\label{Item::Core::Smoothing::SOperator} Let \(S\) be any operator as described in \eqref{Eqn::MainCor::FormOfS}, and set \(T=\Psi_{*}S\Psi^{*}\).
            Then,
            \begin{equation}\label{Eqn::Core::Smothing::CoreBToDomainQ}
                T:\CoreB\rightarrow \DomainQ.
            \end{equation}
    \end{enumerate}
\end{assumption}

\begin{remark}
    In \eqref{Eqn::Core::Smothing::CoreBToDomainQ},
    it is often easier to verify the stronger hypothesis \(T:\CoreB\rightarrow \CoreB\). Since \(\CoreB\subseteq \DomainQ\), this suffices.
\end{remark}

\begin{corollary}\label{Cor::Core::MainCor}
    Under the above assumptions, the conclusions of 
    all the results in Section \ref{Section::MainResult} hold.
    I.e., the conclusions of Theorem \ref{Thm::Result::MainThm::New}
    and Corollaries 
    \ref{Cor::Result::PartialtPluOpLSubEllip},
    \ref{Cor::Result::PartialtPluOpLHypoEllip},
    \ref{Cor::Result::BoundSobolevByHighPowers},
    \ref{Cor::Result::partialtPlusOpLInfinityGivesSmooth},
    \ref{Cor::Result::MainCorWithoutDt},
    \ref{Cor::Result::SubellipWithoutDt},
    \ref{Cor::Result::HypoellipWithoutDt},
    \ref{Cor::Result::BoundSobolevByDerivsWithoutDt},
    \ref{Cor::Result::OpLInfinityGivesSmoothWithoutPartialT},
    and \ref{Cor::Result::EigenvectorsSmooth}
    hold.
    The conclusions of those results also hold with \(\LDomainL\) replaced by
    \((\opL^{*}, \Domain{\opL^{*}})\), throughout.
\end{corollary}

The assumptions of this section remain unchanged if \(\QDomainQ\) is replaced by
\((\FormQ^{*}, \DomainQ)\), where \(\FormQ^{*}(f,g)=\overline{\FormQ[g][f]}\);
in particular, \(\CoreB\) is a core for \((\FormQ^{*}, \DomainQ)\) if and only if it is a core
for \(\QDomainQ\).
Since replacing \(\QDomainQ\) with \((\FormQ^{*}, \DomainQ)\) has the effect of
replacing \(\LDomainL\) with \((\opL^{*}, \Domain{\opL^{*}})\)
(see \cite[Chapter 6, Section 2, Theorem 2.5]{KatoPerturbationTheory}),
once we show Corollary \ref{Cor::Core::MainCor} holds for \(\LDomainL\),
it follows that it also holds for \((\opL^{*}, \Domain{\opL^{*}})\).
With this in mind,
Corollary \ref{Cor::Core::MainCor} follows immediately from the next proposition.

\begin{proposition}\label{Prop::Core::AssumptionsHold}
    Under the above assumptions, the assumptions of Section \ref{Section::MainResult} hold.
\end{proposition}

The remainder of this section is devoted to the proof of Proposition \ref{Prop::Core::AssumptionsHold}.
Since \(\CoreB\) is a core for \(\QDomainQ\), \(\CoreB\) is a dense subspace of
\(\DomainQ\) (with the Hilbert space topology described in \eqref{Eqn::Result::DomainQHilbertNorm}). In fact, this is the definition
of \(\CoreB\) being a core for \(\QDomainQ\); see \cite[Chapter 6, Section 1.3]{KatoPerturbationTheory}.

\begin{lemma}\label{Lemma::Core::MultPsiIsCont}
    \(f\mapsto \psi f\) is a continuous map \(\DomainQ\rightarrow \DomainQ\).
\end{lemma}
\begin{proof}
    Assumption \ref{Assumption::Core::ExistsPsi} and \eqref{Eqn::Core::QNormSameAsFormQ}
    show
    \begin{equation*}
        \QNorm{\psi f}\lesssim \QNorm{f},\quad \forall f\in \CoreB.
    \end{equation*}
    Since \(\CoreB\) is dense in \(\DomainQ\), the result follows.
\end{proof}

\begin{lemma}\label{Lemma::Core::QNormIsHsWNormOnCoreB}
    \begin{equation*}
        \QNorm{\psi f}\approx \HsWNorm{\psi f}{\kappa},\quad \forall f\in \CoreB.
    \end{equation*}
\end{lemma}
\begin{proof}
    By Assumptions \ref{Assumption::Core::ExistsPsi}, \ref{Assumption::Core::LocalizationOfCore}, \ref{Assumption::Core::QEqualsQF} and \eqref{Eqn::Core::QNormSameAsFormQ},
    \begin{equation*}
        \QNorm{\psi f}^2
        \approx \left| \FormQ[\psi f][\psi f] \right| + \LtNorm{\psi f}^2
        =\left| \FormQF[\psi f][\psi f] \right| + \LtNorm{\psi f}^2
        \lesssim \HsWNorm{\psi f}{\kappa}^2,
    \end{equation*}
    where the final \(\lesssim\) used the formula for \(\FormQF\): \eqref{Eqn::Result::FormQFFormula}.

    By Assumption \ref{Assumption::Core::MaxSubAssumption} and \eqref{Eqn::Core::QNormSameAsFormQ},
    \begin{equation*}
        \HsWNorm{\psi f}{\kappa}^2
        \lesssim \left| \FormQ[\psi f][\psi f] \right| + \LtNorm{\psi f}^2
        \approx \QNorm{\psi f}^2,
    \end{equation*}
    completing the proof.
\end{proof}

\begin{lemma}\label{Lemma::Core::QNormIsHkappaWNormOnDomainQ}
    \(\forall f\in \DomainQ\), \(\psi f\in \HsWSpace*{\kappa}[\Omegat; \C^M]\) and
    \begin{equation*}
        \QNorm{\psi f}\approx \HsWNorm{\psi f}{\kappa},\quad \forall f\in \DomainQ.
    \end{equation*}
\end{lemma}
\begin{proof}
    Let \(f\in \DomainQ\). Take \(f_j\in \CoreB\) such that \(f_j\rightarrow f\) in \(\DomainQ\).
    By Assumption \ref{Assumption::Core::ExistsPsi}, \(\psi f_j\in \CoreB\), and by Lemma \ref{Lemma::Core::MultPsiIsCont},
    \(\psi f_j\rightarrow \psi f\) in \(\DomainQ\).
    Therefore \(\psi f_j\) is Cauchy in \(\DomainQ\), and Lemma \ref{Lemma::Core::QNormIsHsWNormOnCoreB}
    implies \(\psi f_j\) is Cauchy in \(\HsWSpace*{\kappa}[\Omegat;\C^M]\).
    We conclude \(\psi f_j\rightarrow \psi f\)in \(\HsWSpace*{\kappa}[\Omegat;\C^M]\) and therefore
    \(\psi f\in \HsWSpace*{\kappa}[\Omegat;\C^M]\) and
    \begin{equation*}
        \HsWNorm{\psi f}{\kappa}
        =\lim_{j\rightarrow \infty} \HsWNorm{\psi f_j}{\kappa}
        \approx\lim_{j\rightarrow \infty} \QNorm{\psi f_j}
        =\QNorm{\psi f},
    \end{equation*}
    where the \(\approx\) uses Lemma \ref{Lemma::Core::QNormIsHsWNormOnCoreB}.
\end{proof}

\begin{lemma}\label{Lemma::Core::MultPsiIsContDomQToHW}
    \(f\mapsto \psi f\) is a continuous map \(\DomainQ\rightarrow \HsWSpace*{\kappa}[\Omegat;\C^M]\).
\end{lemma}
\begin{proof}
    For \(f\in \DomainQ\), \(\psi f\in \HsWSpace*{\kappa}[\Omegat;\C^M]\)
    by Lemmas \ref{Lemma::Core::QNormIsHkappaWNormOnDomainQ} and \ref{Lemma::Core::MultPsiIsCont} we have
    \begin{equation*}
        \HsWNorm{\psi f}{\kappa}\approx \QNorm{\psi f} \lesssim \QNorm{f}.
    \end{equation*}
\end{proof}

\begin{lemma}\label{Lemma::Core::FirstAssumptionsHold}
    Assumptions \ref{Asssumption::Result::TestFunctionsInDomain}, \ref{Asssumption::Result::DomainIsHkappaW},
    and \ref{Asssumption::Result::QIsQF}
    hold.
\end{lemma}
\begin{proof}
    Assumption \ref{Asssumption::Result::TestFunctionsInDomain}:
    By Assumption \ref{Assumption::Core::TestFunctionsInCore},
    we have \(\TestFunctionsZeroCM[\Omega]\subseteq \CoreB\subseteq \DomainQ\).

    Assumption \ref{Asssumption::Result::DomainIsHkappaW}:
    The map \(g\mapsto g\big|_{\Omega}\) is continuous \(\HsWSpace*{\kappa}[\Omegat;\C^M]\rightarrow \HsWSpace*{\kappa}[\Omega;\C^M]\hookrightarrow \HsWlocSpace{\kappa}[\Omega;\C^M]\).
    By Lemma \ref{Lemma::Core::MultPsiIsContDomQToHW}, \(f\mapsto \psi f\) is continuous \(\DomainQ\rightarrow \HsWSpace*{\kappa}[\Omegat;\C^M]\).
    Combining these we see \(f\mapsto f\big|_{\Omega}= \psi f\big|_{\Omega}\) is continuous \(\DomainQ\rightarrow \HsWlocSpace{\kappa}[\Omega;\C^M]\).

    Assumption \ref{Asssumption::Result::QIsQF}:
    By the Leibniz rule, there is a continuous sesquilinear form
    \(\FormRpsi:\HsWSpace{\kappa}[\Omegat;\C^M]\times \HsWSpace{\kappa}[\Omegat;\C^M]\rightarrow \C\)
    such that
    \begin{equation*}
        \FormQF[\psi^{2\kappa+2}u][v]=\FormRpsi[\psi u][\psi v].
    \end{equation*}
    Indeed, we have
    \begin{equation*}
        \FormQF[\psi^{2\kappa+2}u ][v]=\sum_{|\alpha|,|\beta|\leq \kappa} \Ltip*{W^{\alpha} (\psi u)}{c_{\alpha,\beta} W^{\beta}(\psi v)}\eqqcolon \FormRpsi[\psi u][\psi v],
    \end{equation*}
    for some smooth, compactly supported \(c_{\alpha,\beta}\).

    Let \(f,g\in \DomainQ\) such that \(\supp(f)\subseteq \Omega\). Note that \(\psi f=f\).
    Take \(f_j,g_k\in \CoreB\) with \(f_j\rightarrow f\), and \(g_k\rightarrow g\) in \(\DomainQ\).
    By Lemma \ref{Lemma::Core::MultPsiIsCont}, \(\psi f_j\rightarrow \psi f=f\), \(\psi^{2\kappa+2}f_j\rightarrow \psi^{2\kappa+2}f=f\),
    and \(\psi g_k\rightarrow \psi g\) in \(\DomainQ\).
    By Lemma \ref{Lemma::Core::QNormIsHkappaWNormOnDomainQ}, \(f=\psi f, \psi g\in \HsWSpace{\kappa}[\Omegat;\C^M]\)
    and
    by Lemma \ref{Lemma::Core::MultPsiIsContDomQToHW},
    \(\psi f_j\rightarrow \psi f=f\) and \(\psi g_k\rightarrow \psi g\) in \(\HsWSpace{\kappa}[\Omegat;\C^M]\).
    We have, using Assumption \ref{Assumption::Core::QEqualsQF},
    \begin{align*}
    &\FormQ[f][g]=\FormQ[\psi^{2\kappa+2} f][g] && \text{Using }\supp(f)
    \\&=\lim_{j,k\rightarrow \infty} \FormQ[\psi^{2\kappa+2} f_j][g_k] & & \text{Convergence in }\DomainQ 
    \\&=\lim_{j,k\rightarrow \infty} \FormQF[\psi^{2\kappa+2} f_j][g_k]  & & \text{Assumption \ref{Assumption::Core::QEqualsQF}}
    \\&=\lim_{j,k\rightarrow \infty}\FormRpsi[\psi f_j][\psi g_k] && \text{Definition of }\FormRpsi
    \\&=\FormRpsi[\psi f][\psi g] && \text{Convergence in }\HsWSpace{\kappa}[\Omegat]
    \\&=\FormQF[\psi^{2\kappa+2}f][g] && \text{Definition of }\FormRpsi
    \\&=\FormQF[f][g] && \text{Using }\supp(f).
    \end{align*}

\end{proof}

It remains to establish Assumption \ref{Asssumption::Result::BoundaryAssumption::New}.
Let \(\Psi\) be as in Assumption \ref{Asssumption::Core::Smoothing}.

\begin{lemma}\label{Lemma::Cor::SmoothingOpIsCont}
    Let \(K(x,y')\in \CinftycptSpace[\CubengeqOneHalf\times \CubenmoOneHalf]\), and set
    \begin{equation}\label{Eqn::Cor::SmoothingOpIsCont::DefineS}
        S f(x',x_n):=\int K(x,y') f(y',x_n)\: dy',\quad T:=\Psi_{*} S \Psi^{*}.
    \end{equation}
    Then, \(T:\HsWSpace{\kappa}[\Omegat]\rightarrow \HsWSpace{\kappa}[\Omega]\) continuously.
\end{lemma}
\begin{proof}
    Let \(X\) be as in Assumption \ref{Asssumption::Core::Smoothing}.
    We will show
    \begin{equation}\label{Eqn::Cor::SmoothingOpIsCont::ToShow}
        \sum_{|\alpha|\leq \kappa} \LtNorm*{W^{\alpha} T f}[\ManifoldN]\lesssim
        \sum_{j=0}^{\kappa} \LtNorm{X^j f}[\Psi( \CubengeqOneHalf )].
    \end{equation}
    To see why \eqref{Eqn::Cor::SmoothingOpIsCont::ToShow} completes the proof, note
    that 
    \(\supp(Tf)\subseteq\Psi(\CubengeqOneHalf )\subseteq \Omega\) and
    \eqref{Eqn::Cor::SmoothingOpIsCont::ToShow} implies
    \begin{equation*}
        \HsWNorm{T f}{\kappa} = \sum_{|\alpha|\leq \kappa} \LtNorm*{W^{\alpha} T f}[\ManifoldN]\lesssim
        \sum_{j=0}^{\kappa} \LtNorm{X^j f}[\Psi(\CubengeqOneHalf )]
        \lesssim \sum_{|\alpha|\leq \kappa}\LtNorm*{W^{\alpha} f}[\Psi(\CubengeqOneHalf )]
        \leq \HsWNorm{f}{\kappa}.
    \end{equation*}

    We turn to establishing \eqref{Eqn::Cor::SmoothingOpIsCont::ToShow}.  
    Since \(\Vol\) is a smooth, strictly positive density and \(\Psi(\CubengeqOneHalf)\Subset \ManifoldN\),
    we have
    \begin{equation*}
        \LtNorm{f}[\Psi(\CubengeqOneHalf)][\Vol]\approx \LtNorm{\Psi^{*} f}[\CubengeqOneHalf],
    \end{equation*}
    where the \(\LtNorm{\cdot}\) on the right-hand side uses Lebesgue measure.
    We conclude, using the form of \(S\) (see \eqref{Eqn::Cor::SmoothingOpIsCont::DefineS})
    and Assumption \ref{Asssumption::Core::Smoothing}\ref{Item::Core::Smoothing::PullBackX},
    \begin{equation*}
    \begin{split}
         &\sum_{|\alpha|\leq \kappa} \LtNorm*{W^{\alpha}T f}[\ManifoldN]
         =\sum_{|\alpha|\leq \kappa} \LtNorm*{W^{\alpha}T f}[\Psi(\CubengeqOneHalf)]
         \approx \sum_{|\alpha|\leq \kappa} \LtNorm*{\left( \Psi^{*} W \right)^{\alpha}S \Psi^{*} f}[\CubengeqOneHalf]
         \\&
         \lesssim \sum_{|\beta|\leq \kappa}\LtNorm*{\partial_x^{\beta}S \Psi^{*} f}[\CubengeqOneHalf]
         \lesssim \sum_{j=0}^{\kappa} \LtNorm*{\partial_{x_n}^j  \Psi^{*} f}[\CubengeqOneHalf]
         \approx \sum_{j=0}^{\kappa} \LtNorm*{\left( \Psi^{*} X \right)^j  \Psi^{*} f}[\CubengeqOneHalf]
         \\&\approx \sum_{j=0}^{\kappa} \LtNorm{X^j f}[\Psi(\CubengeqOneHalf)][\Vol],
    \end{split}
    \end{equation*}
    establishing \eqref{Eqn::Cor::SmoothingOpIsCont::ToShow} and
    completing the proof.
\end{proof}

\begin{lemma}\label{Lemma::Cor::SmoothingOpOnDomQ}
    Let \(T\) be any operator of the form described in Assumption \ref{Asssumption::Core::Smoothing}\ref{Item::Core::Smoothing::SOperator}.
    Then, \(T\) is a bounded operator \(T:\DomainQ\rightarrow \HsWSpace{\kappa}[\Omega;\C^M]\) and
    \(T:\DomainQ\rightarrow \DomainQ\).
\end{lemma}
\begin{proof}
    Note that, by the definition of \(T\), \(Tf=\psi T f = T\psi f\) for any \(f\).
    Using Lemmas \ref{Lemma::Cor::SmoothingOpIsCont}, \ref{Lemma::Core::QNormIsHkappaWNormOnDomainQ}, and \ref{Lemma::Core::MultPsiIsCont},
    we have for \(f\in \DomainQ\),
    \begin{equation}\label{Eqn::Cor::SmoothingOpOnDomQ::HkappaW}
        \begin{split}
            &\HsWNorm{T f}{\kappa}
            =\HsWNorm{T \psi f}{\kappa}
            \lesssim \HsWNorm{\psi f}{\kappa}
            \approx \QNorm{\psi f}
            \lesssim \QNorm{f}.
        \end{split}
    \end{equation}
    \eqref{Eqn::Cor::SmoothingOpOnDomQ::HkappaW} shows \(T:\DomainQ\rightarrow \HsWSpace{\kappa}[\Omega;\C^M]\)
    and is continuous.
    
    For \(f\in \CoreB\), since \(T:\CoreB\rightarrow  \DomainQ\), by Assumption \ref{Assumption::Core::ExistsPsi},
    we have (using Lemma \ref{Lemma::Core::QNormIsHkappaWNormOnDomainQ}),
    \begin{equation}\label{Eqn::Cor::SmoothingOpOnDomQ::Tmp1}
        \QNorm{Tf}= \QNorm{\psi Tf}\approx \HsWNorm{\psi Tf }{\kappa}=\HsWNorm{Tf }{\kappa}
    \end{equation}
    Combining \eqref{Eqn::Cor::SmoothingOpOnDomQ::Tmp1} with \eqref{Eqn::Cor::SmoothingOpOnDomQ::HkappaW}
    shows
    \begin{equation*}
        \QNorm{Tf}\lesssim \QNorm{f},\quad f\in \CoreB.
    \end{equation*}
    Since \(\CoreB\) is dense in \(\DomainQ\), it follows that \(T:\DomainQ\rightarrow \DomainQ\)
    and is continuous, completing the proof.
\end{proof}

\begin{lemma}\label{Lemma::Cor::TtMapsRightSpaces}
    Let \(\Tt\) be an operator of the form
    of the operator \(T\) described in Assumption \ref{Asssumption::Result::BoundaryAssumption::New}\ref{Item::Result::BoundaryAssumption::SOperator::New}.
    Then, \(\Tt:\LtSpace*[\R][\DomainQ]\rightarrow \CinftycptSpace*[\R][\DomainQ]\) and
                \(\Tt:\LtSpace*[\R][\DomainQ]\rightarrow \CinftycptSpace*[\R][\HsWSpace{\kappa}[\Omega;\C^M]]\),
                and these two maps are continuous.
\end{lemma}
\begin{proof}
    Let \(\Kt(t,s,x,y')\in \CinftycptSpace[\R\times \R\times \CubengeqOneHalf\times \CubenmoOneHalf]\)
be such that
\(\exists \delta>0\) with \(\partial_{x_n}K(t,s,(x',x_n),y')=0\) for \(0\leq x_n<\delta\),
and 
\begin{equation*}
    \Tt u(t,\Psi(x))=\iint \Kt(t,s,x,y')u(s,\Psi(y',x_n))\: ds\: dy'.
\end{equation*}
    Consider the operator \(U\) defined by
    \begin{equation*}
        Uu (t,s,\Psi(x)):=\int \Kt(t,s,x,y') u(s, \Psi(y',x_n))\: dy'.
    \end{equation*}
    Lemma \ref{Lemma::Cor::SmoothingOpOnDomQ} shows
    \begin{equation*}
        U:\LtSpaceNoMeasure*[\R][\DomainQ]\rightarrow \CinftycpttSpace*[\R][\LtSpaceNoMeasurecpts[\R][\DomainQ]],
    \end{equation*}
    \begin{equation*}
        U:\LtSpaceNoMeasure*[\R][\DomainQ]\rightarrow \CinftycpttSpace*[\R][\LtSpaceNoMeasurecpts[\R][\HsWSpace{\kappa}[\Omega]]],
    \end{equation*}
    continuously.
    Integrating in \(s\) gives the result.
\end{proof}


\begin{proof}[Proof of Proposition \ref{Prop::Core::AssumptionsHold}]
    In light of 
    Lemma \ref{Lemma::Core::FirstAssumptionsHold} all that remains
    is 
    to establish Assumption \ref{Asssumption::Result::BoundaryAssumption::New};
    however, this is an immediate consequence of Lemma \ref{Lemma::Cor::TtMapsRightSpaces}.
\end{proof}

\section{Function spaces and pseudodifferential operators}\label{Section::PseudodifferentialOps}
Fix \(n,\kappa\in \Zg\); we work on \(\Ropn:=\R\times \Rnmo\times \R\) with coordinates \((t,x',x_n)=(t,x)\),
and the submanifold with boundary \(\Ropngeq:=\left\{ (t,x',x_n)\in \Ropn : x_n\geq 0 \right\}\).
We define pseudodifferential operators and function spaces which treat \(\partial_t\) as an operator of order
\(2\kappa\).
Let \(\SchwartzSpace[\Ropn]\) be the usual Schwartz space on \(\Ropn\) and let
\(\SchwartzSpace[\Ropngeq]=\left\{ f\in \SchwartzSpace[\Ropn] : \supp(f)\subseteq \Ropngeq \right\}\).
Let \(\TemperedDistributions[\Ropn]\) and \(\TemperedDistributions[\Ropngeq]\) be the corresponding dual spaces.

Let \(\tau\in \R\), \(\xi'\in \Rnmo\), \(\xi_n\in \R\), and \(\xi=(\xi',\xi_n)\in \Rn\) be dual to
\(t\), \(x'\), \(x_n\), and \(x=(x',x_n)\), respectively.
Let \(\Fouriertxp\) denote the partial Fourier transform in \((t,x')\):
\begin{equation*}
    \Fouriertxp f(\tau, \xi', x_n)
    :=
    \iint f(t,x',x_n)e^{-2\pi i (t\tau+x'\cdot \xi')}\: dx'\: dt,
\end{equation*}
and let \(\Fouriertx\) denote the full Fourier transform:
\begin{equation*}
    \Fouriertx f(\tau, \xi) 
    :=
    \iint f(t,x) e^{-2\pi i(t\tau+x\cdot \xi)}\: dx\: dt.
\end{equation*}
Set
\begin{equation*}
    \JBracket*{(\tau, \xi',\xi_n)}=\JBracket*{(\tau,\xi)}:=\left( 1+|\tau|^2 + |\xi|^{4\kappa}   \right)^{1/4\kappa},
\end{equation*}
\begin{equation}\label{Eqn::PDOs::DefinetxpJBracket}
    \JBracket*{(\tau, \xi')}:=\left( 1+|\tau|^2 + |\xi'|^{4\kappa}  \right)^{1/4\kappa}.
\end{equation}

We consider the following symbol classes for pseudodifferential operators, for \(s\in \R\).
\(\Symbolstx{s}\) consists of those \(a(t,x,\tau,\xi)\in \CinftySpace[\Ropn\times \Ropn]\)
satisfying for every \(L\),
\begin{equation}\label{Eqn::PDOs::SymboltxSemiNorm}
    \SymbolstxNorm{a}{s}{L}
    :=\sum_{|\alpha_1|,|\alpha_2|,|\beta_1|,|\beta_2|\leq L}
    \sup_{t,x,\tau,\xi}
    \JBracket*{(\tau,\xi)}^{-s+|\beta_2|+2\kappa |\beta_1|}
    \left| \partial_t^{\alpha_1} \partial_x^{\alpha_2} \partial_\tau^{\beta_1} \partial_\xi^{\beta_2} a(t,x,\tau,\xi) \right|
    <\infty.
\end{equation}
\(\Symbolstxp{s}\) consists of those \(b(t,x,\tau,\xi')\in \CinftySpace[\R\times \Rngeq\times \R\times \Rnmo]\)
satisfying 
\begin{equation}\label{Eqn::PDOs::SymboltxpSemiNorm}
    \SymbolstxpNorm{b}{s}{L}:=
    \sum_{|\alpha_1|,|\alpha_2|,|\beta_1|,|\beta_2|\leq L}
    \sup_{t,x,\tau,\xi'}
    \JBracket*{(\tau,\xi')}^{-s+|\beta_2|+2\kappa |\beta_1|}
    \left| \partial_t^{\alpha_1} \partial_x^{\alpha_2} \partial_\tau^{\beta_1} \partial_{\xi'}^{\beta_2} b(t,x,\tau,\xi') \right|
    <\infty. 
\end{equation}

For \(a\in \Symbolstx{s}\), we set
\begin{equation}\label{Eqn::PDOs::DefineaOp}
    a(t,x,D_t, D_x) f(t,x):=\iint a(t,x,\tau,\xi) \Fouriertx f (\tau,\xi) e^{2\pi i (t\tau+x\cdot \xi)}\: d\tau \: d\xi,
\end{equation}
and for \(b\in \Symbolstxp{s}\), we set
\begin{equation}\label{Eqn::PDOs::DefinebOp}
    b(t,x,D_t, D_{x'}) f(t,x):=\iint b(t,x,\tau,\xi') \Fouriertxp f (\tau,\xi',x_n) e^{2\pi i (t\tau+x\cdot \xi')}\: d\tau \: d\xi'.
\end{equation}
Note that \(b(t,x,D_t, D_{x'})\) is well-defined for \(f\in \TemperedDistributions[\Ropngeq]\); though this is not true for \(a(t,x,D_t, D_x)\).
We call \(a\) and \(b\) the symbols for the operators \(a(t,x,D_t, D_x)\) and \(b(t,x,D_t, D_{x'})\), respectively.

Two special pseudodifferential operators we use are
\(\Lambdatx[s]\) with symbol \(\JBracket*{(\tau,\xi)}\in \Symbolstx{s}\) and \(\Lambdatxp[s]\) with symbol \(\JBracket*{(\tau,\xi')}\in \Symbolstxp{s}\).

We set
\begin{equation*}
    \HsNorm{f}{s}[\Ropn]:=\LtNorm*{\Lambdatx[s] f}[\Ropn],\quad \HsSpace{s}[\Ropn]:=\left\{ f\in \TemperedDistributions[\Ropn] : \HsNorm{f}{s}[\Ropn]<\infty \right\}.
\end{equation*}
Set 
\begin{equation*}
    \HsSpace{s}[\Ropngeq]:=\left\{ f\big|_{\Ropngeq} : f\in \HsSpace{s}[\Ropn] \right\},\quad
    \HsNorm{g}{s}[\Ropngeq] = \inf \left\{ \HsNorm{f}{s}[\Ropn] : f\big|_{\Ropngeq}=g \right\}.
\end{equation*}
\(\HsSpace{s}[\Ropngeq]\) can be identified with a subspace of \(\TemperedDistributions[\Ropngeq]\).
Set
\begin{equation*}
    \GsNorm{f}{s}[\Ropngeq]:=\LtNorm*{\Lambdatxp[s] f}[\Ropngeq],\quad \GsSpace{s}[\Ropngeq]:=\left\{ f\in \TemperedDistributions[\Ropngeq] : \GsNorm{f}{s}[\Ropngeq]<\infty \right\}.
\end{equation*}

\begin{remark}\label{Rmk::PDOs::DefineClassicalHsSpaceOnMfld}
Given a manifold with boundary, \(\ManifoldN\), of dimension \(n\), set
\(\TestFunctionsZero[\R\times \ManifoldN]\) to be the space of those smooth functions with compact support which vanish to infinite order
on \(\R\times \BoundaryN\). Let \(\DistributionsZero[\R\times \ManifoldN]\) be the dual space. 
Let \(\HscptSpace{s}[\R\times \ManifoldN]\) consist of those \(u\in \DistributionsZero[\R\times \ManifoldN]\)
such that \(u\) has compact support in \(\R\times \ManifoldN\), and for each \((t,x)\in \R\times \ManifoldN\),
there is a neighborhood \((t-\delta,t+\delta)\times U\) of \(t,x\), a diffeomorphism \(\Phi:U\xrightarrow{\sim} \Phi(U)\subseteq \Rngeq\),
and a function \(\phi\in \CinftycptSpace[(t-\delta,t+\delta)\times \Phi(U)]\) with \(\phi=1\) on a neighborhood of \((t,\Phi(x))\) with
\(\phi \Phi_{*} u(t,\cdot)\in \HscptSpace{s}[\Ropngeq]\).
\end{remark}

    \subsection{Notation: Estimates in terms of pseudodifferential operators near the boundary}\label{Section::PseudodifferentialOps::Bdry}
    We will prove a number of estimates involving
pseudodifferential operators supported on \(\R\times \CubengeqOne\) which are smoothing
in the \(t,x'\) variables.  In this section we introduce notation that will make these estimates easier to
state and work with.

Fix \(0<\delta_1<\delta_2<1\) and \(s\in \R\),
and open sets  \(U_1\Subset U_2 \Subset \R\). We write \(\Stxp{s}\) for an operator of the form:
\begin{enumerate}[(I)]
    \item\label{Item::PDOS::NearBdry::Formula} \(\Stxp{s}=b(t,x,D_t,D_{x'})\) for some \(b\in \bigcap_{s'\in \R} \Symbolstxp{s'}\); see \eqref{Eqn::PDOs::DefinebOp}.
    \item\label{Item::PDOS::NearBdry::Support} The Schwartz kernel of \(b(t,x,D_t,D_{x'})\) is supported in \(\left(U_1\times \CubengeqDeltaOne \right)\times \left( U_1\times \CubengeqDeltaOne \right)\).
    \item\label{Item::PDOS::NearBdry::DerivZero} \(\partial_{x_n}b(t,x,\tau,\xi')=0\) for \(|x_n|<\delta_0\) for some \(\delta_0>0\).
    \item\label{Item::PDOS::NearBdry::Estimates} The way \(s\) plays a role is only through the estimates as described below.
\end{enumerate}
We write \(\Stxpt{s}\) for the same thing, but with \(\delta_1\) and \(U_1\) replaced
by \(\delta_2\) and \(U_2\), respectively, in \ref{Item::PDOS::NearBdry::Support}.

When we write an estimate like
\begin{equation*}
    \LtNorm*{\Stxp{s}u}[\Ropngeq]\lesssim \LtNorm*{\wtEst \Stxpt{s'}v}[\Ropngeq],
\end{equation*}
it means the following.
Given any \(\Stxp{s}=b(t,x,D_t,D_{x'})\) as described in \ref{Item::PDOS::NearBdry::Formula}--\ref{Item::PDOS::NearBdry::DerivZero}
above, there exist countable collections:
\begin{equation*}
    \Stxpt{s'}[1]=b_1(t,x,D_t,D_{x'}), \Stxpt{s'}[2]=b_2(t,x,D_t,D_{x'}),\ldots,
\end{equation*}
\begin{equation*}
    \wtEst_1,\wtEst_2,\ldots\in \CinftySpace[\Ropngeq],
\end{equation*}
such that
\begin{equation*}
    \LtNorm*{\Stxp{s}u}[\Ropngeq]\leq \sum_{j}\LtNorm*{\wtEst_j \Stxpt{s'}[j]v}[\Ropngeq],
\end{equation*}
and for every \(L\), there exists \(K\) depending only on \(L\), \(s\), \(s'\), and \(n\),
such that
\begin{equation}\label{Eqn::PDOs::NearBdry::CL}
    \sup_j \CjNorm*{\wtEst_j}{L} + \sum_j \SymbolstxpNorm*{b_j}{s'}{L}\leq C_L,
\end{equation}
where \(C_L\) can be chosen to depend only on \(L\), \(s\), \(s'\), \(n\), \(\delta_1\), \(\delta_2\),
\(U_1\), \(U_2\),
and an upper bound for \(\SymbolstxpNorm{b}{s}{K}\).
\(C_L\) will also be allowed to depend on various other relevant ingredients in estimate;
see, e.g., Section \ref{Section::HorVfs}.

If we instead write
\begin{equation*}
    \LtNorm*{\Stxp{s}u}[\Ropngeq]\lesssim \LtNorm*{ \Stxpt{s'}v}[\Ropngeq],
\end{equation*}
it means the same thing but with \(\wtEst_j=1\), \(\forall j\).

For a more complicated example, we will write equations like
\begin{equation*}
      \LtNorm*{\Stxp{s}[1] u_1}[\Ropngeq]+\LtNorm*{\Stxp{s}[2] u_2}[\Ropngeq]
      \lesssim \left| \Ltip*{\wtEst\Stxpt{s'}v_1}{\wtEst\Stxpt{s'}v_2}[\Ropngeq] \right|
\end{equation*}
to mean that for every \(\Stxp{s}[1]\) and \(\Stxp{s}[2]\) as above, there exist
\(\Stxpt{s'}[0],\Stxt{s'}[1],\ldots \)
and \(\wtEst_0,\wtEst_1,\ldots\)
(as described above) such that
\begin{equation*}
    \LtNorm*{\Stx{s}[1] u_1}[\Ropngeq]+\LtNorm*{\Stx{s}[2] u_2}[\Ropngeq]
     \leq C
     \sum_{k}\left| \Ltip*{\wtEst_{2k}\Stxpt{s'}[2k]v_1}{\wtEst_{2k+1}\Stxpt{s'}[2k+1]v_2}[\Ropn] \right|.
\end{equation*}
In other words, each copy of \(\Stxpt{s'}\) and \(\wtEst\) on the right-hand side gets independently replaced by a countable collection of operators and functions of that form which we sum over.
\textbf{In particular, the operators \(\Stxpt{s'}\) and functions \(\wtEst\) may vary from line to line and term to term in our estimates, even though the notation
does not indicate the change.}

\begin{remark}\label{Rmk::PDOs::NotationNearBdry::WhyWtEst}
    The reason the \(\wtEst\) function appears is that we cannot include it in the operator
    \(\Stxpt{s'}\) without possibly destroying property \ref{Item::PDOS::NearBdry::DerivZero}.
    Property \ref{Item::PDOS::NearBdry::DerivZero} is essential to ensure the operators preserve the domain
    of the forms.
\end{remark}

The next lemma explains the main way we use the above operators to conclude results.
\begin{lemma}\label{Lemma::PDOs::NotationNearBdry::ConcludeSobolevEstimates}
    Suppose
    \begin{equation}\label{Eqn::PDOs::NotationNearBdry::AssumeTheIneq}
        \LtNorm*{\Stxp{s}u}[\Ropngeq]\lesssim \LtNorm*{\wtEst \Stxpt{s'}v}[\Ropngeq],
    \end{equation}
    and fix \(\phi_2\in \CinftycptSpace[\R\times \CubengeqOne]\) with \(\phi_2=1\) on a neighborhood of
\(\overline{U_2}\times \CubengeqDeltaTwoClosure\) and \(\phi_1\in \CinftycptSpace[U_1\times \CubengeqDeltaOne]\).
Then,
\begin{equation}\label{Eqn::PDOs::NotationNearBdry::AssumeTheIneq::Tmp3}
    \GsNorm*{\phi_1 u}{s}[\Ropngeq]\lesssim \GsNorm*{\phi_2 v}{s'}[\Ropngeq].
\end{equation}
\end{lemma}
\begin{proof}
    \eqref{Eqn::PDOs::NotationNearBdry::AssumeTheIneq} implies
    \begin{equation}\label{Eqn::PDOs::NotationNearBdry::AssumeTheIneq::Tmp1}
        \LtNorm*{\Stxp{s}u}[\Ropngeq]
        \lesssim
        \GsNorm*{\phi_2 v}{s'}[\Ropngeq].
    \end{equation}

    Take
    \begin{equation*}
        \psi_0,\psi_1
        \in\CinftycptSpace[U_1\times\CubengeqDeltaOne]
    \end{equation*}
    with
    \begin{equation*}
        \phi_1\prec\psi_0\prec\psi_1
    \end{equation*}
    and \(\partial_{x_n}\psi_j=0\) near \(x_n=0\), \(j=0,1\).
    The calculus of pseudodifferential operators gives
    \begin{equation}\label{Eqn::PDOs::NotationNearBdry::ConcludeSobolevEstimates::LocalizedSobolevEstimate}
        \GsNorm*{\phi_1 u}{s}[\Ropngeq]
        \lesssim
        \LtNorm*{\psi_1\Lambdatxp[s]\psi_0u}[\Ropngeq]
        +
        \LtNorm*{\Stxp{s}u}[\Ropngeq],
    \end{equation}
    where the final operator satisfies
    \ref{Item::PDOS::NearBdry::Formula}--%
    \ref{Item::PDOS::NearBdry::DerivZero} from
    Section \ref{Section::PseudodifferentialOps::Bdry}.

    By cutting off in the \((\tau,\xi')\)-variables, we may approximate
    \begin{equation*}
        \psi_1\Lambdatxp[s]\psi_0
    \end{equation*}
    by operators
    \begin{equation*}
        \Stxp{s}[j]=b_j(t,x,D_t,D_{x'})
    \end{equation*}
    such that
    \begin{equation*}
        \Stxp{s}[j]
        \longrightarrow
        \psi_1\Lambdatxp[s]\psi_0
    \end{equation*}
    on distributions and
    \begin{equation*}
        \sup_j\SymbolstxpNorm*{b_j}{s}{L}<\infty,
        \qquad \forall L.
    \end{equation*}
    The two cutoff functions ensure that the Schwartz kernels satisfy
    \ref{Item::PDOS::NearBdry::Support}, and their choice near \(x_n=0\)
    ensures \ref{Item::PDOS::NearBdry::DerivZero}. Thus,
    \eqref{Eqn::PDOs::NotationNearBdry::AssumeTheIneq::Tmp1} and weak
    compactness in \(\LtSpace[\Ropngeq]\) show
    \begin{equation*}
        \LtNorm*{\psi_1\Lambdatxp[s]\psi_0u}[\Ropngeq]
        \lesssim
        \GsNorm*{\phi_2 v}{s'}[\Ropngeq].
    \end{equation*}
    Combining this with
    \eqref{Eqn::PDOs::NotationNearBdry::AssumeTheIneq::Tmp1} and
    \eqref{Eqn::PDOs::NotationNearBdry::ConcludeSobolevEstimates::LocalizedSobolevEstimate}
    gives
    \eqref{Eqn::PDOs::NotationNearBdry::AssumeTheIneq::Tmp3}.
\end{proof}

The next lemma gives useful formulas for working with such operators.
\begin{lemma}\label{Lemma::PDOs::NotationNearBdry::CalcOfPDO}
    \begin{enumerate}[(i)]
        \item\label{Item::PDOs::NotationNearBdry::CalcOfPDO::Composition} \(\Stxp{s_1}[1]\Stxp{s_2}[2]=\Stxpt{s_1+s_2}\).
        \item\label{Item::PDOs::NotationNearBdry::CalcOfPDO::Adjoint} \(\left( \Stxp{s} \right)^{*} = \Stxpt{s}\).
        \item\label{Item::PDOs::NotationNearBdry::CalcOfPDO::CommutatorOfPDOs} \(\left[ \Stxp{s_1}[1], \Stxp{s_2}[2] \right]= \Stxpt{s_1+s_2-1}\).
        \item\label{Item::PDOs::NotationNearBdry::CalcOfPDO::CommutatorOfPartialt} \(\left[ \partial_t, \Stxp{s} \right]= \Stxpt{s}\).
        \item\label{Item::PDOs::NotationNearBdry::CalcOfPDO::CommutatorMultiplication}  Let \(v\in \CinftySpace[\R\times \CubengeqOne]\), and \(\Mult{v}:u\mapsto vu\).
            Then, \(\left[ \Stxp{s}, \Mult{v} \right]= \wtEst \Stxpt{s-1}\); recall,
            \(\wtEst \Stxpt{s-1}\) denotes a countable sum of such terms.
        \item\label{Item::PDOs::NotationNearBdry::CalcOfPDO::VectorField} Let \(Y\) be any smooth vector field on \(\CubengeqOne\) 
        with no \(\partial_{x_n}\) component
        (and therefore, \(Y\) also does not have a \(\partial_t\) component). Then,
            \(\left[ \Stxp{s}, Y \right]=\wtEst \Stxpt{s}\).
    \end{enumerate}
\end{lemma}
\begin{proof}
    \ref{Item::PDOs::NotationNearBdry::CalcOfPDO::Composition},
    \ref{Item::PDOs::NotationNearBdry::CalcOfPDO::Adjoint},
    and \ref{Item::PDOs::NotationNearBdry::CalcOfPDO::CommutatorOfPDOs}
    are simple consequences of the calculus of pseudodifferential operators.
    For \ref{Item::PDOs::NotationNearBdry::CalcOfPDO::CommutatorOfPartialt} note that
    if \(\Stxp{s}=b(t,x,D_t,D_{x'})\), then \(\left[ \partial_t, \Stxp{s} \right]= \left( \partial_t b \right)(t,x,D_t, D_{x'})\),
    and \ref{Item::PDOs::NotationNearBdry::CalcOfPDO::CommutatorOfPartialt} follows.

    We turn to \ref{Item::PDOs::NotationNearBdry::CalcOfPDO::CommutatorMultiplication}.
    Set \(\delta'=(\delta_1+\delta_2)/2\)
    and \(U'\) open with \(U_1\Subset U'\Subset U_2\).
    Take \(\phi\in \CinftycptSpace[U'\times \CubengeqDeltaTwo]\)
    with \(\phi=1\) on a neighborhood of \(\overline{U_1}\times \CubengeqDeltaOneClosure\).
    Note that,
    \begin{equation*}
        \Mult{\phi}\Stxp{s}=\Stxp{s}\Mult{\phi}=\Stxp{s},
    \end{equation*}
    and we may therefore replace \(v\) with 
    \(\phi v\) in the statement of \ref{Item::PDOs::NotationNearBdry::CalcOfPDO::CommutatorMultiplication}.
    As a consequence, we may without loss of generality assume
    \(v\in \CinftycptSpace[U'\times \Cubengeq{\delta'}]\).
    Fix \(\psi_1\in \CinftycptSpace[U'\times \Cubenmo{\delta'}]\), \(\psi_2\in \CinftycptSpace[\lbrack 0, \delta')]\)
    with \(\psi_1(t,x')\psi_2(x_n)v(t,x',x_n)=v(t,x',x_n)\).

    Extend \(v\) to a smooth function \(\vt\in \CinftycptSpace[\Ropn]\) (see \cite{SeeleyExtensionOfCInfinityFunctionsDefinedInAHalfSpace}).
    By a standard result, we may decompose
    \begin{equation*}
        \vt(t,x)=\sum_{j=1}^{\infty} f_1^j(t,x')f_2^j(x_n),
    \end{equation*}
    where \(\CjNorm{f_1^j}{L}+\CjNorm{f_2^j}{L}\lesssim j^{-N}\), \(\forall N,L\); to see this,
    decompose \(\vt(t,x)\) into a Fourier series on a large cube, and multiply by cutoff functions.
    Set \(v_1^j(t,x')=\psi_1(t,x') f_1^j(t,x')\), \(v_2^j(x_n)=\psi_2(x_n)f_2^j(x_n)\).
    We have, for \(x_n\geq 0\),
    \begin{equation*}
        v(t,x',x_n)=\sum_{j=1}^\infty v_1^j(t,x')  v_2^j(x_n).
    \end{equation*}
    We have,
    \begin{equation*}
    \begin{split}
         &\left[ \Stxp{s}, \Mult{v} \right]
         =\sum_{j=1}^\infty\left[ \Stxp{s}, \Mult{v_1^j}\Mult{v_2^j} \right]
         \\&=\sum_{j=1}^\infty\Mult{v_2^j} \left[ \Stxp{s}, \Mult{v_1^j} \right]
         =\sum_{j=1}^\infty\Mult{v_2^j} \Stxpt{s-1}[j],
    \end{split}
    \end{equation*}
    which establishes \ref{Item::PDOs::NotationNearBdry::CalcOfPDO::CommutatorMultiplication}.

    \ref{Item::PDOs::NotationNearBdry::CalcOfPDO::VectorField} follows from 
    \ref{Item::PDOs::NotationNearBdry::CalcOfPDO::CommutatorMultiplication} and the immediate fact
    that \(\left[ \Stxp{s}, \partial_{x_j} \right]=\Stxpt{s}\).
\end{proof}

    \subsection{Notation: Estimates in terms of pseudodifferential operators on the interior}\label{Section::PseudodifferentialOps::Interior}
    When working on interior estimates, we do not need to distinguish the \(x_n\) variable
and do not require a property like  \ref{Item::PDOS::NearBdry::DerivZero},
and can instead work with standard pseudodifferential operators in all variables.
We use a similar notation to Section \ref{Section::PseudodifferentialOps::Bdry},
though it is simpler in this case.

Fix open sets \(U_1\Subset U_2\Subset \Ropn\). We write \(\Stx{s}\) for an operator of the form:
\begin{enumerate}[(A)]
    \item\label{Item::PDOS::Interior::Formula} \(\Stx{s}=a(t,x,D_t,D_{x})\) for some \(a\in \bigcap_{s'\in \R} \Symbolstx{s'}\); see \eqref{Eqn::PDOs::DefineaOp}.
    \item\label{Item::PDOS::Interior::Support} The Schwartz kernel of \(a(t,x,D_t,D_{x})\) is supported in \(U_1\times U_1\).
    \item\label{Item::PDOS::Interior::Estimates} The way \(s\) plays a role is only through the estimates as described below.
\end{enumerate}
We write \(\Stxt{s}\) for the same thing, but with \(U_1\) replaced
by \(U_2\) in \ref{Item::PDOS::Interior::Support}.

Similar to Section \ref{Section::PseudodifferentialOps::Bdry}, when we write estimates like
\begin{equation*}
    \LtNorm*{\Stx{s}u}[\Ropn]\lesssim \LtNorm*{\Stxt{s'}v}[\Ropn],
\end{equation*}
it means the following.  Given any \(\Stx{s}=a(t,x,D_t,D_x)\) as described in \ref{Item::PDOS::Interior::Formula}
and \ref{Item::PDOS::Interior::Support}, there exists a finite\footnote{Unlike Section \ref{Section::PseudodifferentialOps::Bdry},
we only require finite collections when working on the interior. 
This is because we do not need to separate the \(\wtEst\) function 
(because we do not require \ref{Item::PDOS::NearBdry::DerivZero} to hold--see Remark \ref{Rmk::PDOs::NotationNearBdry::WhyWtEst}).
This means that we do not need a countable sum to obtain results like Lemma \ref{Lemma::PDOs::NotationNearBdry::CalcOfPDO}\ref{Item::PDOs::NotationNearBdry::CalcOfPDO::CommutatorMultiplication}.
This is not an essential point: allowing countable sums in our notation would not change the proofs.} 
collection
\(\Stxt{s'}[1]=a_1(t,x,D_t,D_x),\ldots, \Stxt{s'}[L]=a_L(t,x,D_t,D_x)\), as described above, with
\begin{equation*}
    \LtNorm*{\Stx{s}u}[\Ropn]\leq \sum_{j=1}^L\LtNorm*{\Stxt{s'}[j]v}[\Ropn],
\end{equation*}
and for every \(L\), there exists \(K\) depending only on \(L\), \(s\), \(s'\), and \(n\),
such that
\begin{equation}\label{Eqn::PDOs::NotationInterior::CL}
    \sum_{j=1}^L \SymbolstxNorm*{a_j}{s'}{L}\leq C_L,
\end{equation}
where \(C_L\) can be chosen to depend only on \(L\), \(s\), \(s'\), \(n\), \(U_1\), \(U_2\),
and an upper bound for \(\SymbolstxNorm{a}{s}{K}\)
(and any other relevant ingredients in the estimate).

\begin{lemma}\label{Lemma::PDOs::NotationInterior::ConcludeSobolevEstimates}
    Suppose
    \begin{equation}
        \LtNorm*{\Stx{s}u}[\Ropn]\lesssim \LtNorm*{\Stxt{s'}v}[\Ropn],
    \end{equation}
    and fix \(\phi_2\in \CinftycptSpace[\Ropn]\) with \(\phi_2=1\) on a neighborhood of
\(\overline{U_2}\) and \(\phi_1\in \CinftycptSpace[U_1]\).
Then,
\begin{equation}
    \HsNorm*{\phi_1 u}{s}[\Ropn]\lesssim \HsNorm*{\phi_2 v}{s'}[\Ropn].
\end{equation}
\end{lemma}
\begin{proof}
    This follows from 
    a simpler reprise of the proof of Lemma \ref{Lemma::PDOs::NotationNearBdry::ConcludeSobolevEstimates}.
\end{proof}

\begin{lemma}\label{Lemma::PDOs::NotationInterior::CalcOfPDO}
    \begin{enumerate}[(i)]
        \item\label{Item::PDOs::NotationInterior::CalcOfPDO::Composition} \(\Stx{s_1}[1]\Stx{s_2}[2]=\Stxt{s_1+s_2}\).
        \item\label{Item::PDOs::NotationInterior::CalcOfPDO::Adjoint} \(\left( \Stx{s} \right)^{*} = \Stxt{s}\).
        \item\label{Item::PDOs::NotationInterior::CalcOfPDO::CommutatorOfPDOs} \(\left[ \Stx{s_1}[1], \Stx{s_2}[2] \right]= \Stxt{s_1+s_2-1}\).
        \item\label{Item::PDOs::NotationInterior::CalcOfPDO::CommutatorOfPartialt} \(\left[ \partial_t, \Stx{s} \right]= \Stxt{s}\).
        \item\label{Item::PDOs::NotationInterior::CalcOfPDO::CommutatorMultiplication}  Let \(v\in \CinftySpace[\Ropn]\), and \(\Mult{v}:u\mapsto vu\).
            Then, \(\left[ \Stx{s}, \Mult{v} \right]= \Stxt{s-1}\).
        \item\label{Item::PDOs::NotationInterior::CalcOfPDO::VectorField} Let \(Y\) be any smooth vector field without a \(\partial_t\) component. Then,
            \(\left[ \Stx{s}, Y \right]=\Stxt{s}\).
    \end{enumerate}
\end{lemma}
\begin{proof}
    \ref{Item::PDOs::NotationInterior::CalcOfPDO::Composition}, 
    \ref{Item::PDOs::NotationInterior::CalcOfPDO::Adjoint},
    \ref{Item::PDOs::NotationInterior::CalcOfPDO::CommutatorOfPDOs},
    \ref{Item::PDOs::NotationInterior::CalcOfPDO::CommutatorMultiplication},
    and \ref{Item::PDOs::NotationInterior::CalcOfPDO::VectorField}
    follow from the calculus of pseudodifferential operators.
    \ref{Item::PDOs::NotationInterior::CalcOfPDO::CommutatorOfPartialt}
    follows just as in Lemma \ref{Lemma::PDOs::NotationNearBdry::CalcOfPDO}\ref{Item::PDOs::NotationNearBdry::CalcOfPDO::CommutatorOfPartialt}.
\end{proof}

\section{H\"ormander vector fields and dependence of constants}\label{Section::HorVfs}
Let \(W_1,\ldots, W_r\) be smooth vector fields on an open set \(U\subseteq \Rngeq\) or \(U\subseteq \Rn\) satisfying H\"ormander's condition of order \(m\in \Zg\).
By writing \(W_j=\sum_{k=1}^n a_j^k \partial_{x_k}\), we may identify \(W_j=(a_j^1,\ldots, a_j^n)\in \CinftySpace[U][\Rn]\).

Let \(X_1,\ldots, X_q\) be an enumeration of all commutators of \(W_1,\ldots, W_r\) up to order \(m\). By H\"ormander's condition,
for each compact set \(\Compact\Subset U\), \(\exists c_{\Compact}>0\),
\begin{equation}\label{Eqn::HorVfs::LowerBoundDet}
    \inf_{x\in \Compact}\max_{j_1,\ldots, j_n\in \{1,\ldots, q\}} 
    \left| \det \left( X_{j_1}(x)| \cdots | X_{j_n}(x) \right) \right|\geq c_{\Compact}>0,
\end{equation}
where \( \left( X_{j_1}(x)| \cdots | X_{j_n}(x) \right)\) denotes the \(n\times n\) matrix
with columns given by \(X_{j_1}(x),\ldots, X_{j_n}(x)\).

When we write estimates like:
\begin{equation*}
    \LtNorm*{\Stx{s} A}[\Ropn]\lesssim \LtNorm*{\Stxt{s'} B}[\Ropn],
\end{equation*}
where \(A\) and \(B\) are some functions possibly involving \(W_1,\ldots, W_r\),
it means the same thing as in Section \ref{Section::PseudodifferentialOps::Interior}
except \(C_L\) in \eqref{Eqn::PDOs::NotationInterior::CL}
can also depend on \(n\), \(r\), \(c_{\overline{U_2}}^{-1}\), and an upper bound for
\begin{equation*}
    \max_{1\leq j\leq r} \CjNorm*{W_j}{K}[U_2][\Rn],
\end{equation*}
where \(K\) is as described there.
We use a similar definition with \((t,x)\) replaced by \((t,x')\) and using the notation
from Section \ref{Section::PseudodifferentialOps::Bdry}.


    \subsection{Gaining regularity on the interior}
    We use the notation from this section to state the basic way we use H\"ormander's condition
to gain regularity on the interior of the manifold. Using the notation from Section \ref{Section::PseudodifferentialOps::Interior},
for vector fields \(W_1,\ldots, W_r\) 
satisfying H\"ormander's condition of order \(m\) 
on some open set \(U\subseteq \Rn\), we have:

\begin{proposition}\label{Prop::Horvfs::GainInterior::RSGain}
    \begin{equation*}
        \sum_{|\alpha|\leq \kappa-1} \LtNorm*{W^{\alpha} \Stx{s} u}[\Ropn]
        \lesssim
        \sum_{|\alpha|\leq \kappa} \LtNorm*{ W^{\alpha} \Stxt{s-1/m} u}[\Ropn]
        +\LtNorm*{\partial_t \Lambdatx[-\kappa]\Stxt{s-1/m} u}[\Ropn],
    \end{equation*}
    \(\forall u\in \Distributions[\Ropn]\).
\end{proposition}

Proposition \ref{Prop::Horvfs::GainInterior::RSGain} is standard and follows easily from a result of Rothschild and Stein
\cite{RothschildSteinHypoellipticDifferentialOperatorsAndNilpotentGroups}; we include the reduction here
as this forms the core of many of our estimates.

\begin{lemma}\label{Lemma::HorVfs::GainInterior::RSGainLemma}
    Fix relatively compact, open sets \(U_1\Subset V_1\Subset  U\), \(U_2\Subset \R\). Then,
    \begin{equation*}
        \sum_{|\alpha|\leq \kappa-1}
        \HsNorm*{W^{\alpha}u}{1/m}[\Ropn] 
        \leq C_{U_1, V_1,U_2}
        \left( \sum_{|\alpha|\leq \kappa} \LtNorm*{ W^{\alpha}u}[\Ropn] + 
        \LtNorm*{ \partial_t \Lambdatx[-\kappa] u}[\Ropn] \right),
    \end{equation*}
    \(\forall u\in \Distributions[\Ropn]\) with \(\supp(u)\subseteq U_2\times U_1\),
    where if the right-hand side is finite, so is the left-hand side.
    Here, \(C_{U_1, V_1,U_2}\geq 1\) can be chosen to depend only on \(\kappa\), \(n\), \(r\), \(m\), \(U_1\), \(V_1\), \(U_2\),
    \(c_{\overline{V_1}}\) (where \(c_{\overline{V_1}}\) is as in \eqref{Eqn::HorVfs::LowerBoundDet}),
    and an upper bound for
    \begin{equation*}
        \max_{1\leq j\leq r} \CjNorm{W_j}{L}[V_1][\Rn],
    \end{equation*}
    where \(L\) can be chosen to depend only on \(n\), \(r\), and \(m\).
\end{lemma}
\begin{proof}
    We introduce three microlocalizations which are pseudodifferential operators with symbols in \(\Symbolstx{0}\):
    \(\Theta_1\), \(\Theta_2\), and \(\Theta_3\) which are chosen to satisfy
    \(\Theta_1+\Theta_2+\Theta_3=I\).
    \(\Theta_1\) has symbol supported on \(|\xi|^{2\kappa}+|\tau|\lesssim 1\) (and therefore is infinitely smoothing),
    \(\Theta_2\) has symbol supported on \(|\xi|^{2\kappa}+|\tau|\gtrsim 1\) and \(|\xi|^{2\kappa}\lesssim |\tau|\),
    and \(\Theta_3\) has symbol supported on \(|\xi|^{2\kappa}+|\tau|\gtrsim 1\) and \(|\tau|\lesssim |\xi|^{2\kappa}\).

    We have, for any \(u\in \Distributions[\Ropn]\), 
    and any \(\phi\in \CinftycptSpace[U]\),
    \begin{equation}\label{Eqn::HorVfs::GainInterior::RSGainLemma::Tmp1}
    \begin{split}
         &\sum_{|\alpha|\leq \kappa-1} \HsNorm*{\Stx{0}W^{\alpha}\phi u}{1/m}
         \lesssim \sum_{|\alpha|\leq \kappa-1} \sum_{j=1}^3 \LtNorm*{ \Stxt{0} \Lambdatx[1/m] \Theta_j  W^{\alpha} \phi u}.
    \end{split}
    \end{equation}
    \(\Theta_1\) is infinitely smoothing, so we have
    \begin{equation}\label{Eqn::HorVfs::GainInterior::RSGainLemma::Tmp2}
        \sum_{|\alpha|\leq \kappa-1}\LtNorm*{\Stxt{0}\Lambdatx[1/m] \Theta_1  W^{\alpha} \phi u}
        \lesssim \LtNorm*{u}.
    \end{equation}
    For \(\Theta_2\), since \((\kappa-1+1/m)/\kappa\leq 1\) and \(\Theta_2 \partial_t^{-1}\) is an operator of order
    \(-2\kappa\) (in the sense of \(\Symbolstx{-2\kappa}\)), we have
    \begin{equation}\label{Eqn::HorVfs::GainInterior::RSGainLemma::Tmp3}
        \sum_{|\alpha|\leq \kappa-1}\LtNorm*{\Stxt{0}\Lambdatx[1/m] \Theta_2  W^{\alpha} \phi u}
        \lesssim \LtNorm*{\partial_t \Lambdatx[-\kappa] \phi u}+\LtNorm*{u}.
    \end{equation}
    Finally, for \(j=3\), we consider the pseudodifferential operator
    \(\Lambdax[s]\) on \(\Rn\), with symbol \((1+|\xi|^2)^{s/2}\).
    Rothschild and Stein
    \cite{RothschildSteinHypoellipticDifferentialOperatorsAndNilpotentGroups}
    showed that, for all \(v\in\Distributions[\Rn]\) with
    \(\supp(v)\subseteq V_1\),
    \begin{equation*}
        \LtNorm{\Lambdax[1/m]v}[\Rn]
        \lesssim
        \sum_{|\beta|\leq1}\LtNorm{W^\beta v}[\Rn],
    \end{equation*}
    where if the right-hand side is finite, so is the left-hand side.
    See, also, Remark \ref{Rmk::HorVfs::GainInterior::CanUseKohn}.
    
    Choose \(\chi_0,\chi_1\in\CinftycptSpace[V_1]\), with
    \(\chi_0\prec\chi_1\), so that every operator \(\Stxt{0}\)
    appearing below satisfies
    \begin{equation*}
        \Stxt{0}=\Stxt{0}\chi_0.
    \end{equation*}
    The calculus of pseudodifferential operators and the support properties
    of \(\Theta_3\) give
    \begin{equation*}
    \begin{split}
        &\sum_{|\alpha|\leq\kappa-1}
        \LtNorm*{\Stxt{0}\Lambdatx[1/m]\Theta_3W^\alpha\phi u}
        \\
        &\qquad\lesssim
        \sum_{|\alpha|\leq\kappa-1}
        \LtNorm*{\Lambdax[1/m]\chi_1\Theta_3W^\alpha\phi u}
        +
        \sum_{|\alpha|\leq\kappa-1}\LtNorm*{W^\alpha\phi u}.
    \end{split}
    \end{equation*}
    Here, the terms arising from \(1-\chi_1\) are infinitely smoothing.
    Applying the Rothschild--Stein estimate to
    \(\chi_1\Theta_3W^\alpha\phi u\) for each \(t\), and then using the
    calculus of pseudodifferential operators, we conclude
    \begin{equation}\label{Eqn::HorVfs::GainInterior::RSGainLemma::Tmp4}
    \begin{split}
        &\sum_{|\alpha|\leq\kappa-1}
        \LtNorm*{\Stxt{0}\Lambdatx[1/m]\Theta_3W^\alpha\phi u}
        \\
        &\qquad\lesssim
        \sum_{|\alpha|\leq\kappa-1}\sum_{|\beta|\leq1}
        \LtNorm*{W^\beta\!\left(\chi_1\Theta_3W^\alpha\phi u\right)}
        +
        \sum_{|\alpha|\leq\kappa-1}\LtNorm*{W^\alpha\phi u}
        \\
        &\qquad\lesssim
        \sum_{|\alpha|\leq\kappa}\LtNorm*{W^\alpha\phi u}.
    \end{split}
    \end{equation}
    Combining \eqref{Eqn::HorVfs::GainInterior::RSGainLemma::Tmp1}, \eqref{Eqn::HorVfs::GainInterior::RSGainLemma::Tmp2},
    \eqref{Eqn::HorVfs::GainInterior::RSGainLemma::Tmp3}, and \eqref{Eqn::HorVfs::GainInterior::RSGainLemma::Tmp4} yields
    \begin{equation*}
        \sum_{|\alpha|\leq \kappa-1} \HsNorm*{\Stx{0}W^{\alpha}\phi u}{1/m}
        \lesssim \sum_{|\alpha|\leq \kappa} \LtNorm*{W^{\alpha}\phi u}+ \LtNorm*{\partial_t \Lambdatx[-\kappa] \phi u}.
    \end{equation*}
    Restricting to \(u\) with \(\supp(u)\subseteq U_2\times U_1\) and taking \(\phi=1\) on \(U_2\times U_1\), we see
    \begin{equation*}
        \sum_{|\alpha|\leq \kappa-1} \HsNorm*{\Stx{0}W^{\alpha} u}{1/m}
        \lesssim \sum_{|\alpha|\leq \kappa} \LtNorm*{W^{\alpha} u}+ \LtNorm*{\partial_t \Lambdatx[-\kappa] u}.
    \end{equation*}
    Taking \(\Stx{0}\rightarrow I\) completes the proof.
\end{proof}

\begin{proof}[Proof of Proposition \ref{Prop::Horvfs::GainInterior::RSGain}]
    We have, using the calculus of pseudodifferential operators and Lemma \ref{Lemma::HorVfs::GainInterior::RSGainLemma},
    \begin{equation*}
    \begin{split}
         &\sum_{|\alpha|\leq \kappa-1} \LtNorm*{W^\alpha \Stx{s}u}
         \lesssim \sum_{|\alpha|\leq \kappa-1} \LtNorm*{W^\alpha \Lambdatx[1/m] \Stxt{s-1/m}u}
         \\&\lesssim \sum_{|\alpha|\leq \kappa-1} \LtNorm*{\Lambdatx[1/m] W^\alpha  \Stxt{s-1/m}u}
         \\&\lesssim \sum_{|\alpha|\leq \kappa} \LtNorm*{W^\alpha  \Stxt{s-1/m}u} + \LtNorm*{\partial_t \Lambdatx[-\kappa] \Stxt{s-1/m} u}.
    \end{split}
    \end{equation*}
\end{proof}

\begin{remark}\label{Rmk::HorVfs::GainInterior::CanUseKohn}
    In many of our estimates and results, there is a ``gain'' of \(1/m\) derivatives, which comes from
    Proposition \ref{Prop::Horvfs::GainInterior::RSGain}.
    As described in Section \ref{Section::Intro::Quant}, this gain is often not optimal, and achieving
    the optimal gain is not relevant for our purposes (we plan to later use the results of this paper as a step
    in obtaining optimal results).
    If one is willing to replace \(1/m\) with \(2^{1-m}\) (which is fine for the applications we have in mind),
    then one can replace the above application of Rothschild and Stein's results \cite{RothschildSteinHypoellipticDifferentialOperatorsAndNilpotentGroups}
    with a result due to Kohn \cite{KohnLecturesOnDegenerateEllipticProblems} which has a more elementary proof.
    See \cite[Proposition 4.5.8]{StreetMaximalSubellipticity} for an exposition of this more elementary proof.
\end{remark}

    \subsection{Gaining regularity near the boundary}
    Similar to Proposition \ref{Prop::Horvfs::GainInterior::RSGain}, we can use H\"ormander's condition
to gain regularity near the boundary, though (at first) only in directions tangent to the boundary.
For later notational consistency, we use \(Y_0,Y_1,\ldots, Y_r\) to denote the H\"ormander vector fields in this case;
and we use the same dependence of constants as before, with \(Y\) replacing \(W\), throughout.

Let \(Y_0,Y_1,\ldots, Y_r\) be smooth vector fields on \(\CubengeqOne\) satisfying H\"ormander's condition of order \(m\).
Fix \(0<\delta_1<\delta_2<1\), and \(U_1\Subset U_2\Subset \R\) as in Section \ref{Section::PseudodifferentialOps::Interior}.
Suppose \(Y_0=\partial_{x_n}\) and \(Y_1,\ldots, Y_r\) have no \(\partial_{x_n}\) component. 

Recall the space \(\TestFunctionsZero[\Rngeq]\) consisting of those \(f\in \CinftycptSpace[\Rngeq]\)
which vanish to infinite order at the boundary \(\Rnmo\), and the dual space
\(\DistributionsZero[\Rngeq]\). When dealing with functions defined on \(\Rngeq\) (respectively, \(\Ropngeq\)),
we identify them with distributions in \(\DistributionsZero[\Rngeq]\) (respectively, \(\Ropngeq\)).
Thus, we may freely take derivatives without worrying about the boundary.

\begin{proposition}\label{Prop::HorVfs::GainBoundary::MainGainProp}
    For \(u\in \LtSpace[\Ropngeq]\), with \(\partial_{x_n}^ju\in \LtSpace[\Ropngeq]\), \(0\leq j\leq \kappa\), we have
    \begin{equation*}
        \sum_{|\alpha|\leq \kappa-1}
        \LtNorm*{Y^{\alpha} \Stxp{s} u}[\Ropngeq]
        \lesssim
        \sum_{|\alpha|\leq \kappa} \LtNorm*{Y^{\alpha}\Stxpt{s-1/m}u}[\Ropngeq]
        +\LtNorm*{\partial_t \left( 1-\partial_t^2 \right)^{-1/4}\Stxpt{s-1/m} u}[\Ropngeq].
    \end{equation*}
\end{proposition}

To prove Proposition \ref{Prop::HorVfs::GainBoundary::MainGainProp}, we extend the setting so
that Proposition \ref{Prop::Horvfs::GainInterior::RSGain} applies.
By a classic result of Seeley \cite{SeeleyExtensionOfCInfinityFunctionsDefinedInAHalfSpace},
we may extend \(Y_0,Y_1,\ldots Y_r\) to smooth vector fields on \(\CubenOne\); call the extended
vector fields \(W_0, W_1,\ldots, W_r\), with \(W_0=\partial_{x_n}\), satisfying:
\begin{itemize}
    \item \(W_j\big|_{\CubengeqOne}=Y_j\),
    \item For any open \(U\Subset \CubenOne\), \(\CjNorm{W_j}{L}[U][\Rn]\leq C_{L,n,U} \CjNorm{Y_j}{L}[U][\Rn]\),
        where \(C_{L,n,U}\geq 0\) depends only on \(L,n,U\).
\end{itemize}
Fix \(\delta_3\in(\delta_2,1)\).
By continuity, \(W_0,W_1,\ldots,W_r\) satisfy H\"ormander's condition
on a neighborhood of \(\CubengeqDeltaTwoClosure\).
More precisely, we may choose an open set \(V\Subset\CubenOne\), with
\(\CubengeqDeltaTwoClosure\Subset V\), such that the corresponding
lower bound in \eqref{Eqn::HorVfs::LowerBoundDet} satisfies
\(c_{\overline{V}}\gtrsim1\).
The set \(V\) and the implicit constant may be chosen to depend only on
\(n\), \(r\), \(m\), \(\delta_2\), \(\delta_3\),
\(c_{\CubengeqDeltaTwoClosure}^{-1}\), and an upper bound for
\begin{equation*}
    \max_{0\leq j\leq r}
    \CjNorm*{Y_j}{L}[\Cubengeq{\delta_3}][\Rn],
\end{equation*}
where \(L\) can be chosen to depend only on \(n\), \(r\), and \(m\).

Choose \(\psi\in\CinftySpace[\CubenOne]\) such that \(\psi=0\) on a
neighborhood of \(\CubengeqDeltaTwoClosure\) and
\begin{equation*}
    \psi=1
    \quad\text{on}\quad
    \CubenOne\setminus V.
\end{equation*}
Enlarge each of the lists \(Y_0,\ldots,Y_r\) and \(W_0,\ldots,W_r\)
by appending
\begin{equation*}
    \psi\partial_{x_1},\ldots,\psi\partial_{x_{n-1}}.
\end{equation*}
The enlarged list of \(W\)-vector fields satisfies H\"ormander's
condition on all of \(\CubenOne\).
Since the appended vector fields vanish on \(\CubengeqDeltaTwo\), the
conclusion of Proposition
\ref{Prop::HorVfs::GainBoundary::MainGainProp} is unchanged.
Henceforth, we use these enlarged lists.
For \(\Compact_0\Subset \CubengeqOne\) compact, we let
\begin{equation}\label{Eqn::HorVfs::GainBoundar::DefineHsYSpaceOfCompact}
    \HsYSpace{j}[\Compact_0]:=\left\{ u\in \HsYSpace{j}[\CubengeqOne]  : \supp(u)\subseteq \Compact_0\right\}.
\end{equation}
Similarly, for \(\Compact\Subset \CubenOne\) compact, we let
\begin{equation*}
    \HsWSpace{j}[\Compact]:=\left\{ u\in \HsWSpace{j}[\CubenOne]  : \supp(u)\subseteq \Compact\right\}.
\end{equation*}

Clearly, the restriction map \(f\mapsto f\big|_{\Rngeq}\) is continuous
\(\HsWSpace{j}[\Compact]\rightarrow \HsYSpace{j}[\Compact\cap \Rngeq]\). The next result shows that this has a corresponding
right-inverse, continuous, linear extension map. Here, we cannot use the classical extension maps (as discussed
in Section \ref{Section::EstNearBdry::ClassicalExtension}) because they do not respect the spaces.

\begin{lemma}\label{Lemma::HorVfs::GainBoundary::StreetExtension}
    Fix \(N\in \Zg\).
    There is a continuous, linear, extension map
    \begin{equation}\label{Eqn::HorVfs::GainBoundary::StreetExtension::ExtensionMapping}
        \Extension_{N}:\HsYSpace{j}[\CubengeqDeltaOneClosure]\rightarrow \HsWSpace{j}[\CubenDeltaTwoClosure],\quad 0\leq j\leq N,
    \end{equation}
    satisfying
    \begin{equation}\label{Eqn::HorVfs::GainBoundary::StreetExtension::ExtensionRestriction}
        \Extension_{N} f\big|_{\Rngeq} = f.
    \end{equation}
    Moreover,
    \begin{equation*}
        \HsWNorm{\Extension_{N} f}{j}
        \leq C_{N} \HsYNorm{f}{j}, \quad \forall 0\leq j\leq N,\: f\in \HsYSpace{j}[\CubengeqDeltaOneClosure],
    \end{equation*}
    where \(C_{N}\geq 0\) can be chosen to depend only on \(N\), \(n\), \(r\), \(m\), \(c_{\CubenDeltaTwoClosure}\)
    (as in \eqref{Eqn::HorVfs::LowerBoundDet}), and an upper bound for
    \begin{equation*}
        \max_{1\leq j\leq r} \CjNorm{W_j}{L}[\CubenDeltaTwoClosure][\Rn],
    \end{equation*}
    where \(L\) can be chosen to depend only on \(N\), \(n\), \(r\), and \(m\).
\end{lemma}
\begin{proof}
    This is a result from \cite{StreetFunctionSpacesAndTraceTheoremsForMaximallySubellipticBoundaryValueProblems};
    we describe how to see this result in the language of that paper.
    For \(\Omegageq \subseteq \CubengeqOne\) open and \(d\in \Zg\), let \(\FilteredSheafF[\Omegageq][d]\)
    be the \(\CinftySpace[\Omegageq][\R]\)-module generated by \(Y_0,\ldots, Y_r\),
    and for \(\Omega\subseteq \CubenOne\) open and \(d\in \Zg\), let
    \(\FilteredSheafFh[\Omega][d]\) be the \(\CinftySpace[\Omega][\R]\)-module generated by \(W_0,\ldots, W_r\).
    In the language of \cite{StreetFunctionSpacesAndTraceTheoremsForMaximallySubellipticBoundaryValueProblems},
    \(\FilteredSheafF\) and \(\FilteredSheafFh\) are H\"ormander filtrations of sheaves of vector fields
    (see \cite[Definition \ref*{FS::Defn::Filtrations_Sheaves::Hormander_Filtration_Sheaves}]{StreetFunctionSpacesAndTraceTheoremsForMaximallySubellipticBoundaryValueProblems}).
    Since \(Y_0=\partial_{x_n}\) is not tangent to \(\partial \CubengeqOne\), every boundary point of
    \(\partial \CubengeqOne\) is \(\FilteredSheafF\)-non-characteristic in the sense of
    \cite[Definition \ref*{FS::Defn::Filtrations::RestrictingFiltrations::NonCharPoints}]{StreetFunctionSpacesAndTraceTheoremsForMaximallySubellipticBoundaryValueProblems};
    see, also, \cite[Example \ref*{FS::Example::Filtrations::RestrictingFiltrations::NonCharExamples} \ref*{FS::Item::Filtrations::RestrictingFiltrations::NonCharExamples::CharacterizeWhenAllDegsOne}]{StreetFunctionSpacesAndTraceTheoremsForMaximallySubellipticBoundaryValueProblems}.
    Finally, we have \(\RestrictFilteredSheaf{\FilteredSheafFh}{\CubengeqOne}=\FilteredSheafF\) in the sense of
    \cite[Definition \ref*{FS::Defn::Filtrations::RestrictingFiltrations::RestrictedFiltration}]{StreetFunctionSpacesAndTraceTheoremsForMaximallySubellipticBoundaryValueProblems}; see, also,
    \cite[Proposition \ref*{FS::Prop::Filtrations::RestrictingFiltrations::CoDim0Restriction}]{StreetFunctionSpacesAndTraceTheoremsForMaximallySubellipticBoundaryValueProblems}.

    \cite[Proposition \ref*{FS::Prop::Spaces::EqualsSobolev}]{StreetFunctionSpacesAndTraceTheoremsForMaximallySubellipticBoundaryValueProblems}
    shows
    \begin{equation*}
        \TLSpace{j}{2}{2}[\CubengeqDeltaOneClosure][\FilteredSheafF] = \HsYSpace{j}[\CubengeqDeltaOneClosure],\quad \TLSpace{j}{2}{2}[\CubenDeltaTwoClosure][\FilteredSheafFh]=\HsWSpace{j}[\CubenDeltaTwoClosure],
    \end{equation*}
    with equivalence of norms,
    where the \(\TLSpace{s}{p}{q}\) spaces are defined in \cite[Chapter \ref*{FS::Chapter::Spaces}]{StreetFunctionSpacesAndTraceTheoremsForMaximallySubellipticBoundaryValueProblems}.
    From here, the existence of \(\Extension_{N}\) as in \eqref{Eqn::HorVfs::GainBoundary::StreetExtension::ExtensionMapping}
    and \eqref{Eqn::HorVfs::GainBoundary::StreetExtension::ExtensionRestriction}
    is a special case of \cite[Theorem \ref*{FS::Thm::Spaces::Extension} \ref*{FS::Item::Spaces::Extension::Extension}]{StreetFunctionSpacesAndTraceTheoremsForMaximallySubellipticBoundaryValueProblems}.
    The statement about what \(C_{N}\) depends on can be shown by keeping track of constants in the proofs in
    \cite{StreetFunctionSpacesAndTraceTheoremsForMaximallySubellipticBoundaryValueProblems}.
\end{proof}

\begin{proof}[Proof of Proposition \ref{Prop::HorVfs::GainBoundary::MainGainProp}]
    Let \(\Extension=\Extension_{\kappa}\) be as in Lemma \ref{Lemma::HorVfs::GainBoundary::StreetExtension}. 
    For functions \(u(t,x)\), we let \(\Extension\) act in the \(x\)
    variable.  We have, 
    \begin{equation*}
    \begin{split}
         &\sum_{|\alpha|\leq \kappa-1} \LtNorm*{Y^{\alpha}\Stxp{s} u}[\Ropngeq]
         \lesssim \sum_{|\alpha|\leq \kappa-1} \LtNorm*{\Lambdatxp[1/m]  Y^{\alpha}\Stxpt{s-1/m} u}[\Ropngeq]
         \\&\leq \sum_{|\alpha|\leq \kappa-1}
         \LtNorm*{\Lambdatxp[1/m]  W^{\alpha} \Extension \Stxpt{s-1/m} u}[\Ropn]
         \leq \sum_{|\alpha|\leq \kappa-1}
         \LtNorm*{\Lambdatx[1/m]  W^{\alpha} \Extension \Stxpt{s-1/m} u}[\Ropn]
         \\&=\sum_{|\alpha|\leq \kappa-1}\HsNorm*{W^{\alpha} \Extension \Stxpt{s-1/m} u}{1/m}[\Ropn]
         \\&\lesssim \sum_{|\alpha|\leq \kappa} \LtNorm*{W^{\alpha} \Extension \Stxpt{s-1/m} u}[\Ropn]
         +\LtNorm*{\partial_t \Lambdatx[-\kappa] \Extension \Stxpt{s-1/m} u}[\Ropn]
         \\&\leq\sum_{|\alpha|\leq \kappa} \LtNorm*{W^{\alpha} \Extension \Stxpt{s-1/m} u}[\Ropn]
         +\LtNorm*{\partial_t \left( 1-\partial_t^2 \right)^{-1/4} \Extension \Stxpt{s-1/m} u}[\Ropn]
        \\&=\sum_{|\alpha|\leq \kappa} \LtNorm*{W^{\alpha} \Extension \Stxpt{s-1/m} u}[\Ropn]
         +\LtNorm*{ \Extension \partial_t \left( 1-\partial_t^2 \right)^{-1/4} \Stxpt{s-1/m} u}[\Ropn]
         \\&\lesssim \sum_{|\alpha|\leq \kappa} \LtNorm*{Y^{\alpha} \Stxpt{s-1/m} u}[\Ropngeq]
         +\LtNorm*{ \partial_t \left( 1-\partial_t^2 \right)^{-1/4} \Stxpt{s-1/m} u}[\Ropngeq],
    \end{split}
    \end{equation*}
    where in the second inequality we  used \(\Extension f\big|_{\Rngeq} =f\), in the fourth inequality
    we used  Lemma \ref{Lemma::HorVfs::GainInterior::RSGainLemma}, 
    in the final equality we used that \(\Extension\) does not act in the \(t\)-variable,
    and in the final inequality we used
    Lemma \ref{Lemma::HorVfs::GainBoundary::StreetExtension}.
\end{proof}

\section{Estimates near the boundary}\label{Section::EstNearBdry}
Let \(W_0,\ldots, W_r\) be smooth vector fields on \(\CubengeqOne\) satisfying H\"ormander's condition of order \(m\in \Zg\),
with \(W_0=c_0\partial_{x_n}\), where \(c_0\in \R\setminus\{0\}\) is a constant.
We assume \(W_1(x',0),\ldots, W_r(x',0)\) have no \(\partial_{x_n}\) component for \(x'\in \CubenmoOne\).
Fix \(M,\kappa\in \Zg\).  As usual we give \(\Rngeq\) coordinates \(x=(x',x_n)\), where \(x_n\geq 0\),
and we give \(\Ropngeq\) coordinates \((t,x)=(t,x',x_n)\), where \(t\in \R\), \(x\in \Rngeq\).
Fix \(\sigma\in \CinftySpace[\CubengeqOne][(0,\infty)]\).

Let \(a_{\alpha,\beta}\in \CinftySpace[\CubengeqOne][\MatrixSpace[M][M][\C]]\) be given (for \(\alpha,\beta\) lists
of elements of \(\left\{ 0,1,\ldots,r \right\}\),
\(|\alpha|,|\beta|\leq \kappa\)). Define, wherever it makes sense,
\begin{equation}\label{Eqn::EstBdry::FormulaForQ}
    \FormQ[f][g]:=\sum_{|\alpha|,|\beta|\leq \kappa} \Ltip*{W^{\alpha}f}{\sigma a_{\alpha,\beta}W^\beta g}[\Rngeq],
\end{equation}
\begin{equation*}
    \FormQH[u][v]:= \sum_{|\alpha|,|\beta|\leq \kappa} \Ltip*{W^{\alpha}u}{\sigma a_{\alpha,\beta}W^\beta v}[\Ropngeq]
    +\Ltip{u}{\sigma \partial_t v}[\Ropngeq].
\end{equation*}
Set
\begin{equation*}
    \opL:=\sum_{|\alpha|,|\beta|\leq \kappa} \sigma^{-1}\left( W^{\alpha} \right)^{*} \sigma a_{\alpha,\beta} W^{\beta}.
\end{equation*}
Note that,
\begin{equation}\label{Eqn::EstBdry::LGivenByQ}
    \FormQ[f][g]=\Ltip*{f}{\sigma \opL g}[\Rngeq],\quad f,g\in \TestFunctionsZero*[\CubengeqOne],
\end{equation}
\begin{equation*}
    \FormQH[u][v]=\Ltip*{u}{\sigma \left( \partial_t +\opL \right)v}[\Ropngeq], \quad u,v\in \TestFunctionsZero*[\R\times \CubengeqOne].
\end{equation*}

\begin{assumption}\label{Assumption::EstBdry::MaxSubOnInterior}
    \(\exists D_1,D_2\geq 0\), \(\forall f\in \CinftycptSpace[\CubengOneHalf][\C^M]\),
    \begin{equation}\label{Eqn::EstBdry::MaxSubOnInterior}
        \sum_{|\alpha|=\kappa} \LtNorm*{W^\alpha f}[\Rngeq]^2
        \leq
        D_1 \Real \FormQ*[f][f] + D_2 \LtNorm*{f}[\Rngeq]^2.
    \end{equation}
\end{assumption}

For the next assumption, we consider integral kernels
\(K(t,s,x,y')\in \CinftycptSpace[\R\times \R\times \CubengeqOneHalf\times \CubenmoOneHalf]\)
such that \(\exists \delta_0>0\) with \(\partial_{x_n}K(t,s,x,y')=0\), for \(x_n<\delta_0\),
and operators
\begin{equation}\label{Eqn::EstBdry::SOpForAssump}
    S u(t,x)=\iint K(t,s,x,y') u(s,y',x_n)\:ds\: dy'.
\end{equation}


\begin{assumption}\label{Assumption::EstBdry::MaxSubOnSmoothing}
    We are given a \emph{set} 
    \(\DSetL{1}\subseteq \DistributionsZeroCM[\R\times \CubengeqOne]\)
    such that
    \begin{enumerate}[(i)]
        \item\label{Item::EstBdry::MaxSubOnSmoothing::InHkWOnCubegeqOneHalf} \(\forall u\in \DSetL{1}\), \(u\big|_{\R\times \CubengeqOneHalf}\in \LtSpaceNoMeasure[\R][\HsWSpace{\kappa}[\CubengeqOneHalf;\C^M]]\).
        \item\label{Item::EstBdry::MaxSubOnSmoothing::MaxSubEstimate} \(\exists E_1,E_2\geq 0\), \(\forall u\in \DSetL{1}\), \(\forall S\) of the form \eqref{Eqn::EstBdry::SOpForAssump},
            \begin{equation}\label{Eqn::EstBdry::MaxSubOnSmoothing::MaxSubEstimate}
                \sum_{|\alpha|=\kappa} \LtNorm*{W^{\alpha} S u}[\Ropngeq]^2
                \leq E_1 \Real \FormQH*[Su][Su] + E_2 \LtNorm*{Su}[\Ropngeq]^2.
            \end{equation}
        \item\label{Item::EstBdry::MaxSubOnSmoothing::dtplusOplOnDIsFunction}  \(\forall u\in \DSetL{1}\), \(\opL u\big|_{\R\times \CubengeqOneHalf}\)
            (when taken in the sense of \(\DistributionsZeroCM*[\R\times \CubengeqOne]\)) agrees with
            a function in \(\LtSpaceNoMeasure[\R\times \CubengeqOneHalf][\C^M]\).
        \item\label{Item::EstBdry::MaxSubOnSmoothing::FormQHGivesOp} \(\forall u\in \DSetL{1}\), 
        \(\forall S\) of the form \eqref{Eqn::EstBdry::SOpForAssump},
            \begin{equation*}
                \FormQH[S u][u] = \Ltip*{Su}{\sigma \left( \partial_t+\opL \right)u}[\Ropngeq].
            \end{equation*}
    \end{enumerate}
\end{assumption}

\begin{definition}\label{Defn::EstBdry::DSet}
    For \(N\geq 1\), we recursively define 
    \(\DSetL{N}\subseteq\DistributionsZeroCM[\R\times \CubengeqOne]\)
    by: 
    \[\DSetL{N+1}:=\left\{ u(t,x)\in \DSetL{N} : \left( \partial_t +\opL \right)u\in \DSetL{N} \right\},\]
    where \(\DSetL{1}\) is as given in Assumption \ref{Assumption::EstBdry::MaxSubOnSmoothing}.
    In particular, by Assumption \ref{Assumption::EstBdry::MaxSubOnSmoothing}\ref{Item::EstBdry::MaxSubOnSmoothing::InHkWOnCubegeqOneHalf}, every \(u\in \DSetL{N}\)
    restricts to an element of \(\LtSpaceNoMeasure[\R][\HsWSpace{\kappa}[\CubengeqOneHalf;\C^M]]\).
\end{definition}

Set \(\epsilon_0:=\min\{1/m,1/2\}\).

\begin{theorem}\label{Thm::EstBdry::MainThm::New}
    Under the above assumptions (Assumptions \ref{Assumption::EstBdry::MaxSubOnInterior} and \ref{Assumption::EstBdry::MaxSubOnSmoothing})
    the following holds. 
        Fix \(0<\delta_1<\delta_2<1/2\) and \(U_1\Subset U_2\Subset \R\) open sets.
    Let \(\phi_1\in \CinftycptSpace[U_1\times \Cubengeq{\delta_1}]\),
    \(\phi_2\in \CinftycptSpace[\R\times \CubengeqOneHalf]\)
    with 
    \(\phi_2=1\) on a neighborhood of \(\overline{U_1}\times \CubengeqDeltaOneClosure\).
    Then, \(\forall N\in \Zg\), \(\forall l\in \Zg\) with \(l\geq 2\kappa-1\), and
    \(J\in \Zg\) with \(1\leq J\leq \min\left\{ \floor{\frac{l+1}{2\kappa}},N \right\}\),
    there exists \(C_{l,\phi_1,\phi_2,N}\geq 0\), \(\forall u\in \DSetL{N}\),
    \begin{equation}\label{Eqn::EstBdry::MainThm::MainEst::New}
    \begin{split}
            \HsNorm*{\phi_1 u}{l+1}[\Ropngeq]\leq C_{l,\phi_1,\phi_2,N}
            \bigg( &
                \HsNorm*{\phi_2\left( \partial_t+\opL \right)^J u}{l+1-2\kappa J}[\Ropngeq]
                \\&+ \HsNorm*{\phi_2\left( \partial_t+\opL \right)^N u}{\left( l+1-N\epsilon_0 \right)\vee 0}[\Ropngeq]
                +\sum_{k=0}^{N-1} \LtNorm*{\phi_2\left( \partial_t+\opL \right)^k u}[\Ropngeq]
             \bigg),
    \end{split}    
\end{equation}
    where if the right-hand side of \eqref{Eqn::EstBdry::MainThm::MainEst::New} is finite, so is the left-hand side.
    \(C_{l,\phi_1,\phi_2,N}\) can be chosen to depend only on
    \(l\), \(\phi_1\), \(\phi_2\), \(N\), \(n\), \(\kappa\), \(m\), \(r\), \(M\), \(c_{\CubengeqSevenEighthsClosure}\) (as in \eqref{Eqn::HorVfs::LowerBoundDet}),
     \(\delta_1\), \(\delta_2\), \(U_1\), \(U_2\),
    \(D_1\), \(D_2\), \(E_1\), \(E_2\), and an upper bound for
    \begin{equation*}
        \max_{0\leq j\leq r} \CjNorm*{W_j}{L}[\CubengeqSevenEighths][\Rn] + \max_{\alpha,\beta} \CjNorm{a_{\alpha,\beta}}{L}[\CubengeqOneHalf][\MatrixSpace[M][M][\C]]
        +\CjNorm{\sigma}{L}[\CubengeqOneHalf]+\sup_{x\in \CubengeqOneHalf} \sigma(x)^{-1},
    \end{equation*}
    where \(L\) can be chosen to depend only on \(l\), \(m\), \(n\), \(N\), \(\kappa\), and \(r\).
\end{theorem}

\begin{remark}
    Throughout this section, we write \(A\lesssim B\) to mean \(A\leq CB\) where \(C\) is a constant
    which depends only on the correct quantities.
    More precisely, when we combine our estimates, the implicit constants will have the dependence of the constant
    \(C_{l,\phi_1,\phi_2, N}\) 
    in Theorem \ref{Thm::EstBdry::MainThm::New}.
\end{remark}

\begin{remark}
    We present the proof of the boundary estimates (Theorem \ref{Thm::EstBdry::MainThm::New}) before the easier interior results 
    (Theorem \ref{Thm::EstBdryInt::MainThm::New}). 
    This is perhaps logically backwards:
    we use the interior result Theorem \ref{Thm::EstBdryInt::MainThm::New} to help simplify the boundary proof.
    However, the interior proof is largely an easier reprise of the boundary proof, so we present the boundary proof
    first, and then refer back to the changes needed to prove the interior result. This is not circular:
    the proof of the interior result can be extracted in a completely self-contained way.
\end{remark}

\begin{remark}\label{Rmk::EstBdry::c0Approx1}
    Even though \(c_0\) does not appear among the quantities on which the
    constant in Theorem \ref{Thm::EstBdry::MainThm::New} depends,
    \(|c_0|\) is bounded above and below by constants with the dependence
    stated there (and therefore \(|c_0|\approx 1\)).
    The upper bound follows from
    \begin{equation*}
        \CjNorm{W_0}{0}[\CubengeqSevenEighths][\Rn]=|c_0|.
    \end{equation*}
    At \(x_n=0\), every commutator not containing \(W_0\) is tangent to
    the boundary, while every commutator containing
    \(W_0=c_0\partial_{x_n}\) has \(\partial_{x_n}\)-coefficient
    \(O(|c_0|)\), with constants controlled by the stated \(C^L\)-bounds
    and the upper bound for \(|c_0|\).
    Thus, every determinant in \eqref{Eqn::HorVfs::LowerBoundDet} is
    \(O(|c_0|)\) on the boundary.
    The lower bound in \eqref{Eqn::HorVfs::LowerBoundDet} gives
    \(|c_0|\gtrsim1\), and therefore \(|c_0|\approx1\).
\end{remark}

\begin{remark}
    In Assumption \ref{Assumption::EstBdry::MaxSubOnInterior}, we used
    \(\sum_{|\alpha|=\kappa}\), instead of the (a priori) stronger assumption with \(\sum_{|\alpha|\leq \kappa}\)
    on the left-hand side of \eqref{Eqn::EstBdry::MaxSubOnInterior}.
    In Lemma \ref{Lemma::EstBdry::Reduction::BoundAllDerivsByHighDerivs} 
    we show that these two assumptions are  equivalent.
    In future papers, we will only be given good quantitative bounds
    on \(\sum_{|\alpha|=\kappa}\), rather than on \(\sum_{|\alpha|\leq \kappa}\).
    Thus, it is more convenient to state Assumption \ref{Assumption::EstBdry::MaxSubOnInterior}
    as we have for the future applications we have in mind.
\end{remark}

\begin{remark}\label{Rmk::EstBdry::MainThm::tdoesntplayroleInMaxSub}
    Note that, because \(\sigma\) does not depend on \(t\),
    \begin{equation*}
        \Real \Ltip{v}{\sigma \partial_t v}[\Ropngeq]=0.
    \end{equation*}
    Thus, the \(\Real \Ltip{u}{\sigma \partial_t v}[\Ropngeq]\) term of \(\FormQH[u][v]\) 
    does not play a role in Assumption \ref{Assumption::EstBdry::MaxSubOnSmoothing}\ref{Item::EstBdry::MaxSubOnSmoothing::MaxSubEstimate}.
\end{remark}

Two special cases of Theorem \ref{Thm::EstBdry::MainThm::New} are particularly useful.
We state them here for the reader's convenience.

\begin{corollary}\label{Cor::EstBdry::MainThm::CorNEquals1}
    Under the above assumptions (Assumptions \ref{Assumption::EstBdry::MaxSubOnInterior} and \ref{Assumption::EstBdry::MaxSubOnSmoothing})
    the following holds. 
    Let \(\phi_1\) and \(\phi_2\) be as in Theorem \ref{Thm::EstBdry::MainThm::New}.
    Then, \(\forall l\in \Zg\) with \(l\geq 2\kappa-1\),
    there exists \(C_{l,\phi_1,\phi_2}\geq 0\), \(\forall u\in \DSetL{1}\),
    \begin{equation*}
    \begin{split}
            \HsNorm*{\phi_1 u}{l+1}[\Ropngeq]\leq C_{l,\phi_1,\phi_2}
            \bigg( &
                \HsNorm*{\phi_2\left( \partial_t+\opL \right) u}{l+1-\epsilon_0 }[\Ropngeq]
                + \LtNorm*{\phi_2 u}[\Ropngeq]
             \bigg),
    \end{split}    
\end{equation*}
where if the right-hand side is finite, so is the left-hand side.
    Here 
    \(C_{l,\phi_1,\phi_2}\) is as in Theorem \ref{Thm::EstBdry::MainThm::New} (with \(N=1\)).
\end{corollary}
\begin{proof}
    Take \(J=N=1\) in Theorem \ref{Thm::EstBdry::MainThm::New}. 
    Using,
    \begin{equation*}
        \HsNorm*{\phi_2\left( \partial_t+\opL \right) u}{l+1-2\kappa }[\Ropngeq]
        \leq \HsNorm*{\phi_2\left( \partial_t+\opL \right) u}{l+1-\epsilon_0 }[\Ropngeq],
    \end{equation*}
    the result follows immediately from Theorem \ref{Thm::EstBdry::MainThm::New}.
\end{proof}

\begin{corollary}
    Under the above assumptions (Assumptions \ref{Assumption::EstBdry::MaxSubOnInterior} and \ref{Assumption::EstBdry::MaxSubOnSmoothing})
    the following holds. 
    Let \(\phi_1\) and \(\phi_2\) be as in Theorem \ref{Thm::EstBdry::MainThm::New}.
    Then, \(\forall N\in \Zg\), \(\forall l'\in \{1,2,\ldots, N\}\), \(\exists C_{l',\phi_1,\phi_2,N}\geq 0\),
    \(\forall u\in \DSetL{N}\),
    \begin{equation*}
        \HsNorm*{\phi_1 u}{2\kappa l'}[\Ropngeq]
        \leq C_{l',\phi_1,\phi_2,N}
        \left( 
            \HsNorm*{\phi_2\left(\partial_t+\opL  \right)^N u}{(2\kappa l'-N\epsilon_0)\vee 0}[\Ropngeq]
            +\sum_{k=0}^{N-1} \LtNorm*{\phi_2 \left( \partial_t+\opL \right)^k u}[\Ropngeq]
         \right),
    \end{equation*}
    where if the right-hand side is finite, so is the left-hand side and
    \(C_{l',\phi_1,\phi_2,N}\geq 0\) is as in  Theorem \ref{Thm::EstBdry::MainThm::New}.
    In particular, if \(N\epsilon_0\geq 2\kappa l'\), for \(u\in \DSetL{N}\), we have
    \begin{equation*}
        \HsNorm*{\phi_1 u}{2\kappa l'}[\Ropngeq]
        \leq C_{l',\phi_1,\phi_2,N}
            \sum_{k=0}^{N} \LtNorm*{\phi_2 \left( \partial_t+\opL \right)^k u}[\Ropngeq].
    \end{equation*}
\end{corollary}
\begin{proof}
    Take \(l+1=2\kappa l'\) and \(J=l'\) in Theorem \ref{Thm::EstBdry::MainThm::New}.
\end{proof}

\begin{remark}\label{Rmk::EstBdry::MainThm::CanUseSigmaInsteadOfLebesgue}
    Because \(\sigma\approx 1\) on \(\CubengeqOneHalf\), and every \(\LtNorm{\cdot}\) above is localized
    to \(\CubengeqOneHalf\), one can replace Lebesgue measure on \(\CubengeqOneHalf\) with
    \(\sigma(x)\: dx\) in every \(\LtSpace\)-norm in the above assumptions and conclusions (though not in the inner products), and the assumptions and results
    are unchanged.
\end{remark}

    \subsection{Classical extension operators}\label{Section::EstNearBdry::ClassicalExtension}
    We require classical extension operators in our proofs. In this section, we review some basic formulas and results we need.
Fix \(L\in \Zgeq\); 
in proofs, one takes any \(L\geq \floor{s}\) where \(s\) is the number of derivatives under consideration.
Let \(0<\lambda_0<\lambda_1<\cdots<\lambda_L\) and \(a_0,\ldots, a_L\) be real numbers with
\(\sum_{k=0}^L a_k (-\lambda_k)^{l}=1\) for \(l=0,\ldots, L\).
For \(f:\Rng\rightarrow \C\) define \(\Extension_L f:\Rn\rightarrow \C\) by
\begin{equation}\label{Eqn::EstBdry::Extension::FormulaForExtension}
    \Extension_L f(x',x_n)
    =
    \begin{cases}
        f(x),& x_n>0,\\
        \sum_{k=0}^L a_k f(x', -\lambda_k x_n), &x_n\leq 0.
    \end{cases}
\end{equation}

The basic fact we use is:
\begin{proposition}[See {\cite[Section 4.5.2]{TriebelTheoryOfFunctionSpacesII}}]\label{Prop::EstBdry::Extension::ClassicalExtension}
    \(\Extension_L: \HsSpace{s}[\Rngeq]\rightarrow \HsSpace{s}[\Rn]\)
    continuously, for \(L\geq \floor{s}\).
\end{proposition}

From \(\Extension_L\) we may construct other extension operators which are useful for our purposes.
To explain this we introduce some new spaces.
First, we begin with the classical, isotropic, spaces on \(\Rn\).
Let \(\Omega\subset \Rn\) be a bounded, smoothly bounded open set.
Set,
\begin{equation*}
    \HsSpace{s}[\Omega] := \left\{ u\big|_{\Omega} : u\in \HsSpace{s}[\Rn] \right\},
\end{equation*}
\begin{equation*}
    \HsNorm{v}{s}[\Omega]:=\inf\left\{ \HsNorm{u}{s}[\Rn]: u\big|_{\Omega}=v \right\}.
\end{equation*}

\begin{lemma}\label{Lemma::EstBdry::Extension::ClassicalBoundedExtension}
    Fix \(L\in \Zgeq\). Then there exists a  continuous, linear extension map \(\Extension_L\) such that
    \begin{equation*}
        \Extension_L : \HsSpace{s}[\Omega]\rightarrow \HsSpace{s}[\Rn], \quad \floor{s}\leq L,
    \end{equation*}
    and \(\Extension_L f\big|_{\Omega}=f\).
\end{lemma}
\begin{proof}[Proof sketch]
    By using a partition of unity, one may reduce this to a local result near each point of the boundary.
    By a change of variables, this local result follows from Proposition \ref{Prop::EstBdry::Extension::ClassicalExtension}.
    See \cite[Section 5.1.3]{TriebelTheoryOfFunctionSpacesII} for details.
\end{proof}

Next, we work on the space \(\R\times \Rnmo\) with coordinates \((t,x')\).
Here we introduce  \(\LtSpace\)-Sobolev spaces where \(\partial_t\) is an operator of order \(2\kappa\).
Using the notation \(\Lambdatxp[s]\) from Section \ref{Section::PseudodifferentialOps}, we take
\begin{equation*}
    \HstxpNorm{u}{s}[\R\times \Rnmo] := \LtNorm{\Lambdatxp[s]u}[\R\times \Rnmo],
\end{equation*}
\begin{equation*}
    \HstxpSpace{s}[\R\times \Rnmo] :=\left\{ u\in \TemperedDistributions[\R\times \Rnmo] : \HstxpNorm{u}{s}[\R\times \Rnmo]<\infty \right\}.
\end{equation*}
Let \((a,b)\subseteq \R\) be a (possibly infinite) interval and \(\Omega\Subset \Rnmo\) be a bounded, smoothly bounded open set.
Set
\begin{equation*}
    \HstxpSpace{s}[(a,b)\times \Omega] :=\left\{ u\big|_{(a,b)\times \Omega} : u\in \HstxpSpace{s}[\R\times \Rnmo] \right\},
\end{equation*}
\begin{equation*}
    \HstxpNorm{v}{s}[(a,b)\times \Omega] :=\inf \left\{ \HstxpNorm{u}{s}[\R\times \Rnmo] : u\big|_{(a,b)\times \Omega}=v \right\}.
\end{equation*}

\begin{lemma}\label{Lemma::EstBdry::Extension::HstxpExtension}
    Fix \(L\in \Zgeq\). Then there exists a continuous, linear extension map \(\Extension_L\) such that
    \begin{equation*}
        \Extension_L : \HstxpSpace{s}[(a,b)\times \Omega]\rightarrow \HstxpSpace{s}[\R\times \Rnmo], \quad \floor{s}\leq L,
    \end{equation*}
    and \(\Extension_L f\big|_{(a,b)\times \Omega}=f\).
\end{lemma}
\begin{proof}[Proof sketch]
    Lemma \ref{Lemma::EstBdry::Extension::ClassicalBoundedExtension} applied in the \(t\)-variable gives an extension map
    \(\Extension_L^1:\HstxpSpace{s}[(a,b)\times \Omega]\rightarrow \HstxpSpace{s}[\R\times \Omega]\),
    for \(\floor{s}\leq L\).
    Lemma \ref{Lemma::EstBdry::Extension::ClassicalBoundedExtension} applied in the \(x'\) variable gives a second extension map
    \(\Extension_L^2:\HstxpSpace{s}[\R\times \Omega]\rightarrow \HstxpSpace{s}[\R\times \Rnmo]\).
    The result follows by taking \(\Extension_L=\Extension_L^2\Extension_L^1\).
\end{proof}

Fix (possibly infinite) intervals \((a,b),(c,d)\subseteq \R\) and let \(\Omega\) equal either \(\Rnmo\)
or a bounded, smoothly bounded open subset of \(\Rnmo\).
For \(l\in \Zgeq\), \(s\in \R\), we define
\begin{equation}\label{Eqn::EstBdry::Estension::DefinVlsNorm}
   \VlsNorm{u}{l}{s}[(c,d)\times \Omega\times (a,b)]^2:=\sum_{j=0}^l \int_{a}^b \HstxpNorm*{\partial_{x_n}^j u(\cdot,\cdot, x_n)}{s-j}[(c,d)\times \Omega]^2\: dx_n,
\end{equation}
\begin{equation}\label{Eqn::EstBdry::Estension::DefinVlsSpace}
    \VlsSpace{l}{s}[(c,d)\times \Omega\times (a,b)]:= \left\{ u(t,x',x_n)\in \LtxnSpace[(a,b)][\HstxpSpace{s}[\Omega]] : \partial_{x_n}^j u\in \LtxnSpace[(a,b)][\HstxpSpace{s-j}[\Omega]], 0\leq j\leq l \right\}.
\end{equation}
\begin{lemma}\label{Lemma::EstBdry::Extension::VlsExtension}
    Fix \(L\in \Zgeq\). Then, there exists a continuous, linear extension map \(\Extension_L\) such that
    \begin{equation*}
        \Extension_L:\VlsSpace{l}{s}[(c,d)\times \Omega\times (a,b)]\rightarrow \VlsSpace{l}{s}[\R\times \Rnmo\times \R], \quad \max\{\floor{s},0\}+l\leq L,
    \end{equation*}
    and \(\Extension_L u\big|_{(c,d)\times \Omega\times (a,b)}=u\).
\end{lemma}
\begin{proof}[Proof sketch]
    Lemma \ref{Lemma::EstBdry::Extension::HstxpExtension}, applied in the \((t,x')\) variables, gives
    an extension map \(\Extension_L^1:\VlsSpace{l}{s}[(c,d)\times \Omega\times (a,b)]\rightarrow \VlsSpace{l}{s}[\R\times \Rnmo\times (a,b)]\).
    Lemma \ref{Lemma::EstBdry::Extension::ClassicalBoundedExtension} applied in just the \(x_n\)
    variable, gives an extension map
    \(\Extension_L^2:\VlsSpace{l}{s}[\R\times \Rnmo\times (a,b)]\rightarrow \VlsSpace{l}{s}[\R\times \Rnmo\times \R]\).
    The result follows by taking \(\Extension_L:=\Extension_L^2\Extension_L^1\).
\end{proof}

\begin{remark}
    The above results used a fixed \(L\in \Zgeq\) and only considered spaces with \(\floor{s}\leq L\);
    this is sufficient for our purposes. It is possible to exhibit a single universal extension map which works for all \(s\).
    See the elegant paper of Seeley \cite{SeeleyExtensionOfCInfinityFunctionsDefinedInAHalfSpace}
    and the earlier result of Mitjagin
    \cite{MitjaginApproximateDimensionAndBasesInNuclearSpaces}.
    Even though the formula for this universal extension map is quite simple, it is a little more
    complicated than \eqref{Eqn::EstBdry::Extension::FormulaForExtension} which complicates some of our proofs.
    Thus, we use the simpler \eqref{Eqn::EstBdry::Extension::FormulaForExtension} in this paper.
\end{remark}

    \subsection{Large and small constants}\label{Section::EstNearBdry::LargeSmallSonsts}
    We write \(\lconst\) and \(\sconst\) to denote ``large constant'' and ``small constant,'' respectively.
If we write
\begin{equation*}
    A\leq \lconst B +\sconst D
\end{equation*}
it means \(\forall \epsilon>0\), \(\exists C_\epsilon\geq 1\) with
\begin{equation*}
    A\leq C_\epsilon B + \epsilon D.
\end{equation*}
Here, \(C_\epsilon\) will only depend on \(\epsilon\) and the appropriate parameters in what we are trying to prove
(see, e.g., the statement of Theorem \ref{Thm::EstBdry::MainThm::New} and Sections \ref{Section::PseudodifferentialOps::Bdry} and \ref{Section::HorVfs}).

We will repeatedly use the elementary inequality \(AB\leq \sconst A^2 +\lconst B^2\).
Also, note that if \(A\leq \sconst A+\lconst B\), then \(A\lesssim B\) (by subtracting \(\sconst A\) from both sides).

    \subsection{Reductions}
    \begin{proposition}\label{Prop::EstBdry::Reduction::MainReduction}
    It suffices to prove Theorem \ref{Thm::EstBdry::MainThm::New} with the following changes:
    \begin{enumerate}[(i)]
        \item\label{Item::EstBdry::Reduction::MainReduction::ReplaceWWithY} We may replace the 
        H\"ormander vector fields \(W_0,\ldots, W_r\) with H\"ormander vector fields \(Y_0,\ldots, Y_r\) on \(\CubengeqOne\)
            where \(Y_0=\partial_{x_n}\), \(Y_1,\ldots, Y_r\) have no \(\partial_{x_n}\) component,
            and \(Y_0,\ldots, Y_r\)  satisfy H\"ormander's condition of order \(m\) on \(\CubengeqOne\).
        \item\label{Item::EstBdry::Reduction::MainReduction::BoundAllDerivs} 
        In place of Assumption \ref{Assumption::EstBdry::MaxSubOnInterior}, we may assume the stronger hypothesis: 
            \(\exists D_1',D_2'\geq 0\),
            \begin{equation*}
            \sum_{|\alpha|\leq \kappa} \LtNorm*{Y^\alpha f}[\Rngeq]^2
            \leq D_1' \Real \FormQ[f][f]
            +D_2' \LtNorm*{f}[\Rngeq]^2,\quad \forall f\in \CinftycptSpace[\CubengOneHalf][\C^M].
        \end{equation*}
        Similarly, in place of Assumption \ref{Assumption::EstBdry::MaxSubOnSmoothing}\ref{Item::EstBdry::MaxSubOnSmoothing::MaxSubEstimate},
        we may assume the stronger hypothesis: \(\exists E_1',E_2'\geq 0\),
        \(\forall u\in \DSetL{1}\), \(\forall S\) of the form \eqref{Eqn::EstBdry::SOpForAssump},
        \begin{equation*}
            \sum_{|\alpha|\leq \kappa} \LtNorm*{Y^{\alpha} Su}[\Ropngeq]^2
            \leq E_1' \Real \FormQH*[Su][Su] + E_2' \LtNorm*{Su}[\Ropngeq]^2.
        \end{equation*}
        \item\label{Item::EstBdry::Reduction::MainReduction::HighOrderTermPos} We may assume \(\opL\) is of the form
            \begin{equation*}
                a(x)\partial_{x_n}^{2\kappa} + \text{terms which are order }<2\kappa\text{ in }\partial_{x_n},
            \end{equation*}
            where \(a(x)\) is invertible, \(\forall x\in \CubengeqOneHalf\), and \(\MatrixNorm{a(x)^{-1}}[M][M]\lesssim 1\), \(\forall x\in \CubengeqOneHalf\) (i.e., the norm is bounded uniformly
            on \(\CubengeqOneHalf\)).
    \end{enumerate}
\end{proposition}

We begin by establishing the reduction Proposition \ref{Prop::EstBdry::Reduction::MainReduction}\ref{Item::EstBdry::Reduction::MainReduction::HighOrderTermPos}.

\begin{lemma}\label{Lemma::EstBdry::Reduction::aalpha0Invertible}
    Let \(\alpha_0=(0,0,\ldots,0)\) be the list consisting of \(\kappa\) zeros.
    There exists \(c>0\) such that
    \begin{equation}\label{Eqn::EstBdry::Reduction::aalpha0Invertible::ToShow1}
        \Real \CNip{v}{a_{\alpha_0,\alpha_0}(x',0)v}[M]
        \geq c|v|^2,
        \qquad
        x'\in \CubenmoOneHalf,\quad v\in \C^M.
    \end{equation}
    In particular, \(a_{\alpha_0,\alpha_0}(x',0)\) is invertible for every
    \(x'\in \CubenmoOneHalf\), and
    \begin{equation*}
        \sup_{x'\in \CubenmoOneHalf}
        \MatrixNorm{a_{\alpha_0,\alpha_0}(x',0)^{-1}}[M][M]
        \lesssim 1.
    \end{equation*}
\end{lemma}
\begin{proof}
    Since \(W_j(x',0)\) has no \(\partial_{x_n}\)-component for
    \(1\leq j\leq r\), we may write
    \begin{equation*}
        W_j
        =
        \sum_{\ell=1}^{n-1}b_{j,\ell}(x)\partial_{x_\ell}
        +x_n\widetilde b_j(x)\partial_{x_n},
        \qquad 1\leq j\leq r,
    \end{equation*}
    for smooth functions \(b_{j,\ell}\) and \(\widetilde b_j\). Recall also that
    \begin{equation*}
        W_0=c_0\partial_{x_n},
        \qquad c_0\neq0.
    \end{equation*}

    Fix \(x'_0\in \CubenmoOneHalf\) and \(v\in \C^M\) with \(|v|=1\).
    Choose
    \begin{equation*}
        \psi\in \CinftycptSpace[\Bnmo{1}[0]],
        \qquad
        \LtNorm*{\psi}[\Rnmo]=1,
    \end{equation*}
    and choose
    \begin{equation*}
        \phi\in \CinftycptSpace[(0,1/4)]
    \end{equation*}
    such that \(\phi^{(\kappa)}\not\equiv0\). For \(\delta>0\) sufficiently small, set
    \begin{equation*}
        \psi_\delta(x')
        :=
        \delta^{-(n-1)/2}
        \psi\left(\delta^{-1}(x'-x'_0)\right),
    \end{equation*}
    so that \(\supp(\psi_\delta)\subset \CubenmoOneHalf\) and
    \(\LtNorm*{\psi_\delta}[\Rnmo]=1\). For \(\lambda\geq1\), define
    \begin{equation*}
        f_{\lambda,\delta}(x',x_n)
        :=
        \lambda^{1/2}\psi_\delta(x')\phi(\lambda x_n)v.
    \end{equation*}
    Then
    \begin{equation*}
        f_{\lambda,\delta}
        \in \CinftycptSpace[\CubengOneHalf][\C^M].
    \end{equation*}

    For a list \(\alpha\), let \(\nu(\alpha)\) denote the number of entries of
    \(\alpha\) equal to \(0\). After making the change of variables
    \begin{equation*}
        s=\lambda x_n,
    \end{equation*}
    the vector fields become
    \begin{equation*}
        W_0=\lambda c_0\partial_s
    \end{equation*}
    and, for \(1\leq j\leq r\),
    \begin{equation*}
        W_j
        =
        \sum_{\ell=1}^{n-1}
        b_{j,\ell}(x',s/\lambda)\partial_{x_\ell}
        +
        s\widetilde b_j(x',s/\lambda)\partial_s.
    \end{equation*}
    Consequently, for each fixed \(\delta>0\), an induction on \(|\alpha|\) gives
    \begin{equation*}
        \LtNorm*{W^\alpha f_{\lambda,\delta}}[\Rngeq]
        =
        O_\delta\left(\lambda^{\nu(\alpha)}\right),
        \qquad |\alpha|\leq\kappa.
    \end{equation*}

    For \(\alpha=\alpha_0\), we have the exact formula
    \begin{equation*}
        W^{\alpha_0}f_{\lambda,\delta}
        =
        c_0^\kappa\lambda^{\kappa+1/2}
        \psi_\delta(x')\phi^{(\kappa)}(\lambda x_n)v.
    \end{equation*}
    Thus, setting
    \begin{equation*}
        C_\phi:=\LtNorm*{\phi^{(\kappa)}}[\R]^2>0,
    \end{equation*}
    we obtain
    \begin{equation}\label{Eqn::EstBdry::Reduction::aalpha0Invertible::Walpha0Norm}
        \LtNorm*{W^{\alpha_0}f_{\lambda,\delta}}[\Rngeq]^2
        =
        |c_0|^{2\kappa}\lambda^{2\kappa}C_\phi.
    \end{equation}

    If \(|\alpha|,|\beta|\leq\kappa\) and
    \((\alpha,\beta)\neq(\alpha_0,\alpha_0)\), then
    \begin{equation*}
        \nu(\alpha)+\nu(\beta)\leq2\kappa-1.
    \end{equation*}
    It follows from the preceding estimates and \eqref{Eqn::EstBdry::FormulaForQ} that
    \begin{equation}\label{Eqn::EstBdry::Reduction::aalpha0Invertible::FormQAsymptotic}
        \begin{split}
            \Real \FormQ*[f_{\lambda,\delta}][f_{\lambda,\delta}]
            ={}&
            |c_0|^{2\kappa}\lambda^{2\kappa}C_\phi
            \int_{\Rnmo}
            |\psi_\delta(x')|^2
            \sigma(x',0)
            \\
            &\hspace{35mm}\times
            \Real \CNip{v}{a_{\alpha_0,\alpha_0}(x',0)v}[M]
            \:dx'
            +O_\delta\left(\lambda^{2\kappa-1}\right).
        \end{split}
    \end{equation}
    Here, the replacement of the coefficients at \(x_n\) by their values at
    \(x_n=0\) also contributes only
    \(O_\delta\left(\lambda^{2\kappa-1}\right)\), since
    \(x_n=O(\lambda^{-1})\) on the support of \(f_{\lambda,\delta}\).

    Applying Assumption \ref{Assumption::EstBdry::MaxSubOnInterior} to
    \(f_{\lambda,\delta}\), using
    \eqref{Eqn::EstBdry::Reduction::aalpha0Invertible::Walpha0Norm} and
    \(\LtNorm*{f_{\lambda,\delta}}[\Rngeq]=O(1)\), and then dividing by
    \begin{equation*}
        |c_0|^{2\kappa}\lambda^{2\kappa}C_\phi,
    \end{equation*}
    we obtain
    \begin{equation*}
        \begin{split}
            1
            \leq{}&
            D_1
            \int_{\Rnmo}
            |\psi_\delta(x')|^2
            \sigma(x',0)
            \Real \CNip{v}{a_{\alpha_0,\alpha_0}(x',0)v}[M]
            \:dx'
            \\
            &\quad+O_\delta(\lambda^{-1})+O(\lambda^{-2\kappa}).
        \end{split}
    \end{equation*}

    Letting \(\lambda\rightarrow\infty\) while keeping \(\delta\) fixed gives
    \begin{equation}\label{Eqn::EstBdry::Reduction::aalpha0Invertible::ApproxIdentityEstimate}
        1
        \leq
        D_1
        \int_{\Rnmo}
        |\psi_\delta(x')|^2
        \sigma(x',0)
        \Real \CNip{v}{a_{\alpha_0,\alpha_0}(x',0)v}[M]
        \:dx'.
    \end{equation}
    In particular, Assumption \ref{Assumption::EstBdry::MaxSubOnInterior}
    forces \(D_1>0\).

    Since \(|\psi_\delta|^2\,dx'\) is an approximate identity at \(x'_0\), letting
    \(\delta\downarrow0\) in
    \eqref{Eqn::EstBdry::Reduction::aalpha0Invertible::ApproxIdentityEstimate}
    yields
    \begin{equation*}
        1
        \leq
        D_1\sigma(x'_0,0)
        \Real \CNip{v}{a_{\alpha_0,\alpha_0}(x'_0,0)v}[M].
    \end{equation*}
    Therefore,
    \begin{equation*}
        \Real \CNip{v}{a_{\alpha_0,\alpha_0}(x'_0,0)v}[M]
        \geq
        \frac{1}{D_1\sigma(x'_0,0)}
        \geq
        \frac{1}{D_1\left\|\sigma\right\|_{L^\infty(\CubengeqOneHalf)}}.
    \end{equation*}
    Since \(x'_0\) and the unit vector \(v\) were arbitrary, \eqref{Eqn::EstBdry::Reduction::aalpha0Invertible::ToShow1} follows.

    Finally, writing \(A=a_{\alpha_0,\alpha_0}(x',0)\), we have for every
    \(v\in\C^M\),
    \begin{equation*}
        c|v|^2
        \leq
        \Real \CNip{v}{Av}[M]
        \leq
        |v|\,|Av|.
    \end{equation*}
    Hence
    \begin{equation*}
        |Av|\geq c|v|.
    \end{equation*}
    Thus \(A\) is invertible and
    \begin{equation*}
        \MatrixNorm{A^{-1}}[M][M]\leq c^{-1},
    \end{equation*}
    uniformly in \(x'\), completing the proof.
\end{proof}

\begin{proof}[Proof of Proposition \ref{Prop::EstBdry::Reduction::MainReduction}\ref{Item::EstBdry::Reduction::MainReduction::HighOrderTermPos}]
    Let \(\alpha_0=(0,0,\ldots,0)\) be as in Lemma \ref{Lemma::EstBdry::Reduction::aalpha0Invertible}.
    Note that, on \(\CubenmoOne\),
        \begin{equation*}
        \opL = (-1)^{\kappa}c_0^{2\kappa} a_{\alpha_0,\alpha_0}(x',0)\partial_{x_n}^{2\kappa}+\text{terms which are order }<2\kappa\text{ in }\partial_{x_n},
    \end{equation*}
    by \eqref{Eqn::EstBdry::LGivenByQ} and \eqref{Eqn::EstBdry::FormulaForQ}.
    Write
    \begin{equation*}
        \opL = a(x)\partial_{x_n}^{2\kappa} + \sum_{\substack{|\beta|\leq 2\kappa \\ \beta_n<2\kappa}} b_{\beta}(x)\partial_{x}^{\beta},
    \end{equation*}
    so that \(a(x',0)= (-1)^{\kappa}c_0^{2\kappa}a_{\alpha_0,\alpha_0}(x',0)\).
    Lemma \ref{Lemma::EstBdry::Reduction::aalpha0Invertible} implies  \ref{Item::EstBdry::Reduction::MainReduction::HighOrderTermPos}
    holds near \(\CubenmoOneHalf\).  Informally, we will separate the proof of Theorem \ref{Thm::EstBdry::MainThm::New} into the part
    near the boundary (where Lemma \ref{Lemma::EstBdry::Reduction::aalpha0Invertible} applies) and away from the boundary,
    where the interior result Corollary \ref{Cor::EstBdryInt::MainCor} applies.

    Let \(\phi_1,\phi_2\in \CinftycptSpace[\R\times\CubengeqOneHalf]\) with
    \(\phi_1\prec\phi_2\) be as in the statement of
    Theorem \ref{Thm::EstBdry::MainThm::New}.
    Choose
    \begin{equation*}
        \psi_1,\psit_1\in \CinftycptSpace[\CubengeqOneHalf],
        \qquad
        \psi_2,\psit_2\in \CinftycptSpace[\CubengOneHalf],
    \end{equation*}
    such that \(\psi_j\prec\psit_j\), \(j=1,2\),
    \begin{equation*}
        \psi_1+\psi_2=1
        \quad\text{on}\quad
        \left\{x:\exists t,\ (t,x)\in\supp(\phi_2)\right\},
    \end{equation*}
    and \(\psit_1\) is supported where
    \begin{equation*}
        \MatrixNorm*{a(x)^{-1}}[M][M]\lesssim1.
    \end{equation*}
    In particular,
    \begin{equation*}
        \psi_j\phi_1\prec\psit_j\phi_2,
        \qquad j=1,2.
    \end{equation*}
    Let \(l\), \(J\), \(N\), and \(\DSetL{N}\) be as in the statement of
    Theorem \ref{Thm::EstBdry::MainThm::New}.
    By a simple translation and scaling in the \(x_n\)-variable,
    Corollary \ref{Cor::EstBdryInt::MainCor}, with \(\phi_1,\phi_2\) replaced by
    \(\psi_2\phi_1,\psit_2\phi_2\), respectively, implies, for every
    \(u\in\DistributionsZeroCM[\R\times\CubengeqOne]\),
    \begin{equation}\label{Eqn::EstBdry::Reduction::MainReduction::HighOrderTermPos::AwayFromBoundaryEstimate}
        \begin{split}
            \HsNorm*{\psi_2\phi_1 u}{l+1}
            &\lesssim
            \HsNorm*{\psit_2\phi_2(\partial_t+\opL)^N u}{l+1-N\epsilon_0}
            +\LtNorm*{\psit_2\phi_2 u}
            \\
            &\lesssim
            \HsNorm*{\phi_2(\partial_t+\opL)^N u}{l+1-N\epsilon_0}
            +\LtNorm*{\phi_2 u}.
        \end{split}
    \end{equation}
    After a simple scaling in the \(x_n\)-variable, Theorem
    \ref{Thm::EstBdry::MainThm::New}, under the additional assumption
    \begin{equation*}
        \MatrixNorm*{a(x)^{-1}}[M][M]\lesssim1,
    \end{equation*}
    and with \(\phi_1,\phi_2\) replaced by
    \(\psi_1\phi_1,\psit_1\phi_2\), respectively, implies, for
    \(u\in\DSetL{N}\),
    \begin{equation}\label{Eqn::EstBdry::Reduction::MainReduction::HighOrderTermPos::NearBoundaryEstimate}
        \begin{split}
            &\HsNorm*{\psi_1\phi_1 u}{l+1}[\Ropngeq]
            \\&\lesssim
            \HsNorm*{\psit_1\phi_2(\partial_t+\opL)^J u}{l+1-2\kappa J}[\Ropngeq]
            +
            \HsNorm*{\psit_1\phi_2(\partial_t+\opL)^N u}{(l+1-N\epsilon_0)\vee0}[\Ropngeq]
            +
            \sum_{k=0}^{N-1}
            \LtNorm*{\psit_1\phi_2(\partial_t+\opL)^k u}[\Ropngeq]
            \\
            &\lesssim
            \HsNorm*{\phi_2(\partial_t+\opL)^J u}{l+1-2\kappa J}[\Ropngeq]
            +
            \HsNorm*{\phi_2(\partial_t+\opL)^N u}{(l+1-N\epsilon_0)\vee0}[\Ropngeq]
            +
            \sum_{k=0}^{N-1}
            \LtNorm*{\phi_2(\partial_t+\opL)^k u}[\Ropngeq].
        \end{split}
    \end{equation}
    Combining
    \eqref{Eqn::EstBdry::Reduction::MainReduction::HighOrderTermPos::AwayFromBoundaryEstimate}
    and
    \eqref{Eqn::EstBdry::Reduction::MainReduction::HighOrderTermPos::NearBoundaryEstimate},
    and using
    \begin{equation*}
        (\psi_1+\psi_2)\phi_1=\phi_1,
    \end{equation*}
    completes the proof.
\end{proof}

Henceforth we assume Proposition \ref{Prop::EstBdry::Reduction::MainReduction}\ref{Item::EstBdry::Reduction::MainReduction::HighOrderTermPos} holds,
and establish Proposition \ref{Prop::EstBdry::Reduction::MainReduction}\ref{Item::EstBdry::Reduction::MainReduction::ReplaceWWithY},\ref{Item::EstBdry::Reduction::MainReduction::BoundAllDerivs}.
We begin by establishing Proposition \ref{Prop::EstBdry::Reduction::MainReduction}\ref{Item::EstBdry::Reduction::MainReduction::BoundAllDerivs}
with \(Y\) replaced by \(W\). This follows from the next lemma.

\begin{lemma}\label{Lemma::EstBdry::Reduction::BoundAllDerivsByHighDerivs}
    \begin{equation*}
        \sum_{|\alpha|\leq \kappa}\LtNorm{W^{\alpha} f}[\Rngeq]
        \lesssim \sum_{|\alpha|=\kappa} \LtNorm{W^{\alpha}f}[\Rngeq]+ \LtNorm{f}[\Rngeq],\quad \forall f\in \HsWSpace{\kappa}[\CubengeqOneHalfClosure],
    \end{equation*}
    where \(\HsWSpace{\kappa}[\CubengeqOneHalfClosure]\) is as in
    \eqref{Eqn::HorVfs::GainBoundar::DefineHsYSpaceOfCompact}.
\end{lemma}

To prove Lemma \ref{Lemma::EstBdry::Reduction::BoundAllDerivsByHighDerivs}, we use another lemma.
\begin{lemma}\label{Lemma::EstBdry::Reduction::BoundLowerxnDerivsByHigher}
    For \(f(x_n)\in \HsSpace{2}[\lbrack 0 ,\infty)]\), we have
    \begin{equation*}
        \LtNorm{\partial_{x_n}f}[\lbrack 0 ,\infty)]
        \leq \sconst \LtNorm{\partial_{x_n}^2 f}[\lbrack 0 ,\infty)]
        +\lconst \LtNorm{f}[\lbrack 0 ,\infty)].
    \end{equation*}
\end{lemma}
\begin{proof}
    First we show the result holds with \(\lbrack 0 ,\infty)\) replaced by \(\R\).
    Indeed, for \(g(x_n)\in \HsSpace{2}[\R]\), we have
    \begin{equation*}
    \begin{split}
         & \LtNorm*{\partial_{x_n}g}[\R]^{2}
         =\left| \Ltip*{\partial_{x_n} g}{\partial_{x_n}g}[\R] \right|
         =\left| \Ltip*{ g}{\partial_{x_n}^2g}[\R] \right|
         \leq \sconst \LtNorm*{\partial_{x_n}^2 g}[\R]^2+ \lconst \LtNorm*{ g}[\R]^2.
    \end{split}
    \end{equation*}

    We use the extension map \(\Extension_2:\HsSpace{2}[\lbrack 0 ,\infty)]\rightarrow \HsSpace{2}[\R]\) given by \eqref{Eqn::EstBdry::Extension::FormulaForExtension}.
    Note that
    \begin{equation}\label{Eqn::EstBdry::Reduction::BoundLowerxnDerivsByHigher::Tmp1}
        \LtNorm*{\partial_{x_n}^2 \Extension_2 f}[\R]\lesssim \LtNorm*{\partial_{x_n}^2 f}[ \lbrack 0 ,\infty) ],
    \end{equation}
    as can be seen from the formula \eqref{Eqn::EstBdry::Extension::FormulaForExtension}.
    Using Proposition \ref{Prop::EstBdry::Extension::ClassicalExtension} and \eqref{Eqn::EstBdry::Reduction::BoundLowerxnDerivsByHigher::Tmp1}, we have
    \begin{equation*}
        \begin{split}
            &\LtNorm*{\partial_{x_n}f}[ \lbrack 0,\infty) ]
            \lesssim \HsNorm{f}{1}[ \lbrack 0,\infty)]
            \approx \HsNorm{\Extension_2 f}{1}[\R] 
            \\&\leq \sconst \LtNorm*{\partial_{x_n}^2 \Extension_2 f}[\R] + \lconst \LtNorm*{\Extension_2 f}[\R]
            \lesssim \sconst \LtNorm*{\partial_{x_n}^2 f}[ \lbrack 0,\infty)] + \lconst \LtNorm*{f}[ \lbrack 0,\infty)].
        \end{split}
    \end{equation*}
\end{proof}

\begin{proof}[Proof of Lemma \ref{Lemma::EstBdry::Reduction::BoundAllDerivsByHighDerivs}]
    For \(\kappa=0,1\), the result is immediate, so we assume \(\kappa\geq 2\).
    We will show
    \begin{equation}\label{Eqn::EstBdry::Reduction::BoundAllDerivsByHighDerivs::ToShow}
        \sum_{|\alpha|\leq \kappa}\LtNorm{W^{\alpha} u}[\Rngeq]
        \lesssim \sum_{|\alpha|=\kappa} \LtNorm{W^{\alpha}u}[\Rngeq]+ \LtNorm{u}[\Rngeq],\quad \forall u\in \CinftycptSpace[\CubengeqThreeFourths].
    \end{equation}

    First we see why \eqref{Eqn::EstBdry::Reduction::BoundAllDerivsByHighDerivs::ToShow} completes the proof.
    As described in the proof of Lemma \ref{Lemma::HorVfs::GainBoundary::StreetExtension}, we have
    \begin{equation}\label{Eqn::EstBdry::Reduction::BoundAllDerivsByHighDerivs::EqualSpaces}
        \TLSpace{j}{2}{2}[\CubengeqOneHalfClosure][\FilteredSheafF] = \HsWSpace{j}[\CubengeqOneHalfClosure],
    \end{equation}
    for certain spaces \(\TLSpace{j}{2}{2}[\CubengeqOneHalfClosure][\FilteredSheafF]\) defined in \cite[Chapter \ref*{FS::Chapter::Spaces}]{StreetFunctionSpacesAndTraceTheoremsForMaximallySubellipticBoundaryValueProblems}
    (and as described in the proof of Lemma \ref{Lemma::HorVfs::GainBoundary::StreetExtension}, every boundary point is
    \(\FilteredSheafF\)-non-characteristic in the sense of
    \cite[Definition \ref*{FS::Defn::Filtrations::RestrictingFiltrations::NonCharPoints}]{StreetFunctionSpacesAndTraceTheoremsForMaximallySubellipticBoundaryValueProblems}).
    Using \eqref{Eqn::EstBdry::Reduction::BoundAllDerivsByHighDerivs::EqualSpaces},
    \cite[Corollary \ref*{FS::Cor::Spaces::Approximation::SmoothFunctionsAreDense} \ref*{FS::Item::Spaces::Approximation::SmoothFunctionsAreDense::ApproxInST}]{StreetFunctionSpacesAndTraceTheoremsForMaximallySubellipticBoundaryValueProblems}
    shows that every element of \(\HsWSpace{\kappa}[\CubengeqOneHalfClosure]\) can be approximated in the strong topology on
    \(\HsWSpace{\kappa}[\CubengeqThreeFourthsClosure]\) by elements of \(\CinftycptSpace[\CubengeqThreeFourths]\).
    Thus, once we prove \eqref{Eqn::EstBdry::Reduction::BoundAllDerivsByHighDerivs::ToShow}, the result follows.

    We turn to establishing \eqref{Eqn::EstBdry::Reduction::BoundAllDerivsByHighDerivs::ToShow}.
    First we claim, \(\forall N\geq 1\),
    \begin{equation}\label{Eqn::EstBdry::Reduction::BoundAllDerivsByHighDerivs::ToShow2}
        \sum_{|\alpha|=N} \LtNorm*{W^{\alpha} u}[\Rngeq]
        \leq \sconst \sum_{|\alpha|=N+1} \LtNorm*{W^{\alpha}u}[\Rngeq] + \lconst \sum_{|\alpha|=N-1} \LtNorm*{W^{\alpha}u}[\Rngeq],
        \quad \forall u\in \CinftycptSpace[\CubengeqThreeFourths].
    \end{equation}
    Let \(|\alpha|=N\), so that \(W^{\alpha}=W_{j_0}W^{\beta}\), where \(|\beta|=N-1\).
    We separate the proof of \eqref{Eqn::EstBdry::Reduction::BoundAllDerivsByHighDerivs::ToShow2} into two cases:
    when \(j_0=0\) and when \(1\leq j_0\leq r\).
    When \(j_0=0\), \(W_{j_0}=W_0=c_0\partial_{x_n}\), where \(c_0\approx 1\) (see Remark \ref{Rmk::EstBdry::c0Approx1}).
    Thus, using Lemma \ref{Lemma::EstBdry::Reduction::BoundLowerxnDerivsByHigher}, we have
    \begin{equation*}
        \begin{split}
            &\LtNorm*{W^{\alpha}u}[\Rngeq] \approx \LtNorm*{\partial_{x_n}W^{\beta}u}[\Rngeq]
            \leq \sconst \LtNorm*{\partial_{x_n}^2 W^{\beta}u}[\Rngeq] + \lconst \LtNorm*{W^{\beta}u}[\Rngeq]
            \\&\approx \sconst \LtNorm*{W_0^2 W^{\beta}u}[\Rngeq] + \lconst \LtNorm*{W^{\beta}u}[\Rngeq],
        \end{split}
    \end{equation*}
    establishing \eqref{Eqn::EstBdry::Reduction::BoundAllDerivsByHighDerivs::ToShow2} in this case.
    When \(j_0\ne 0\), by hypothesis, \(W_{j_0}(x',0)\) has no \(\partial_{x_n}\) component. As a consequence,
    we can integrate by parts with \(W_{j_0}\) and will not have any boundary terms. Note that
    \(W_{j_0}^{*}=-W_{j_0}+f_{j_0}\) for some smooth function \(f_{j_0}\). We have,
    \begin{equation}\label{Eqn::EstBdry::Reduction::BoundAllDerivsByHighDerivs::Tmp1}
        \begin{split}
            &\LtNorm*{W^{\alpha}u}[\Rngeq]^2
            =\left| \Ltip*{W_{j_0}W^{\beta}u}{W^{\alpha}u}[\Rngeq] \right|
            \\&\leq \left| \Ltip*{W^{\beta}u}{W_{j_0} W^{\alpha}u}[\Rngeq] \right|
            +\left| \Ltip*{W^{\beta}u}{f_{j_0} W^{\alpha}u}[\Rngeq] \right|
            \\&\leq \sconst \left( \LtNorm*{W_{j_0}W^{\alpha} u}[\Rngeq]^2 + \LtNorm*{W^{\alpha}u}[\Rngeq]^2\right)
            +\lconst \LtNorm*{W^{\beta}u}[\Rngeq]^2.
        \end{split}
    \end{equation}
    Subtracting \(\sconst \LtNorm*{W^{\alpha}u}[\Rngeq]^2\) from both sides
    of \eqref{Eqn::EstBdry::Reduction::BoundAllDerivsByHighDerivs::Tmp1}
    yields
    \begin{equation*}
        \LtNorm*{W^{\alpha}u}[\Rngeq]^2\leq \sconst  \LtNorm*{W_{j_0}W^{\alpha} u}[\Rngeq]^2
            +\lconst \LtNorm*{W^{\beta}u}[\Rngeq]^2,
    \end{equation*}
    completing the proof of \eqref{Eqn::EstBdry::Reduction::BoundAllDerivsByHighDerivs::ToShow2}.

    We next claim, \(\forall 1\leq j\leq N\)
    \begin{equation}\label{Eqn::EstBdry::Reduction::BoundAllDerivsByHighDerivs::ToShow3}
        \sum_{j\leq|\alpha|\leq N} \LtNorm*{W^{\alpha}u}[\Rngeq] 
        \leq \sconst \sum_{|\alpha|=N+1} \LtNorm*{W^{\alpha}u}[\Rngeq] + \lconst \sum_{|\alpha|=j-1} \LtNorm*{W^{\alpha}u}[\Rngeq].
    \end{equation}
    We proceed by induction on \(j\). The base case, \(j=N\), is \eqref{Eqn::EstBdry::Reduction::BoundAllDerivsByHighDerivs::ToShow2}.
    We assume \eqref{Eqn::EstBdry::Reduction::BoundAllDerivsByHighDerivs::ToShow3} for some \(j+1\leq N\) and prove it for \(j\).
    Fix \(\epsilon>0\). The inductive hypothesis shows \(\exists C_\epsilon\geq 1\) with
    \begin{equation*}
    \begin{split}
         & \sum_{j\leq |\alpha|\leq N} \LtNorm*{W^{\alpha}u}[\Rngeq] 
         =\sum_{j+1\leq |\alpha|\leq N} \LtNorm*{W^{\alpha}u}[\Rngeq]  + \sum_{|\alpha|=j} \LtNorm*{W^{\alpha}u}[\Rngeq] 
         \\&\leq \epsilon \sum_{|\alpha|=N+1} \LtNorm*{W^{\alpha}u}[\Rngeq] + C_{\epsilon} \sum_{|\alpha|=j} \LtNorm*{W^{\alpha}u}[\Rngeq]
         \\&\leq \epsilon \sum_{|\alpha|=N+1} \LtNorm*{W^{\alpha}u}[\Rngeq] + C_{\epsilon}  \sconst \sum_{|\alpha|=j+1}\LtNorm*{W^{\alpha}u}[\Rngeq]
         +C_{\epsilon} \lconst \sum_{|\alpha|=j-1}\LtNorm*{W^{\alpha}u}[\Rngeq],
    \end{split}
    \end{equation*}
    where the final estimate used \eqref{Eqn::EstBdry::Reduction::BoundAllDerivsByHighDerivs::ToShow2} with \(N=j\).
    Choose \(\sconst\) so small \(C_\epsilon \sconst\leq 1/2\) to see
    \begin{equation*}
        \sum_{j\leq |\alpha|\leq N} \LtNorm*{W^{\alpha}u}[\Rngeq] 
        \leq  2\epsilon \sum_{|\alpha|=N+1} \LtNorm*{W^{\alpha}u}[\Rngeq] + C_{\epsilon}' \sum_{|\alpha|=j-1}\LtNorm*{W^{\alpha}u}[\Rngeq].
    \end{equation*}
    This completes the inductive step and establishes \eqref{Eqn::EstBdry::Reduction::BoundAllDerivsByHighDerivs::ToShow3}.

    Taking \(j=1\) and \(N=\kappa-1\) in \eqref{Eqn::EstBdry::Reduction::BoundAllDerivsByHighDerivs::ToShow3}
    gives \eqref{Eqn::EstBdry::Reduction::BoundAllDerivsByHighDerivs::ToShow} and completes the proof.
\end{proof}

\begin{proof}[Completion of the proof of Proposition \ref{Prop::EstBdry::Reduction::MainReduction}]
    Set \(Y_0=\partial_{x_n}=c_0^{-1}W_0\) and for \(1\leq j\leq r\) let
    \(Y_j=W_j-a_j Y_0\), where \(a_j\) is the \(\partial_{x_n}\) component of \(W_j\).
    All our hypothesis are equivalent whether using \(W\) or \(Y\); the only ones which are
    non-obvious are Assumption \ref{Assumption::EstBdry::MaxSubOnInterior} and Assumption \ref{Assumption::EstBdry::MaxSubOnSmoothing}\ref{Item::EstBdry::MaxSubOnSmoothing::MaxSubEstimate}.
    However, as shown in Lemma \ref{Lemma::EstBdry::Reduction::BoundAllDerivsByHighDerivs},
    we may replace \(\sum_{|\alpha|=\kappa}\) in those assumptions
    with \(\sum_{|\alpha|\leq \kappa}\).
    Then, 
     Assumption \ref{Assumption::EstBdry::MaxSubOnInterior} and Assumption \ref{Assumption::EstBdry::MaxSubOnSmoothing}\ref{Item::EstBdry::MaxSubOnSmoothing::MaxSubEstimate}
    with \(W\) are clearly equivalent
    to the same assumptions with \(Y\); establishing both \ref{Item::EstBdry::Reduction::MainReduction::ReplaceWWithY}
    and \ref{Item::EstBdry::Reduction::MainReduction::BoundAllDerivs}.
\end{proof}

    \subsection{Reducing to tangent derivatives}
    One main difficulty when studying boundary value problems is that techniques involving integration by parts
introduce troublesome boundary terms. 
A key way to avoid this problem is to use extension operators to reduce the problem to a manifold without boundary; we've already seen  examples
of this in Proposition \ref{Prop::HorVfs::GainBoundary::MainGainProp} and Lemma \ref{Lemma::EstBdry::Reduction::BoundLowerxnDerivsByHigher}.
In this section we introduce a more complicated version of this, following ideas of Kohn
and Nirenberg \cite{KohnNirenbergNonCoerciveBoundaryValueProblems}.

Let \(\opP\) be a differential operator on \(\Ropn\) of the form
\begin{equation*}
    \opP= a_0\partial_t + b_0 \partial_{x_n}^{2\kappa} + \sum_{\substack{|\alpha|\leq 2\kappa \\ \alpha_n<2\kappa}} c_\alpha \partial_x^{\alpha},
    \quad a_0, b_0, c_\alpha\in \CinftySpace[\CubengeqOne][\MatrixSpace[M][M]],
\end{equation*}
and assume \(b_0(x)\) is invertible, \(\forall x\in \CubengeqOneHalf\) with \(\sup_{x\in \CubengeqOneHalf}\MatrixNorm{b_0(x)^{-1}}[M][M][\C]\leq C_0<\infty\).

\begin{proposition}\label{Prop::EstBdry::ReduceTangent::MainReduceTangentProp}
    For all \(\phi_1,\phi_2\in \CinftycptSpace[\R\times \CubengeqOneHalf]\) with \(\phi_1\prec \phi_2\),
    \(\forall l\in \Zgeq\), \(l\geq 2\kappa-1\), \(\forall s\in \R\), \(\exists C\geq 1\), \(\forall u\in \DistributionsZero[\Ropngeq]\),
    \begin{equation*}
        \sum_{j=0}^{l+1} \GsNorm*{\partial_{x_n}^j \phi_1 u}{s-j}[\Ropngeq]
        \leq C
        \left( \sum_{j=0}^{l+1-2\kappa} \GsNorm*{\partial_{x_n}^j \phi_2 \opP u}{s-2\kappa-j}[\Ropngeq] 
        +\GsNorm*{\phi_2 u}{s}[\Ropngeq]\right),
    \end{equation*}
    where if the right-hand side is finite, so is the left-hand side.
    \(C\geq 1\) can be chosen to depend only on \(n\), \(\phi_1\), \(\phi_2\), \(l\), \(s\), \(M\), \(\kappa\), 
    \(C_0\), and upper bounds for
    \(\CjNorm{a_0}{L}[\CubengeqOneHalf][\MatrixSpace[M][M]]\), \(\CjNorm{b_0}{L}[\CubengeqOneHalf][\MatrixSpace[M][M]]\),
    \(\max_{\alpha}\CjNorm{c_\alpha}{L}[\CubengeqOneHalf][\MatrixSpace[M][M]]\), where \(L\) can be chosen to depend only
    on \(n\), \(l\), and \(s\).
\end{proposition}

\begin{remark}
    Proposition \ref{Prop::EstBdry::ReduceTangent::MainReduceTangentProp} shows that if we understand the regularity of
    \(\opP u\), then we can use that to control the \(\partial_{x_n}\) derivatives for \(u\); thereby reducing
    understanding the regularity of \(u\) to understanding the regularity of \(u\) in directions tangent to the boundary.
    Once that is done, we can study the regularity in directions tangent to the boundary without boundary terms in the relevant
    integrations by parts.
\end{remark}

Due to the cutoff functions \(\phi_1\) and \(\phi_2\), only the values of the coefficients on \(\R\times \CubengeqOneHalf\)
play a role.
By multiplying \(\opP\) by \(b_0^{-1}\) it suffices to assume \(b_0=1\). Since the result does not depend on the values of \(a_0\) and \(c_\alpha\)
outside of \(\CubengeqOneHalf\), we may henceforth assume \(a_0,c_\alpha\in \CinftycptSpace[\Rn][\MatrixSpace[M][M]]\).
Thus, we assume \(\opP\) is of the form
\begin{equation}\label{Eqn::EstBdry::ReduceTangent::ReducedFormulaForP}
    \opP = a_0 \partial_t +\partial_{x_n}^{2\kappa} + \sum_{\substack{|\alpha|\leq 2\kappa \\ \alpha_n<2\kappa}} c_\alpha \partial_x^{\alpha}, \quad
    a_0, c_\alpha\in \CinftycptSpace[\Rn][\MatrixSpace[M][M]].
\end{equation}

For the remainder of this section, the fact that the coefficients are matrix valued does not play a role. Thus, we drop
it from our notation, with the understanding that all functions are vector valued (taking values in \(\C^M\)), and the coefficients
of \(\opP\) are matrix valued (in \(\MatrixSpace[M][M][\C]\)).

\begin{lemma}\label{Lemma::EstBdry::ReduceTangent::TradePowersForConsts}
    Fix \(l\in \Zgeq\). Then,
    \begin{equation*}
        a^j\leq \sconst a^{l+1}+\lconst 1, \quad \forall a\geq 0,\:  0\leq j\leq l.
    \end{equation*}
\end{lemma}
\begin{proof}
    Fix \(\epsilon\in (0,1]\). If \(a>1/\epsilon\), then \(a^j\leq a^l \leq \epsilon a^{l+1}\).
    If \(a<1/\epsilon\), then \(a^{j}\leq \epsilon^{-l}\). We conclude \(a^j\leq \epsilon a^{l+1}+ \epsilon^{-l}\).
\end{proof}

We use the spaces \(\VlsSpace{l}{s}\) defined in \eqref{Eqn::EstBdry::Estension::DefinVlsNorm} and \eqref{Eqn::EstBdry::Estension::DefinVlsSpace}.

\begin{lemma}\label{Lemma::EstBdry::ReduceTangent::EstVNormOnWholeSpace}
    \(\forall l\in \Zgeq\), \(\forall s\in \R\),
    \begin{equation*}
        \VlsNorm*{u}{l}{s}[\R\times \Rnmo\times \R]
        \leq \sconst \VlsNorm*{\partial_{x_n}^{l+1} u}{0}{s-l-1}[\R\times \Rnmo\times \R]
        +\lconst \VlsNorm*{u}{0}{s}[\R\times \Rnmo\times \R], \quad \forall u\in \TemperedDistributions[\Ropn],
    \end{equation*}
    where if the right-hand side is finite, so is the left-hand side.
\end{lemma}
\begin{proof}
    Using Lemma \ref{Lemma::EstBdry::ReduceTangent::TradePowersForConsts}, we have with \(\JBracket{(\tau,\xi')}\) as in \eqref{Eqn::PDOs::DefinetxpJBracket},
    \begin{equation}\label{Eqn::EstBdry::ReduceTangent::EstVNormOnWholeSpace::Tmp1}
        |\xi_n|^j \JBracket{(\tau,\xi')}^{-j} \leq \sconst |\xi_n|^{l+1} \JBracket{(\tau,\xi')}^{-l-1} +\lconst 1.
    \end{equation}
    Using \eqref{Eqn::EstBdry::ReduceTangent::EstVNormOnWholeSpace::Tmp1} and the definition \eqref{Eqn::EstBdry::Estension::DefinVlsNorm}, we have
    \begin{equation*}
    \begin{split}
         &\VlsNorm*{u}{l}{s}[\R\times \Rnmo\times \R]^2 =
         \sum_{j=0}^l \LtNorm*{\partial_{x_n}^l \Lambdatxp[s-j] u}[\Ropn]^2
         = \sum_{j=0}^l (2\pi)^{2j} \LtNorm*{ |\xi_n|^j  \JBracket{(\tau,\xi')}^{s-j} \hat{u} }[\Ropn]^2
         \\&\leq \sconst \LtNorm*{|\xi_n|^{l+1} \JBracket{(\tau,\xi')}^{s-(l+1)} \hat{u}}[\Ropn]^2
         +\lconst \LtNorm*{\JBracket{(\tau,\xi')}^{s} \hat{u}}[\Ropn]^2
         \\&=\sconst \LtNorm*{\partial_{x_n}^{l+1}\Lambdatxp[s-(l+1)] u}[\Ropn]^2 + \lconst \LtNorm*{\Lambdatxp[s] u}[\Ropn]^2
         \\&=\sconst \VlsNorm*{\partial_{x_n}^{l+1} u}{0}{s-l-1}[\R\times \Rnmo\times \R]
        +\lconst \VlsNorm*{u}{0}{s}[\R\times \Rnmo\times \R].
    \end{split}
    \end{equation*}
\end{proof}

\begin{lemma}\label{Lemma::EstBdry::ReduceTangent::EstVNormOnSubspace}
    Fix intervals \((a,b),(c,d)\subseteq \R\) and a bounded, smoothly bounded, open set \(\Omega\Subset \Rnmo\). Then,
    \(\forall l\in \Zgeq\), \(\forall s\in \R\),
    \begin{equation}\label{Eqn::EstBdry::ReduceTangent::EstVNormOnSubspace::Conclusion}
         \VlsNorm*{u}{l}{s}[(c,d)\times \Omega\times (a,b)]
         \leq \sconst \VlsNorm*{\partial_{x_n}^{l+1} u}{0}{s-l-1}[(c,d)\times \Omega\times (a,b)]
        +\lconst \VlsNorm*{u}{0}{s}[(c,d)\times \Omega\times (a,b)],
    \end{equation}
    where if the right-hand side is finite, so is the left-hand side.
\end{lemma}
\begin{proof}
    Let \(\Extension_L\) be as in Lemma \ref{Lemma::EstBdry::Extension::VlsExtension}, with \(l+1+\max\{\floor{s},0\}\leq L\).
    We have, using Lemma \ref{Lemma::EstBdry::ReduceTangent::EstVNormOnWholeSpace},
    \begin{equation}\label{Eqn::EstBdry::ReduceTangent::EstVNormOnSubspace::Tmp1}
        \begin{split}
            &\VlsNorm*{u}{l}{s}[(c,d)\times \Omega\times (a,b)]
            \approx \VlsNorm*{\Extension_L u}{l}{s}[\R\times \Rnmo\times \R]
            \\&\leq \sconst \VlsNorm*{\partial_{x_n}^{l+1}\Extension_L u}{0}{s-l-1}[\R\times \Rnmo\times \R]
            +\lconst \VlsNorm*{\Extension_L u}{0}{s}[\R\times \Rnmo\times \R]
            \\&\leq \sconst \VlsNorm*{\partial_{x_n}^{l+1}\Extension_L u}{0}{s-l-1}[\R\times \Rnmo\times \R]
            +\lconst \VlsNorm*{u}{0}{s}[(c,d)\times \Omega\times (a,b)].
        \end{split}
    \end{equation}
    We have, using  Lemma \ref{Lemma::EstBdry::Extension::VlsExtension},
    \begin{equation}\label{Eqn::EstBdry::ReduceTangent::EstVNormOnSubspace::Tmp2}
    \begin{split}
         &\VlsNorm*{\partial_{x_n}^{l+1} \Extension_L u}{0}{s-l-1}[\R\times \Rnmo\times \R]^2
         \leq \sum_{j=0}^{l+1}\VlsNorm*{\partial_{x_n}^{j} \Extension_L u}{0}{s-j}[\R\times \Rnmo\times \R]^2
         \\&=\VlsNorm*{\Extension_L u}{l+1}{s}[\R\times \Rnmo\times \R]^2
         \approx \VlsNorm*{u}{l+1}{s}[(c,d)\times \Omega\times (a,b)]^2
         \\&=\VlsNorm*{u}{l}{s}[(c,d)\times \Omega\times (a,b)]^2 +  \VlsNorm*{\partial_{x_n}^{l+1}u}{0}{s-l-1}[(c,d)\times \Omega\times (a,b)]^2.
    \end{split}
    \end{equation}
    Plugging \eqref{Eqn::EstBdry::ReduceTangent::EstVNormOnSubspace::Tmp2} into the right-hand side of
    \eqref{Eqn::EstBdry::ReduceTangent::EstVNormOnSubspace::Tmp1} and subtracting
    \(\sconst \VlsNorm*{u}{l}{s}[(c,d)\times \Omega\times (a,b)]\) from both sides of the resulting inequality completes the proof.
\end{proof}

\begin{lemma}\label{Lemma::EstBdry::ReduceTangent::IntroduceP}
    Fix intervals \((a,b),(c,d)\subseteq \R\) and a bounded, smoothly bounded, open set \(\Omega\Subset \Rnmo\). Then,
    \(\forall l\in \Zgeq\), \(l\geq 2\kappa-1\), \(\forall s\in \R\),
    \begin{equation*}
        \VlsNorm*{u}{l+1}{s}[(c,d)\times \Omega\times (a,b)]
        \lesssim \VlsNorm*{\opP u}{l+1-2\kappa}{s-2\kappa}[(c,d)\times \Omega\times (a,b)] + \VlsNorm*{u}{0}{s}[(c,d)\times \Omega\times (a,b)],
    \end{equation*}
    where if the right-hand side is finite, so is the left-hand side.
\end{lemma}
\begin{proof}
    We have, using the formula \eqref{Eqn::EstBdry::ReduceTangent::ReducedFormulaForP} for \(\opP\),
    \begin{equation}\label{Eqn::EstBdry::ReduceTangent::IntroduceP::Tmp1}
    \begin{split}
            &\VlsNorm*{\partial_{x_n}^{l+1} u}{0}{s-l-1}[(c,d)\times \Omega\times (a,b)]
            \\&\lesssim \VlsNorm*{\partial_{x_n}^{l+1-2\kappa}\opP u}{0}{s-l-1}[(c,d)\times \Omega\times (a,b)]
            +\sum_{\substack{|\alpha|\leq 2\kappa \\ \alpha_n<2\kappa}} \VlsNorm*{\partial_{x_n}^{l+1-2\kappa}\partial_x^\alpha u}{0}{s-l-1}[(c,d)\times \Omega\times (a,b)]
            +\VlsNorm*{\partial_{x_n}^{l+1-2\kappa}\partial_t u}{0}{s-l-1}[(c,d)\times \Omega\times (a,b)]
            \\&=\VlsNorm*{\partial_{x_n}^{l+1-2\kappa}\opP u}{0}{s-l-1}[(c,d)\times \Omega\times (a,b)]
            +\sum_{\alpha_n=0}^{2\kappa-1} \sum_{|\alpha'|\leq 2\kappa-\alpha_n} \VlsNorm*{\partial_{x'}^{\alpha'} \partial_{x_n}^{l+1-2\kappa+\alpha_n} u }{0}{s-l-1}[(c,d)\times \Omega\times (a,b)]
            +\VlsNorm*{\partial_{x_n}^{l+1-2\kappa}\partial_t u}{0}{s-l-1}[(c,d)\times \Omega\times (a,b)]
            \\&\lesssim \VlsNorm*{\partial_{x_n}^{l+1-2\kappa}\opP u}{0}{s-l-1}[(c,d)\times \Omega\times (a,b)]
            +\sum_{\alpha_n=0}^{2\kappa-1} \VlsNorm*{\partial_{x_n}^{l+1-2\kappa+\alpha_n} u}{0}{s-l-1+2\kappa-\alpha_n}[(c,d)\times \Omega\times (a,b)]
            +\VlsNorm*{\partial_{x_n}^{l+1-2\kappa}\partial_t u}{0}{s-l-1}[(c,d)\times \Omega\times (a,b)]
            \\&\lesssim \VlsNorm*{\partial_{x_n}^{l+1-2\kappa}\opP u}{0}{s-l-1}[(c,d)\times \Omega\times (a,b)] + \VlsNorm*{u}{l}{s}[(c,d)\times \Omega\times (a,b)].
    \end{split}    
    \end{equation}
    Plugging \eqref{Eqn::EstBdry::ReduceTangent::IntroduceP::Tmp1} into the right-hand side of
    \eqref{Eqn::EstBdry::ReduceTangent::EstVNormOnSubspace::Conclusion} and subtracting
    \(\sconst \VlsNorm*{u}{l}{s}[(c,d)\times \Omega\times (a,b)]\) from both sides yields
    \begin{equation}\label{Eqn::EstBdry::ReduceTangent::IntroduceP::Tmp2}
        \VlsNorm*{u}{l}{s}[(c,d)\times \Omega\times (a,b)]
        \lesssim \VlsNorm*{\partial_{x_n}^{l+1-2\kappa}\opP u}{0}{s-l-1}[(c,d)\times \Omega\times (a,b)]
        +\VlsNorm*{u}{0}{s}[(c,d)\times \Omega\times (a,b)].
    \end{equation}
    Using \eqref{Eqn::EstBdry::Estension::DefinVlsNorm}, \eqref{Eqn::EstBdry::ReduceTangent::IntroduceP::Tmp1}, and 
    \eqref{Eqn::EstBdry::ReduceTangent::IntroduceP::Tmp2}, we have
    \begin{equation*}
    \begin{split}
         &\VlsNorm*{u}{l+1}{s}[(c,d)\times \Omega\times (a,b)]
         \approx \VlsNorm*{\partial_{x_n}^{l+1}u}{0}{s-l-1}[(c,d)\times \Omega\times (a,b)] + \VlsNorm*{u}{l}{s}[(c,d)\times \Omega\times (a,b)]
         \\&\lesssim \VlsNorm*{\partial_{x_n}^{l+1-2\kappa}\opP u}{0}{s-l-1}[(c,d)\times \Omega\times (a,b)]
         +\VlsNorm*{u}{l}{s}[(c,d)\times \Omega\times (a,b)]
         \\&\lesssim \VlsNorm*{\partial_{x_n}^{l+1-2\kappa}\opP u}{0}{s-l-1}[(c,d)\times \Omega\times (a,b)] + \VlsNorm*{u}{0}{s}[(c,d)\times \Omega\times (a,b)]
         \\&\lesssim \VlsNorm*{\opP u}{l+1-2\kappa}{s-2\kappa}[(c,d)\times \Omega\times (a,b)] + \VlsNorm*{u}{0}{s}[(c,d)\times \Omega\times (a,b)]
    \end{split}
    \end{equation*}
\end{proof}

Let \(\phi_1\) and \(\phi_2\)
be as in Proposition \ref{Prop::EstBdry::ReduceTangent::MainReduceTangentProp}.
Let \(V_1\) and \(V_2\) be open sets with \(\supp(\phi_1)\Subset V_1 \Subset V_2 \Subset \left\{ \phi_2=1 \right\} \).
Set \(\Compact=\overline{V_1}\Subset V_2\), so that \(\Compact\Subset \Ropngeq\) is compact.
Cover \(\Compact\) by a finite collection of open sets:
\begin{equation*}
    (c_1,d_1)\times B_1\times (a_1,b_1),\ldots, (c_K, d_K)\times B_K\times (a_K, b_K),
\end{equation*}
where each \(B_j\subset \Rnmo\) is a ball, \(a_j\) can be negative (though \(b_j>0\)), and
\begin{equation*}
    \left[ (c_j, d_j)\times B_j \times (a_j,b_j) \right]\cap \Rngeq \subseteq V_2, \quad 1\leq j\leq K.
\end{equation*}

\begin{lemma}\label{Lemma::EstBdry::ReduceTangent::BoundGsByV}
    For all \(l\in \Zgeq\) and \(s\in \R\),
    \begin{equation*}
        \sum_{j=0}^{l+1} \GsNorm*{\partial_{x_n}^j \phi_1 u}{s-j}[\Ropngeq]
        \lesssim \sum_{k=1}^K \VlsNorm{u}{l+1}{s}[(c_k,d_k)\times B_k \times (a_k\vee 0,b_k)],
    \end{equation*}
    where if the right-hand side is finite, so is the left-hand side.
\end{lemma}
\begin{proof}
    Pick \(\psi_k\in \CinftySpace[(c_k,d_k)\times B_k \times (a_k,b_k)]\) such that
    \(\sum_{k=1}^K \psi_k=1\) on a neighborhood of \(\Compact\).
    We have,
    \begin{equation*}
    \begin{split}
         &\sum_{j=0}^{l+1} \GsNorm*{\partial_{x_n}^j \phi_1 u}{s-j}[\Ropngeq]^2
         \leq \sum_{j=0}^{l+1}\sum_{k=1}^K \GsNorm*{\psi_k \partial_{x_n}^j \phi_1 u}{s-j}[\Ropngeq]^2
         \\& \lesssim  \sum_{j=0}^{l+1}\sum_{k=1}^K \int_0^\infty \HstxpNorm*{\psi_k(\cdot,x_n) \partial_{x_n}^j u(\cdot, x_n)}{s-j}[\R\times \Rnmo]^2\: dx_n
         \\&=\sum_{j=0}^{l+1}\sum_{k=1}^K \int_{a_k\vee 0}^{b_k} \HstxpNorm*{\psi_k(\cdot,x_n) \partial_{x_n}^j u(\cdot, x_n)}{s-j}[\R\times \Rnmo]^2\: dx_n
         \\&\lesssim \sum_{j=0}^{l+1}\sum_{k=1}^K \int_{a_k\vee 0}^{b_k} \HstxpNorm*{\partial_{x_n}^j u(\cdot, x_n)}{s-j}[(c_k,d_k)\times B_k]^2\: dx_n
         \\&= \sum_{k=1}^K \VlsNorm*{u}{l+1}{s}[(c_k,d_k)\times B_k\times (a_k\vee 0,b_k)]^2.
    \end{split}
    \end{equation*}
\end{proof}

\begin{lemma}\label{Lemma::EstBdry::ReduceTangent::BoundVByGs}
    For \(l\in \Zgeq\), \(s\in \R\), and \(1\leq k\leq K\),
    \begin{equation*}
        \VlsNorm{u}{l}{s}[(c_k,d_k)\times B_k\times (a_k\vee 0,b_k)]^2
        \leq \sum_{j=0}^l \GsNorm{\partial_{x_n}^j \phi_2 u}{s-j}[\Ropngeq]^2,
    \end{equation*}
    where if the right-hand side is finite, so is the left-hand side.
\end{lemma}
\begin{proof}
    Since \(\phi_2=1\) on \(V_2\supseteq (c_k,d_k)\times B_k\times (a_k\vee 0,b_k)\),
    \(\partial_{x_n}^j \phi_2 u(\cdot, x_n)\) is an extension  of \(\partial_{x_n}^j u(\cdot,x_n) \big|_{(c_k,d_k)\times B_k}\).
    Thus, we have
    \begin{equation*}
        \begin{split}
            &\VlsNorm*{u}{l}{s}[(c_k,d_k)\times B_k\times (a_k\vee 0,b_k)]^2
            =\sum_{j=0}^l \int_{a_k\vee 0}^{b_k} \HstxpNorm*{ \partial_{x_n}^j u(\cdot, x_n) }{s-j}[(c_k,d_k)\times B_k]^2 \: dx_n
            \\&\leq \sum_{j=0}^l \int_{a_k\vee 0}^{b_k} \HstxpNorm*{ \partial_{x_n}^j \phi_2 u(\cdot, x_n) }{s-j}[\R\times \Rnmo]^2 \: dx_n
            \leq \sum_{j=0}^l \int_0^\infty \HstxpNorm*{ \partial_{x_n}^j \phi_2 u(\cdot, x_n) }{s-j}[\R\times \Rnmo]^2 \: dx_n
            \\&=\sum_{j=0}^l \GsNorm*{\partial_{x_n}^j \phi_2 u}{s-j}[\Ropngeq]^2.
        \end{split}
    \end{equation*}
\end{proof}

\begin{proof}[Proof of Proposition \ref{Prop::EstBdry::ReduceTangent::MainReduceTangentProp}]
    \begin{align}
        &\sum_{j=0}^{l+1} \GsNorm*{\partial_{x_n}^j \phi_1 u}{s-j}[\Ropngeq]
        \lesssim \sum_{k=1}^K \VlsNorm*{u}{l+1}{s}[(c_k,d_k)\times B_k\times(a_k\vee 0,b_k)]\label{Eqn::EstBdry::ReduceTangent::MainReduceTangentProp::Tmp1}
        \\&\lesssim \sum_{k=1}^K \VlsNorm*{\opP u}{l+1-2\kappa}{s-2\kappa}[(c_k,d_k)\times B_k\times (a_k\vee 0,b_k)]
            +\sum_{k=1}^K \VlsNorm*{u}{0}{s}[(c_k,d_k)\times B_k\times (a_k\vee 0,b_k)]\label{Eqn::EstBdry::ReduceTangent::MainReduceTangentProp::Tmp2}
        \\&\lesssim \sum_{j=0}^{l+1-2\kappa} \GsNorm{\partial_{x_n}^j \phi_2 \opP u}{s-2\kappa-j}[\Ropngeq]
        +\GsNorm*{\phi_2 u}{s}[\Ropngeq],\label{Eqn::EstBdry::ReduceTangent::MainReduceTangentProp::Tmp3}
    \end{align}
    where \eqref{Eqn::EstBdry::ReduceTangent::MainReduceTangentProp::Tmp1} uses Lemma \ref{Lemma::EstBdry::ReduceTangent::BoundGsByV},
    \eqref{Eqn::EstBdry::ReduceTangent::MainReduceTangentProp::Tmp2} uses Lemma \ref{Lemma::EstBdry::ReduceTangent::IntroduceP},
    and \eqref{Eqn::EstBdry::ReduceTangent::MainReduceTangentProp::Tmp3} uses Lemma \ref{Lemma::EstBdry::ReduceTangent::BoundVByGs}.
\end{proof}

    \subsection{Estimates in tangent directions}\label{Section::EstNearBdry::EstimatesInTangentDirections}
    \begin{proposition}\label{Prop::EstBdry::Tangent::MainTangentProp}
    In the same setting as Theorem \ref{Thm::EstBdry::MainThm::New} (after the reductions from Proposition \ref{Prop::EstBdry::Reduction::MainReduction}
    have been applied), we have, \(\forall s\in \R\),
    \begin{equation}\label{Eqn::EstBdry::Tangent::MainTangentProp::Conclusion}
        \GsNorm*{\phi_1 u}{s+\epsilon_0}[\Ropngeq]
        \lesssim
            \GsNorm*{\phi_2 \left( \partial_t+\opL \right)u}{s}[\Ropngeq] + \LtNorm*{\phi_2 u}[\Ropngeq],
        \quad \forall u\in \DSetL{1},
    \end{equation}
    where if the right-hand side is finite, so is the left-hand side.
\end{proposition}


Fix an open set \(U_2'\) and \(\delta_2'\in(\delta_1,\delta_2)\) such that 
\(U_1\Subset U_2'\Subset U_2\) and \(\phi_2=1\)
on a neighborhood of
\begin{equation*}
\overline{U_2'\times \Cubengeq{\delta_2'}}.    
\end{equation*}
Let \(\Stxp{s}\) and \(\Stxpt{s}\) be as described in Section \ref{Section::PseudodifferentialOps::Bdry}
with \(\delta_1,\delta_2',U_1,U_2'\) in place of \(\delta_1,\delta_2,U_1,U_2\). 
We use the \(\sconst\) and \(\lconst\) notation from Section \ref{Section::EstNearBdry::LargeSmallSonsts}.

\begin{lemma}\label{Lemma::EstBdry::Tangent::MovingdtsPastW}
    Let \(w\in \CinftySpace[\Ropngeq]\) be such that \(\CjNorm{w}{j}\lesssim 1\), \(\forall j\).
    Then, for \(u\in \DSetL{1}\),
    \begin{equation*}
        \begin{split}
            &\left| \Ltip*{\partial_t \left( 1-\partial_t^2 \right)^{-1/2} \Stxp{s}[1]u }{ w \Stxp{s-1}[2] \partial_t u}[\Ropngeq] \right|
            \\&\leq \sconst\LtNorm*{  \partial_t \left( 1-\partial_t^2 \right)^{-1/4} \Stxp{s}[1]u}[\Ropngeq]^2
            +\lconst \LtNorm*{\partial_t \left( 1-\partial_t^2 \right)^{-1/4} \Stxpt{s-1} u}[\Ropngeq]^2
            \\&\quad\quad\quad+\lconst \LtNorm*{\Stxpt{s-1}u}[\Ropngeq]^2.
        \end{split}
    \end{equation*}
\end{lemma}
\begin{proof}
    We have
    \begin{equation}\label{Eqn::EstBdry::Tangent::MovingdtsPastW::Tmp1}
        \begin{split}
            &\left| \Ltip*{\partial_t \left( 1-\partial_t^2 \right)^{-1/2} \Stxp{s}[1]u }{ w \Stxp{s-1}[2] \partial_t u}[\Ropngeq] \right|
            \\&\leq \sconst \LtNorm*{  \partial_t \left( 1-\partial_t^2 \right)^{-1/4} \Stxp{s}[1]u}[\Ropngeq]^2
            +\lconst \LtNorm*{\left( 1-\partial_t^2 \right)^{-1/4}w \Stxp{s-1}[2] \partial_t u}[\Ropngeq]^2.
        \end{split}
    \end{equation}
    The classical calculus of pseudodifferential operators in the \(t\)-variable shows
    \begin{equation}\label{Eqn::EstBdry::Tangent::MovingdtsPastW::Tmp2}
        \LtOpNorm*{\left( 1-\partial_t^2 \right)^{-1/4} w \left( 1-\partial_t^2 \right)^{1/4}}\lesssim 1.
    \end{equation}
    \eqref{Eqn::EstBdry::Tangent::MovingdtsPastW::Tmp2} implies
    \begin{equation}\label{Eqn::EstBdry::Tangent::MovingdtsPastW::Tmp3}
        \LtNorm*{\left( 1-\partial_t^2 \right)^{-1/4}w \Stxp{s-1}[2] \partial_t u}[\Ropngeq]
        \lesssim \LtNorm*{\left( 1-\partial_t^2 \right)^{-1/4} \Stxp{s-1}[2] \partial_t u}[\Ropngeq].
    \end{equation}
    Lemma \ref{Lemma::PDOs::NotationNearBdry::CalcOfPDO}\ref{Item::PDOs::NotationNearBdry::CalcOfPDO::CommutatorOfPartialt}
    shows
    \begin{equation}\label{Eqn::EstBdry::Tangent::MovingdtsPastW::Tmp4}
        \begin{split}
            &\LtNorm*{\left( 1-\partial_t^2 \right)^{-1/4} \Stxp{s-1}[2] \partial_t u}[\Ropngeq]
            \\&\lesssim  \LtNorm*{\left( 1-\partial_t^2 \right)^{-1/4} \partial_t \Stxp{s-1}[2]  u}[\Ropngeq]
            +\LtNorm*{\left( 1-\partial_t^2 \right)^{-1/4}  \Stxpt{s-1}  u}[\Ropngeq]
            \\&\leq \LtNorm*{\left( 1-\partial_t^2 \right)^{-1/4} \partial_t \Stxp{s-1}[2]  u}[\Ropngeq]
            +\LtNorm*{ \Stxpt{s-1}  u}[\Ropngeq].
        \end{split}
    \end{equation}
    Combining \eqref{Eqn::EstBdry::Tangent::MovingdtsPastW::Tmp1}, \eqref{Eqn::EstBdry::Tangent::MovingdtsPastW::Tmp3},
    and \eqref{Eqn::EstBdry::Tangent::MovingdtsPastW::Tmp4} completes the proof.
\end{proof}

\begin{lemma}\label{Lemma::EstBdry::Tangent::BounddtStxp}
    There exist \(C_1,C_2\geq 1\),
    \(\forall u\in \DSetL{1}\),
    \begin{equation*}
    \begin{split}
            &\LtNorm*{ \partial_t \left( 1-\partial_t^2 \right)^{-1/4} \Stxp{s} u}[\Ropngeq]^2
            \\&\leq C_1\Real \FormQH*[ \left( \Stxp{s} \right)^{*} \partial_t(1-\partial_t^2)^{-1/2} \Stxp{s} u][u]
             \\&\quad+            
                C_2\left| \int_{-\infty}^{\infty} \FormQ*[ \left( \Stxp{s} \right)^{*} \partial_t\left( 1-\partial_t^2 \right)^{-1/2} \Stxp{s}u(t,\cdot) ][u(t,\cdot)] \: dt \right| 
            + C_2\LtNorm*{\Stxpt{s}u}[\Ropngeq]^2
            \\&\quad+C_2\LtNorm*{ \partial_t \left( 1-\partial_t^2 \right)^{-1/4} \Stxpt{s-1} u}[\Ropngeq]^2.
    \end{split}    
    \end{equation*}
\end{lemma}
\begin{proof}
    Because \(\sigma\) does not depend on \(t\), \(\sigma\approx 1\) on \(\CubengeqOneHalf\), and due to the support of the kernel
    of \(\Stxp{s}\), we have
    \begin{equation}\label{Eqn::EstBdry::Tangent::BounddtStxp::Tmp1}
        \LtNorm*{ \partial_t \left( 1-\partial_t^2 \right)^{-1/4} \Stxp{s} u}[\Ropngeq]
        \approx \LtNorm*{\sigma^{1/2} \partial_t \left( 1-\partial_t^2 \right)^{-1/4} \Stxp{s} u}[\Ropngeq].
    \end{equation}
    We have, using Lemma \ref{Lemma::PDOs::NotationNearBdry::CalcOfPDO}\ref{Item::PDOs::NotationNearBdry::CalcOfPDO::CommutatorOfPartialt},\ref{Item::PDOs::NotationNearBdry::CalcOfPDO::CommutatorMultiplication},
    and that \(\sigma\) does not depend on \(t\),
    \begin{equation}\label{Eqn::EstBdry::Tangent::BounddtStxp::Tmp2}
    \begin{split}
         &\LtNorm*{\sigma^{1/2} \partial_t \left( 1-\partial_t^2 \right)^{-1/4} \Stxp{s} u}[\Ropngeq]^2
         =\Real \Ltip*{\partial_t (1-\partial_t^2)^{-1/2} \Stxp{s}u}{\sigma \partial_t \Stxp{s}u}[\Ropngeq]
         \\&=\Real \Ltip*{\left( \Stxp{s} \right)^{*}\partial_t (1-\partial_t^2)^{-1/2} \Stxp{s}u}{\sigma \partial_t u}[\Ropngeq]
         \\&\quad\quad+\Real \Ltip*{\partial_t (1-\partial_t^2)^{-1/2} \Stxp{s}u}{\sigma  \Stxpt{s}u}[\Ropngeq]
         \\&\quad\quad + \Real \Ltip*{\partial_t (1-\partial_t^2)^{-1/2} \Stxp{s}u}{ \wtEst \Stxpt{s-1}\partial_t u}[\Ropngeq]
         \\&=\Real \FormQH*[\left( \Stxp{s} \right)^{*}\partial_t (1-\partial_t^2)^{-1/2} \Stxp{s}u][u]
         \\&\quad\quad-\int_{-\infty}^{\infty} \Real \FormQ*[\left( \Stxp{s} \right)^{*}\partial_t (1-\partial_t^2)^{-1/2} \Stxp{s}u(t,\cdot)][u(t,\cdot)]\: dt
         \\&\quad\quad + \Real \Ltip*{\partial_t (1-\partial_t^2)^{-1/2} \Stxp{s}u}{ \sigma \Stxpt{s}u}[\Ropngeq]
         \\&\quad\quad+ \Real \Ltip*{\partial_t (1-\partial_t^2)^{-1/2} \Stxp{s}u}{ \wtEst \Stxpt{s-1}\partial_t u}[\Ropngeq].
    \end{split}
    \end{equation}
    We have,
    \begin{equation}\label{Eqn::EstBdry::Tangent::BounddtStxp::Tmp3}
         \begin{split}
            &\left| \Ltip*{\partial_t (1-\partial_t^2)^{-1/2} \Stxp{s}u}{ \sigma \Stxpt{s}u}[\Ropngeq] \right|
            \\&\lesssim \LtNorm*{\partial_t (1-\partial_t^2)^{-1/2} \Stxp{s}u}[\Ropngeq]^2 + \LtNorm*{\Stxpt{s}u}[\Ropngeq]^2
            \\&\lesssim \LtNorm*{\Stxpt{s}u}[\Ropngeq]^2.
         \end{split}
    \end{equation}
    Applying Lemma \ref{Lemma::EstBdry::Tangent::MovingdtsPastW} to the final term on the right-hand side of
    \eqref{Eqn::EstBdry::Tangent::BounddtStxp::Tmp2}, and using \eqref{Eqn::EstBdry::Tangent::BounddtStxp::Tmp1},
    \eqref{Eqn::EstBdry::Tangent::BounddtStxp::Tmp2}, and \eqref{Eqn::EstBdry::Tangent::BounddtStxp::Tmp3},
    we see
    \begin{equation}\label{Eqn::EstBdry::Tangent::BounddtStxp::Tmp4}
    \begin{split}
         &\LtNorm*{ \partial_t \left( 1-\partial_t^2 \right)^{-1/4} \Stxp{s} u}[\Ropngeq]
         \leq
            C_1' \Real \FormQH*[\left( \Stxp{s} \right)^{*}\partial_t (1-\partial_t^2)^{-1/2} \Stxp{s}u][u]
            \\&\quad+            
                C_2'\left| \int_{-\infty}^{\infty} \FormQ*[ \left( \Stxp{s} \right)^{*} \partial_t\left( 1-\partial_t^2 \right)^{-1/2} \Stxp{s}u(t,\cdot) ][u(t,\cdot)] \: dt \right| 
         \\&\quad+\sconst\LtNorm*{  \partial_t \left( 1-\partial_t^2 \right)^{-1/4} \Stxp{s}u}[\Ropngeq]^2
            \\&\quad+\lconst \LtNorm*{\partial_t \left( 1-\partial_t^2 \right)^{-1/4} \Stxpt{s-1} u}[\Ropngeq]^2
            \\&\quad+\lconst \LtNorm*{\Stxpt{s}u}[\Ropngeq]^2.
    \end{split}
    \end{equation}
    Subtracting \(\sconst\LtNorm*{  \partial_t \left( 1-\partial_t^2 \right)^{-1/4} \Stxp{s}u}[\Ropngeq]^2\) from both sides
    of \eqref{Eqn::EstBdry::Tangent::BounddtStxp::Tmp4} completes the proof.
\end{proof}


\begin{lemma}\label{Lemma::EstBdry::Tangent::BoundIntQByQH}
    \(\forall u\in \DSetL{1}\),
    \begin{equation*}
    \begin{split}
         &\left| \int_{-\infty}^{\infty} \FormQ*[ \left( \Stxp{s} \right)^{*} \partial_t\left( 1-\partial_t^2 \right)^{-1/2} \Stxp{s}u(t,\cdot) ][u(t,\cdot)] \: dt \right|
         \\&\leq C_1 \Real \FormQH*[\Stxpt{s} u][\Stxpt{s}u]
         +C_2 \LtNorm*{ \Stxpt{s}u}[\Ropngeq]^2.
    \end{split}
    \end{equation*}
\end{lemma}
\begin{proof}
    We have, recalling the vector fields \(Y_0,\ldots, Y_r\) do not depend on \(t\), and using 
    Lemma \ref{Lemma::PDOs::NotationNearBdry::CalcOfPDO}
    \begin{equation*}
    \begin{split}
         &\left| \int_{-\infty}^{\infty} \FormQ*[ \left( \Stxp{s} \right)^{*} \partial_t\left( 1-\partial_t^2 \right)^{-1/2} \Stxp{s}u(t,\cdot) ][u(t,\cdot)] \: dt \right|
         \\&\leq \sum_{|\alpha|,|\beta|\leq \kappa} \left| \int_{-\infty}^{\infty} \Ltip*{Y^{\alpha}\left( \Stxp{s} \right)^{*} \partial_t\left( 1-\partial_t^2 \right)^{-1/2} \Stxp{s}u(t,\cdot) }{\sigma a_{\alpha,\beta}Y^{\beta} u(t,\cdot)}[\Rngeq]\: dt \right|
         \\&\lesssim \sum_{|\alpha|,|\beta|\leq \kappa} \int_{-\infty}^{\infty} \left| \Ltip*{Y^{\alpha} \partial_t\left( 1-\partial_t^2 \right)^{-1/2} \Stxp{s} u(t,\cdot)}{Y^{\beta} \wtEst\Stxpt{s} u(t,\cdot)}[\Rngeq] \right|\: dt
         \\&=\sum_{|\alpha|,|\beta|\leq \kappa} \int_{-\infty}^{\infty} \left| \Ltip*{ \partial_t\left( 1-\partial_t^2 \right)^{-1/2}Y^{\alpha} \Stxp{s} u(t,\cdot)}{Y^{\beta} \wtEst\Stxpt{s} u(t,\cdot)}[\Rngeq] \right|\: dt
         \\&\lesssim \sum_{|\alpha|\leq \kappa} \left( \LtNorm*{\partial_t\left( 1-\partial_t^2 \right)^{-1/2}Y^{\alpha} \Stxp{s} u}[\Ropngeq]^2 +  \LtNorm*{Y^{\alpha}\wtEst\Stxpt{s}u}[\Ropngeq]^2 \right)
         \\&\lesssim \sum_{|\alpha|\leq \kappa}\LtNorm*{Y^{\alpha}\Stxpt{s}u}[\Ropngeq]^2.
    \end{split}
    \end{equation*}
    From here, the result follows from Proposition \ref{Prop::EstBdry::Reduction::MainReduction}\ref{Item::EstBdry::Reduction::MainReduction::BoundAllDerivs}
    with \(S=\Stxpt{s}\).
\end{proof}

\begin{lemma}\label{Lemma::EstBdry::Tangent::BoundMaxSubPlusTFormQH}
    \(\forall u\in \DSetL{1}\),
    \begin{equation}\label{Eqn::EstBdry::Tangent::BoundMaxSubPlusTFormQH::Conclusion}
    \begin{split}
         &\sum_{|\alpha|\leq \kappa} \LtNorm*{Y^{\alpha} \Stxp{s}u}[\Ropngeq]^2 + \LtNorm*{\partial_t \left( 1-\partial_t^2 \right)^{-1/4}\Stxp{s}u}[\Ropngeq]^2
         \\&\leq C_1 \Real \FormQH*[ \Stxpt{s} u][\Stxpt{s}u]
         +C_2 \Real \FormQH*[\Stxpt{s}\partial_t\left( 1-\partial_2^2 \right)^{-1/2} \Stxpt{s} u][u]
         \\&\quad\quad+C_3 \LtNorm*{\Stxpt{s}u}[\Ropngeq]^2
         +C_4\LtNorm*{ \partial_t \left( 1-\partial_t^2 \right)^{-1/4} \Stxpt{s-1} u}[\Ropngeq]^2.
    \end{split}
    \end{equation}
\end{lemma}
\begin{proof}
    This follows from Proposition \ref{Prop::EstBdry::Reduction::MainReduction}\ref{Item::EstBdry::Reduction::MainReduction::BoundAllDerivs}
    with \(S=\Stxp{s}\), and Lemmas
    \ref{Lemma::EstBdry::Tangent::BounddtStxp}
    and \ref{Lemma::EstBdry::Tangent::BoundIntQByQH}.
\end{proof}

By taking \(S=\Stxpt{s}\partial_t\left( 1-\partial_t^2 \right)^{-1/2} \Stxpt{s}\),
it follows from Assumption \ref{Assumption::EstBdry::MaxSubOnSmoothing}\ref{Item::EstBdry::MaxSubOnSmoothing::FormQHGivesOp}
that
\begin{equation*}
    \FormQH*[S u][u]=
    \Ltip*{S u}{\sigma \left( \partial_t+\opL \right)u}[\Ropngeq].
\end{equation*}
Thus, the second term on the right-hand side of \eqref{Eqn::EstBdry::Tangent::BoundMaxSubPlusTFormQH::Conclusion}
can be written in terms of \(\left( \partial_t+\opL \right)u\).  However, that is not immediately true of the first term.
We turn to studying this term, which requires a preliminary result.

\begin{lemma}\label{Lemma::EstBdry::Tangent::Changes1s2Tor1r2}
    Let \(b_{\alpha,\beta}, w_1,w_2\in \CinftySpace[\CubengeqOne][\MatrixSpace[M][M][\C]]\) be given functions
    and let \(s_1,s_2,r_1,r_2\in \R\) satisfy \(s_1+s_2=r_1+r_2\).  Then, \(\forall \kappa_1,\kappa_2\in \Zgeq 0\), with \(\kappa_1,\kappa_2\leq \kappa\),
    \begin{enumerate}[(i)]
    \item \label{Item::EstBdry::Tangent::SmoothOnBothSides}
    \(\forall u,v\in \LtSpaceNoMeasure[\R][\HsYSpace{\kappa}[\CubengeqOneHalf]]\),
    \begin{equation*}
    \begin{split}
         &\sum_{\substack{|\alpha|\leq \kappa_1 \\ |\beta|\leq \kappa_2}} 
         \left| \Ltip*{Y^{\alpha} w_1 \Stxp{s_1}[1] u}{b_{\alpha,\beta} Y^{\beta}w_2\Stxp{s_2}[2]v}[\Ropngeq] \right|
         \\&\lesssim \sum_{|\alpha|\leq \kappa_1} \LtNorm*{Y^{\alpha} \Stxpt{r_1}u}[\Ropngeq]^2
         +\sum_{|\beta|\leq \kappa_2} \LtNorm*{Y^{\beta} \Stxpt{r_2}v}[\Ropngeq]^2.
    \end{split}
    \end{equation*}

    \item \label{Item::EstBdry::Tangent::SmoothOnOneSide}
    \(\forall u,v\in \LtSpaceNoMeasure[\R][\HsYSpace{\kappa}[\CubengeqOneHalf]]\),
    \begin{equation*}
    \begin{split}
         &\sum_{\substack{|\alpha|\leq \kappa_1 \\ |\beta|\leq \kappa_2}} 
         \left| \Ltip*{Y^{\alpha} w_1\Stxp{s_2}[2]\Stxp{s_1}[1] u}{b_{\alpha,\beta} Y^{\beta}v}[\Ropngeq] \right|
         \\&\lesssim \sum_{|\alpha|\leq \kappa_1} \LtNorm*{Y^{\alpha} \Stxpt{r_1}u}[\Ropngeq]^2
         +\sum_{|\beta|\leq \kappa_2} \LtNorm*{Y^{\beta} \Stxpt{r_2}v}[\Ropngeq]^2.
    \end{split}
    \end{equation*}

    \item \label{Item::EstBdry::Tangent::SmoothOnBothSidesWithGain}
    \(\forall u,v\in \LtSpaceNoMeasure[\R][\HsYSpace{\kappa}[\CubengeqOneHalf]]\), if \(\kappa_1\geq 1\),
    \begin{equation*}
    \begin{split}
         &\sum_{\substack{|\alpha|\leq \kappa_1-1 \\ |\beta|\leq \kappa_2}} 
         \left| \Ltip*{Y^{\alpha} w_1\Stxp{s_1}[1] u}{b_{\alpha,\beta} Y^{\beta}w_2\Stxp{s_2}[2]v}[\Ropngeq] \right|
         \lesssim \sum_{|\alpha|\leq \kappa_1} \LtNorm*{Y^{\alpha} \Stxpt{r_1-\epsilon_0}u}[\Ropngeq]^2
         \\&\quad\quad+\sum_{|\beta|\leq \kappa_2} \LtNorm*{Y^{\beta} \Stxpt{r_2}v}[\Ropngeq]^2
         +\LtNorm*{\partial_t \left( 1-\partial_t^2 \right)^{-1/4}\Stxpt{r_1-\epsilon_0} u}[\Ropngeq]^2.
    \end{split}
    \end{equation*}

    \item \label{Item::EstBdry::Tangent::SmoothOnOneSideWithGain}
    \(\forall u,v\in \LtSpaceNoMeasure[\R][\HsYSpace{\kappa}[\CubengeqOneHalf]]\), if \(\kappa_1\geq 1\),
    \begin{equation*}
    \begin{split}
         &\sum_{\substack{|\alpha|\leq \kappa_1-1 \\ |\beta|\leq \kappa_2}} 
         \left| \Ltip*{Y^{\alpha} w_1\Stxp{s_2}[2]\Stxp{s_1}[1] u}{b_{\alpha,\beta} Y^{\beta}v}[\Ropngeq] \right|
         \lesssim \sum_{|\alpha|\leq \kappa_1} \LtNorm*{Y^{\alpha} \Stxpt{r_1-\epsilon_0}u}[\Ropngeq]^2
         \\&\quad\quad+\sum_{|\beta|\leq \kappa_2} \LtNorm*{Y^{\beta} \Stxpt{r_2}v}[\Ropngeq]^2
         +\LtNorm*{\partial_t \left( 1-\partial_t^2 \right)^{-1/4}\Stxpt{r_1-\epsilon_0} u}[\Ropngeq]^2.
    \end{split}
    \end{equation*}

    \end{enumerate}
\end{lemma}
\begin{proof}
    We begin with \ref{Item::EstBdry::Tangent::SmoothOnBothSides}.
    We have
    \begin{equation*}
    \begin{split}
         &\sum_{\substack{|\alpha|\leq \kappa_1\\ |\beta|\leq \kappa_2}}
        \left| \Ltip*{Y^{\alpha} w_1\Stxp{s_1}[1] u}{b_{\alpha,\beta} Y^{\beta}w_2\Stxp{s_2}[2]v}[\Ropngeq] \right|
        \\&=\sum_{\substack{|\alpha|\leq \kappa_1\\ |\beta|\leq \kappa_2}}
        \left| \Ltip*{\Lambdatxp[s_1-r_1]\Lambdatxp[r_1-s_1]Y^{\alpha} w_1\Stxp{s_1}[1] u}{b_{\alpha,\beta} Y^{\beta}w_2\Stxp{s_2}[2]v}[\Ropngeq] \right|
        \\&\lesssim
        \sum_{\substack{|\alpha|\leq \kappa_1\\ |\beta|\leq \kappa_2}}
        \left| \Ltip*{Y^{\alpha} \wtEst\Stxpt{r_1}[1] u}{ Y^{\beta}\wtEst\Stxpt{r_2}[2]v}[\Ropngeq] \right|,
    \end{split}
    \end{equation*}
    where the final line uses the calculus of pseudodifferential operators 
    (see Lemma \ref{Lemma::PDOs::NotationNearBdry::CalcOfPDO}\ref{Item::PDOs::NotationNearBdry::CalcOfPDO::CommutatorMultiplication}).
    From here, \ref{Item::EstBdry::Tangent::SmoothOnBothSides} follows by using the Cauchy--Schwarz inequality.
    
    We turn to \ref{Item::EstBdry::Tangent::SmoothOnBothSidesWithGain}.  Using \ref{Item::EstBdry::Tangent::SmoothOnBothSides},
     Proposition \ref{Prop::HorVfs::GainBoundary::MainGainProp}, and the fact that \(\epsilon_0=\min\{1/m,1/2\}\leq 1/m\), we have
    \begin{equation*}
    \begin{split}
         &\sum_{\substack{|\alpha|\leq \kappa_1-1 \\ |\beta|\leq \kappa_2}} 
         \left| \Ltip*{Y^{\alpha} w_1\Stxp{s_1}[1] u}{b_{\alpha,\beta} Y^{\beta}w_2\Stxp{s_2}[2]v}[\Ropngeq] \right|
         \\&\lesssim \sum_{|\alpha|\leq \kappa_1-1} \LtNorm*{Y^{\alpha} \Stxpt{r_1}u}[\Ropngeq]^2
         +\sum_{|\beta|\leq \kappa_2} \LtNorm*{Y^{\beta} \Stxpt{r_2}v}[\Ropngeq]^2
         \\&\lesssim \sum_{|\alpha|\leq \kappa_1} \LtNorm*{Y^{\alpha} \Stxpt{r_1-\epsilon_0}u}[\Ropngeq]^2
         \\&\quad\quad+\sum_{|\beta|\leq \kappa_2} \LtNorm*{Y^{\beta} \Stxpt{r_2}v}[\Ropngeq]^2
         +\LtNorm*{\partial_t \left( 1-\partial_t^2 \right)^{-1/4}\Stxpt{r_1-\epsilon_0} u}[\Ropngeq]^2.
    \end{split}
    \end{equation*}

    The calculus of pseudodifferential operators (see Lemma \ref{Lemma::PDOs::NotationNearBdry::CalcOfPDO};
    in particular Lemma \ref{Lemma::PDOs::NotationNearBdry::CalcOfPDO}\ref{Item::PDOs::NotationNearBdry::CalcOfPDO::CommutatorMultiplication}) shows that 
    the left-hand side of \ref{Item::EstBdry::Tangent::SmoothOnOneSide} (respectively, \ref{Item::EstBdry::Tangent::SmoothOnOneSideWithGain})
    can be written in the same form as the left-hand side of \ref{Item::EstBdry::Tangent::SmoothOnBothSides} (respectively, \ref{Item::EstBdry::Tangent::SmoothOnBothSidesWithGain}),
    with \(\Stxp{s}\) replaced by \(\Stxpt{s}\),
    so \ref{Item::EstBdry::Tangent::SmoothOnOneSide} (respectively, \ref{Item::EstBdry::Tangent::SmoothOnOneSideWithGain}) follows.
\end{proof}

\begin{lemma}\label{Lemma::EstBdry::Tangent::NoDerivsGain}
    \(\forall u\in \LtSpaceNoMeasure[\R][\HsYSpace{\kappa}[\CubengeqOneHalf]]\),
    \begin{equation*}
        \LtNorm*{\Stxp{s}u}[\Ropngeq]^2
        \lesssim \sum_{|\alpha|\leq \kappa} \LtNorm*{Y^{\alpha} \Stxpt{s-\epsilon_0}u}[\Ropngeq]^2 + \LtNorm*{\partial_t\left( 1-\partial_t^2 \right)^{-1/4}\Stxpt{s-\epsilon_0}u}[\Ropngeq]^2.
    \end{equation*}
\end{lemma}
\begin{proof}
    Using \(\LtNorm*{\Stxp{s}u}[\Ropngeq]^2 \leq \sum_{|\alpha|\leq \kappa-1}\LtNorm*{Y^{\alpha} \Stxp{s}u}[\Ropngeq]^2\),
    the result follows from Proposition \ref{Prop::HorVfs::GainBoundary::MainGainProp}
    and the fact that \(\epsilon_0=\min\{1/m,1/2\}\leq 1/m\).
\end{proof}

\begin{lemma}\label{Lemma::EstBdry::Tangent::MovePDOsToOtherSideOfQH}
    \(\forall u\in \DSetL{1}\),
    \begin{equation*}
    \begin{split}
         &\FormQH*[\Stxp{s}[1]u][\Stxp{s}[2]u]
         =\FormQH*[\left( \Stxp{s}[2] \right)^{*} \Stxp{s}[1] u][u]
         \\&\quad\quad +O\left( \sum_{|\alpha|\leq \kappa} \LtNorm*{Y^{\alpha}\Stxpt{s-\epsilon_0/2}u}[\Ropngeq]^2 
            +\LtNorm*{\partial_t\left( 1-\partial_t^2 \right)^{-1/4}\Stxpt{s-\epsilon_0/2}u}[\Ropngeq]^2
         \right).
    \end{split}
    \end{equation*}
\end{lemma}
\begin{proof}
    Using the calculus of pseudodifferential operators (see Lemma \ref{Lemma::PDOs::NotationNearBdry::CalcOfPDO}),
    we have
    \begin{equation}\label{Eqn::EstBdry::Tangent::MovePDOsToOtherSideOfQH::Tmp1}
    \begin{split}
         &\FormQH*[\Stxp{s}[1]u][\Stxp{s}[2]u]
         \\&=\sum_{|\alpha|,|\beta|\leq \kappa} \Ltip*{Y^{\alpha} \Stxp{s}[1]u}{\sigma a_{\alpha,\beta}Y^\beta \Stxp{s}[2]u}[\Ropngeq]
         +\Ltip*{\Stxp{s}[1]u}{\sigma \partial_t\Stxp{s}[2]u}[\Ropngeq]
         \\&= \sum_{|\alpha|,|\beta|\leq \kappa} \Ltip*{Y^{\alpha}\left( \Stxp{s}[2] \right)^{*} \Stxp{s}[1] u }{\sigma a_{\alpha,\beta}Y^\beta u}[\Ropngeq]
         +\Ltip*{\left( \Stxp{s}[2] \right)^{*}\Stxp{s}[1]u}{\sigma \partial_tu}[\Ropngeq]
         \\&\quad+
         O\Bigg(
            \sum_{\substack{|\alpha|\leq \kappa \\ |\beta|\leq \kappa-1}} \left| \Ltip*{ Y^{\alpha} \Stxp{s}[1]u}{a_{\alpha,\beta} Y^{\beta} \wtEst\Stxpt{s} u}[\Ropngeq]\right|
            \\&\quad\quad\quad\quad+\sum_{\substack{|\alpha|,|\beta|\leq \kappa}} \left| \Ltip*{  Y^{\alpha} \wtEst\Stxpt{s-1} \Stxp{s}[1]u}{Y^{\beta} u}[\Ropngeq] \right|
            \\&\quad\quad\quad\quad+\sum_{\substack{|\alpha|\leq \kappa-1\\|\beta|\leq \kappa}} \left| \Ltip*{ Y^{\alpha}\wtEst\Stxpt{s}  \Stxp{s}[1] u }{Y^{\beta}u}[\Ropngeq] \right|
            \\&\quad\quad\quad\quad+ \left| \Ltip*{\Stxp{s}[1]u}{\sigma \Stxpt{s} u}[\Ropngeq] \right|
            \\&\quad\quad\quad\quad+ \left| \Ltip*{\Stxp{s}[1]u}{\wtEst \Stxpt{s-1}\partial_t u}[\Ropngeq] \right|
         \Bigg).
    \end{split}
    \end{equation}
    By the definition of \(\FormQH\), we have
    \begin{equation}\label{Eqn::EstBdry::Tangent::MovePDOsToOtherSideOfQH::Tmp2}
    \begin{split}
         &\sum_{|\alpha|,|\beta|\leq \kappa} \Ltip*{Y^{\alpha}\left( \Stxp{s}[2] \right)^{*} \Stxp{s}[1] u }{\sigma a_{\alpha,\beta}Y^\beta u}[\Ropngeq]
         +\Ltip*{\left( \Stxp{s}[2] \right)^{*}\Stxp{s}[1]u}{\sigma \partial_tu}[\Ropngeq]
         \\&=\FormQH*[\left( \Stxp{s}[2] \right)^{*} \Stxp{s}[1] u][u].
    \end{split}
    \end{equation}

    Thus, to complete the proof, we need to estimate the terms inside the \(O(\cdot)\) in \eqref{Eqn::EstBdry::Tangent::MovePDOsToOtherSideOfQH::Tmp1}.
    Using Lemma \ref{Lemma::EstBdry::Tangent::Changes1s2Tor1r2}\ref{Item::EstBdry::Tangent::SmoothOnOneSideWithGain}
    with \(s_1=s_2=s\), \(r_1=s+\epsilon_0/2\), \(r_2=s-\epsilon_0/2\), we have
    \begin{equation}\label{Eqn::EstBdry::Tangent::MovePDOsToOtherSideOfQH::Tmp3}
    \begin{split}
         &\sum_{\substack{|\alpha|\leq \kappa-1\\|\beta|\leq \kappa}} \left| \Ltip*{ Y^{\alpha}\wtEst\Stxpt{s}  \Stxp{s}[1] u }{Y^{\beta}u}[\Ropngeq] \right|
         \\&\lesssim 
         \sum_{|\alpha|\leq \kappa} \LtNorm*{Y^{\alpha}\Stxpt{s-\epsilon_0/2}u}[\Ropngeq]^2 
             +\LtNorm*{\partial_t\left( 1-\partial_t^2 \right)^{-1/4}\Stxpt{s-\epsilon_0/2}u}[\Ropngeq]^2.
    \end{split}
    \end{equation}
    Similarly, reversing the roles of 
    \(\alpha\) and \(\beta\) and using Lemma \ref{Lemma::EstBdry::Tangent::Changes1s2Tor1r2}\ref{Item::EstBdry::Tangent::SmoothOnBothSidesWithGain},
    we have
    \begin{equation}\label{Eqn::EstBdry::Tangent::MovePDOsToOtherSideOfQH::Tmp4}
    \begin{split}
         &\sum_{\substack{|\alpha|\leq \kappa \\ |\beta|\leq \kappa-1}} \left| \Ltip*{ Y^{\alpha} \Stxp{s}[1]u}{a_{\alpha,\beta} Y^{\beta} \wtEst\Stxpt{s} u}[\Ropngeq]\right|
         \\&\lesssim 
         \sum_{|\alpha|\leq \kappa} \LtNorm*{Y^{\alpha}\Stxpt{s-\epsilon_0/2}u}[\Ropngeq]^2 
             +\LtNorm*{\partial_t\left( 1-\partial_t^2 \right)^{-1/4}\Stxpt{s-\epsilon_0/2}u}[\Ropngeq]^2.
    \end{split}
    \end{equation}
    Next, we use Lemma \ref{Lemma::EstBdry::Tangent::Changes1s2Tor1r2}\ref{Item::EstBdry::Tangent::SmoothOnOneSide}
    with \(s_2=s-1\), \(s_1=s\), \(r_1=r_2=s-1/2\) and \(\epsilon_0=\min\{1/m,1/2\}\leq 1\) to see
    \begin{equation}\label{Eqn::EstBdry::Tangent::MovePDOsToOtherSideOfQH::Tmp5}
    \begin{split}
         &\sum_{\substack{|\alpha|,|\beta|\leq \kappa}} \left| \Ltip*{  Y^{\alpha} \wtEst\Stxpt{s-1} \Stxp{s}[1]u}{Y^{\beta} u}[\Ropngeq] \right|
         \\&\lesssim 
         \sum_{|\alpha|\leq \kappa} \LtNorm*{Y^{\alpha}\Stxpt{s-\epsilon_0/2}u}[\Ropngeq]^2 
             +\LtNorm*{\partial_t\left( 1-\partial_t^2 \right)^{-1/4}\Stxpt{s-\epsilon_0/2}u}[\Ropngeq]^2.
    \end{split}
    \end{equation}
    Next, Lemma \ref{Lemma::EstBdry::Tangent::NoDerivsGain} shows
    \begin{equation}\label{Eqn::EstBdry::Tangent::MovePDOsToOtherSideOfQH::Tmp6}
    \begin{split}
         &\left| \Ltip*{\Stxp{s}[1]u}{\sigma \Stxpt{s} u}[\Ropngeq] \right|
         \lesssim \LtNorm*{\Stxpt{s} u}[\Ropngeq]^2
         \\&\lesssim 
         \sum_{|\alpha|\leq \kappa} \LtNorm*{Y^{\alpha}\Stxpt{s-\epsilon_0/2}u}[\Ropngeq]^2 
             +\LtNorm*{\partial_t\left( 1-\partial_t^2 \right)^{-1/4}\Stxpt{s-\epsilon_0/2}u}[\Ropngeq]^2.
    \end{split}
    \end{equation}
    Using the calculus of pseudodifferential operators (see Lemma \ref{Lemma::PDOs::NotationNearBdry::CalcOfPDO}), we have,
    using \(\LtNorm*{\left( 1-\partial_t^2 \right)^{-1/4}\wtEst \left( 1-\partial_t^2 \right)^{1/4}}\lesssim 1\),
    \begin{equation}\label{Eqn::EstBdry::Tangent::MovePDOsToOtherSideOfQH::Tmp7}
    \begin{split}
         &\left| \Ltip*{\Stxp{s}[1]u}{\wtEst \Stxpt{s-1}\partial_t u}[\Ropngeq] \right|
         \lesssim \left| \Ltip*{\Stxpt{s-1/2}u}{\wtEst \Stxpt{s-1/2}\partial_t u}[\Ropngeq] \right|
         \\&\lesssim \left| \Ltip*{\Stxpt{s-1/2}u}{\wtEst \partial_t\Stxpt{s-1/2} u}[\Ropngeq] \right|
         \\&\lesssim \left| \Ltip*{\left( 1-\partial_t^2 \right)^{1/4}\Stxpt{s-1/2}u}{\left( 1-\partial_t^2 \right)^{-1/4}\wtEst \left( 1-\partial_t^2 \right)^{1/4} \left( 1-\partial_t^2 \right)^{-1/4} \partial_t\Stxpt{s-1/2} u}[\Ropngeq] \right|
         \\&\lesssim \LtNorm*{\left( 1-\partial_t^2 \right)^{1/4}\Stxpt{s-1/2}u}[\Ropngeq]^2
         +\LtNorm*{\left( 1-\partial_t^2 \right)^{-1/4} \partial_t\Stxpt{s-1/2} u}[\Ropngeq]^2
         \\&\lesssim \LtNorm*{\left( 1-\partial_t^2 \right)^{-1/4} \partial_t\Stxpt{s-1/2} u}[\Ropngeq]^2 + \LtNorm*{\Stxpt{s-1/2} u}[\Ropngeq]^2.
    \end{split}
    \end{equation}
    Using \eqref{Eqn::EstBdry::Tangent::MovePDOsToOtherSideOfQH::Tmp3}, \eqref{Eqn::EstBdry::Tangent::MovePDOsToOtherSideOfQH::Tmp4},
    \eqref{Eqn::EstBdry::Tangent::MovePDOsToOtherSideOfQH::Tmp5}, \eqref{Eqn::EstBdry::Tangent::MovePDOsToOtherSideOfQH::Tmp6},
    and \eqref{Eqn::EstBdry::Tangent::MovePDOsToOtherSideOfQH::Tmp7} to estimate the terms in the \(O(\cdot)\)
    in \eqref{Eqn::EstBdry::Tangent::MovePDOsToOtherSideOfQH::Tmp1} completes the proof.
\end{proof}

\begin{lemma}\label{Lemma::EstBdry::Tangent::EstMaxSubBySumOfForms}
    \(\forall u\in \DSetL{1}\),
    \begin{equation}\label{Eqn::EstBdry::Tangent::EstMaxSubBySumOfForms::Conclusion}
    \begin{split}
         &\sum_{|\alpha|\leq \kappa} \LtNorm*{ Y^{\alpha} \Stxp{s} u}[\Ropngeq]^2
            +\LtNorm*{\partial_t \left( 1-\partial_t^2 \right)^{-1/4} \Stxp{s}u}[\Ropngeq]^2
            \\&\leq C_1 \Real \FormQH*[ \Stxpt{s} \Stxpt{s} u][u]
            +C_2 \Real \FormQH*[\Stxpt{s} \partial_t \left( 1-\partial_t^2 \right)^{-1/2}\Stxpt{s}u][u]
            \\&\quad\quad+C_3\sum_{|\alpha|\leq \kappa} \LtNorm*{Y^\alpha \Stxpt{s-\epsilon_0/2}u}[\Ropngeq]^2
            +C_3 \LtNorm*{\partial_t \left( 1-\partial_t^2 \right)^{-1/4} \Stxpt{s-\epsilon_0/2}u}[\Ropngeq]^2.
    \end{split}
    \end{equation}
\end{lemma}
\begin{proof}
    Combining Lemmas \ref{Lemma::EstBdry::Tangent::BoundMaxSubPlusTFormQH} and \ref{Lemma::EstBdry::Tangent::MovePDOsToOtherSideOfQH}
    we see
    \begin{equation}\label{Eqn::EstBdry::Tangent::EstMaxSubBySumOfForms::Tmp1}
    \begin{split}
        &\sum_{|\alpha|\leq \kappa} \LtNorm*{ Y^{\alpha} \Stxp{s} u}[\Ropngeq]^2
            +\LtNorm*{\partial_t \left( 1-\partial_t^2 \right)^{-1/4} \Stxp{s}u}[\Ropngeq]^2
            \\&\leq C_1 \Real \FormQH*[ \Stxpt{s} \Stxpt{s} u][u]
            +C_2 \Real \FormQH*[\Stxpt{s} \partial_t \left( 1-\partial_t^2 \right)^{-1/2}\Stxpt{s}u][u]
            \\&\quad\quad+C_3\sum_{|\alpha|\leq \kappa} \LtNorm*{Y^\alpha \Stxpt{s-\epsilon_0/2}u}[\Ropngeq]^2
            +C_3 \LtNorm*{\partial_t \left( 1-\partial_t^2 \right)^{-1/4} \Stxpt{s-\epsilon_0/2}u}[\Ropngeq]^2
            \\&\quad\quad+C_4 \LtNorm*{\Stxpt{s} u}[\Ropngeq]^2.
    \end{split}
    \end{equation}
    Only the final term on the right-hand side of \eqref{Eqn::EstBdry::Tangent::EstMaxSubBySumOfForms::Tmp1} is not of the desired form.
    However, Lemma \ref{Lemma::EstBdry::Tangent::NoDerivsGain} shows that this term can be bounded by terms of the desired form,
    completing the proof.
\end{proof}

\begin{lemma}\label{Lemma::EstBdry::Tangent::EstMaxSubByOps}
    For \(u\in \DSetL{1}\),
    \begin{equation*}
    \begin{split}
         &\sum_{|\alpha|\leq \kappa} \LtNorm*{ Y^{\alpha} \Stxp{s} u}[\Ropngeq]
            +\LtNorm*{\partial_t \left( 1-\partial_t^2 \right)^{-1/4} \Stxp{s}u}[\Ropngeq]
        \\&\lesssim \LtNorm*{\Stxpt{s}\sigma \left( \partial_t+\opL \right)u}[\Ropngeq]
        +\sum_{|\alpha|\leq \kappa} \LtNorm*{Y^\alpha \Stxpt{s-\epsilon_0/2}u}[\Ropngeq]
        +\LtNorm*{\partial_t \left( 1-\partial_t^2 \right)^{-1/4} \Stxpt{s-\epsilon_0/2}u}[\Ropngeq].
    \end{split}
    \end{equation*}
\end{lemma}
\begin{proof}
    We have, using Assumption \ref{Assumption::EstBdry::MaxSubOnSmoothing}\ref{Item::EstBdry::MaxSubOnSmoothing::FormQHGivesOp}
    with \(S=\Stxp{s}[1] \Stxp{s}[2]\),
    \begin{equation}\label{Eqn::EstBdry::Tangent::EstMaxSubByOps::Tmp1}
    \begin{split}
        &\left| \FormQH*[ \Stxp{s}[1] \Stxp{s}[2] u][u] \right|
        =\left| \Ltip*{\Stxp{s}[1] \Stxp{s}[2] u}{\sigma \left( \partial_t+\opL \right)u}[\Ropngeq] \right|
        \\&\lesssim \LtNorm*{\Stxpt{s}u}[\Ropngeq]^2 + \LtNorm*{\Stxpt{s} \sigma \left( \partial_t+\opL \right)u}[\Ropngeq]^2
        \\&\lesssim  \LtNorm*{\Stxpt{s} \sigma \left( \partial_t+\opL \right)u}[\Ropngeq]^2
        +\sum_{|\alpha|\leq \kappa} \LtNorm*{Y^\alpha \Stxpt{s-\epsilon_0/2}u}[\Ropngeq]^2
        \\&\quad\quad\quad+\LtNorm*{\partial_t \left( 1-\partial_t^2 \right)^{-1/4} \Stxpt{s-\epsilon_0/2}u}[\Ropngeq]^2,
    \end{split}
    \end{equation}
    where in the final estimate
    we used Lemma \ref{Lemma::EstBdry::Tangent::NoDerivsGain}.
    Similarly, using Assumption \ref{Assumption::EstBdry::MaxSubOnSmoothing}\ref{Item::EstBdry::MaxSubOnSmoothing::FormQHGivesOp}
    with \(S=\Stxp{s}[1] \partial_t \left( 1-\partial_t^2 \right)^{-1/2}\Stxp{s}[2]\),
    \begin{equation}\label{Eqn::EstBdry::Tangent::EstMaxSubByOps::Tmp2}
    \begin{split}
         &\left| \FormQH*[\Stxp{s}[1] \partial_t \left( 1-\partial_t^2 \right)^{-1/2}\Stxp{s}[2]u][u] \right|
         \\&=\left| \Ltip*{\Stxp{s}[1]  \partial_t \left( 1-\partial_t^2 \right)^{-1/2}\Stxp{s}[2] u}{\sigma \left( \partial_t+\opL \right)u}[\Ropngeq]   \right|
         \\&\lesssim \LtNorm*{\partial_t \left( 1-\partial_t^2 \right)^{-1/2}\Stxpt{s}u}[\Ropngeq]^2 + \LtNorm*{\Stxpt{s}\sigma  \left( \partial_t+\opL \right)u}[\Ropngeq]^2
         \\&\lesssim \LtNorm*{\Stxpt{s}u}[\Ropngeq]^2 + \LtNorm*{\Stxpt{s}\sigma  \left( \partial_t+\opL \right)u}[\Ropngeq]^2
        \\&\lesssim  \LtNorm*{\Stxpt{s} \sigma \left( \partial_t+\opL \right)u}[\Ropngeq]^2
        +\sum_{|\alpha|\leq \kappa} \LtNorm*{Y^\alpha \Stxpt{s-\epsilon_0/2}u}[\Ropngeq]^2
        \\&\quad\quad\quad+\LtNorm*{\partial_t \left( 1-\partial_t^2 \right)^{-1/4} \Stxpt{s-\epsilon_0/2}u}[\Ropngeq]^2.
    \end{split}
    \end{equation} 
    Using \eqref{Eqn::EstBdry::Tangent::EstMaxSubByOps::Tmp1}
    and \eqref{Eqn::EstBdry::Tangent::EstMaxSubByOps::Tmp2} to bound the right-hand side of \eqref{Eqn::EstBdry::Tangent::EstMaxSubBySumOfForms::Conclusion},
    the result follows.
\end{proof}

\begin{lemma}\label{Lemma::EstBdry::Tangent::EstMaxSubByOpsNTimes}
        For \(u\in \DSetL{1}\), \(N\in \Zgeq\),
    \begin{equation*}
    \begin{split}
         &\sum_{|\alpha|\leq \kappa} \LtNorm*{ Y^{\alpha} \Stxp{s} u}[\Ropngeq]
            +\LtNorm*{\partial_t \left( 1-\partial_t^2 \right)^{-1/4} \Stxp{s}u}[\Ropngeq]
        \\&\lesssim \LtNorm*{\Stxpt{s}\sigma \left( \partial_t+\opL \right)u}[\Ropngeq]
        +\sum_{|\alpha|\leq \kappa} \LtNorm*{Y^\alpha \Stxpt{s-N\epsilon_0/2}u}[\Ropngeq]
        +\LtNorm*{\partial_t \left( 1-\partial_t^2 \right)^{-1/4} \Stxpt{s-N\epsilon_0/2}u}[\Ropngeq].
    \end{split}
    \end{equation*}
\end{lemma}
\begin{proof}
    For \(N=0\), there is nothing to prove. For \(N\geq 1\),
    apply Lemma \ref{Lemma::EstBdry::Tangent::EstMaxSubByOps} \(N\) times.
\end{proof}

\begin{proof}[Proof of Proposition \ref{Prop::EstBdry::Tangent::MainTangentProp}]

        Let \(\phi_1\), \(\phi_2\) be as in the statement of the proposition (see the statement of Theorem \ref{Thm::EstBdry::MainThm::New}).
    Fix 
    \(\psi_1,\psi_2\in \CinftycptSpace[\R\times \CubengeqOneHalf]\)
    with
    \(\phi_1\prec \psi_1\prec \psi_2 \prec \phi_2\)
    and \(\psi_1=1\) on a neighborhood of 
    \(\overline{U_2' \times \Cubengeq{\delta_2'}}\).
    We will show
    \begin{equation}\label{Eqn::EstBdry::Tangent::MainTangentProp::ToShow}
        \LtNorm*{\Stxp{s+\epsilon_0} u}[\Ropngeq]
        \lesssim \GsNorm*{\phi_2 \left( \partial_t+\opL \right)u}{s}[\Ropngeq]+\LtNorm*{\phi_2 u}[\Ropngeq].
    \end{equation}
    \eqref{Eqn::EstBdry::Tangent::MainTangentProp::Conclusion} follows from \eqref{Eqn::EstBdry::Tangent::MainTangentProp::ToShow}
    as in the proof of Lemma \ref{Lemma::PDOs::NotationNearBdry::ConcludeSobolevEstimates}.

    Fix \(N\) so large \(s-N\epsilon_0/2+2\kappa\leq 0\) and \(-N\epsilon_0/2+\kappa\leq 0\). Note that,
    \begin{equation}\label{Eqn::EstBdry::Tangent::MainTangentProp::Tmp1}
    \begin{split}
        &\LtNorm*{\partial_t \left( 1-\partial_t^2 \right)^{-1/4}\Stxpt{s-N\epsilon_0/2}u}[\Ropngeq]
        \lesssim \LtNorm*{\partial_t \Stxpt{s-N\epsilon_0/2}u}[\Ropngeq]
        \\&\lesssim \LtNorm*{\Stxpt{s-N\epsilon_0/2+2\kappa}u}[\Ropngeq]
        \lesssim \LtNorm*{\phi_2 u}[\Ropngeq].
    \end{split}
    \end{equation}
    Similarly, using the form of \(Y_0,\ldots, Y_r\) described in Proposition \ref{Prop::EstBdry::Reduction::MainReduction} \ref{Item::EstBdry::Reduction::MainReduction::ReplaceWWithY},
    we have
    \begin{equation}\label{Eqn::EstBdry::Tangent::MainTangentProp::Tmp2}
    \begin{split}
         &\sum_{|\alpha|\leq \kappa}\LtNorm*{Y^{\alpha} \Stxpt{s-N\epsilon_0/2} u}[\Ropngeq]
         \lesssim \sum_{j=0}^\kappa \LtNorm*{\partial_{x_n}^j \Stxpt{s-N\epsilon_0/2+\kappa-j} u}[\Ropngeq]
         \\&\lesssim \sum_{j=0}^\kappa \GsNorm*{\partial_{x_n}^j \psi_2 u}{s-N\epsilon_0/2+\kappa-j}[\Ropngeq]
         \lesssim \sum_{j=0}^{2\kappa} \GsNorm*{\partial_{x_n}^j \psi_2 u}{s-N\epsilon_0/2+\kappa-j}[\Ropngeq].
    \end{split}
    \end{equation}

    Using the form of \(\opL\) described in Proposition \ref{Prop::EstBdry::Reduction::MainReduction} \ref{Item::EstBdry::Reduction::MainReduction::HighOrderTermPos},
    we see that \(\partial_t+\opL\) satisfies the conditions of \(\opP\) in Proposition \ref{Prop::EstBdry::ReduceTangent::MainReduceTangentProp}.
    Proposition \ref{Prop::EstBdry::ReduceTangent::MainReduceTangentProp} shows
    \begin{equation}\label{Eqn::EstBdry::Tangent::MainTangentProp::Tmp3}
    \begin{split}
         &\sum_{j=0}^{2\kappa} \GsNorm*{\partial_{x_n}^j \psi_2 u}{s-N\epsilon_0/2+\kappa-j}[\Ropngeq]
         \lesssim \GsNorm*{\phi_2\left( \partial_t+\opL \right)u}{s-N\epsilon_0/2-\kappa}[\Ropngeq]
         +\GsNorm*{\phi_2 u}{s-N\epsilon_0/2+\kappa}[\Ropngeq]
         \\&\lesssim \GsNorm*{\phi_2\left( \partial_t+\opL \right)u}{s}[\Ropngeq]
         +\LtNorm*{\phi_2 u}[\Ropngeq].
    \end{split}
    \end{equation}
    Combining \eqref{Eqn::EstBdry::Tangent::MainTangentProp::Tmp2} and \eqref{Eqn::EstBdry::Tangent::MainTangentProp::Tmp3}, we see
    \begin{equation}\label{Eqn::EstBdry::Tangent::MainTangentProp::Tmp4}
    \begin{split}
         &\sum_{|\alpha|\leq \kappa}\LtNorm*{Y^{\alpha} \Stxpt{s-N\epsilon_0/2} u}[\Ropngeq]
         \lesssim \GsNorm*{\phi_2\left( \partial_t+\opL \right)u}{s}[\Ropngeq]
         +\LtNorm*{\phi_2 u}[\Ropngeq].
    \end{split}
    \end{equation}
    Applying Lemmas \ref{Lemma::EstBdry::Tangent::NoDerivsGain} and \ref{Lemma::EstBdry::Tangent::EstMaxSubByOpsNTimes}, we see
    \begin{equation*}
    \begin{split}
         &\LtNorm*{\Stxp{s+\epsilon_0} u}[\Ropngeq]
         \lesssim \sum_{|\alpha|\leq \kappa} \LtNorm*{ Y^{\alpha} \Stxp{s} u}[\Ropngeq]
            +\LtNorm*{\partial_t \left( 1-\partial_t^2 \right)^{-1/4} \Stxp{s}u}[\Ropngeq]
        \\&\lesssim 
        \LtNorm*{\Stxpt{s}\sigma \left( \partial_t+\opL \right)u}[\Ropngeq]
        +\sum_{|\alpha|\leq \kappa} \LtNorm*{Y^\alpha \Stxpt{s-N\epsilon_0/2}u}[\Ropngeq]
        \\&\quad\quad\quad+\LtNorm*{\partial_t \left( 1-\partial_t^2 \right)^{-1/4} \Stxpt{s-N\epsilon_0/2}u}[\Ropngeq]
        \\&\lesssim \LtNorm*{\Stxpt{s}\sigma \left( \partial_t+\opL \right)u}[\Ropngeq]
        + \GsNorm*{\phi_2\left( \partial_t+\opL \right)u}{s}[\Ropngeq]
         +\LtNorm*{\phi_2 u}[\Ropngeq]
         \\&\lesssim \GsNorm*{\phi_2\left( \partial_t+\opL \right)u}{s}[\Ropngeq]
         +\LtNorm*{\phi_2 u}[\Ropngeq],
    \end{split}
    \end{equation*}
    where the second to last estimate used \eqref{Eqn::EstBdry::Tangent::MainTangentProp::Tmp1} and \eqref{Eqn::EstBdry::Tangent::MainTangentProp::Tmp4}.
    This establishes \eqref{Eqn::EstBdry::Tangent::MainTangentProp::ToShow} and completes the proof.
\end{proof}

    \subsection{Completion of the proof of Theorem \texorpdfstring{\ref{Thm::EstBdry::MainThm::New}}{\ref*{Thm::EstBdry::MainThm::New}}}
    \begin{lemma}\label{Lemma::EstBdry::CompletePF::WorkWithHsNorms}
     Let \(\phi_1,\phi_2\in \CinftycptSpace[\R\times \CubengeqOneHalf]\) with
     \(\phi_1\prec\phi_2\).
     Then, for \(l\in \Zgeq\) with \(l\geq 2\kappa-1\), and \(J\in \Zg\)
     with \(J 2\kappa \leq l+1\), we have \(\forall u\in \DistributionsZeroCM[\R\times \CubengeqOne]\),
     \begin{equation*}
          \HsNorm*{\phi_1 u}{l+1}[\Ropngeq]
          \lesssim \HsNorm*{\phi_2 \left( \partial_t+\opL  \right)^J u}{l+1-J2\kappa}[\Ropngeq]
          +\sum_{j=0}^{J-1} \GsNorm*{\phi_2 \left( \partial_t+\opL \right)^j u}{l+1-j2\kappa}[\Ropngeq].
     \end{equation*}
\end{lemma}
\begin{proof}
     We prove the result by induction on \(J\).  We begin with the base case, \(J=1\).
     By Proposition \ref{Prop::EstBdry::Reduction::MainReduction}\ref{Item::EstBdry::Reduction::MainReduction::HighOrderTermPos},
     \(\partial_t+\opL\) is of the form studied in Proposition \ref{Prop::EstBdry::ReduceTangent::MainReduceTangentProp}.
     We have, using Proposition \ref{Prop::EstBdry::ReduceTangent::MainReduceTangentProp},
     \(\forall u\in \DistributionsZero[\R\times \CubengeqOne]\)
     \begin{equation}\label{Eqn::EstBdry::CompletePF::WorkWithHsNorms::BaseCase}
     \begin{split}
          &\HsNorm*{\phi_1 u}{l+1}[\Ropngeq]
          \approx \sum_{j=0}^{l+1}\GsNorm{\partial_{x_n}^j \phi_1 u}{l+1-j}[\Ropngeq]
          \\&\lesssim \sum_{j=0}^{l+1-2\kappa} \GsNorm*{\partial_{x_n}^j \phi_2 \left( \partial_t+\opL \right) u}{l+1-2\kappa-j}[\Ropngeq] +\GsNorm*{\phi_2 u}{l+1}[\Ropngeq]
          \\&\lesssim \HsNorm*{\phi_2 \left( \partial_t+\opL \right) u}{l+1-2\kappa}[\Ropngeq]+\GsNorm*{\phi_2 u}{l+1}[\Ropngeq].
     \end{split}
     \end{equation}
     Proposition \ref{Prop::EstBdry::ReduceTangent::MainReduceTangentProp} was stated for \(u\in \DistributionsZero[\Ropngeq]\),
     but due to the cutoff functions holds \(\forall u\in \DistributionsZero[\R\times \CubengeqOne]\).
     This completes the proof of the base case.

     We turn to the inductive step and assume the result for \(J\) and prove it for \(J+1\).
     Take \(\psi\in \CinftycptSpace[\R\times \CubengeqOneHalf]\) with \(\phi_1\prec\psi\prec\phi_2\).
     By the inductive hypothesis with \(\phi_2\) replaced by \(\psi\), we have
     \begin{equation}\label{Eqn::EstBdry::CompletePF::WorkWithHsNorms::Inductive1}
     \begin{split}
               \HsNorm*{\phi_1 u}{l+1}[\Ropngeq]
               &\lesssim \HsNorm*{\psi \left( \partial_t+\opL  \right)^J u}{l+1-J2\kappa}[\Ropngeq]
               +\sum_{j=0}^{J-1} \GsNorm*{\psi \left( \partial_t+\opL \right)^j u}{l+1-j2\kappa}[\Ropngeq]
               \\&\lesssim \HsNorm*{\psi \left( \partial_t+\opL  \right)^J u}{l+1-J2\kappa}[\Ropngeq]
               +\sum_{j=0}^{J-1} \GsNorm*{\phi_2 \left( \partial_t+\opL \right)^j u}{l+1-j2\kappa}[\Ropngeq]
     \end{split}     \end{equation}
     Applying \eqref{Eqn::EstBdry::CompletePF::WorkWithHsNorms::BaseCase} with \(\phi_1\)
     replaced by \(\psi\) and \(u\) replaced by \(\left( \partial_t+\opL  \right)^J u\)
     we have
     \begin{equation}\label{Eqn::EstBdry::CompletePF::WorkWithHsNorms::Inductive2}
          \begin{split}
               &\HsNorm*{\psi \left( \partial_t+\opL  \right)^J u}{l+1-J2\kappa}[\Ropngeq]
               \\&\lesssim \HsNorm*{\phi_2 \left( \partial_t+\opL \right)^{J+1}u}{l+1-(J+1)2\kappa}[\Ropngeq]
               +\GsNorm*{\phi_2 \left( \partial_t+\opL \right)^J u}{l+1-J2\kappa}[\Ropngeq].
          \end{split}
     \end{equation}
     Plugging \eqref{Eqn::EstBdry::CompletePF::WorkWithHsNorms::Inductive2} into the right-hand side of
     \eqref{Eqn::EstBdry::CompletePF::WorkWithHsNorms::Inductive1} completes the proof.
\end{proof}

\begin{lemma}\label{Lemma::EstBdry::CompletePF::WorkWithGsNorms}
    Fix \(0<\delta_1<\delta_2<1/2\) and \(U_1\Subset U_2\Subset \R\) open.
    Let \(\phi_1\in \CinftycptSpace[U_1\times \Cubengeq{\delta_1}]\),
    \(\phi_2\in \CinftycptSpace[\R\times \CubengeqOneHalf]\)
    with 
    \(\phi_2=1\) on a neighborhood of \(\overline{U_1}\times \CubengeqDeltaOneClosure\),
    and fix \(s\in \R\) and \(N\in \Zg\).
    Then, for \(u\in \DSetL{N}\) and \(0\leq j\leq N\),
    \begin{equation*}
    \begin{split}
      &\GsNorm*{\phi_1 \left( \partial_t+\opL \right)^j u}{s-j\epsilon_0}[\Ropngeq]
      \lesssim \GsNorm*{\phi_2\left( \partial_t+\opL \right)^N u}{s-N\epsilon_0}[\Ropngeq]
      +\sum_{k=j}^{N-1}\LtNorm*{\phi_2 \left( \partial_t+\opL \right)^k u}[\Ropngeq],
    \end{split}
    \end{equation*}
    where when \(j=N\) we take the convention \(\sum_{k=N}^{N-1}=0\).
\end{lemma}
\begin{proof}
     We prove the result by induction on \(j\). The base case, \(j=N\), is trivial.
     We assume the result for \(j\in \{1,2,\ldots, N\}\) and prove it for \(j-1\).
     Fix \(\delta'\in (\delta_1,\delta_2)\) and \(V\Subset U_2\) open with \(U_1\Subset V\)
     and \(\phi_2 =1\) on a neighborhood of \(\overline{V}\times \overline{\Cubengeq{\delta'}}\).
     Take \(\psi\in \CinftycptSpace[V\times \Cubengeq{\delta'}]\)
     with \(\psi=1\) on a neighborhood of \(\overline{U_1}\times \CubengeqDeltaOneClosure\).
     Let \(u\in \DSetL{N}\), then \((\partial_t+\opL)^{j-1}u\in \DSetL{N-(j-1)}\subseteq \DSetL{1}\).
     Applying  Proposition \ref{Prop::EstBdry::Tangent::MainTangentProp}
     with \(\phi_2\) replaced by \(\psi\) and \(u\) replaced by \((\partial_t+\opL)^{j-1}u\),
     we have
     \begin{equation*}
     \begin{split}
           \GsNorm*{\phi_1 \left( \partial_t+\opL \right)^{j-1}u}{s-(j-1)\epsilon_0}[\Ropngeq]
           &\lesssim \GsNorm*{\psi \left( \partial_t+\opL \right)^{j}u}{s-j\epsilon_0}[\Ropngeq]
           +\LtNorm*{\psi\left( \partial_t+\opL \right)^{j-1}u}[\Ropngeq]
           \\&\lesssim \GsNorm*{\psi \left( \partial_t+\opL \right)^{j}u}{s-j\epsilon_0}[\Ropngeq]
           +\LtNorm*{\phi_2\left( \partial_t+\opL \right)^{j-1}u}[\Ropngeq].
     \end{split}
     \end{equation*}
     Applying the inductive hypothesis to \( \GsNorm*{\psi \left( \partial_t+\opL \right)^{j}u}{s-j\epsilon_0}[\Ropngeq]\)
     (with \(\phi_1\) replaced by \(\psi\))
     completes the proof.
\end{proof}

\begin{proof}[Proof of Theorem \ref{Thm::EstBdry::MainThm::New}]
     Fix \(\delta'\in (\delta_1,\delta_2)\) and \(V\Subset U_2\) open with \(U_1\Subset V\)
     and \(\phi_2 =1\) on a neighborhood of \(\overline{V}\times \overline{\Cubengeq{\delta'}}\).
     Take \(\psi\in \CinftycptSpace[V\times \Cubengeq{\delta'}]\)
     with \(\psi=1\) on a neighborhood of \(\overline{U_1}\times \CubengeqDeltaOneClosure\).
     Applying   Lemma \ref{Lemma::EstBdry::CompletePF::WorkWithHsNorms}
     with \(\phi_2\) replaced by \(\psi\) shows
     \(\forall u\in \DistributionsZero[\R\times \CubengeqOne]\)
     \begin{equation}\label{Eqn::EstBdry::CompletePF::FinalProof::Tmp1}
\begin{split}
          \HsNorm*{\phi_1 u}{l+1}[\Ropngeq]
               &\lesssim \HsNorm*{\psi \left( \partial_t+\opL  \right)^J u}{l+1-J2\kappa}[\Ropngeq]
               +\sum_{j=0}^{J-1} \GsNorm*{\psi \left( \partial_t+\opL \right)^j u}{l+1-j2\kappa}[\Ropngeq]
          \\&\lesssim \HsNorm*{\phi_2 \left( \partial_t+\opL  \right)^J u}{l+1-J2\kappa}[\Ropngeq]
               +\sum_{j=0}^{J-1} \GsNorm*{\psi \left( \partial_t+\opL \right)^j u}{l+1-j2\kappa}[\Ropngeq].
\end{split}     \end{equation}
     Now assuming \(u\in \DSetL{N}\) and using that \(J\leq N\), Lemma \ref{Lemma::EstBdry::CompletePF::WorkWithGsNorms}
     shows for \(0\leq j\leq J-1\),
     \begin{equation}\label{Eqn::EstBdry::CompletePF::FinalProof::Tmp2}
     \begin{split}
          &\GsNorm*{\psi \left( \partial_t+\opL \right)^j u}{l+1-j2\kappa}[\Ropngeq]
          \leq \GsNorm*{\psi \left( \partial_t+\opL \right)^j u}{l+1-j\epsilon_0}[\Ropngeq]
          \\&\lesssim \GsNorm*{\phi_2 \left( \partial_t+\opL \right)^N u }{l+1-N\epsilon_0}[\Ropngeq]
          +\sum_{k=j}^{N-1} \LtNorm*{\phi_2 \left( \partial_t +\opL\right)^k u}[\Ropngeq].
     \end{split}
     \end{equation}
     Plugging \eqref{Eqn::EstBdry::CompletePF::FinalProof::Tmp2} into \eqref{Eqn::EstBdry::CompletePF::FinalProof::Tmp1}
     shows
     \begin{equation}\label{Eqn::EstBdry::CompletePF::FinalProof::Tmp3}
     \begin{split}
           \HsNorm*{\phi_1 u}{l+1}[\Ropngeq]
           \lesssim  &\HsNorm*{\phi_2 \left( \partial_t+\opL  \right)^J u}{l+1-J2\kappa}[\Ropngeq]
           +\GsNorm*{\phi_2 \left( \partial_t+\opL \right)^N u }{l+1-N\epsilon_0}[\Ropngeq]
           \\&+\sum_{k=0}^{N-1} \LtNorm*{\phi_2 \left( \partial_t +\opL\right)^k u}[\Ropngeq].
     \end{split}
     \end{equation}
     Finally, using
     \(\GsNorm{\cdot}{s}[\Ropngeq]\lesssim \HsNorm{\cdot}{s\vee 0}[\Ropngeq]\), \(\forall s\in \R\),
     \eqref{Eqn::EstBdry::CompletePF::FinalProof::Tmp3} implies \eqref{Eqn::EstBdry::MainThm::MainEst::New}
     and completes the proof.
\end{proof}

\section{Estimates on the interior}\label{Section::EstInt}
Let \(W_0,W_1,\ldots, W_r\in \CinftySpace[\CubenOne][\Rn]\) be smooth vector fields on \(\CubenOne\) satisfying
H\"ormander's condition of order \(m\). For \(|\alpha|\leq 2 \kappa\), let \(b_{\alpha}\in \CinftySpace[\CubenOne][\MatrixSpace[M][M][\C]]\)
be a matrix and set
\begin{equation*}
\begin{split}
     &\opL:=\sum_{|\alpha|\leq 2\kappa} b_{\alpha}(x)W^{\alpha},
\end{split}
\end{equation*}
so that \(\opL\) is a partial differential operator of order \(2\kappa\).
Let \(\sigma\in \CinftySpace[\CubenOne][(0,\infty)]\) be given.

\begin{assumption}\label{Assumption::EstBdryInt::MaxSubAssump}
    \(\exists C_1, C_2\geq 0\),
    such that
    \begin{equation}\label{Eqn::EstBdryInt::MaxSubAssumption}
        \sum_{|\alpha|=\kappa}\LtNorm*{W^\alpha f}[\Rn]^2\leq 
         C_1 \Real \Ltip*{f}{\sigma \opL f}[\Rn]
         +C_2 \LtNorm{f}[\Rn]^2,\quad \forall f\in \CinftycptSpace[\CubenOneHalf][\C^M].
    \end{equation}
\end{assumption}

Set \(\epsilon_0=\min\{1/m,1/2\}\).

\begin{theorem}\label{Thm::EstBdryInt::MainThm::New}
    Under Assumption \ref{Assumption::EstBdryInt::MaxSubAssump}, the following holds.
    Let \(\phi_1,\phi_2\in \CinftycptSpace[\R\times \CubenOneHalf]\) with \(\phi_1\prec \phi_2\).
    Then, \(\forall s\in \R\), \(\forall N\in \Zgeq\), \(\exists C_{s,N,\phi_1,\phi_2}\geq 0\),
    \begin{equation*}
        \HsNorm*{\phi_1 u}{s+\epsilon_0}[\Ropn]
        \leq C_{s,N,\phi_1,\phi_2}
        \left( \HsNorm*{\phi_2 \left( \partial_t+\opL \right)u}{s}[\Ropn] + \HsNorm*{\phi_2 u}{s-N}[\Ropn] \right),
    \end{equation*}
    \(\forall u\in \Distributions*[\R\times \CubenOne][\C^M]\), where if the right-hand side is finite, so is the left-hand side.
    \(C_{s,N,\phi_1,\phi_2}\) can be chosen to depend only on
    \(s\), \(n\), \(m\), \(r\), \(\kappa\), \(M\), \(N\), \(\phi_1\), \(\phi_2\), \(c_{\CubenSevenEighthsClosure}\) (as in \eqref{Eqn::HorVfs::LowerBoundDet}),
    \(C_1\) and \(C_2\), and an upper bound for
    \begin{equation*}
        \max_{0\leq j\leq r} \CjNorm*{W_j}{L}[\Cuben{7/8}][\Rn] + \max_{\alpha}\CjNorm*{b_{\alpha}}{L}[\CubenOneHalf]+\max_{1\leq j\leq 2}\CjNorm*{\phi_j}{L}+\CjNorm{\sigma}{L}[\CubenOneHalf]+\sup_{x\in \CubenOneHalf} \sigma(x)^{-1},
    \end{equation*}
    where \(L\) can be chosen to depend only on \(N\), \(s\), \(m\), \(n\), \(\kappa\), and \(r\).
\end{theorem}

The proof of Theorem \ref{Thm::EstBdryInt::MainThm::New} is an easier reprise of the proof of Proposition \ref{Prop::EstBdry::Tangent::MainTangentProp}.
The changes needed are:
\begin{itemize}
    \item We do not need any of the reductions from Proposition \ref{Prop::EstBdry::Reduction::MainReduction}.
    \item Replace \(\GsSpace{s}\) with \(\HsSpace{s}\), throughout.
    \item Replace \(\Stxp{s}\), \(\Stxpt{s}\), and \(\Lambdatxp[s]\) with \(\Stx{s}\), \(\Stxt{s}\), and \(\Lambdatx[s]\), respectively.
        Here we can use the simpler notation from Section \ref{Section::PseudodifferentialOps::Interior}, and so do not need
        to use the \(\wtEst\) functions in our estimates.
    \item Formally replace \(\FormQ[u][v]\) with \(\Ltip*{u}{\sigma \opL v}[\Rn]\) and replace \(\FormQH[u][v]\) with \(\Ltip*{u}{\sigma \left( \partial_t+\opL \right)v}[\Ropn]\),
    wherever they appear.
    \item Because \(\Stx{s}\) and \(\Stxt{s}\) are infinitely smoothing in all variables (and have compact support), one may work with general distributions at every step.
    \item Since we are working with all distributions, we no longer need to discuss the set \(\DSetL{1}\).
    \item In Proposition \ref{Prop::EstBdry::Tangent::MainTangentProp}, the \(\partial_{x_n}\) derivatives are treated differently. This is because
        integration by parts in the \(x_n\)-variable would introduce troublesome boundary terms. In the setting of Theorem \ref{Thm::EstBdryInt::MainThm::New},
        we do not need to distinguish any direction, as integration by parts never introduces boundary terms.
    \item Because \(\HsSpace{s}\) covers regularity in all variables, we do not need to use Proposition \ref{Prop::EstBdry::ReduceTangent::MainReduceTangentProp} as
        is used at the end of the proof of Proposition \ref{Prop::EstBdry::Tangent::MainTangentProp}.
    \item Though not necessary, one may replace \(1-\partial_t^2\) wherever it appears by \(\Lambdatx[4\kappa]\). This is  because
        Proposition \ref{Prop::Horvfs::GainInterior::RSGain} gives a better estimate than Proposition \ref{Prop::HorVfs::GainBoundary::MainGainProp}.
        This is convenient, since \(\Lambdatx[-4\kappa]\) is a pseudodifferential operator (in all variables), while \(\left( 1-\partial_t^2 \right)^{-1}\) is not.
        This allows several steps to be (slightly) simplified using the calculus of pseudodifferential operators.
    \item Assumption \ref{Assumption::EstBdryInt::MaxSubAssump} takes the place of
        Assumption \ref{Assumption::EstBdry::MaxSubOnSmoothing}\ref{Item::EstBdry::MaxSubOnSmoothing::MaxSubEstimate}.
        Indeed, by integrating \eqref{Eqn::EstBdryInt::MaxSubAssumption} in \(t\) one obtains an estimate of the form \eqref{Eqn::EstBdry::MaxSubOnSmoothing::MaxSubEstimate};
        see Remark \ref{Rmk::EstBdry::MainThm::tdoesntplayroleInMaxSub}.
\end{itemize}

Because Theorem \ref{Thm::EstBdryInt::MainThm::New} holds for all distributions and all \(s\) and \(N\)
we can iterate the result. This leads to the next corollary.

\begin{corollary}\label{Cor::EstBdryInt::MainCor}
    Under Assumption \ref{Assumption::EstBdryInt::MaxSubAssump}, the following holds.
    Let \(\phi_1,\phi_2\in \CinftycptSpace[\R\times \CubenOneHalf]\) with \(\phi_1\prec \phi_2\).
    Then, \(\forall s\in \R\), \(\forall N\in \Zg\), \(\exists C_{s,N,\phi_1,\phi_2}\geq 0\),
    \(\forall u\in \Distributions*[\R\times \CubenOne][\C^M]\),
    \begin{equation*}
        \HsNorm*{\phi_1 u}{s+\epsilon_0}[\Ropn]
        \leq C_{s,N,\phi_1,\phi_2}
        \left( 
            \HsNorm*{\phi_2 \left( \partial_t+\opL \right)^N u}{s-(N-1)\epsilon_0}[\Ropn]
            +\LtNorm*{\phi_2 u}[\Ropn]
         \right),
    \end{equation*}
    where if the right-hand side is finite, so is the left-hand side.
    Here, \(C_{s,N,\phi_1,\phi_2}\geq 0\) can be chosen to depend on only the same quantities as in Theorem \ref{Thm::EstBdryInt::MainThm::New}.
\end{corollary}
\begin{proof}
    We prove the result by induction on \(N\). The base case, \(N=1\) follows by taking \(N\geq s\)
    in Theorem \ref{Thm::EstBdryInt::MainThm::New}.  We assume the result for \(N\geq 1\) and prove it for \(N+1\).
    Let \(\psi_1,\psi_2 \in \CinftycptSpace[\R\times \CubenOneHalf]\) with \(\phi_1\prec\psi_1\prec\psi_2\prec\phi_2\).
    By the inductive hypothesis applied with
    \(\phi_2\) replaced by \(\psi_1\), we have
    \begin{equation}\label{Eqn::Cor::EstBdryInt::MainCor::Tmp1}
\begin{split}
            \HsNorm*{\phi_1 u}{s+\epsilon_0}[\Ropn]
            &\lesssim 
                \HsNorm*{\psi_1 \left( \partial_t+\opL \right)^N u}{s-(N-1)\epsilon_0}[\Ropn]
                +\LtNorm*{\psi_1 u}[\Ropn]
            \\&\lesssim \HsNorm*{\psi_1 \left( \partial_t+\opL \right)^N u}{s-(N-1)\epsilon_0}[\Ropn]
                +\LtNorm*{\phi_2 u}[\Ropn].
\end{split}
    \end{equation}
    Applying Theorem \ref{Thm::EstBdryInt::MainThm::New} with \(u\) replaced by
    \(\left( \partial_t+\opL \right)^N u\), \(\phi_1\) and \(\phi_2\) replaced by \(\psi_1\) and \(\psi_2\), respectively,
    \(s\) replaced by \(s-N\epsilon_0\),
    and \(N\) replaced by \(N'\geq s-N\epsilon_0+2\kappa N\), we have
    \begin{equation}\label{Eqn::Cor::EstBdryInt::MainCor::Tmp2}
    \begin{split}
         &\HsNorm*{\psi_1 \left( \partial_t+\opL \right)^N u}{s-(N-1)\epsilon_0}[\Ropn]
        \\&\lesssim
        \HsNorm*{\psi_2 \left( \partial_t+\opL \right)^{N+1} u}{s-N\epsilon_0}[\Ropn]
        +\HsNorm*{\psi_2 \left( \partial_t+\opL \right)^N u}{s-N\epsilon_0-N'}[\Ropn]
        \\&\lesssim \HsNorm*{\phi_2 \left( \partial_t+\opL \right)^{N+1} u}{s-N\epsilon_0}[\Ropn]
        +\HsNorm*{\phi_2 u}{s-N\epsilon_0+2\kappa N-N'}[\Ropn]
        \\&\lesssim \HsNorm*{\phi_2 \left( \partial_t+\opL \right)^{N+1} u}{s-N\epsilon_0}[\Ropn]
        +\LtNorm*{\phi_2 u}[\Ropn].
    \end{split}
    \end{equation}
    Plugging \eqref{Eqn::Cor::EstBdryInt::MainCor::Tmp2} into \eqref{Eqn::Cor::EstBdryInt::MainCor::Tmp1}
    completes the proof.
\end{proof}

\begin{remark}\label{Rmk::EstBdryInt::MainCor::CanUseSigmaInsteadOfLebesgue}
    Similar to Remark \ref{Rmk::EstBdry::MainThm::CanUseSigmaInsteadOfLebesgue},
    one can replace Lebesgue measure on \(\CubenOneHalf\)
    with \(\sigma(x)\: dx\) in every \(\LtSpace\)-norm in the above assumptions and conclusions (though not in the inner products), and the assumptions and results
     are unchanged.
\end{remark}


\section{Completion of the proof of Theorem \texorpdfstring{\ref{Thm::Result::MainThm::New}}{\ref*{Thm::Result::MainThm::New}}}
In this section, we complete the proof of Theorem \ref{Thm::Result::MainThm::New}
by reducing it to Theorem \ref{Thm::EstBdry::MainThm::New} and Corollary \ref{Cor::EstBdryInt::MainCor}.
Fix \(x_0\in \Omega\), \(t_0\in \R\) and an open neighborhood \(V\subseteq \R\times \Omega\) of \((t_0,x_0)\).  
We will show that there exist \(\phi_1,\phi_2\in \CinftycptSpace[V]\) 
with \(\phi_1\prec\phi_2\) and \(\phi_1=1\) on a neighborhood of \((x_0, t_0)\),
such that the conclusions of Theorem \ref{Thm::Result::MainThm::New} hold
(this is done in Propositions \ref{Prop::Proof::Boundary} and \ref{Prop::Proof::Interior}).
The general statement of Theorem \ref{Thm::Result::MainThm::New} (for general \(\phi_1,\phi_2\in \CinftycptSpace[\R\times \Omega]\) with
\(\phi_1\prec\phi_2\)) then follows by a simple partition of unity argument, which we leave to the reader.
We may, without loss of generality, assume \(t_0=0\), by translation invariance in the \(t\)-variable.


We begin by explaining the maximally subelliptic estimates we use.

\begin{lemma}[The interior maximally subelliptic estimate]\label{Lemma::Proof::New::InteriorMaxSubEstimate}
    \(\forall U\Subset \Omega\cap\InteriorN\), relatively compact and open, \(\exists C_U\geq 0\), \(\forall f\in \CinftycptSpace[U][\C^M]\),
    \begin{equation*}
        \HsWNorm{f}{\kappa}^2\leq C_U \left( \left\lvert \FormQ[f][f] \right\rvert + \LtNorm{f}^2  \right). 
    \end{equation*}
\end{lemma}
\begin{proof}
    This follows from Remark \ref{Rmk::Result::DomainIsHkappaW::MaxSub}. Indeed, that remark shows
    \begin{equation*}
        \HsWNorm{f}{\kappa}\leq C_U' \DomainQNorm{f}. 
    \end{equation*}
    Since \(\DomainQNorm{f}^2\lesssim \left( \left\lvert \FormQ[f][f] \right\rvert + \LtNorm{f}^2  \right)\), the result follows.
\end{proof}

\begin{lemma}\label{Lemma::Proof::New::BoundarySubellipticEstimate}
    Fix \(x_0\in \Omega\cap \BoundaryN\).  Then \(\exists A=A(x_0)\geq 0\) such that the following holds.
    Let \(T\) be as in Assumption \ref{Asssumption::Result::BoundaryAssumption::New}. 
    Then, \(T:\LtSpaceNoMeasure*[\R][\DomainL]\rightarrow \CinftycptSpace*[\R][\DomainQ]\)
    and \(T:\LtSpaceNoMeasure*[\R][\DomainL]\rightarrow \CinftycptSpace*[\R][\HsWSpace{\kappa}[\Omega;\C^M]]\)
    and these maps are continuous.  Moreover,
            \(\forall u\in \LtSpaceNoMeasure*[\R][\DomainL]\),
                \begin{equation}\label{Eqn::Proof::New::BoundarySubellipticEstimate}
                    \int \HsWNorm{Tu (t,\cdot)}{\kappa}^2 \: dt
                    \leq A
                    \left( \int \left| \FormQ[Tu(t,\cdot)][Tu(t,\cdot)] \right|\: dt 
                    + \LtNorm*{T u}[\R\times \ManifoldN][dt\times d\Vol]^2  \right).
                \end{equation}
\end{lemma}
\begin{proof}
    Since \(K\) smooth with compact support, we have \(T:\LtSpaceNoMeasure*[\R\times \ManifoldN][\C^M]\rightarrow \LtSpaceNoMeasure*[\R\times \ManifoldN][\C^M]\) continuously, and therefore the map \(T:\LtSpaceNoMeasure*[\R][\DomainL]\rightarrow \LtSpaceNoMeasure[\R][\DomainQ]\) is closed. 
    The Closed Graph Theorem then implies \(T:\LtSpaceNoMeasure*[\R][\DomainL]\rightarrow \LtSpaceNoMeasure*[\R][\DomainQ]\)
    is continuous (Assumption \ref{Asssumption::Result::BoundaryAssumption::New} did not assume the map was continuous, but here the Closed Graph Theorem
    gives the continuity). 
    Since \(\partial_t^j T\) is of the same form as \(T\), we conclude \(\partial_t^jT: \LtSpaceNoMeasure*[\R][\DomainL]\rightarrow \LtSpaceNoMeasure*[\R][\DomainQ]\)
    and is continuous, \(\forall j\in \Zgeq\). The vector valued Sobolev Embedding Theorem (along with the compact support in \(t\) of the Schwartz
    kernel of \(T\)) shows \(T:\LtSpaceNoMeasure*[\R][\DomainL]\rightarrow \CinftycptSpace*[\R][\DomainQ]\), continuously. 

    Set \(\Compact_0\coloneqq \overline{\Psi(\CubengeqOneHalf)}\). Note that \(\Compact_0\Subset \Psi(\CubengeqOne)\subseteq \Omega\)
    and therefore \(\Compact_0\) is a compact subset of \(\Omega\).
    Also, \(T u(t,\cdot)\) is supported in \(\Compact_0\Subset \Psi(\CubengeqOne)\Subset \Omega\), for any \(u\) and \(t\). 
    Remark \ref{Rmk::Result::DomainIsHkappaW::MaxSub} shows
    \begin{equation*}
        \left\{ f\in \DomainQ : \supp(f)\subseteq \Compact_0 \right\} \hookrightarrow \HsWSpace{\kappa}[\Omega;\C^M],
    \end{equation*}
    continuously. We conclude, \(\exists A',A\geq 0\) such that
    any \(T\) as in the lemma takes \(T:\LtSpaceNoMeasure*[\R][\DomainL]\rightarrow \CinftycptSpace*[\R][\HsWSpace{\kappa}[\Omega;\C^M]]\) continuously and 
    \begin{equation*}
    \begin{split}
        & \HsWNorm*{Tu(t,\cdot)}{\kappa}^2 \leq A' \DomainQNorm{Tu(t,\cdot)}^2
        \leq A \left( \left\lvert \FormQ[Tu(t,\cdot)][Tu(t,\cdot)] \right\rvert + \LtNorm*{Tu(t,\cdot)}[\ManifoldN][\Vol]^2  \right).
    \end{split}
    \end{equation*}
    Integrating in \(t\) gives \eqref{Eqn::Proof::New::BoundarySubellipticEstimate}.
\end{proof}

We turn to the harder case \(x_0\in \Omega\cap\BoundaryN\), which is completed by the next proposition.
\begin{proposition}\label{Prop::Proof::Boundary}
    Assume all the assumptions from Section \ref{Section::MainResult}.
    Let \(x_0\in \Omega\cap \BoundaryN\), and \(\Psi:\CubengeqOne\xrightarrow{\sim}\Psi(\CubengeqOne)\subseteq \Omega\) as in Assumption \ref{Asssumption::Result::BoundaryAssumption::New}.
    Take \(U_1\Subset U_2\Subset \CubengeqOneHalf\) neighborhoods of \(0\)
    and \(0<\delta_1<\delta_2<1\) so small that \([-\delta_2,\delta_2]\times \Psi(\overline{U_2} )\Subset V\).
    Let \(\phih_1\in \CinftycptSpace*[(-\delta_1,\delta_1)\times U_1]\) with \(\phih_1=1\) on a neighborhood of \((0,0)\)
    and \(\phih_2\in \CinftycptSpace*[(-\delta_2,\delta_2)\times U_2]\) be such that \(\phih_2=1\)
    on a neighborhood of \([-\delta_1,\delta_1]\times \overline{U_1}\).
    Set \(\phi_j(t,\xi):=\phih_j(t,\Psi^{-1}(\xi))\in \CinftycptSpace[V]\).
    Then, the conclusions of Theorem \ref{Thm::Result::MainThm::New} hold for this choice
    of \(\phi_1,\phi_2\).
\end{proposition}

We turn to the proof of Proposition \ref{Prop::Proof::Boundary}.
All the assumptions from Section \ref{Section::MainResult} hold with \(\Omega\) replaced with the smaller open set \(\Omegat=\Psi(\CubengeqOne)\).
Note that the vector field \(X\)
from Assumption \ref{Asssumption::Result::BoundaryAssumption::New} is defined on all of \(\Omegat\).

Set \(\Wt_0:=\frac{1}{c_0\circ \Psi^{-1}}X\) and let \(a_j\in \CinftySpace[\Psi(\CubengeqOne)][\R]\), \(1\leq j\leq r\)
be such that \(W_j(x')-a_j(x')\Wt_0(x')\in \TangentSpace{x'}{\BoundaryN}\), \(\forall x'\in \BoundaryN\cap \Omegat\).
Set \(\Wt_j:=W_j-a_j \Wt_0\).
We may replace \(W_1,\ldots, W_r\) with \(\Wt_0,\ldots,\Wt_r\) and \(\Omega\) with \(\Omegat\)
throughout (and appropriately modifying \(a_{\alpha,\beta}\); call the new functions \(\at_{\alpha,\beta}\))
and all our assumptions hold.
We henceforth make this replacement.

Let \(\sigma\in \CinftySpace[\CubengeqOne][(0,\infty)]\) be defined by
\(d \Psi^{*}\Vol = \sigma(x) dx\). 
Set \(\Wh_j:=\Psi^{*} \Wt_j\); notice that \(\Wh_0=\partial_{x_n}\) and for \(1\leq j\leq r\),
\(\Wh_j(x',0)\) does not have \(\partial_{x_n}\) component \(\forall x'\in \CubenmoOne\).
\(\Wh_0,\ldots, \Wt_r\) satisfy H\"ormander's condition of order \(m\) because
\(\Wt_0,\ldots, \Wt_r\) do.
Set \(\ah_{\alpha,\beta}:= \at_{\alpha,\beta}\circ \Psi\),
 and, whenever it makes sense
\begin{equation}\label{Eqn::Proof::QhEqualsPullbackQF}
    \FormQh[f][g]:=\sum_{|\alpha|,|\beta|\leq \kappa}
    \Ltip*{\Wh^{\alpha}f}{\sigma \ah_{\alpha,\beta} \Wh^{\beta} g}[\CubengeqOne]
    =\FormQF*[f\circ \Psi^{-1}][g\circ \Psi^{-1}].
\end{equation}
Set
\begin{equation*}
    \opLh:=\sum_{|\alpha|,|\beta|\leq \kappa} \sigma^{-1}\left( \Wh^{\alpha} \right)^{*} \sigma \ah_{\alpha,\beta} \Wh^{\beta},
\end{equation*}
\begin{equation*}
    \FormQHh[u][v]:=
    \sum_{|\alpha|,|\beta|\leq \kappa}
    \Ltip*{\Wh^{\alpha}u}{\sigma \ah_{\alpha,\beta} \Wh^{\beta} v}[\R\times \CubengeqOne]
    +\Ltip*{u}{\sigma \partial_t v}[\R\times \CubengeqOne].
\end{equation*}
We verify the assumptions of Theorem \ref{Thm::EstBdry::MainThm::New} for \(\opLh\).
For this, we put a \(\wedge\) on all the notation 
from Theorem \ref{Thm::EstBdry::MainThm::New};
so for example we use \(\opLh\), \(\Wh_0,\ldots, \Wh_r\), \(\FormQh\), \(\FormQHh\), and \(\ah_{\alpha,\beta}\) in place of
\(\opL\), \(W_0,\ldots, W_r\), \(\FormQ\), \(\FormQH\), and \(a_{\alpha,\beta}\) in the setup of that theorem.

For \(f\in \LtSpace[\CubengeqOneHalf]\), we have
\begin{equation}\label{Eqn::Proof::LtNormPullBackIsEquiv}
    \LtNorm{f}[\Rngeq]\approx
    \LtNorm{\sigma^{1/2}f}[\Rngeq]=
    \LtNorm{f}[\Rngeq][\Psi^{*}\Vol] = \LtNorm{f\circ\Psi^{-1}}[\ManifoldN][\Vol].
\end{equation}

Combining \eqref{Eqn::Proof::QhEqualsPullbackQF} with Assumption \ref{Asssumption::Result::QIsQF},
we see \(\forall f,g\in \DomainQ\) with \(\supp(f)\Subset \Omegat\), we have
\begin{equation}\label{Eqn::Proof::QhEqualsPullbackQ}
    \FormQh*[f\circ \Psi][g\circ \Psi] = \FormQF[f][g]= \FormQ[f][g].
\end{equation}

\begin{lemma}\label{Lemma::EstBdry::InteriorMaxSubHolds}
Assumption \ref{Assumption::EstBdry::MaxSubOnInterior} holds.    
\end{lemma}
\begin{proof}
    Using \eqref{Eqn::Proof::LtNormPullBackIsEquiv},
    Remark \ref{Rmk::Result::DomainIsHkappaW::MaxSub}
    with \(\Compact_0= \Psi\left( \overline{\CubengeqOneHalf} \right) \),
    Assumption \ref{Asssumption::Result::TestFunctionsInDomain},
    and \eqref{Eqn::Proof::QhEqualsPullbackQ},
    we have for \(f\in \CinftycptSpace[\CubengOneHalf]\),
    \begin{equation}\label{Eqn::Proof::IntMaxSub::Tmp1}
    \begin{split}
        &\sum_{|\alpha|= \kappa} \LtNorm*{\Wh^{\alpha}f}[\Rngeq]^2
        \approx \sum_{|\alpha|= \kappa} \LtNorm*{\Wt^{\alpha}f\circ \Psi^{-1}}[\ManifoldN][\Vol]^2
        \lesssim \sum_{|\alpha|\leq \kappa} \LtNorm*{W^{\alpha}f\circ \Psi^{-1}}[\ManifoldN][\Vol]^2
        \\&=\HsWNorm*{f\circ \Psi^{-1}}{\kappa}^2
        \lesssim \left| \FormQ*[f\circ \Psi^{-1}][f\circ \Psi^{-1}] \right| + \LtNorm*{f\circ \Psi^{-1}}[\ManifoldN][\Vol]^2
        \\&\approx \left| \FormQ*[f\circ \Psi^{-1}][f\circ \Psi^{-1}] \right| + \LtNorm*{f}[\Rngeq]^2
    \end{split}
    \end{equation}
    Using \eqref{Eqn::Result::SectorialConsequence} and \eqref{Eqn::Proof::LtNormPullBackIsEquiv}, we have
    \begin{equation}\label{Eqn::Proof::IntMaxSub::Tmp2}
    \begin{split}
        &\left| \FormQ*[f\circ \Psi^{-1}][f\circ \Psi^{-1}] \right|
        \leq C_1 \Real \FormQ*[f\circ \Psi^{-1}][f\circ \Psi^{-1}] + C_2 \LtNorm*{f\circ\Psi^{-1}}[\ManifoldN][\Vol]^2
        \\&\leq C_1 \Real \FormQ*[f\circ \Psi^{-1}][f\circ \Psi^{-1}] + C_3 \LtNorm*{f}[\Rngeq]^2
        =C_1 \Real \FormQh*[f][f] + C_3 \LtNorm*{f}[\Rngeq]^2.
    \end{split}
    \end{equation}
    Combining \eqref{Eqn::Proof::IntMaxSub::Tmp1} and \eqref{Eqn::Proof::IntMaxSub::Tmp2},
    Assumption \ref{Assumption::EstBdry::MaxSubOnInterior} follows.
\end{proof}

\begin{lemma}\label{Lemma::Proof::opLhIsPullbackOpL}
    For \(f\in \DomainL\), 
    \begin{equation*}
        \opLh \left( f\circ \Psi \right) = \left( \opL f \right)\circ \Psi.
    \end{equation*}
\end{lemma}
\begin{proof}
    Let \(f\in \DomainL\) and 
    \(g\in \TestFunctionsZero[\CubengOneHalf]\).
    We have \(g\circ \Psi^{-1}\in \DomainQ\)
    (see Assumption \ref{Asssumption::Result::TestFunctionsInDomain}), and therefore, using \eqref{Eqn::Proof::QhEqualsPullbackQ},
    \begin{equation*}
    \begin{split}
         & \Ltip*{g}{\sigma \left( \opL f \right)\circ \Psi}[\Rngeq]
         =\Ltip{g\circ \Psi^{-1}}{\opL f}[\ManifoldN][\Vol]
         \\&=\FormQ[g\circ \Psi^{-1}][f] = \FormQh[g][f\circ \Psi]
         =\Ltip*{g}{\sigma \opLh \left( f\circ \Psi \right)}[\Rngeq].
    \end{split}
    \end{equation*}
\end{proof}

We turn to Assumption \ref{Assumption::EstBdry::MaxSubOnSmoothing}.
Set
\begin{equation*}
    \DSetL{1}:=\left\{ u(t, \Psi(x)) : u(t,\xi)\in \FSpace{1} \right\}.
\end{equation*}
Since \(\FSpace{l+1}\subseteq \FSpace{l}\) and \((\partial_t+\opL)\FSpace{l+1}\subseteq \FSpace{l}\)
(see Definition \ref{Defn::Result::FSpace}), using Lemma \ref{Lemma::Proof::opLhIsPullbackOpL}, we see
\begin{equation}\label{Eqn::Proof::DSetContainsFSpace}
  \DSetL{N}\supseteq \left\{u(t, \Psi(x)) : u(t,\xi)\in \FSpace{N}  \right\}  
\end{equation}
(see Definition \ref{Defn::EstBdry::DSet}).

\begin{lemma}
    Assumption \ref{Assumption::EstBdry::MaxSubOnSmoothing}\ref{Item::EstBdry::MaxSubOnSmoothing::InHkWOnCubegeqOneHalf}
    holds.
    That is, \(\forall \uh\in \DSetL{1}\), 
    \(\uh\big|_{\R\times \CubengeqOneHalf}\in \LtSpaceNoMeasure*[\R][\HsWhSpace*{\kappa}[\CubengeqOneHalf;\C^M]] \).
\end{lemma}
\begin{proof}
    Let \(u(t,\xi)\in \FSpace{1}\subseteq \LtSpaceNoMeasure*[\R][\DomainL]\subseteq \LtSpaceNoMeasure*[\R][\DomainQ]\).
    By Assumption \ref{Asssumption::Result::DomainIsHkappaW},
    \(u\in \LtSpaceNoMeasure*[\R][\HsWlocSpace*{\kappa}[\Omegat;\C^M]]=\LtSpaceNoMeasure*[\R][\HsWtlocSpace*{\kappa}[\Omegat;\C^M]]\).
    Using \eqref{Eqn::Proof::LtNormPullBackIsEquiv}, we conclude
    \(u(t, \Psi(x))\big|_{\R\times \CubengeqOneHalf}\in \LtSpaceNoMeasure*[\R][\HsWhSpace*{\kappa}[\CubengeqOneHalf;\C^M]].\)
\end{proof}

\begin{lemma}
    Assumption \ref{Assumption::EstBdry::MaxSubOnSmoothing}\ref{Item::EstBdry::MaxSubOnSmoothing::dtplusOplOnDIsFunction} holds.
\end{lemma}
\begin{proof}
    Let \(\uh(t,x)\in \DSetL{1}\). Take \(u(t,\xi)\in \FSpace{1}= \LtSpaceNoMeasure*[\R][\DomainL]\) such that
    \(\uh(t,x)=u(t,\Psi(x))\).  
    Using Lemma \ref{Lemma::Proof::opLhIsPullbackOpL}, we have
    \begin{equation*}
        \opLh \uh(t,x)\big|_{\R\times \CubengeqOneHalf} = 
        \opLh u(t,\Psi(x))\big|_{\R\times \CubengeqOneHalf}
        =\left( \opL u \right)(t,\Psi(x))\big|_{\R\times \CubengeqOneHalf}.
    \end{equation*}
    Since \(u(t,\xi)\in \LtSpaceNoMeasure*[\R][\DomainL]\),
    \(\opL u(t,\xi)\in \LtSpace[\R\times \ManifoldN][dt\times d\Vol]\).
    It follows from \eqref{Eqn::Proof::LtNormPullBackIsEquiv}
    that \(\left( \opL u \right)(t,x)\big|_{\R\times \CubengeqOneHalf}\in \LtSpace[\R\times \CubengeqOneHalf]\),
    completing the proof.
\end{proof}

\begin{lemma}\label{Lemma::Proof::MaxSubSmoothingAssumpHolds}
    Assumption \ref{Assumption::EstBdry::MaxSubOnSmoothing}\ref{Item::EstBdry::MaxSubOnSmoothing::MaxSubEstimate},\ref{Item::EstBdry::MaxSubOnSmoothing::FormQHGivesOp} hold.
\end{lemma}
\begin{proof}
    Let \(S\) be an operator as in \eqref{Eqn::EstBdry::SOpForAssump}.
    Define \(T=\Psi_{*}S\Psi^{*}\) as in 
    Assumption \ref{Asssumption::Result::BoundaryAssumption::New}.
    Take \(\uh\in \DSetL{1}\) and let \(u\in \FSpace{1}\)
    be such that \(\uh(t,x)=u(t,\Psi(x))\); i.e., \(\uh= \Psi^{*} u\). 

    We begin with Assumption \ref{Assumption::EstBdry::MaxSubOnSmoothing}\ref{Item::EstBdry::MaxSubOnSmoothing::MaxSubEstimate}.
    Using \eqref{Eqn::Proof::LtNormPullBackIsEquiv} and the support of \(K\) (see the discussion preceding \eqref{Eqn::EstBdry::SOpForAssump}), we have
    \begin{equation}\label{Eqn::Proof::MaxSubSmoothingAssumpHolds::Tmp1}
        \begin{split}
            &\sum_{|\alpha|= \kappa} \LtNorm*{\Wh^{\alpha} S\uh}[\Ropngeq]
            = \sum_{|\alpha|=\kappa}\LtNorm*{\Wh^{\alpha} S\Psi^{*} u}[\Ropngeq]
            \\&\leq\sum_{|\alpha|\leq \kappa}\LtNorm*{\Wh^{\alpha} S\Psi^{*} u}[\Ropngeq]
            \approx \sum_{|\alpha|\leq \kappa}\LtNorm*{\sigma^{1/2}\Wh^{\alpha} S\Psi^{*} u}[\Ropngeq]
            \\&=\sum_{|\alpha|\leq \kappa}\LtNorm*{\Wt^{\alpha} T u}[\R\times \Psi(\CubengeqOneHalf)][dt\times d\Vol]
            \approx \sum_{|\alpha|\leq \kappa}\LtNorm*{W^{\alpha} T u}[\R \times \ManifoldN][dt\times d\Vol].
        \end{split}
    \end{equation}
    By Lemma \ref{Lemma::Proof::New::BoundarySubellipticEstimate},
    \eqref{Eqn::Result::SectorialConsequence}, and Assumption \ref{Asssumption::Result::QIsQF},
    \begin{equation}\label{Eqn::Proof::MaxSubSmoothingAssumpHolds::Tmp2}
    \begin{split}
         &\sum_{|\alpha|\leq \kappa}\LtNorm*{W^{\alpha} T u}[\R\times \ManifoldN][dt\times d\Vol]^2
         \lesssim \left( \int \left| \FormQ[Tu(t,\cdot)][Tu(t,\cdot)] \right|\: dt + \LtNorm{Tu}[\R\times \ManifoldN][dt\times d\Vol]^2 \right)
         \\&\leq C_1 \int \Real \FormQ[Tu(t,\cdot)][Tu(t,\cdot)] \: dt + C_2\LtNorm{Tu}[\R\times \ManifoldN][dt\times d\Vol]^2
         \\&=C_1 \int \Real \FormQF[Tu(t,\cdot)][Tu(t,\cdot)] \: dt + C_2\LtNorm{Tu}[\R\times \ManifoldN][dt\times d\Vol]^2
         \\&=C_1 \int \Real \FormQh[S\Psi^{*}u(t,\cdot)][S\Psi^{*}u(t,\cdot)] \: dt + C_2\LtNorm{Tu}[\R\times \ManifoldN][dt\times d\Vol]^2
         \\&=C_1 \int \Real \FormQh[S\uh(t,\cdot)][S\uh(t,\cdot)] \: dt + C_2\LtNorm{Tu}[\R\times \ManifoldN][dt\times d\Vol]^2.
    \end{split}
    \end{equation}
    We have, using that \(\sigma\) does not depend on \(t\) (see, also, Remark \ref{Rmk::EstBdry::MainThm::tdoesntplayroleInMaxSub})
    \begin{equation}\label{Eqn::Proof::MaxSubSmoothingAssumpHolds::Tmp3}
    \begin{split}
         &\int \Real \FormQh[S\uh(t,\cdot)][S\uh(t,\cdot)] \: dt
         =\int \Real \FormQh[S\uh(t,\cdot)][S\uh(t,\cdot)] \: dt + \Real \Ltip*{S\uh}{\sigma \partial_tS\uh}[\Ropngeq]
         \\&=\Real \FormQHh[S\uh][S\uh].
    \end{split}
    \end{equation}
    Using \eqref{Eqn::Proof::LtNormPullBackIsEquiv} and the support of \(K\) (see the discussion preceding \eqref{Eqn::EstBdry::SOpForAssump}), we have
    \begin{equation}\label{Eqn::Proof::MaxSubSmoothingAssumpHolds::Tmp4}
    \begin{split}
         &\LtNorm{Tu}[\R\times \ManifoldN][dt\times \Vol]
         = \LtNorm{Tu}[\R\times \Psi\left(\CubengeqOneHalf  \right)][dt\times d\Vol]
         \\&=\LtNorm{\sigma^{1/2}S \Psi^{*} u}[\R\times \CubengeqOneHalf]
         \approx \LtNorm{S \Psi^{*} u}[\R\times \CubengeqOneHalf]
         =\LtNorm{S \uh}[\Ropngeq].
    \end{split}
    \end{equation}
    Combining \eqref{Eqn::Proof::MaxSubSmoothingAssumpHolds::Tmp1},
    \eqref{Eqn::Proof::MaxSubSmoothingAssumpHolds::Tmp2},
    \eqref{Eqn::Proof::MaxSubSmoothingAssumpHolds::Tmp3},
    and \eqref{Eqn::Proof::MaxSubSmoothingAssumpHolds::Tmp4}
    establishes Assumption \ref{Assumption::EstBdry::MaxSubOnSmoothing}\ref{Item::EstBdry::MaxSubOnSmoothing::MaxSubEstimate}.

    We turn to Assumption \ref{Assumption::EstBdry::MaxSubOnSmoothing}\ref{Item::EstBdry::MaxSubOnSmoothing::FormQHGivesOp}.
    We have,
    \begin{equation}\label{Eqn::Proof::MaxSubSmoothingAssumpHolds::Tmp5}
    \begin{split}
         &\FormQHh[S\uh][\uh]
         =\int \FormQh[S\uh (t,\cdot)][\uh(t,\cdot)]\: dt + \Ltip{S \uh}{\sigma \partial_t \uh}[\Ropngeq]
         \\&=\int \FormQh[S\Psi^{*}u (t,\cdot)][\Psi^{*}u(t,\cdot)]\: dt + \Ltip{S \uh}{\sigma \partial_t \uh}[\Ropngeq]
         \\&=\int \FormQF[Tu (t,\cdot)][u(t,\cdot)]\: dt + \Ltip{S \uh}{\sigma \partial_t \uh}[\Ropngeq].
    \end{split}
    \end{equation}
    For almost every \(t\), \(\supp(Tu(t,\cdot))\subseteq \Psi(\CubengeqOneHalf)\Subset \Omega\),
    and \(Tu(t,\cdot)\in \DomainQ\) and \(u(t,\cdot)\in \DomainL\subseteq \DomainQ\).
    Assumption \ref{Asssumption::Result::QIsQF} shows
    \begin{equation}\label{Eqn::Proof::MaxSubSmoothingAssumpHolds::Tmp6}
    \begin{split}
         & \FormQF[Tu (t,\cdot)][u(t,\cdot)] = 
         \FormQ[Tu (t,\cdot)][u(t,\cdot)].
    \end{split}
    \end{equation}
    By the definition of \(\opL\),
    and using \(Tu(t,\cdot)\in \DomainQ\) and \(u(t,\cdot)\in \DomainL\), \(\forall t\), and Lemma \ref{Lemma::Proof::opLhIsPullbackOpL}, we have
    \begin{equation}\label{Eqn::Proof::MaxSubSmoothingAssumpHolds::Tmp7}
    \begin{split}
         &
         \FormQ[Tu (t,\cdot)][u(t,\cdot)]
         =\Ltip*{Tu(t,\cdot)}{\opL u(t,\cdot)}[\ManifoldN][\Vol]
         \\&=\Ltip*{\Psi^{*}Tu(t,\cdot)}{\sigma \opLh \Psi^{*} u(t,\cdot)}[\Rngeq]
         =\Ltip*{S\uh(t,\cdot)}{\sigma \opLh \uh(t,\cdot)}[\Rngeq].
    \end{split}
    \end{equation}
    Combining \eqref{Eqn::Proof::MaxSubSmoothingAssumpHolds::Tmp5}, \eqref{Eqn::Proof::MaxSubSmoothingAssumpHolds::Tmp6},
    and \eqref{Eqn::Proof::MaxSubSmoothingAssumpHolds::Tmp7}
    establishes  Assumption \ref{Assumption::EstBdry::MaxSubOnSmoothing}\ref{Item::EstBdry::MaxSubOnSmoothing::FormQHGivesOp}.
\end{proof}

\begin{proof}[Proof of Proposition \ref{Prop::Proof::Boundary}]
    We have shown above that all the assumptions for Theorem \ref{Thm::EstBdry::MainThm::New} hold. 
    Let \(N\), \(l\), and \(J\) be as in the statement of Theorem \ref{Thm::Result::MainThm::New}.
    Take \(u\in \FSpace{N}\) and set \(\uh(t,x):=u(t,\Psi(x))\).
    By \eqref{Eqn::Proof::DSetContainsFSpace}, \(\uh\in \DSetL{N}\).
    Applying Theorem \ref{Thm::EstBdry::MainThm::New}, we see
    \begin{equation}\label{Eqn::Proof::ApplyBdryEstToHat}
    \begin{split}
            \HsNorm*{\phih_1 \uh}{l+1}[\Ropngeq]\lesssim
            \bigg( &
                \HsNorm*{\phih_2\left( \partial_t+\opLh \right)^J \uh}{l+1-2\kappa J}[\Ropngeq]
                \\&+ \HsNorm*{\phih_2\left( \partial_t+\opLh \right)^N \uh}{\left( l+1-N\epsilon_0 \right)\vee 0}[\Ropngeq]
                +\sum_{k=0}^{N-1} \LtNorm*{\phih_2\left( \partial_t+\opLh \right)^k \uh}[\Ropngeq]
             \bigg),
    \end{split}    
\end{equation}
    where if the right-hand side is finite, so is the left-hand side.
    Using that each object with a \(\wedge\) is the pullback via \(\Psi\) of the object without the \(\wedge\),
    and using \eqref{Eqn::Proof::LtNormPullBackIsEquiv}
    and the equivalence of the compactly supported Sobolev norms under smooth coordinate changes
    (as in Remark \ref{Rmk::PDOs::DefineClassicalHsSpaceOnMfld}),
    this implies \eqref{Eqn::Result::MainThm::MainEqn} holds and completes the proof.
\end{proof}

Finally we deal with the easy case  \(x_0\in \Omega\cap\InteriorN\), where we conclude a result even stronger
than Theorem \ref{Thm::Result::MainThm::New}.  This follows from the next proposition.
\begin{proposition}\label{Prop::Proof::Interior}
    Assume all the assumptions from Section \ref{Section::MainResult}.
    Let \(x_0\in \Omega\cap \InteriorN\), and \(\Psi:\CubenOne\xrightarrow{\sim}\Psi(\CubenOne)\subseteq \Omega\) 
    be any smooth coordinate chart with \(\Psi(0)=x_0\).
    Let \(\phih_1,\phih_2\in \CinftycptSpace[\left(\mathrm{id}\times \Psi \right)^{-1}(V)\cap (\R\times \CubenOneHalf)]\) with \(\phih_1=1\) on a 
    neighborhood of \((0,0)\)
    and 
    \(\phih_1\prec \phih_2\).
    Set \(\phi_j(t,\xi):=\phih_j(t,\Psi^{-1}(\xi))\in \CinftycptSpace[V]\).
    Then, \(\forall s\in \R\), \(\forall N\in \Zg\), \(\exists C\geq 0\), \(\forall u\in \Distributions[\R\times \InteriorN][\C^M]\),
    \begin{equation*}
        \HsNorm*{\phi_1 u}{s+\epsilon_0}
        \leq C
        \left( 
            \HsNorm*{\phi_2 \left( \partial_t+\opL \right)^N u}{s-(N-1)\epsilon_0}
            +\LtNorm*{\phi_2 u}
         \right),
    \end{equation*}
    where if the right-hand side is finite, so is the left-hand side.
\end{proposition}
\begin{proof}[Proof sketch]
    This is a simpler reprise of the proof of Proposition \ref{Prop::Proof::Boundary}.
    We apply Corollary \ref{Cor::EstBdryInt::MainCor}
    with
    \(W_0,\ldots, W_r\) replaced by \(\Psi^{*}W_1,\ldots,\Psi^{*} W_r\) and 
    \(\opL\) replaced by \(\Psi^{*}\opL\Psi_{*}\).
    That Assumption \ref{Assumption::EstBdryInt::MaxSubAssump} holds follows as in Lemma \ref{Lemma::EstBdry::InteriorMaxSubHolds}.
    From here, the result follows from Corollary \ref{Cor::EstBdryInt::MainCor},
    just as Proposition \ref{Prop::Proof::Boundary} followed from Theorem \ref{Thm::EstBdry::MainThm::New};
    we leave the details to the reader.
\end{proof}

\section{Boundary Conditions}\label{Section::BoundaryCond}
In this section, we describe the connection between the choice
of \(\FormQF\) and \(\CoreB\) and boundary conditions
as outlined in Example \ref{Example::Intro::Examples}.
First, in Section \ref{Section::BoundaryCond::IdentifyDomains}, we describe the standard way in which we connect forms to boundary conditions.
In Section \ref{Section::BoundaryCond::SecondOrder}, we describe Example \ref{Example::Intro::Examples}\ref{Item::Intro::Examples::SubLaplace},
while in Section \ref{Section::BoundaryCond::GeneralOps} we describe Example \ref{Example::Intro::Examples}\ref{Item::Intro::Examples::NonDirichlet}.

    \subsection{Identifying Domains}\label{Section::BoundaryCond::IdentifyDomains}
    In this section, we describe the standard way to understand the operator \(\LDomainL\)
in terms of the corresponding form acting on a core. For simplicity,
we work on a bounded smooth domain \(\Omega\subset \Rn\), though the same results hold for a manifold
with boundary.

Let \(\QDomainQ\) be a closed, densely defined, sectorial form on \(\LtSpace[\Omega]\),
and let \(\CoreB\) be a core for \(\QDomainQ\) with \(\TestFunctionsZero[\Omega]\subseteq \CoreB\). 
Let \(\LDomainL\) be the corresponding closed, densely defined, m-sectorial operator
(see \cite[Chapter 6, Section 2, Theorem 2.1]{KatoPerturbationTheory}).

Suppose that we have
\begin{equation*}
    \FormQ[f][g]=\FormQF[f][g]:=\sum_{|\alpha|,|\beta|\leq \kappa} \Ltip*{\partial_x^{\alpha}f}{a_{\alpha,\beta}\partial_x^{\beta}g}[\Omega],\quad f,g\in \CoreB,
\end{equation*}
where \(a_{\alpha,\beta}\in \CinftySpace[\Omega]\) are given functions.
Define
\begin{equation*}
    \opL_0:=\sum_{|\alpha|,|\beta|\leq \kappa} (-1)^{|\alpha|} \partial_x^{\alpha} a_{\alpha,\beta} \partial_x^{\beta}.
\end{equation*}

\begin{proposition}\label{Prop::BoundaryCond::IdentifyL}
    Fix \(g\in \LtSpace[\Omega]\). The following are equivalent:
    \begin{enumerate}[(i)]
        \item\label{Item::BoundaryCond::IdentifyL::gInDomain} \(g\in \DomainL\).
        \item\label{Item::BoundaryCond::IdentifyL::IntByParts} 
        \(g\in \DomainQ\),
        \(\opL_0 g\in \LtSpace[\Omega]\) and \(\FormQ[f][g]=\Ltip{f}{\opL_0g}\), \(\forall f\in \CoreB\); here \(\opL_0 g\) is taken in the sense of \(\DistributionsZero[\Omega]\).
    \end{enumerate}
    In the above case, we have \(\opL g=\opL_0 g\).
\end{proposition}
\begin{proof}
    By \cite[Chapter 6, Section 2, Theorem 2.1]{KatoPerturbationTheory},
    \ref{Item::BoundaryCond::IdentifyL::gInDomain} holds if and only if
    \(g\in\DomainQ\) and there exists \(w\in\LtSpace[\Omega]\) such that
    \begin{equation}\label{Eqn::BoundaryCond::IdentifyL::Existw}
        \FormQ[f][g]=\Ltip{f}{w},
        \qquad \forall f\in\CoreB.
    \end{equation}
    Here we used that \(\CoreB\) is a core for \(\QDomainQ\).

    Suppose \eqref{Eqn::BoundaryCond::IdentifyL::Existw} holds.
    Take \(g_j\in\CoreB\) such that
    \(g_j\rightarrow g\) in \(\DomainQ\).
    For \(f\in\TestFunctionsZero[\Omega]\subseteq\CoreB\), we have
    \begin{align*}
        \Ltip{f}{w}
        &=
        \FormQ[f][g]
        =
        \lim_{j\to\infty}\FormQ[f][g_j]
        =
        \lim_{j\to\infty}\FormQF[f][g_j]
        \\
        &=
        \lim_{j\to\infty}\Ltip{f}{\opL_0g_j}
        =
        \Ltip{f}{\opL_0g},
    \end{align*}
    where the final equality uses
    \(g_j\rightarrow g\) in \(\LtSpace[\Omega]\).
    Thus \(w=\opL_0g\) in \(\DistributionsZero[\Omega]\).
    In particular, \(\opL_0g\in\LtSpace[\Omega]\), and
    \ref{Item::BoundaryCond::IdentifyL::IntByParts} holds.
    Conversely, if
    \ref{Item::BoundaryCond::IdentifyL::IntByParts} holds, then
    \eqref{Eqn::BoundaryCond::IdentifyL::Existw} holds with
    \(w=\opL_0g\).
    This proves
    \ref{Item::BoundaryCond::IdentifyL::gInDomain}
    \(\Longleftrightarrow\)
    \ref{Item::BoundaryCond::IdentifyL::IntByParts}.

    Finally,
    \cite[Chapter 6, Section 2, Theorem 2.1]{KatoPerturbationTheory}
    gives
    \begin{equation*}
        \opL g=w=\opL_0g,
    \end{equation*}
    completing the proof.
\end{proof}

In the examples in  Sections \ref{Section::BoundaryCond::SecondOrder} and \ref{Section::BoundaryCond::GeneralOps}, we will see for sufficiently smooth \(g\)
\begin{equation*}
    \FormQF[f][g]=\Ltip{f}{\opL_0 g}+\text{Boundary Integral},\quad \forall f\in \CoreB.
\end{equation*}
where the Boundary Integral comes from integration by parts. Thus, for \(g\in \DomainL\) it is necessary
and sufficient that the Boundary Integrals vanish \(\forall f\in \CoreB\). This will lead to the desired ``boundary conditions.''

    \subsection{Second Order Operators}\label{Section::BoundaryCond::SecondOrder}
    Let \(\ManifoldN\) be a compact manifold with boundary, with smooth strictly
positive density \(\Vol\), and let \(W_1,\ldots, W_r\) be smooth vector fields
on \(\ManifoldN\) satisfying H\"ormander's condition.
Let
\begin{equation*}
    \opL_0 f(x) = -\sum_{i,j=1}^r a_{i,j} W_i W_j f(x) + \sum_{j=1}^r a_j W_j f(x) +a(x)f(x),
\end{equation*}
where 
\(a_{i,j}\in \CinftySpace[\ManifoldN][\R]\)
\(a_j, a\in \CinftySpace[\ManifoldN][\C]\) and if
\(A(x)=\left( a_{i,j}(x) \right)_{1\leq i,j,\leq r}\in \CinftySpace[\ManifoldN][\MatrixSpace[r][r][\R]]\),
then \(A(x)\) is a real, symmetric, strictly positive definite matrix, \(\forall x\in \ManifoldN\).

\begin{assumption}\label{Assumption::BoundaryCond::SecondOrder::NonChar}
    We assume every point of \(\BoundaryN\) is \(W\)-non-characteristic
    (see Definition \ref{Defn::Intro::NonChar::NonCharDefn}).
\end{assumption}

\begin{remark}
    In Assumption \ref{Assumption::BoundaryCond::SecondOrder::NonChar}, we have
    assumed every boundary point is \(W\)-non-characteristic. It is possible
    to instead give local versions of the results in this section near a
    \(W\)-non-characteristic boundary point, even if some other boundary
    points are not \(W\)-non-characteristic. We leave such generalizations to the reader.
\end{remark}



The main result of this section is the next proposition.

\begin{proposition}\label{Prop::BoundaryConds::SecondOrder::MainSecondOrder}
    There exists a vector field \(X\) in the \(\CinftySpace[\ManifoldN][\R]\)-module
    generated by \(W_1,\ldots, W_r\), with \(X(x')\not \in \TangentSpace{x'}{\BoundaryN}\),
    \(\forall x'\in \BoundaryN\), and such that \(\forall B\in \CinftySpace[\BoundaryN][\C]\),
    there exist \(e_j^1, e_j^2, e\in \CinftySpace[\ManifoldN][\C]\) such that if we set
    \begin{equation}\label{Eqn::BoundaryConds::SecondOrder::MainSecondOrder::FormQF}
        \FormQF[f][g]:=
        \sum_{1\leq i,j\leq r}
        \Ltip*{W_if}{a_{i,j}W_jg}[\ManifoldN]
        +\sum_{j=1}^r \Ltip*{f}{e_j^1 W_j g}[\ManifoldN]
        +\sum_{j=1}^r\Ltip*{W_j e_j^2 f}{g}[\ManifoldN]
        +\Ltip*{f}{eg}[\ManifoldN],
    \end{equation}
    then,
    \begin{enumerate}[(i)]
        \item\label{Item::BoundaryConds::SecondOrder::MainSecondOrder::MaxSub} \(\left( \FormQF, \CinftySpace[\ManifoldN] \right)\) is maximally subelliptic
            in the sense that Assumption \ref{Assumption::Intro2::MaxSub}
            holds with \(\CoreB=\CinftySpace[\ManifoldN]\).
        \item\label{Item::BoundaryConds::SecondOrder::MainSecondOrder::LEqualsL0} If \(\LDomainL\) is the operator associated to the closure
            of \(\left( \FormQF, \CinftySpace[\ManifoldN] \right)\) (as in the introduction\footnote{See the discussion following Lemma \ref{Lemma::Intro2::QFIsCloseable}.}),
            then \(\opL f=\opL_0 f\), \(\forall f\in \DomainL\), where \(\opL_0 f\)
            is taken in the sense of \(\DistributionsZero[\ManifoldN]\).
        \item\label{Item::BoundaryConds::SecondOrder::MainSecondOrder::BoundaryCond} 
        For \(f\in \CjSpace{2}[\ManifoldN]\),
            \begin{equation*}
                f\in \DomainL \iff \left( Xf+Bf \right)\big|_{\BoundaryN}=0.
            \end{equation*}
    \end{enumerate}
\end{proposition}

\begin{remark}
    When considering elliptic boundary value problems, one has more flexibility in
    choosing the form \(\FormQF\). This comes from the following idea.
    Consider the operator
    \begin{equation*}
        \sum_{i,j} a_{i,j}\partial_{x_i}\partial_{x_j}.
    \end{equation*}
    Because \(\left[ \partial_{x_i},\partial_{x_j} \right]=0\),
    one can replace \(a_{i,j}\) with any matrix \(b_{i,j}\)
    with \(2a_{i,j}=b_{i,j}+b_{j,i}\). This flexibility allows one
    to choose forms for a wider class of boundary conditions;
    see \cite[Example 7.11]{FollandIntroductionPDE}.
    In our case, \(\left[ W_i, W_j \right]\) may not be in the
    \(\CinftySpace[\ManifoldN][\R]\)-module generated by \(W_1,\ldots, W_r\),
    so we lack the same flexibility.
    Roughly, speaking, if we had this flexibility, in Proposition \ref{Prop::BoundaryConds::SecondOrder::MainSecondOrder}
    instead of saying \(\exists X\), we could choose a form \(\FormQF\) for any
    \(X\) in the \(\CinftySpace[\ManifoldN][\R]\)-module generated by \(W_1,\ldots, W_r\)
    with \(X(x')\not\in \TangentSpace{x'}{\BoundaryN}\), \(\forall x'\in \BoundaryN\). As is, we stick to proving
    Proposition \ref{Prop::BoundaryConds::SecondOrder::MainSecondOrder} for just one
    such \(X\).
    It seems likely that the methods of this paper could be extended (with a few new estimates
    and relying on the second order nature of \(\opL_0\)),
    to prove such a version of Proposition \ref{Prop::BoundaryConds::SecondOrder::MainSecondOrder}.
    However, to do this for the more general (higher order) operators covered in this paper
    would likely require some new ideas. We do not pursue such questions here.
\end{remark}

We let \(W_j^{*}\) denote the formal \(\LtSpace[\ManifoldN][\Vol]\)-adjoint of \(W_j\);
and similarly for any other vector field.

\begin{lemma}\label{Lemma::BoundaryCond::SecondOrder::ChooseWt}
    There exist \(\Wt_0,\Wt_1,\ldots, \Wt_r\) smooth vector fields on \(\ManifoldN\)
    and \(\at_{i,j}\in \CinftySpace[\ManifoldN][\R]\), \(0\leq i,j,\leq r\)
    such that:
    \begin{enumerate}[(i)]
        \item\label{Item::BoundaryCond::SecondOrder::ChooseWt::SameModule} The \(\CinftySpace[\ManifoldN][\R]\)-module generated by \(\Wt_0,\ldots, \Wt_r\)
            equals the \(\CinftySpace[\ManifoldN][\R]\)-module generated by \(W_1,\ldots, W_r\).
        \item\label{Item::BoundaryCond::SecondOrder::ChooseWt::NotInTangentSpace} \(\Wt_0(x')\not \in \TangentSpace{x'}{\BoundaryN}\), \(\forall x'\in \BoundaryN\).
        \item\label{Item::BoundaryCond::SecondOrder::ChooseWt::InTangentSpace} \(\Wt_j(x')\in \TangentSpace{x'}{\BoundaryN}\), \(\forall x'\in \BoundaryN\), \(1\leq j\leq r\).
        \item\label{Item::BoundaryCond::SecondOrder::ChooseWt::SameForm} \(\sum_{0\leq i,j\leq r}\Ltip{\Wt_i f}{\at_{i,j} \Wt_j g}[\ManifoldN]=\sum_{1\leq i,j\leq r} \Ltip{W_i f}{a_{i,j}W_j g}[\ManifoldN]\) for any \(f\) and \(g\) for which both sides make sense.
        \item\label{Item::BoundaryCond::SecondOrder::ChooseWt::SameOperator} \(\sum_{0\leq i,j\leq r} \Wt_i^{*}\at_{i,j}\Wt_j=\sum_{1\leq i,j\leq r} W_i^{*} a_{i,j}W_j\). 
        \item\label{Item::BoundaryCond::SecondOrder::ChooseWt::at00Postive} \(\at_{0,0}(x')>0\), \(\forall x'\in \BoundaryN\).
    \end{enumerate}
\end{lemma}
\begin{proof}
    Pick \(\Wt_0\) in the \(\CinftySpace[\ManifoldN][\R]\)-module generated
    by \(W_1,\ldots, W_r\) such that \ref{Item::BoundaryCond::SecondOrder::ChooseWt::NotInTangentSpace}
    holds; this is always possible by Assumption \ref{Assumption::BoundaryCond::SecondOrder::NonChar}.
    For \(1\leq j\leq r\), there exists a unique \(c_j\in \CinftySpace[\BoundaryN][\R]\)
    such that \(W_j(x')-c_j(x')\Wt_0(x')\in \TangentSpace{x'}{\BoundaryN}\), \(\forall x'\in \BoundaryN\).
    Assumption \ref{Assumption::BoundaryCond::SecondOrder::NonChar} implies
    \begin{equation}\label{Eqn::BoundaryCond::SecondOrder::ChooseWt::Tmp0}
        \left( c_1(x'),\ldots, c_r(x') \right)\ne 0,\quad \forall x'\in \BoundaryN.
    \end{equation}
    Let \(m_j\in \CinftySpace[\ManifoldN][\R]\) be such that \(m_j\big|_{\BoundaryN}=c_j\)
    and define
    \(\Wt_j = W_j -m_j W_0\), \(1\leq j\leq r\);
    \ref{Item::BoundaryCond::SecondOrder::ChooseWt::SameModule} and
    \ref{Item::BoundaryCond::SecondOrder::ChooseWt::InTangentSpace} follow.
    Let \(M(x)\) denote the \(r\times (r+1)\) matrix
    \begin{equation*}
        M(x)=
        \begin{bmatrix}
            m_1(x) & 1&0&\cdots &0\\
            m_2(x) &0 &1&\cdots &0\\
            \ldots &0 &0 & \ddots &0\\
            m_r(x) &0 &0 &\cdots &1
        \end{bmatrix},
    \end{equation*}
    so that
    \begin{equation}\label{Eqn::BoundaryCond::SecondOrder::ChooseWt::MatrixEqn}
        \begin{bmatrix}
            W_1(x)\\
            W_2(x)\\
            \vdots \\
            W_r(x)
        \end{bmatrix}
        =M(x)
        \begin{bmatrix}
            \Wt_0(x)\\
            \Wt_1(x)\\
            \vdots\\
            \Wt_r(x)
        \end{bmatrix}.
    \end{equation}
    Define \(\at_{i,j}\) by letting \(\At(x)=(\at_{i,j})_{0\leq i,j\leq r}\) and
    setting \(\At(x)=M(x)^{\transpose}A(x)M(x)\); \ref{Item::BoundaryCond::SecondOrder::ChooseWt::SameForm}
    and \ref{Item::BoundaryCond::SecondOrder::ChooseWt::SameOperator} follow.

    Using \eqref{Eqn::BoundaryCond::SecondOrder::ChooseWt::MatrixEqn}, we have
    \begin{equation}\label{Eqn::BoundaryCond::SecondOrder::ChooseWt::Tmp1}
        \at_{0,0}(x') = \sum_{i,j}m_i(x') a_{i,j}(x')m_j(x')= \sum_{i,j}c_i(x') a_{i,j}(x')c_j(x'),\quad \forall x'\in \BoundaryN.
    \end{equation}
    From \eqref{Eqn::BoundaryCond::SecondOrder::ChooseWt::Tmp1},
    \ref{Item::BoundaryCond::SecondOrder::ChooseWt::at00Postive} follows from
    \eqref{Eqn::BoundaryCond::SecondOrder::ChooseWt::Tmp0} and the fact
    that \(A=(a_{i,j})\) is assumed strictly positive definite.
\end{proof}

For the rest of this section, let \(\Wt_0,\ldots, \Wt_r\) and \(\at_{i,j}\in \CinftySpace[\ManifoldN][\R]\)
be as in Lemma \ref{Lemma::BoundaryCond::SecondOrder::ChooseWt}.
Choose \(\at_j,\at\in \CinftySpace[\ManifoldN][\C]\) such that
\begin{equation*}
    \opL_0=\sum_{0\leq i,j\leq r} \Wt_i^{*}a_{i,j}\Wt_j+\sum_{j=0}^r \at_j \Wt_j+\at.
\end{equation*}
This is clearly possible by Lemma \ref{Lemma::BoundaryCond::SecondOrder::ChooseWt}\ref{Item::BoundaryCond::SecondOrder::ChooseWt::SameModule},\ref{Item::BoundaryCond::SecondOrder::ChooseWt::SameOperator}.

Let \(h_j\in \CinftySpace[\ManifoldN][\R]\) be defined by \(\Wt_j^{*}=-\Wt_j+h_j\).
Take \(b_j^1,b_j^2\in \CinftySpace[\ManifoldN][\C]\) satisfying
\(b_j^1-\overline{b_j^2}=\at_j\) and \(B=\overline{b_0^2}\big|_{\BoundaryN}\).
Define \(b\in \CinftySpace[\ManifoldN][\C]\)
by \(b+\sum_{j=0}^r h_j b_j^2=\at\).
Define
\begin{equation}\label{Eqn::BoundaryCond::SecondOrder::FormulaFormQF}
    \begin{split}
        \FormQF[f][g]
        :=
        &\sum_{0\leq i,j\leq r} \Ltip*{\Wt_i f}{\at_{i,j}\Wt_j g}[\ManifoldN]
        +\sum_{j=0}^r \Ltip*{f}{b_j^1 \Wt_j g}[\ManifoldN]
        \\&+\sum_{j=0}^r \Ltip*{\Wt_j b_j^2 f}{g}[\ManifoldN]
        +\Ltip*{f}{bg}[\ManifoldN].
    \end{split}
\end{equation}
\(\FormQF\) is of the form \eqref{Eqn::BoundaryConds::SecondOrder::MainSecondOrder::FormQF}
by Lemma \ref{Lemma::BoundaryCond::SecondOrder::ChooseWt}\ref{Item::BoundaryCond::SecondOrder::ChooseWt::SameModule},\ref{Item::BoundaryCond::SecondOrder::ChooseWt::SameForm}.
Set \(X:=\sum_{j=0}^r \at_{0,j} \Wt_j\);
that \(X\) is in the \(\CinftySpace[\ManifoldN][\R]\)-module generated by 
\(W_1,\ldots, W_r\) follows from Lemma \ref{Lemma::BoundaryCond::SecondOrder::ChooseWt}\ref{Item::BoundaryCond::SecondOrder::ChooseWt::SameModule}
and that \(X(x')\not \in \TangentSpace{x'}{\BoundaryN}\), \(\forall x'\in \BoundaryN\)
follows from Lemma \ref{Lemma::BoundaryCond::SecondOrder::ChooseWt}\ref{Item::BoundaryCond::SecondOrder::ChooseWt::NotInTangentSpace},\ref{Item::BoundaryCond::SecondOrder::ChooseWt::InTangentSpace},\ref{Item::BoundaryCond::SecondOrder::ChooseWt::at00Postive}.

\begin{lemma}\label{Lemma::BoundaryCond::SecondOrder::IntByParts}
    For \(f,g\in \CjSpace{1}[\ManifoldN]\),
    \begin{equation}\label{Eqn::BoundaryCond::SecondOrder::IntByParts::jNonZero}
        \Ltip*{\Wt_j f}{g}[\ManifoldN] = \Ltip*{f}{\Wt_j^{*}g}[\ManifoldN],\quad 1\leq j\leq r,
    \end{equation}
    \begin{equation}\label{Eqn::BoundaryCond::SecondOrder::IntByParts::jZero}
        \Ltip*{\Wt_0 f}{g}[\ManifoldN] = \Ltip*{f}{\Wt_0^{*} g}[\ManifoldN] + \int_{\BoundaryN} f(x')\overline{g(x')}\: d\sigma(x'),
    \end{equation}
    where \(\sigma\) is a smooth, nowhere-vanishing density on \(\BoundaryN\).
\end{lemma}
\begin{proof}
    \eqref{Eqn::BoundaryCond::SecondOrder::IntByParts::jNonZero}
    follows from Lemma \ref{Lemma::BoundaryCond::SecondOrder::ChooseWt}\ref{Item::BoundaryCond::SecondOrder::ChooseWt::InTangentSpace}
    and integration by parts, while \eqref{Eqn::BoundaryCond::SecondOrder::IntByParts::jZero}
    follows from Lemma \ref{Lemma::BoundaryCond::SecondOrder::ChooseWt}\ref{Item::BoundaryCond::SecondOrder::ChooseWt::NotInTangentSpace}
    and integration by parts.
\end{proof}

\begin{lemma}\label{Lemma::BoundaryCond::SecondOrder::IntByPartsWholeForm}
    \begin{enumerate}[(i)]
        \item\label{Item::BoundaryCond::SecondOrder::IntByPartsWholeForm::C2} For \(f,g\in \CjSpace{2}[\ManifoldN]\),
            \begin{equation*}
                \FormQF[f][g]
                =\Ltip*{f}{\opL_0 g}[\ManifoldN]
                +\int_{\BoundaryN} f(x') \overline{\left( X g(x')+B(x')g(x') \right)}\: d\sigma(x'),
            \end{equation*}
            where \(\sigma\) is the smooth, nowhere-vanishing density from
            Lemma \ref{Lemma::BoundaryCond::SecondOrder::IntByParts}.
        \item\label{Item::BoundaryCond::SecondOrder::IntByPartsWholeForm::Distribution} For \(f\in \TestFunctionsZero[\ManifoldN]\), \(g\in \DistributionsZero[\ManifoldN]\),
            \begin{equation*}
                \FormQF[f][g]=\Ltip{f}{\opL_0 g}[\ManifoldN].
            \end{equation*}
    \end{enumerate}
\end{lemma}
\begin{proof}
    \ref{Item::BoundaryCond::SecondOrder::IntByPartsWholeForm::C2} follows by applying Lemma \ref{Lemma::BoundaryCond::SecondOrder::IntByParts}
    to \eqref{Eqn::BoundaryCond::SecondOrder::FormulaFormQF} and unravelling the definitions.
    \ref{Item::BoundaryCond::SecondOrder::IntByPartsWholeForm::Distribution} is the same, except
    there are no boundary terms since \(f\in \TestFunctionsZero[\ManifoldN]\).
\end{proof}

\begin{lemma}\label{Lemma::BoundaryCond::SecondOrder::MaxSub}
    \(\FormQF\) is maximally subelliptic in the sense that \(\exists C_1,C_2\geq 0\),
    \begin{equation*}
        \LtNorm{f}[\ManifoldN]^2+\sum_{j=1}^r \LtNorm{W_j f}[\ManifoldN]^2
        \leq C_1 \Real \FormQF[f][f] + C_2 \LtNorm{f}[\ManifoldN]^2,\quad \forall f\in \CinftySpace[\ManifoldN].
    \end{equation*}
\end{lemma}
\begin{proof}
    Since \(A(x)=(a_{i,j}(x))_{1\leq i,j\leq r}\) is assumed strictly positive definite,
    we have
    \begin{equation}\label{Eqn::BoundaryCond::SecondOrder::MaxSub::MainPartLowerBound}
        \begin{split}
            &\sum_{j=1}^r \LtNorm{W_j f}[\ManifoldN]^2
            \lesssim \sum_{1\leq i,j\leq r} \Ltip*{W_i f}{a_{i,j} W_j f}[\ManifoldN]
            \\&=\sum_{1\leq i,j\leq r} \Real \Ltip*{W_i f}{a_{i,j} W_j f}[\ManifoldN]
            =\sum_{0\leq i,j\leq r} \Real \Ltip*{\Wt_i f}{\at_{i,j}\Wt_j f}[\ManifoldN]
        \end{split}
    \end{equation}
    where the final equality follows from 
    Lemma \ref{Lemma::BoundaryCond::SecondOrder::ChooseWt}\ref{Item::BoundaryCond::SecondOrder::ChooseWt::SameForm}.

    Using the \(\sconst\), \(\lconst\) notation from Section \ref{Section::EstNearBdry::LargeSmallSonsts}
    and Lemma \ref{Lemma::BoundaryCond::SecondOrder::ChooseWt}\ref{Item::BoundaryCond::SecondOrder::ChooseWt::SameModule},
    we have
    \begin{equation}\label{Eqn::BoundaryCond::SecondOrder::MaxSub::BoundWtjOnRight}
        \begin{split}
            \sum_{j=0}^r \left| \Ltip*{f}{b_j^1 \Wt_j f}[\ManifoldN] \right|
            &\leq  \sconst \sum_{j=0}^r \LtNorm*{\Wt_j f}[\ManifoldN]^2+\lconst \LtNorm*{f}[\ManifoldN]^2 
            \\&\leq \sconst \sum_{j=1}^r \LtNorm*{W_j f}[\ManifoldN]^2+\lconst \LtNorm*{f}[\ManifoldN]^2,
        \end{split}
    \end{equation}
    and similarly,
    \begin{equation}\label{Eqn::BoundaryCond::SecondOrder::MaxSub::BoundWtjOnLeft}
        \sum_{j=0}^r \left| \Ltip*{\Wt_j b_j^2 f}{f}[\ManifoldN] \right|
        \leq \sconst \sum_{j=1}^r \LtNorm*{W_j f}[\ManifoldN]^2+\lconst \LtNorm*{f}[\ManifoldN]^2.
    \end{equation}
    Using the formula \eqref{Eqn::BoundaryCond::SecondOrder::FormulaFormQF}
    and \eqref{Eqn::BoundaryCond::SecondOrder::MaxSub::MainPartLowerBound}, 
    \eqref{Eqn::BoundaryCond::SecondOrder::MaxSub::BoundWtjOnRight},
    and \eqref{Eqn::BoundaryCond::SecondOrder::MaxSub::BoundWtjOnLeft}, we see
    \begin{equation}\label{Eqn::BoundaryCond::SecondOrder::MaxSub::FinalEst}
        \sum_{j=1}^r \LtNorm{W_j f}[\ManifoldN]^2
        \lesssim \Real \FormQF[f][f]
        + \sconst \sum_{j=1}^r \LtNorm{W_j f}[\ManifoldN]^2
        +\lconst \LtNorm{f}[\ManifoldN]^2.
    \end{equation}
    Subtracting 
    \(\sconst \sum_{j=1}^r \LtNorm{W_j f}[\ManifoldN]^2\)
    from both sides of \eqref{Eqn::BoundaryCond::SecondOrder::MaxSub::FinalEst} completes the proof.
\end{proof}

\begin{proof}[Proof of Proposition \ref{Prop::BoundaryConds::SecondOrder::MainSecondOrder}]
    We have already verified that \(\FormQF\) is of the form \eqref{Eqn::BoundaryConds::SecondOrder::MainSecondOrder::FormQF}
    and that \(X\) is in the \(\CinftySpace[\ManifoldN][\R]\)-module
    generated by \(W_1,\ldots, W_r\), with \(X(x')\not \in \TangentSpace{x'}{\BoundaryN}\),
    \(\forall x'\in \BoundaryN\).
    \ref{Item::BoundaryConds::SecondOrder::MainSecondOrder::MaxSub} is Lemma \ref{Lemma::BoundaryCond::SecondOrder::MaxSub}.
    \ref{Item::BoundaryConds::SecondOrder::MainSecondOrder::LEqualsL0}
    follows from Lemma \ref{Lemma::BoundaryCond::SecondOrder::IntByPartsWholeForm}\ref{Item::BoundaryCond::SecondOrder::IntByPartsWholeForm::Distribution}.

    We turn to \ref{Item::BoundaryConds::SecondOrder::MainSecondOrder::BoundaryCond}; this uses the same method
    as the proof of Proposition \ref{Prop::BoundaryCond::IdentifyL}.
    First we claim \(\CjSpace{2}[\ManifoldN]\subseteq \DomainQ\). Indeed, Lemma \ref{Lemma::BoundaryCond::SecondOrder::MaxSub}
    shows  \(\HsWNorm{f}{1}\approx \DomainQNorm{f}\) on \(\CinftySpace[\ManifoldN]\).  Since \(\CinftySpace[\ManifoldN]\)
    is dense in \(\HsWSpace{1}[\ManifoldN]\) (see \cite[Corollary \ref*{FS::Cor::Spaces::Approximation::SmoothFunctionsAreDense}\ref*{FS::Item::Spaces::Approximation::SmoothFunctionsAreDense::ApproxInDist}]{StreetFunctionSpacesAndTraceTheoremsForMaximallySubellipticBoundaryValueProblems})
    this shows \(\CjSpace{2}[\ManifoldN]\subseteq \HsWSpace{1}[\ManifoldN]\subseteq \DomainQ\).
    The same density argument shows \(\FormQ\) agrees with the continuous extension of \(\FormQF\)
    on \(\HsWSpace{1}[\ManifoldN]\times \HsWSpace{1}[\ManifoldN]\). In particular,
    \begin{equation*}
        \FormQ[f][g]=\FormQF[f][g],\quad f\in \CinftySpace[\ManifoldN],\: g\in \CjSpace{2}[\ManifoldN].
    \end{equation*}

    Now, fix \(g\in \CjSpace{2}[\ManifoldN]\subseteq\DomainQ\).
    Consider the condition
    \begin{equation}\label{Eqn::BoundaryConds::SecondOrder::MainSecondOrder::1}
        \exists w\in \LtSpace[\ManifoldN][\Vol]\text{ with } \FormQF[f][g]=\Ltip{f}{w}[\ManifoldN],\quad \forall f\in \CinftySpace[\ManifoldN].
    \end{equation}
    If \eqref{Eqn::BoundaryConds::SecondOrder::MainSecondOrder::1} holds, then
    Lemma \ref{Lemma::BoundaryCond::SecondOrder::IntByPartsWholeForm}\ref{Item::BoundaryCond::SecondOrder::IntByPartsWholeForm::Distribution}
    shows \(w\) must equal \(\opL_0 g\). We conclude 
    \eqref{Eqn::BoundaryConds::SecondOrder::MainSecondOrder::1} holds if and only if
    \begin{equation}\label{Eqn::BoundaryConds::SecondOrder::MainSecondOrder::2}
        \FormQF[f][g]=\Ltip{f}{\opL_0 g}[\ManifoldN],\quad \forall f\in \CinftySpace[\ManifoldN].
    \end{equation}
    By
    Lemma \ref{Lemma::BoundaryCond::SecondOrder::IntByPartsWholeForm}\ref{Item::BoundaryCond::SecondOrder::IntByPartsWholeForm::C2},
    \eqref{Eqn::BoundaryConds::SecondOrder::MainSecondOrder::2} holds if and only if
    \begin{equation}\label{Eqn::BoundaryConds::SecondOrder::MainSecondOrder::3}
        X(x')g(x')+B(x')g(x')=0,\quad \forall x'\in \BoundaryN.
    \end{equation}
    By \cite[Chapter 6, Section 2, Theorem 2.1]{KatoPerturbationTheory},
    \eqref{Eqn::BoundaryConds::SecondOrder::MainSecondOrder::1} holds if and only if \(g\in \DomainL\).
    We conclude \(g\in \DomainL\) if and only if \eqref{Eqn::BoundaryConds::SecondOrder::MainSecondOrder::3} holds;
    completing the proof.
\end{proof}

    \subsection{General Operators}\label{Section::BoundaryCond::GeneralOps}
    Let \(\ManifoldN\) be a smooth manifold with boundary of dimension \(n\geq 2\), endowed with a smooth, strictly positive
density \(\Vol\). Let \(W_1,\ldots, W_r\) be vector fields on \(\ManifoldN\)
satisfying H\"ormander's condition.

\begin{assumption}\label{Assumption::BoundaryCond::GeneralOps::NonChar}
    We assume every point of \(\BoundaryN\) is \(W\)-non-characteristic
    (see Definition \ref{Defn::Intro::NonChar::NonCharDefn}).
\end{assumption}

The results in this section can be seen as a generalization of classical results;
see \cite[Section 7.B]{FollandIntroductionPDE}.

\begin{remark}
    The computations in this section are local, so the global Assumption \ref{Assumption::BoundaryCond::GeneralOps::NonChar}
    is not necessary: one can work near a \(W\)-non-characteristic boundary point.
\end{remark}

Fix \(\kappa\in \Zg\) and set
\begin{equation*}
    \FormQF[f][g]:=\sum_{|\alpha|,|\beta|\leq \kappa} \Ltip*{W^{\alpha}f}{a_{\alpha,\beta}W^\beta g}[\ManifoldN],\quad f,g\in \CinftycptSpace[\ManifoldN],
\end{equation*}
where \(a_{\alpha,\beta}\in \CinftySpace[\ManifoldN]\) are given functions.
Set
\begin{equation*}
    \opL_0:=\sum_{|\alpha|,|\beta|\leq \kappa} \left( W^{\alpha} \right)^{*} a_{\alpha,\beta} W^\beta,
\end{equation*}
where \(*\) denotes the formal \(\LtSpace[\ManifoldN][\Vol]\) adjoint.
In light of Proposition \ref{Prop::BoundaryCond::IdentifyL}, we consider the following question.

\begin{question}\label{Question::BoundaryCond::GeneralOps::MainQuestion}
    Given a subspace \(\CoreB\subseteq \CinftycptSpace[\ManifoldN]\) with \(\TestFunctionsZero[\ManifoldN]\subseteq \CoreB\),
    for what \(g\in \CinftycptSpace[\ManifoldN]\) do we have
    \begin{equation*}
        \FormQF[f][g]=\Ltip*{f}{\opL_0 g}[\ManifoldN]?
    \end{equation*}
\end{question}

In this section we address Question \ref{Question::BoundaryCond::GeneralOps::MainQuestion} for 
\(\CoreB\) given by certain boundary conditions; this, in particular, covers
settings like Example \ref{Example::Intro::Examples}\ref{Item::Intro::Examples::NonDirichlet}
(see Example \ref{Example::BoundaryCond::GeneralOps::CoversIntro}).

\begin{definition}
    Fix \(j\in \Zgeq\) and let \(B\) be a partial differential operator on \(\ManifoldN\). We say
    \(B\) has \(W\)-degree \(\leq j\) if
    \begin{equation*}
        B=\sum_{|\alpha|\leq j} b_{\alpha}(x) W^{\alpha},\quad b_\alpha\in \CinftySpace[\ManifoldN].
    \end{equation*}
\end{definition}

\begin{definition}
    Fix \(j\in \Zgeq\). We say \(B\) is a \(W\)-normal partial differential operator of degree \(j\)
    if:
    \begin{itemize}
        \item \(B\) has \(W\)-degree \(\leq j\).
        \item \(\BoundaryN\) is non-characteristic for \(B\) in the classical sense.
            In other words, let \(\Psi:\CubengeqOne\xrightarrow{\sim} \Psi(\CubengeqOne)\subseteq \ManifoldN\)
            be any coordinate chart near \(\BoundaryN\),  we assume 
            \begin{equation*}
                \Psi^{*}B\Psi_{*} =a(x)\partial_{x_n}^j +\text{ terms which are lower order in }\partial_{x_n},
            \end{equation*}
            where \(a(x',0)\ne 0\), \(\forall x'\in \CubenmoOne\).
    \end{itemize}
\end{definition}

\begin{definition}
    For \(J\subseteq \Zgeq\), we say \(\sB=(B_j)_{j\in J}\) is a \(W\)-normal \(J\)-system if every \(B_j\)
    is a \(W\)-normal partial differential operator of degree \(j\).
\end{definition}

\begin{example}\label{Example::BoundaryCond::GeneralOps::CoversIntro}
    Suppose \(X\) is in the \(\CinftySpace[\ManifoldN][\R]\)-module generated by \(W_1,\ldots,W_r\)
    and \(X(x')\not \in \TangentSpace{x'}{\BoundaryN}\), \(\forall x'\in \BoundaryN\).
    Then, \((X^j)_{j\in J}\) is a \(W\)-normal \(J\)-system.
    Using this (with \(\ManifoldN\) replaced by \(\Omega\)), the setting of 
    Example \ref{Example::Intro::Examples}\ref{Item::Intro::Examples::NonDirichlet} can be addressed
    by Proposition \ref{Prop::BoundaryCond::GeneralOps::CharacterizeDomain}.
\end{example}

\begin{proposition}\label{Prop::BoundaryCond::GeneralOps::IntByParts}
    Let \(\sigma\) be a smooth, strictly positive density on \(\BoundaryN\).
    Given a \(W\)-normal \(\left\{ 0,1,\ldots, \kappa-1 \right\}\)-system \(\left\{ M_0,M_1,\ldots, M_{\kappa-1} \right\}\),
    there exists \(N_j\), \(\kappa\leq j\leq 2\kappa-1\) such that each \(N_j\) is a partial differential operator of \(W\)-degree
    \(\leq j\), such that
    \begin{equation*}
        \FormQF[f][g]-\Ltip*{f}{\opL_0 g}[\ManifoldN]
        =\sum_{j=0}^{\kappa-1} \int_{\BoundaryN} \left( M_j f(x') \right) \overline{\left( N_{2\kappa-1-j}g(x')  \right)}\: d\sigma(x'),\quad \forall f,g\in \CinftySpace[\ManifoldN].
    \end{equation*}
    Moreover, if \(\opL_0\) is a \(W\)-normal partial differential operator of degree \(2\kappa\),
    then \(\left\{ N_{\kappa},N_{\kappa+1},\ldots, N_{2\kappa-1} \right\}\) is a \(W\)-normal
    \(\left\{ \kappa, \kappa+1,\ldots, 2\kappa-1 \right\}\)-system.
\end{proposition}
\begin{proof}[Comments on the proof]
    This is a straightforward modification of \cite[Theorem 7.4]{FollandIntroductionPDE};
    see, also, the remark on \cite[page 236]{FollandIntroductionPDE}.
\end{proof}

\begin{remark}
    If \(\FormQF\) is maximally subelliptic as covered by this paper (for example, satisfying
    the assumptions of Section \ref{Section::MainCor}), then \(\opL_0\) is a \(W\)-normal partial differential
    operator of degree \(2\kappa\). Indeed, \(\opL_0\) clearly has \(W\)-degree \(\leq 2\kappa\).
    That \(\BoundaryN\) is non-characteristic for \(\opL_0\) follows from Proposition
    \ref{Prop::EstBdry::Reduction::MainReduction}\ref{Item::EstBdry::Reduction::MainReduction::HighOrderTermPos}.
\end{remark}

\begin{proposition}\label{Prop::BoundaryCond::GeneralOps::ExtensionMap}
    Let \(\left\{ M_0,M_1,\ldots, M_{\kappa-1} \right\}\) be a \(W\)-normal \(\left\{ 0,1,\ldots, \kappa-1 \right\}\)-system.
    Given \(f_0,f_1,\ldots, f_{\kappa-1}\in \CinftycptSpace[\BoundaryN]\), there exists \(w\in \CinftycptSpace[\ManifoldN]\)
    with
    \begin{equation*}
        M_j w\big|_{\BoundaryN}=f_j.
    \end{equation*}
\end{proposition}
\begin{proof}
    One can prove this in local coordinates as in \cite[Proposition 7.7]{FollandIntroductionPDE}.
    It also follows directly from \cite[Theorem \ref*{FS::Thm::Trace::Dirichlet::MainInverseThm}]{StreetFunctionSpacesAndTraceTheoremsForMaximallySubellipticBoundaryValueProblems}.
    In fact, that theorem gives more detailed mapping properties in terms of adapted Besov and Triebel--Lizorkin spaces,
    which is useful in more refined arguments similar to the ones in this section.
\end{proof}

\begin{proposition}\label{Prop::BoundaryCond::GeneralOps::CharacterizeDomain}
    Let \(\left\{ M_0,M_1,\ldots, M_{\kappa-1} \right\}\) be a \(W\)-normal \(\left\{ 0,1,\ldots, \kappa-1 \right\}\)-system.
    Let \(J\subseteq \left\{ 0,1,\ldots, \kappa-1 \right\}\) and \(\JComplement=\left\{ 0,1,\ldots, \kappa-1 \right\}\setminus J\).
    Let
    \begin{equation*}
        \CoreB:=\left\{ f\in \CinftycptSpace[\ManifoldN] : M_j f\big|_{\BoundaryN}=0,\forall j\in J \right\}.
    \end{equation*}
    Let \(\left\{ N_{\kappa},N_{\kappa+1},\ldots, N_{2\kappa-1} \right\}\) be as in Proposition \ref{Prop::BoundaryCond::GeneralOps::IntByParts}.
    Fix \(g\in \CinftycptSpace[\ManifoldN]\). Then,
    \begin{equation}\label{Eqn::BoundaryCond::GeneralOps::CharacterizeDomain::IntByParts}
        \FormQF[f][g]=\Ltip*{f}{\opL_0 g}[\ManifoldN],\quad \forall f\in \CoreB
    \end{equation}
    if and only if
    \begin{equation}\label{Eqn::BoundaryCond::GeneralOps::CharacterizeDomain::BoundaryCond}
        N_{2\kappa-1-j}g\big|_{\BoundaryN}=0,\quad \forall j\in \JComplement.
    \end{equation}
\end{proposition}
\begin{proof}
    \eqref{Eqn::BoundaryCond::GeneralOps::CharacterizeDomain::BoundaryCond}\(\Rightarrow\)\eqref{Eqn::BoundaryCond::GeneralOps::CharacterizeDomain::IntByParts}
    follows immediately from Proposition \ref{Prop::BoundaryCond::GeneralOps::IntByParts}.

    Suppose \eqref{Eqn::BoundaryCond::GeneralOps::CharacterizeDomain::IntByParts} holds.
    By Proposition \ref{Prop::BoundaryCond::GeneralOps::ExtensionMap}, we may
    take \(f\in \CinftycptSpace[\ManifoldN]\) such that
    \begin{equation*}
        M_j f\big|_{\BoundaryN}
        =
        \begin{cases}
            0, & j\in J,\\
            N_{2\kappa-1-j}g\big|_{\BoundaryN}, &j\in \JComplement.
        \end{cases}
    \end{equation*}
    Note that \(f\in \CoreB\) by the definition of \(\CoreB\).
    Since \eqref{Eqn::BoundaryCond::GeneralOps::CharacterizeDomain::IntByParts} holds, by assumption,
    Proposition \ref{Prop::BoundaryCond::GeneralOps::IntByParts} shows
    \begin{equation*}
        \sum_{j\in \JComplement}\int_{\BoundaryN} \left| N_{2\kappa-1-j}g(x') \right|^2\: d\sigma(x')=0.
    \end{equation*}
    \eqref{Eqn::BoundaryCond::GeneralOps::CharacterizeDomain::BoundaryCond} follows.
\end{proof}

\section{Proofs of the results from the introduction}
\label{Section::IntroProofs}
\begin{proof}[Proof of Lemma \ref{Lemma::Intro2::QFIsCloseable}]
    Assumption \ref{Assumption::Intro2::MaxSub} implies \((\FormQF, \CoreB)\)
    is sectorial (see \eqref{Eqn::Result::SectorialDefn}).
    To see that \((\FormQF, \CoreB)\) is closable and its closure
    is sectorial, we will use \cite[Chapter 6, Section 1.4, Theorem 1.17]{KatoPerturbationTheory};
    this will complete the proof, since \(\CoreB\subseteq \DomainQ\)
    and \(\TestFunctionsZero[\ManifoldN]\subseteq \CoreB\), and therefore
    \(\QDomainQ\) is densely defined (since \(\TestFunctionsZero[\ManifoldN]\), and therefore
    \(\DomainQ\), is dense in \(\LtSpaceNoMeasure[\ManifoldN]\)).

    Using \eqref{Eqn::Intro2::DefineFormQF}, we clearly have
    \begin{equation*}
        \left| \FormQF[f][f] \right| + \LtNorm{f}^2\lesssim \HsWNorm{f}{\kappa}^2,\quad f\in \CoreB.
    \end{equation*}
    Combining this with Assumption \ref{Assumption::Intro2::MaxSub}, we see
    \begin{equation*}
    \begin{split}
         &\left| \FormQF[f][f] \right|+\LtNorm{f}^2\lesssim \HsWNorm{f}{\kappa}^2
         \leq C_1 \Real \FormQF[f][f] + C_2 \LtNorm{f}^2
         \lesssim \left| \FormQF[f][f] \right| + \LtNorm{f}^2,\quad f\in \CoreB.
    \end{split}
    \end{equation*}
    This implies
    \begin{equation}\label{Eqn::PfIntro::QFIsCloseable::EquivNorms}
        \left| \FormQF[f][f] \right|+\LtNorm{f}^2\approx \HsWNorm{f}{\kappa}^2, \quad f\in \CoreB.
    \end{equation}

    Suppose \(f_j\in \CoreB\) is a sequence with \(\LtNorm{f_j}\rightarrow 0\)
    and 
    \begin{equation*}
        \lim_{j,k\rightarrow\infty} \FormQF[f_j-f_k][f_j-f_k]=0.
    \end{equation*}
    Then, by \eqref{Eqn::PfIntro::QFIsCloseable::EquivNorms}, \(f_j\)
    is a Cauchy sequence in \(\HsWSpace{\kappa}[\ManifoldN]\), and is therefore convergent
    in \(\HsWSpace{\kappa}[\ManifoldN]\).  Since \(f_j\rightarrow 0\) in \(\LtSpace[\ManifoldN]\),
    we conclude \(f_j\rightarrow 0\) in \(\HsWSpace{\kappa}[\ManifoldN]\).
    Using \eqref{Eqn::PfIntro::QFIsCloseable::EquivNorms} we see
    \(\left| \FormQF[f_j][f_j] \right|\rightarrow 0\).
    From here, \cite[Chapter 6, Section 1.4, Theorem 1.17]{KatoPerturbationTheory} applies
    to show \((\FormQF, \CoreB)\) is closable and its closure is sectorial. 
\end{proof}

\begin{proof}[Proof of Corollaries \ref{Cor::Intro::dtPlusOplSubellipic}--\ref{Cor::Intro::EigenVectorsSmooth}]
    We apply Corollary \ref{Cor::Core::MainCor} with
    \(\Omegat=\ManifoldN\), \(\Omega=\Omega\), \(M=1\), and \(\psi=1\in \CinftycptSpace[\Omegat]\)
    (here we are using that \(\ManifoldN\) is compact).
    The other objects (\(\FormQF\), \(\CoreB\), \(X\), and \(\QDomainQ\))
    have the same name in Section \ref{Section::Intro} as they do in the application
    of Corollary \ref{Cor::Core::MainCor}.

    Assumptions \ref{Assumption::Core::ExistsPsi},
    \ref{Assumption::Core::TestFunctionsInCore},
    \ref{Assumption::Core::LocalizationOfCore},
    \ref{Assumption::Core::QEqualsQF},
    and \ref{Assumption::Core::MaxSubAssumption}
    follow immediately from the definitions.
    All that remains is to verify Assumption \ref{Asssumption::Core::Smoothing}.
    Fix \(x_0\in \Omega\cap \BoundaryN\) for which we wish to verify
    Assumption \ref{Asssumption::Core::Smoothing} and let \(S\)
    be as in \eqref{Eqn::MainCor::FormOfS}.

    Pick any coordinate chart \(\Psi:\CubengeqOne\xrightarrow{\sim} \Psi(\CubengeqOne)\)
    with \(\Psi(0)=x_0\) and \(\Psi^{*}X=c_0\partial_{x_n}\) for some \textit{constant} \(c_0\ne 0\);
    by the choice of \(\Omega\) and \(x_0\) it is clear such a \(\Psi\) exists.
    Set \(T=\Psi_{*} S\Psi^{*}\) as in Assumption \ref{Asssumption::Core::Smoothing}\ref{Item::Core::Smoothing::SOperator}.
    Note that
    \begin{equation*}
        S:\CinftySpace[\CubengeqOne]\rightarrow \CinftycptSpace[\CubengeqOne],
    \end{equation*}
    \begin{equation*}
        g\big|_{\{x_n=0\}}=0\implies Sg\big|_{\{x_n=0\}}=0, \quad \forall g\in \CinftySpace[\CubengeqOne],
    \end{equation*}
        \begin{equation*}
        \partial_{x_n}^j [\partial_{x_n},S]g\big|_{\{x_n=0\}}=0,\quad \forall g\in \CinftySpace[\CubengeqOne],\:\forall j\in \Zgeq.
    \end{equation*}
    and therefore,
    \begin{equation}\label{Eqn::PfIntro::MainSubCor::TPerservesSmooth}
        T:\CinftySpace[\ManifoldN]\rightarrow \CinftySpace[\ManifoldN],
    \end{equation}
    \begin{equation}\label{Eqn::PfIntro::MainSubCor::TPerservesBoundaryVanish}
        f\big|_{\BoundaryN}=0\implies Tf\big|_{\BoundaryN}=0,\quad \forall f\in \CinftySpace[\ManifoldN],
    \end{equation}
    \begin{equation}\label{Eqn::PfIntro::MainSubCor::TCommutesWithX}
        X^j \left[ X,T \right]f\big|_{\BoundaryN}=0,\quad \forall f\in \CinftySpace[\ManifoldN],\: \forall j\in \Zgeq.
    \end{equation}

    Let \(j\in \JSet\).  Then,
    using \eqref{Eqn::PfIntro::MainSubCor::TPerservesBoundaryVanish} and \eqref{Eqn::PfIntro::MainSubCor::TCommutesWithX},
    we have
    \begin{equation*}
    \begin{split}
         &X^j Tf\big|_{\BoundaryN} = TX^j f\big|_{\BoundaryN}=0, \quad \forall f\in \CoreB.
    \end{split}
    \end{equation*}
    Combining this with \eqref{Eqn::PfIntro::MainSubCor::TPerservesSmooth},
    we conclude \(T:\CoreB\rightarrow \CoreB\).
     Assumption \ref{Asssumption::Core::Smoothing} follows.

    Applying Corollary \ref{Cor::Core::MainCor}
    shows that the 
    conclusions of the corollaries in Section \ref{Section::MainResult}
    hold.  Corollaries \ref{Cor::Intro::dtPlusOplSubellipic}--\ref{Cor::Intro::EigenVectorsSmooth}
    have the same conclusions as corollaries in Section \ref{Section::MainResult}.
\end{proof}

\section{Elliptic Regularization}\label{Section::EllipticRegularization}
\cite{KohnNirenbergNonCoerciveBoundaryValueProblems} introduced an elliptic regularization
procedure which has since been used in many papers. We do not use it in this paper;
however, in this section we show how to incorporate this regularization into our framework.

For convenience we work on a bounded smooth domain \(\EllipOmega\subset \Rn\); though this is not essential.
Let \(\QDomainQ\) be a closed, densely defined, sectorial form on \(\LtSpace[\EllipOmega]\). Fix \(\kappa\in \Zg\).
We assume throughout this section the following assumption.
\begin{assumption}\label{Assumption::EllipticReg::SobolevIsCore}
    We assume \(\DSet:=\HsSpace{\kappa}[\EllipOmega]\cap \DomainQ\) is a core for \(\QDomainQ\).
\end{assumption}

Set
\begin{equation*}
    \Formt[f][g]:=\sum_{j=1}^n \Ltip*{\partial_{x_j}^\kappa f}{\partial_{x_j}^\kappa g}[\EllipOmega].
\end{equation*}
It is easy to see \(\tDomaint\) is a closed, densely defined, sectorial form.
Set, for \(\epsilon>0\),
\begin{equation*}
    \FormQepsilon[f][g]:=\FormQ[f][g]+\epsilon \Formt[f][g],\quad f,g\in \DSet.
\end{equation*}
It follows from \cite[Chapter 6, Section 1, Theorem 1.31]{KatoPerturbationTheory}
that \(\QepsilonDomainQepsilon\) is a closed, densely defined, sectorial form.
Let \(\gamma\in \R\) be a vertex for \(\FormQ\);
i.e., \(\gamma\) satisfies \eqref{Eqn::Result::SectorialDefn} (see \cite[Chapter 8, Section 1.1]{KatoPerturbationTheory}).
Note that \(\gamma\) is a vertex for \(\FormQepsilon\), \(\forall \epsilon>0\).

Let \(\LDomainL\) and \(\LepsilonDomainLepsilon\) be the m-sectorial operators
associated to \(\QDomainQ\) and \(\QepsilonDomainQepsilon\) by
\cite[Chapter 6, Section 2, Theorem 2.1]{KatoPerturbationTheory}.


\begin{example}\label{Example::EllipticReg::IntroExample}
    A main example to keep in mind is \(\QDomainQ\) as given in Section \ref{Section::Intro} (where we take \(\ManifoldN=\overline{\EllipOmega}\)).
    There, \(\CoreB\subseteq \CinftySpace[\overline{\EllipOmega}]\) and therefore 
    \(\CoreB=\CoreB\cap \HsSpace{\kappa}[\EllipOmega]\subseteq \DomainQ \cap \HsSpace{\kappa}[\EllipOmega]\)
    is a core for \(\QDomainQ\) (by assumption) and Assumption \ref{Assumption::EllipticReg::SobolevIsCore} holds.
    Replacing \(\FormQ\) with \(\FormQepsilon\) has the effect of adding \(\epsilon \sum_{j=1}^n (-1)^\kappa \partial_{x_j}^{2\kappa}\)
    to \(\opL_0\).  \(\FormQepsilon\) is coercive in the sense of \cite[Section 7.D]{FollandIntroductionPDE}
    and therefore one has a wide range of results which apply to each \(\FormQepsilon\) (though not uniformly in \(\epsilon>0\)).
    The main point is that the methods of this paper apply uniformly in \(\epsilon>0\), giving estimates 
    on \(\opL_\epsilon\)
    which are uniform in \(\epsilon>0\).
    One can then take \(\epsilon\downarrow 0\) to conclude estimates about \(\opL\).
\end{example}

As described in Example \ref{Example::EllipticReg::IntroExample}, the goal is to take \(\epsilon\downarrow 0\).
The remainder of this section is devoted to explaining that this limit behaves quite well.

\begin{proposition}\label{Prop::EllipticReg::GeneralizedStrongConverg}
    \(\opL_\epsilon\xrightarrow{\epsilon\downarrow 0}\opL\) strongly in the generalized sense
    (see \cite[Chapter 8, Section 1.1]{KatoPerturbationTheory}).
    In particular, 
    for \(\Real \zeta <\gamma\) we have
    \(\left( \opL_{\epsilon}-\zeta \right)^{-1}\xrightarrow{\epsilon\downarrow 0}\left( \opL-\zeta \right)^{-1}\)
    in the strong operator topology, where \(\left( \opL_{\epsilon}-\zeta \right)^{-1}\)
    and \(\left( \opL-\zeta \right)^{-1}\) denote the resolvents.
\end{proposition}
\begin{proof}
    To prove the proposition, we will apply \cite[Chapter 8, Section 3, Theorem 3.6]{KatoPerturbationTheory}.
    \(\Domain{\FormQepsilon}=\DSet\subseteq \DomainQ\), \(\forall \epsilon>0\). We have,
    \begin{equation*}
        \left| \Imaginary \left( \FormQepsilon[f][f]-\FormQ[f][f] \right) \right|
        =0\leq 
        \epsilon \Formt[f][f]=
        \Real\left( \FormQepsilon[f][f]-\FormQ[f][f] \right),\quad \forall f\in \DSet.
    \end{equation*}
    Clearly, \(\forall f\in \DSet\),
    \(\lim_{\epsilon\downarrow 0} \FormQepsilon[f][f]=\FormQ[f][f]\),
    and by Assumption \ref{Assumption::EllipticReg::SobolevIsCore},
    \(\DSet\) is a core for \(\QDomainQ\).
    From here, the result follows from \cite[Chapter 8, Section 3, Theorem 3.6]{KatoPerturbationTheory}
    (applied to any sequence \(\epsilon_m\downarrow 0\)).
\end{proof}

As described in Section \ref{Section::Intro::Quant}, a main motivation of this paper is to help
study heat operators. We turn to discussing the limit \(\epsilon\downarrow 0\) as applied to the corresponding heat operators.
We begin with a lemma.

\begin{lemma}\label{Lemma::EllipticReg::GenerateSemiGroups}
    \((-\opL, \DomainL)\) and \((-\opL_\epsilon,\DomainLepsilon)\) are generators for strongly continuous semigroups
    satisfying
    \begin{equation*}
        \LtOpNorm*{e^{-t\opL}}, \LtOpNorm*{e^{-t\opL_{\epsilon}}}\leq e^{-\gamma t},\quad \forall t\geq 0.
    \end{equation*}
    Recall, \(\gamma\) may be negative.
\end{lemma}
\begin{proof}
    \((\opL, \DomainL)\) and \((\opL_\epsilon,\DomainLepsilon)\) are m-sectorial operators with vertex \(\gamma\).
    Thus, by definition (see \cite[Chapter 5, Section 3.10]{KatoPerturbationTheory})
    we have \((\opL-\gamma, \DomainL)\) and \((\opL_\epsilon-\gamma,\DomainLepsilon)\) are m-accretive.
    From here, the result follows from the Lumer--Phillips Theorem
    \cite[Chapter 2, Theorem 3.15]{EngelNagelAShortCourseOnOperatorSemigroups}

    Alternatively, by replacing \(\FormQ[f][g]\) and \(\FormQepsilon[f][g]\) with \(\FormQ[f][g]-\gamma\Ltip{f}{g}\) and
    \(\FormQepsilon[f][g]-\gamma\Ltip{f}{g}\), respectively, we may assume \(\gamma=0\).
    From here, the result follows from \cite[Chapter 9, Section 1, Theorem 1.24]{KatoPerturbationTheory}.
    In fact, we conclude that \(e^{-t\opL}\) and \(e^{-t\opL_{\epsilon}}\) are holomorphic semigroups,
    though we do not use this fact.

\end{proof}


\begin{proposition}\label{Prop::EllipticReg::ConvergeHeat}
    \(e^{-t\opL_{\epsilon}}\xrightarrow{\epsilon\downarrow 0}e^{-t\opL}\) in the strong operator topology,
    uniformly on compact subsets of \(t\in [0,\infty)\).
\end{proposition}
\begin{proof}
    Using 
    Proposition \ref{Prop::EllipticReg::GeneralizedStrongConverg}
    and Lemma \ref{Lemma::EllipticReg::GenerateSemiGroups} this follows from
    the  First Trotter--Kato Approximation Theorem
    (see \cite[Chapter 3, Theorem 4.8]{EngelNagelOneParameterSemigroupsForLinearEvolutionEquations}).
\end{proof}

\begin{remark}
    In the setting of Example \ref{Example::EllipticReg::IntroExample},
    and with \(\Omega\subseteq \overline{\Gamma}=\ManifoldN\) as in Section \ref{Section::Intro},
    the results of this paper imply\footnote{This implication is standard.  See 
    \cite[Proposition \ref*{Heat::Prop::SubProof::QualSubellip::ExistsHeatKernel}]{StreetHypoellipticityAndHigherOrderGaussianBounds} for an exposition
    on manifolds without boundary. The proof can be easily adapted to manifolds with boundary.}
     that there is a unique smooth
    function \(K(t,x,y)\in \CinftySpace*[(0,\infty)\times \Omega\times \Omega][\MatrixSpace[M][M]]\) such that
    \begin{equation*}
        e^{-t\opL} f(x) = \int_{\Omega} K(t,x,y) f(y)\: d\Vol(y),\quad \forall f\in \LtSpace[\Omega][\Vol][\C^M],\:\forall x\in \Omega.
    \end{equation*}
    Similarly, there are unique smooth functions \(K_t^{\epsilon}(x,y)\in \CinftySpace*[(0,\infty)\times \Omega\times \Omega][\MatrixSpace[M][M]]\) corresponding to \(e^{-t\opL_\epsilon}\), \(\epsilon>0\).
    Since the results of this paper hold uniformly for \(\epsilon\in (0,1]\),
    it follows that \(\left\{ K_t^{\epsilon} \right\}\subset \CinftySpace*[(0,\infty)\times \Omega\times \Omega][\MatrixSpace[M][M]]\)
    is a bounded set. Combining this with Proposition \ref{Prop::EllipticReg::ConvergeHeat}
    shows \(K_t^{\epsilon}\rightarrow K_t\) in \(\CinftySpace*[(0,\infty)\times \Omega\times \Omega][\MatrixSpace[M][M]]\).
\end{remark}

\bibliographystyle{amsalpha}

\bibliography{bibliography}

\center{\it{University of Wisconsin-Madison, Department of Mathematics, 480 Lincoln Dr., Madison, WI, 53706}}

\center{\it{street@math.wisc.edu}}

\center{MSC 2020:  35H20, 35G15, 35B45}

\end{document}